\documentclass[reqno,11pt]{amsart}
\usepackage{amsmath,amssymb,mathrsfs,color,amsthm,mathtools}
\usepackage{esint}
\usepackage{cases}
\usepackage{graphicx}
\usepackage{epstopdf}
\usepackage{soul}
\usepackage{tikz}
\usetikzlibrary{arrows}
\usepackage[margin=1in]{geometry}
\usepackage{
indentfirst, latexsym, bm, enumerate,
wrapfig, adjustbox, subfigure, appendix,
float, bbding, xcolor, multirow, framed, lipsum, epsdice,
url, booktabs, makecell, caption, algpseudocode, pifont,
balance, dcolumn, fontenc, extarrows, arydshln
}

\newtheorem{definition}{Definition}[section]
\newtheorem{theorem}{Theorem}[section]
\newtheorem{lemma}{Lemma}[section]
\newtheorem{proposition}{Proposition}[section]
\newtheorem{corollary}{Corollary}[section]

\newtheorem{remark}{\bf{Remark}}[section]

\makeatletter
\newcommand{\rmnum}[1]{\romannumeral #1}
\newcommand{\Rmnum}[1]{\expandafter\@slowromancap\romannumeral #1@}
\makeatother

\numberwithin{equation}{section}

\usepackage{hyperref}
\usepackage{bookmark}

\hypersetup{
	colorlinks = true,
	linkcolor = blue,
	citecolor = blue,
	urlcolor = blue,
    pdfpagemode = UseOutlines,      
    pdfborder = {0 0 0}             
}

\bookmarksetup{depth=4}

\renewcommand{\hl}[1]{#1}

\begin{document}

\title[A stochastic convex integration scheme for intermittent Onsager theorem]{A stochastic convex integration scheme for the intermittent Onsager theorem in two-dimensional domains}

\author{Zhenxin Liu}
\address{Z. Liu: School of Mathematical Sciences,
Dalian University of Technology, Dalian 116024, P. R. China}
\email{zxliu@dlut.edu.cn}

\author{Wenhu Zhong}
\address{W. Zhong (Corresponding author): School of Mathematical Sciences,
Dalian University of Technology, Dalian 116024, P. R. China}
\email{zhongwenhu@dlut.edu.cn}

\date{September 14, 2026}

\begin{abstract}
For any $\gamma \in \left[ {0,\frac{1}{3}} \right)$, we construct $\gamma$-H\"{o}lder-continuous martingale solutions to the two-dimensional stochastic incompressible Euler equations. These martingale solutions exhibit dissipative behavior and intermittency, thereby establishing the flexible part of the intermittent Onsager theorem. The proof rests on a stochastic and intermittent variant of the Newton--Nash iteration combined with the Littlewood--Paley decomposition scheme. This composite iteration scheme incorporates new stochastic pressure, Reynolds stress and intermittent perturbations to formalize the stochastic and intermittent fluctuations.
\end{abstract}

\maketitle
\tableofcontents  

\section{Introduction}\label{Intro}
The present paper investigates the two-dimensional stochastic incompressible Euler equations
\begin{equation}
\left\{ {\begin{array}{*{20}{l}}
\partial _t {u} + \left( {{u} \cdot \nabla } \right){u} +  \nabla \mathfrak{p} =  \mathfrak{\dot W} , &\mbox{on}\ \left[ {0, T } \right]  \times  \mathbb{T}^2, \\
\nabla  \cdot {u} = 0,
\end{array}} \right.  \label{2DEEUGE1}
\end{equation}
where $\mathbb{T}^2 \triangleq {\mathbb{R}^2} / 2\pi {\mathbb{Z}^2}$ is the two-dimensional torus, and $T$ is an arbitrarily given positive real number. The velocity ${u}\left( {t,x} \right)$ is an $\mathbb{R}^2$-valued function on $\left[ {0, T } \right] \times \mathbb{T}^2$, satisfying the zero-mean condition $\int_{\mathbb{T}^2} {{u}\left( {t,x} \right)}\ \mathrm{d}x = \mathbf{0}$ for any $t \in \left[ {0, T } \right]$. The pressure $\mathfrak{p} \left( {t , x} \right)$ is an $\mathbb{R}$-valued function on $ \left[ {0, T } \right] \times \mathbb{T}^2$. The Wiener process $\mathfrak{ W} \left( {t} \right)$ belongs to the collection $\mathscr{W}$ specified in Appendix B, and its formal time derivative $ {\mathfrak{\dot W}} \left( t \right)$ is interpreted as temporal white noise. We denote the kinetic energy by $\mathcal{E}\left( t \right) \triangleq \left\| {u\left( t \right)} \right\|_{{L^2}}^2 $.
\par
Turbulence is characterized by anomalous dissipation of energy and intermittency that have motivated extensive mathematical investigations of weak solutions to the incompressible Euler equations (see Cheskidov and Shvydkoy \cite{Zbl1512.76040}, Dascaliuc and Gruji\'{c} \cite{Zbl1235.76049}, as well as related works \cite{MR4645737,Zbl1475.76019,Zbl1536.35261,Zbl1508.35069,Zbl1553.35159,Zbl08114326}). A fundamental principle in this field originates from Onsager's pioneering work \cite{MR0036116}: the Euler equations exhibit anomalous dissipation when the H\"{o}lder exponent of solutions is less than 1/3. The threshold condition reveals the intrinsic mechanism of anomalous dissipation and lays a mathematical foundation for the celebrated K41 theory of turbulence \cite{Kol1941a,Kol1941b,Kol1941c,Kol1941d}. The mathematical resolution of Onsager critical solutions has led to a profound development of convex integration methods (see Buckmaster and Vicol \cite{Zbl1461.35186}). Most convex integration constructions of dissipative and intermittent solutions are based on the critical regularity criterion.
\par
Building on the $L^{\infty }$-convex integration method introduced by Shnirelman \cite{ShnirelmanA-1,ShnirelmanA-2}, the dissipative solution to the Euler equations was constructed by De Lellis and Sz\'{e}kelyhidi \cite{Zbl1350.35146,Zbl1280.35103,Zbl1307.35205}, as well as Daneri and Sz\'{e}kelyhidi \cite{Zbl1372.35221}. Their construction was further refined by Buckmaster and his collaborators \cite{MR3302631,MR3374958,Zbl1480.35317}, as well as Drivas and Nguyen \cite{Zbl1401.76068}. These groundbreaking works triggered the resolution of the flexible part of the Onsager theorem by Isett \cite{Zbl1416.35194,Zbl1547.35524}, who adopted the Nash iteration scheme and the Littlewood--Paley decomposition to construct a non-trivial H\"{o}lder-continuous solution for the three-dimensional deterministic Euler equations. Recently, combining the Newton iteration scheme and the Nash iteration scheme, Giri and Radu \cite{Zbl1556.35231} established the Onsager theorem for the two-dimensional deterministic Euler equations. Moreover, we refer to \cite{Zbl1520.35112,Zbl1508.76067,Zbl1412.35215,arXiv:2305.18509,arXiv:2502.04803,Zbl07890701} for more mathematical results concerning the Onsager theorem of other deterministic fluid equations.
\par
However, Cheskidov and Shvydkoy \cite{Zbl1295.76010} emphasized that intermittent fluctuations are not adequately described by statistically self-similar scaling laws predicted by the K41 theory of turbulence (also see \cite{MR3943488,Zbl07201254,Zbl07892740}). Consequently, intermittency has become a fundamental component in the mathematical theory of turbulence,  manifested through the spatial concentration of energetic structures and anomalous scaling of fluctuations. Employing the intermittent Beltrami wave introduced by Buckmaster and Vicol \cite{Zbl1412.35215}, the flexible part of the Onsager theorem was reconfirmed by Novack and Vicol \cite{MR4601999}, who constructed intermittent H\"{o}lder-continuous solutions to the three-dimensional deterministic Euler equations. This intermittent convex integration construction introduced a framework for capturing turbulent small-scale structures. Later, Rosa, Drivas, Inversi and Isett \cite{arXiv:2502.10032} established a close link between dissipative behavior and intermittency.
\par
Despite these advances, most existing convex integration theory has been developed predominantly in deterministic settings. However, environmental fluctuations, unresolved degrees of freedom, and external random forcing naturally introduce stochastic components into fluid models. Biferale, Cencini, Pierotti and Vulpiani \cite{Zbl0939.76040} clarified that stochastic effects dynamically amplify intermittent oscillations. In fact, the regularity of driving noise affects the critical regularity threshold of solutions (see \cite{Zbl1426.35183}). Meanwhile, the interplay between nonlinearity and stochastic forcing spontaneously gives rise to non-Gaussian heavy tails (see \cite{Falkovichsfa}), which constitute the mathematical signature of intermittency. These observations reveal a fundamental mathematical obstruction: the perturbations in the stochastic setting must simultaneously incorporate random temporal oscillations, stochastic transport effects, and intermittent concentrations. A stochastic convex integration theory capable of constructing intermittent Onsager-critical solutions remains undeveloped. Accordingly, establishing the intermittent Onsager theorem for the two-dimensional stochastic Euler equations \eqref{2DEEUGE1} remains a fundamental open problem at the crossroads of convex integration theory, stochastic analysis and turbulent phenomenology.
\par
The present paper develops a systematic stochastic convex integration scheme to establish the intermittent Onsager theorem for the two-dimensional stochastic Euler equations \eqref{2DEEUGE1}. More precisely, for any $\gamma \in \left[ {0,\frac{1}{3}} \right)$, we prove the existence of $\gamma$-H\"{o}lder-continuous martingale solutions that dissipate kinetic energy and exhibit intermittency. The main results are presented as follows.
\par
\begin{theorem}\label{CTSE-O-C}
For any $\gamma  \in \left[ {0,\frac{1}{3}} \right)$, let the Wiener process ${\mathfrak{W}}\left( {t} \right) \in \mathscr{W}$ with covariance operator $Q$ satisfy $\mathrm{Tr} \, Q + \mathrm{Tr} \left( {{Q^{\frac{1}{2}}}{{\left( { - \Delta } \right)}^{\frac{1 + \gamma }{2} + \rho}}{{\left( {{Q^{\frac{1}{2}}}} \right)}^*}} \right) <  \infty$ for some $0 < \rho \ll 1$. Then, there exists a stopping time $\mathfrak{t}$ satisfying $\mathbb{P}\left( {0 < {\mathfrak{t}} \leqslant T} \right) = 1 $ such that Eq.\eqref{2DEEUGE1} admits a martingale solution 
\begin{equation*}
{u }\left( t ,x \right) \in C\left( {\left[ {0,T} \right];{C^\gamma }\left( {{\mathbb{T}^2};{\mathbb{R}^2}} \right)} \right) ,\ \ {\mathrm{supp}_\mathrm{t}} \, u \subset \left[ {0,\mathfrak{t}} \right]   ,\  \mathbb{P}\mbox{-a.s.} ,
\end{equation*}
satisfying the strict energy inequality
\begin{equation}\label{SHUDJ-NENH-B-SIDIS}
\mathcal{E}\left( t \right) <  \mathcal{E}\left( \tau \right) + 2 \int_s^t {{{{\left\langle { {u\left( \tau \right),\mathrm{d}\mathfrak{ W}\left( \tau \right)} } \right\rangle }_{{\mathbb{T}^2}}}}}  + \left( {t - s} \right) \cdot \mathrm{Tr}\, Q , \ \ \forall \, 0 \leqslant s < t \leqslant \mathfrak{t} ,\  \mathbb{P}\mbox{-a.s.}
\end{equation}
\end{theorem}
\par
\begin{remark}\label{OKDIJF-RE}
\rm{For the two-dimensional stochastic Euler equations \eqref{2DEEUGE1}, a straightforward application of It\^{o}'s formula shows that general weak solutions satisfy the kinetic energy balance
\begin{equation}\label{NHUY-SDF-W-1}
\mathbb{E}\mathcal{E}\left( t \right) = \mathbb{E}\mathcal{E}\left( 0 \right) + t \cdot \mathrm{Tr}\, Q ,\ \ \forall \, t \in \left[ {0,T} \right] ,
\end{equation}
where the symbol $\mathbb{E}$ denotes the expectation operator. Notice that the kinetic energy balance \eqref{NHUY-SDF-W-1} precludes conservation of kinetic energy. Consequently, in the stochastic version of the flexible part of the Onsager theorem, the assertion of non-conservation is replaced by the dissipative behavior \eqref{SHUDJ-NENH-B-SIDIS}. In fact, the dissipative behavior \eqref{SHUDJ-NENH-B-SIDIS} implies that the total energy $\mathbb{E}\mathcal{E}\left( t \right) - t \cdot \mathrm{Tr}\, Q$ is not conserved. A similar discussion was provided by L\"{u}, L\"{u} and Zhu \cite{arXiv:2505.06915}, who examined the Onsager theorem for the three-dimensional stochastic Euler equations. In this paper, our constructed martingale solutions further exhibit intermittency and inertial-range intermittency.}
\end{remark}
\par
\begin{theorem}\label{CTSE-O-C-1B}
Let ${u }\left( t ,x \right)$ be the $\gamma$-H\"{o}lder-continuous martingale solution to Eq.\eqref{2DEEUGE1} obtained in Theorem \ref{CTSE-O-C}. Suppose that $M$ is in $\left( {\frac{1}{2},\frac{b}{2}} \right)$, and the exponents $p_1$ and $p_2$ admit ${p_1} > {p_2} > \frac{{2M}}{{\gamma M + M - \beta  + \alpha }}$, where the basic parameters $\alpha$ and $\beta$ are specified in Appendix C. Then, for any $t \in {\mathrm{supp}_\mathrm{t}} \, u  \setminus \left\{ 0 \right\}$, the $\gamma$-H\"{o}lder-continuous martingale solution ${u }\left( t ,x \right)$ exhibits frequency-space intermittency
\begin{equation}\label{Thes-sdgh-ASSDFG-1}
\frac{{{{{\left\| {\sum\limits_{j \geqslant q + 1}{{\dot \Delta }_{{R_j}}}u\left( t \right)} \right\|}_{{L^{{p_1}}}}}} }}{{ {{{\left\| {\sum\limits_{j \geqslant q + 1}{{\dot \Delta }_{{R_j}}}u\left( t \right)} \right\|}_{{L^{{p_2}}}}}} }} \gtrsim {r_q^{2\left( {\frac{1}{{{p_2}}} - \frac{1}{{{p_1}}}} \right)}} \to \infty  , \ \ \mbox{as} \ q \to  \infty , \  \mathbb{P}\mbox{-a.s.} ,
\end{equation}
where $r_q = \lambda_q^M$ is an intermittency parameter of the Dirichlet kernel (see \eqref{XINHUI-SJI-A} in Appendix D), and ${{{\dot \Delta }_{{R_j}}}}$ with $R_j = \lambda_j$ is the homogeneous Littlewood--Paley projector. Furthermore, there exists an inertial range $\left[ {{l_D},{l_I}} \right]$ with the dissipative scale $ {l_D} \left( {q} \right)$ and the integral scale ${l_I}\left( {q} \right)$ satisfying $\lambda _{q + 1}^{ - \frac{1}{2} - M}  \lesssim  {l_D} \left( {q} \right) \ll {l_I}\left( {q} \right) \to 0$ as $q \to \infty$ such that the $\gamma$-H\"{o}lder-continuous martingale solution ${u }\left( t ,x \right)$ exhibits inertial-range intermittency
\begin{equation}\label{Thes-sdgh-ASSDFG-2}
\frac{{{{\left\| {u\left( { t,x + l\mathbf{n}} \right) - u\left( t,x \right)} \right\|}_{{L^{{p_1}}}}}}}{{{{\left\| {u\left( {t,x + l\mathbf{n}} \right) - u\left( t,x \right)} \right\|}_{{L^{{p_2}}}}}}}  \gtrsim {r_q^{2\left( {\frac{1}{{{p_2}}} - \frac{1}{{{p_1}}}} \right)}} \to \infty  , \ \ \mbox{as} \ q \to  \infty ,  \ \forall \, l \in \left[ {{l_D},{l_I}} \right] , \  \mathbb{P}\mbox{-a.s.},
\end{equation}
where $\mathbf{n}$ is the unit spatial direction vector on the two-dimensional torus $\mathbb{T}^2$.
\end{theorem}
\par
\begin{remark}\label{OKDIJF-RE-INIJU-A-1}
\rm{The homogeneous Littlewood--Paley projector ${{{\dot \Delta }_{{R_j}}}}$ extracts the velocity components localized at different frequency scales, thereby revealing the amplitude distribution and spatial concentration of velocity fields. The two intermittency estimates \eqref{Thes-sdgh-ASSDFG-1} and \eqref{Thes-sdgh-ASSDFG-2} show that the $L^{p_1}$-to-$L^{p_2}$ ratio grows without bound as the Littlewood--Paley frequency scale increases and the spatial scale decreases, respectively. These estimates indicate increasingly concentrated small-scale structures consistent with the intermittency framework described by Frisch \cite[Section $8.2$]{TurbulenceFrisch}.}
\end{remark}
\par
To establish Theorem \ref{CTSE-O-C} and Theorem \ref{CTSE-O-C-1B}, we devise an intermittency-compatible stochastic variant of the Newton--Nash iteration coupled with a Littlewood--Paley decomposition scheme, which captures the regularity loss and intermittency of approximate solutions during the iterative construction. The Newton iteration layer introduces fast temporal oscillations to ensure that Nash perturbations at different scales do not interact with each other. The Nash iteration procedure progressively reduces iterative errors to approximate H\"{o}lder-continuous solutions. The Littlewood--Paley projector isolates high-frequency components of the dominant intermittent perturbation, whose estimates control the high-frequency parts of all remaining perturbations. Our stochastic and intermittent Newton--Nash iteration scheme is not merely a stochastic adaptation of the classical convex integration. Instead, it provides a unified framework simultaneously compatible with stochastic oscillations and intermittency, one that systematically controls stochastic perturbations and intermittent concentrations to construct martingale solutions whose H\"{o}lder exponent lies strictly beneath the Onsager threshold $1/3$.
\par
Three fundamental challenges arise in the stochastic and intermittent Newton--Nash iteration scheme. First, the nonlinear term accumulates stochastic oscillations over successive iterations, thereby leading to the total Reynolds error stemming from stochastic effects that cannot be directly controlled by the classical Newton--Nash iteration scheme. To separate out these stochastic oscillations, we adopt the decomposition ${u}\left( {t,x} \right) = {v}\left( {t,x} \right) + {\mathfrak{W}}\left( {t,x} \right) $ to establish the stochastic Euler--Reynolds system
\begin{equation*}
\left\{ {\begin{array}{*{20}{l}}
{\partial _t}{v_q} + \left( {{u_q} \cdot \nabla } \right){u_q} + \nabla {\mathfrak{p}_q} = \mathrm{div}\,   {\mathring{R}_q},\\
{u_q} = {v_q} + \mathfrak{ W}_q ,\\
\mathrm{div}\,   {u_q} = 0 .
\end{array}} \right.  
\end{equation*}
This decomposition reveals a compensation mechanism: instead of treating stochastic oscillations as merely perturbative errors, we incorporate them directly into the nonlinear transport structure and construct a stochastic pressure ${\mathfrak{p}_q}$ to compensate for the accumulated oscillatory errors. Furthermore, the energy gap in the iteration is coupled with the random forcing and contains additional correction terms induced by the It\^{o} integrals.
\par
Second, classical Beltrami waves fail to be compatible with the stochastic Newton--Nash iteration scheme and do not exist in the two-dimensional setting. To circumvent this limitation, we construct the two-dimensional analogues of intermittent Beltrami waves inspired by the construction of Buckmaster and Vicol \cite{Zbl1412.35215} (see Appendix D for the construction of two-dimensional Beltrami analogues). Nevertheless, our modified construction leads to an additional stress residual whose divergence lacks commensurate-order small estimates. Alternatively, we incorporate these intermittent Beltrami analogues into the iterative procedure through random backward flows and directly derive inductive estimates to control their intermittency intensity. Since the Beltrami analogues give rise to intermittent perturbations, the Reynolds stress ${\mathring{R}_q}$ carries both random and intermittent oscillations. We bound these oscillations through inductive estimates of Newton--Nash elements and perturbations, thereby closing the recursive loop of the stochastic and intermittent Newton--Nash iteration scheme. 
\par
Third, since the constructed martingale solution $u(t,x)$ takes the form of a superposition of infinitely many perturbation blocks, its intermittency is encoded in the multiscale interaction among these blocks. To overcome this issue, we use the homogeneous Littlewood--Paley projector $\dot\Delta_{R_j}$ to isolate the individual intermittent components of $u(t,x)$, thereby making the high-frequency structure accessible. However, the spatial oscillations of principal perturbation are transported by the Lagrangian flow, so that the standard Littlewood--Paley projector no longer directly reflects the intrinsic frequency localization of the oscillatory components. We further introduce a Lagrangian Littlewood--Paley localization $\dot \Delta _{{R_{j}}}^{Lag}$, which allows us to separate the dominant intermittent component from the deformation and leakage effects. Consequently, both the frequency-space intermittency \eqref{Thes-sdgh-ASSDFG-1} and the inertial-range intermittency \eqref{Thes-sdgh-ASSDFG-2} are carried  by the dominant component of $u(t,x)$.
\par
We remark that this paper develops a framework combining stochastic convex integration with Lagrangian Littlewood--Paley localization for turbulent fluid equations exhibiting intermittent multiscale structures. The two-dimensional stochastic incompressible Euler equations \eqref{2DEEUGE1} serve as a fundamental model in which this framework is realized. The main challenge is not merely the presence of random forcing, but also the necessity of controlling the interaction between randomness, nonlinear transport, and intermittent multiscale structures. The stochastic Euler–Reynolds framework, the intermittent building blocks, and the Lagrangian frequency localization developed in this paper provide a potential foundation for extending convex integration methods to a broader class of stochastic fluid equations and turbulent systems.
\par
The present paper is organized as follows. Section 2 is devoted to proving Theorem \ref{CTSE-O-C} and Theorem \ref{CTSE-O-C-1B}. Some preliminary results concerning the stochastic and intermittent Newton--Nash iteration scheme are provided in the subsequent sections. More precisely, Section 3 introduces Newtonian perturbations and pressures to formulate the Newton--Euler--Reynolds system. Section 4 constructs Nash perturbations and Reynolds stresses for the stochastic Euler--Reynolds system. Section 5 provides inductive estimates for Newton--Nash elements, oscillations and energy gaps. In the appendix, we give some fundamental setups including H\"{o}lder spaces, martingale solutions, iteration parameters and intermittency.
\section{Intermittent Onsager theorem}
The proofs of Theorem \ref{CTSE-O-C} and Theorem \ref{CTSE-O-C-1B} are built upon the stochastic and intermittent Newton--Nash iteration and the Littlewood--Paley decomposition scheme. We first outline the main iterative framework and provide the key inductive estimates for Newton--Nash elements. Subsequently, we reduce the iterative error in the required H\"{o}lder space as the iteration step tends to infinity. In the limit, the H\"{o}lder-continuous martingale solutions to Eq.\eqref{2DEEUGE1} are recovered by the constructed approximate solutions. Finally, we employ the Littlewood--Paley decomposition scheme to investigate intermittency.
\subsection{Heuristic outline of the Newton--Nash iteration}\label{siuju-sij}
We collect all inductive parameters in Appendix C. Let ${\delta _0}$ and ${\lambda _0}$ be the initial parameters given by \eqref{parameters-S-2}, satisfying $0 < {\delta _0} < 1 < {\lambda _0} \leqslant C$ and $1 \leqslant \delta _0^{\frac{1}{2}}{\lambda _0} \leqslant C$ for a large positive constant $C$. Define the initial Wiener process 
\begin{equation*}
{\mathfrak{W}_0}\left( {t,x} \right) =\frac{{\varsigma _0}}{\sqrt 2 \pi }{\mathfrak{B}_0}\left( t \right)\delta _1^{\frac{1}{2}}\cos \left( {{\lambda _0}{x_1}} \right){{\mathbf{e}}_2} , \ \ {\varsigma _0} \in \mathbb{R} \backslash \left\{ 0 \right\} ,
\end{equation*}
where ${{\mathbf{e}}_2} \triangleq \left( {0,1} \right)$ is the unit vector in $\mathbb{R}^2$, and ${\mathfrak{B}_0}\left( t \right)$ is the real-valued Brownian motion on the complete probability space $\left( {\Omega ,\mathcal{F},\mathbb{P}} \right)$. It is easy to see that
\begin{equation}\label{HIE-WW-1.1-initial}
{\mathop {\sup }\limits_{t \in \left[ {0,{T}} \right]} }\mathbb{E}\left\| {{\mathfrak{W}_0}{\left( t \right)}} \right\|_N^2 \leqslant \frac{T{\varsigma _0 ^2}}{\pi ^2 }  {\delta _1}\lambda _0^{2N} , \ \ \forall \, N \in {\mathbb{Z}^ + } \cup \left\{ 0 \right\} .
\end{equation}
\par
Let $\kappa $ be a positive real number. Define the stopping times
\begin{equation*}
\mathfrak{t}_0 ^*\left( {\kappa,N} \right) \triangleq \inf \left\{ {t \geqslant 0:\mathop {\sup }\limits_{s \in \left[ {0,t} \right]} {{\left\| {{\mathfrak{W}_0}\left( s \right)} \right\|}_N} > C\kappa \delta _1^{\frac{1}{2}}\lambda _0^N} \right\} \wedge \kappa  ,
\end{equation*}
and
\begin{equation*}
{\mathfrak{t}_0} \triangleq T \wedge \min \left\{ {\mathfrak{t}_0 ^*\left( {\kappa,0} \right),  \cdots , \mathfrak{t}_0 ^*\left( {\kappa,N} \right)} \right\}  .
\end{equation*}
Since the Wiener process ${{\mathfrak{W}_0}\left( t \right)}$ is a martingale, it follows from Chebyshev inequality, Doob maximal inequality and \eqref{HIE-WW-1.1-initial} that
\begin{align*}
\mathbb{P}\left\{ {\mathop {\sup }\limits_{t \in \left[ {0,{T}} \right]} {{\left\| {{\mathfrak{W}_0}\left( t \right)} \right\|}_N} \leqslant C\kappa \delta _1^{\frac{1}{2}}\lambda _0^N} \right\} \geqslant & 1 - {C^{ - 2}}{\kappa ^{ - 2}}\delta _1^{ - 1}\lambda _0^{ - 2N}\mathbb{E}\mathop {\sup }\limits_{s \in \left[ {0,{T}} \right]} \left\| {{\mathfrak{W}_0}\left( t \right)} \right\|_N^2 \\
\geqslant & 1 - 4{C^{ - 2}}{\kappa ^{ - 2}}\delta _1^{ - 1}\lambda _0^{ - 2N}  \mathbb{E}\left\| {{\mathfrak{W}_0}\left( {T} \right)} \right\|_N^2  \\
= & 1 -  \frac{{2{\varsigma _0^2 }}}{{{C^{  2}} {\pi ^{ 2}} {\kappa ^{  2}}}}  {T} ,
\end{align*}
which implies that
\begin{equation*}
\mathop {\lim }\limits_{\kappa  \to \infty } \mathbb{P} \left( {\mathfrak{t}_{0}^* \left( {\kappa,N} \right) > \rho } \right) = 1 , \ \ \forall \, \rho > 0 , \ \forall \, N \in \mathbb{Z}^+ \cup \left\{ 0 \right\} .
\end{equation*}
In addition, it is a standard fact that almost all sample paths of ${\mathfrak{B}_0}\left( t \right)$ are continuous at $t = 0$ with ${\mathfrak{B}_0}\left( 0 \right) = 0$. Then
\begin{equation*}
\mathbb{P}\left\{ {\mathop {\lim }\limits_{t \to {0^ + }} \mathop {\sup }\limits_{s \in \left[ {0,t} \right]} {{\left\| {{\mathfrak{W}_0}\left( s \right)} \right\|}_N} = 0} \right\} = 1  ,
\end{equation*}
which implies that
\begin{equation*}
\mathbb{P}\left( {\mathfrak{t}_{0}^* \left( {\kappa,N} \right) > 0} \right) = 1 , \ \  \forall \, N \in \mathbb{Z}^+ \cup \left\{ 0 \right\}  .
\end{equation*}
Therefore, the stopping time ${\mathfrak{t}_0}$ satisfies 
\begin{equation}\label{chshdjhfy-a-1}
\mathbb{P}\left( {0 < {\mathfrak{t}_0} \leqslant T} \right) = 1 .
\end{equation}
Moreover, we have
\begin{equation}\label{HIE-WW-1.1-iDFGRFnitial}
\left\| {{\mathfrak{W}_0}} \right\|_N \lesssim {\delta _1^{\frac{1}{2}}}\lambda _0^{N} \lesssim {\delta _0^{\frac{1}{2}}}\lambda _0^{N} , \ \ \forall \, t \in \left[ {0,{\mathfrak{t}_0}} \right] , \ \forall \, N \in \mathbb{Z}^+ \cup \left\{ 0 \right\} , \ \mathbb{P}\mbox{-a.s.}  
\end{equation}
\par
Assume that $\mathfrak{h}\left( t \right):\ \mathbb{R} \longrightarrow \left[ {0,1} \right]$ is a smooth function with compact support $\left[ { 0,\frac{3}{4} T} \right]$, and satisfies
\begin{equation*}
\mathfrak{h}\left( t \right) = 1 , \ \ \forall \, t \in \left[ { 0,\frac{1}{4} T} \right] ,
\end{equation*}
and
\begin{equation*}
\left| {\mathfrak{h}\left( t \right)} \right| < 1 , \ \ \forall \, t \in \left( {\frac{1}{4}T,\frac{3}{4}T} \right] .
\end{equation*}
Choose the initial auxiliary function
\begin{equation*}
{v_0}\left( {t,x} \right) = \mathfrak{h}\left( t \right)\delta _0^{\frac{1}{2}} \cos \left( {{\lambda _0}{x_1}} \right) {{\mathbf{e}}_2} ,
\end{equation*}
and define the initial velocity
\begin{equation*}
{u_0}\left( {t,x} \right) = {v_0}\left( {t,x} \right) + {\mathfrak{W}_0}\left( {t,x} \right) .
\end{equation*}
Then
\begin{equation*}
\left( {{u_0} \cdot \nabla } \right){u_0} = 0 ,\ \ \mathrm{div}\,   {u_0} = 0  , \ \ \forall \, t \in \left[ {0,{\mathfrak{t}_0}} \right] .
\end{equation*}
Take the initial pressure ${\mathfrak{p}_0} = 0$ and the initial Reynolds stress
\begin{equation}\label{RRR-JIU-SUPP}
{\mathring{R}_0}\left( {t,x} \right) = {\partial _t}\mathfrak{h}\left( t \right)\lambda _0^{ - 1}\delta _0^{\frac{1}{2}}\left[ {\begin{array}{*{20}{c}}
0&{\sin \left( {{\lambda _0}{x_1}} \right)}\\
{\sin \left( {{\lambda _0}{x_1}} \right)}&0
\end{array}} \right]  .
\end{equation}
For any $t \in \left[ {0,{\mathfrak{t}_0}} \right]$, we have
\begin{equation}\label{ERNE-3A.1-initial}
\left\{ {\begin{array}{*{20}{l}}
{{\partial _t}{v_0} + \left( {{u_0} \cdot \nabla } \right){u_0} + \nabla {\mathfrak{p}_0} = \mathrm{div}\,   {\mathring{R}_0},}\\
{{u_0} = {v_0} + \mathfrak{ W}_0 ,}\\
\mathrm{div}\,   {u_0} = 0 .
\end{array}} \right.   
\end{equation}
\par
It follows from the definitions of $v_0$ and $\mathfrak{p}_0$ that
\begin{equation}\label{HIE-u-1.1-initial}
{\left\| {{v_0}} \right\|_N} + {\left\| {{\mathfrak{p}_0}} \right\|_N} \lesssim \delta _0^{\frac{1}{2}}\lambda _0^N, \ \ \forall \, t \in \left[ {0,T} \right],  \ \forall \, N \in {\mathbb{Z}^ + } \cup \left\{ 0 \right\}  .
\end{equation}
Consequently,
\begin{equation}\label{HIE-WW-1.1-initial-HUH}
{\left\| {{u_0} } \right\|_N}  \leqslant {\left\| {{v_0}} \right\|_N} + {\left\| {{\mathfrak{ W}_0}} \right\|_N} \lesssim \delta _0^{\frac{1}{2}}\lambda _0^N, \ \ \forall \, t \in \left[ {0,{\mathfrak{t}_0}} \right], \ \forall \,  N \in {\mathbb{Z}^ + } \cup \left\{ 0 \right\}, \ \mathbb{P}\mbox{-a.s.}  
\end{equation}
In view of \eqref{parameters-S-2}, there exists a parameter ${\delta _1}$ and a real number $\alpha \in \left( {\frac{\beta }{{b}},\beta } \right)$ such that $\lambda _0^{\alpha  - 1} \lesssim {\delta _1}$ for any $\beta  \in \left( {\gamma,\frac{1}{3}} \right)$. Then, we deduce that
\begin{equation}\label{HIE-R-1.1-initial}
{\left\| {{\mathring{R}_0}} \right\|_N} \lesssim \delta _0^{\frac{1}{2}}\lambda _0^{N - 1} \lesssim \delta _1 \lambda _0^{N - \alpha}, \ \  \forall\,  t \in \left[ {0,T} \right] , \ \forall \, N \in {\mathbb{Z}^ + } \cup \left\{ 0 \right\} ,
\end{equation}
and
\begin{equation}\label{MDHIE-R-1.1-initial}
{\left\| {{D_{t,0}}{\mathring{R}_0}} \right\|_N} \lesssim  {\left\| {{\partial _t}{\mathring{R}_0}} \right\|_N} + {\left\| {\left( {{v_0} \cdot \nabla } \right){\mathring{R}_0}} \right\|_N} \lesssim \delta _0^{\frac{1}{2}}{\delta _1}\lambda _0^{N + 1 - \alpha} , \ \  \forall\,  t \in \left[ {0,T} \right] , \  \forall \, N \in {\mathbb{Z}^ + } \cup \left\{ 0 \right\}  ,
\end{equation}
where ${D_{t,0}} \triangleq {\partial _t} + {v_0} \cdot \nabla $ is the material derivative corresponding to ${v_0}$.
\par
For $q  \in \mathbb{Z}^+$, define the $q$-step stopping times 
\begin{equation*}
\mathfrak{t}_{q}^*\left( {\kappa,N}  \right) = \inf \left\{ {t \geqslant 0:\mathop {\sup }\limits_{s \in \left[ {0,t} \right]} {{\left\| {{\mathfrak{W}_q}\left( s \right)} \right\|}_N} > C\kappa \sum\limits_{j = 0}^q {\delta _{q + 1}^{\frac{1}{2}}\lambda _q^N} } \right\} \wedge \kappa   ,
\end{equation*}
and
\begin{equation}\label{KMCNHSGFLDPF-1}
{\mathfrak{t}_q} = T \wedge \min \left\{ {\mathfrak{t}_{q}^*\left( {\kappa,0}  \right) , \cdots , \mathfrak{t}_{q}^*\left( {\kappa,N}  \right)} \right\} ,
\end{equation}
where the $q$-step Wiener process reads
\begin{equation}\label{Q-WINER-P}
{\mathfrak{W}_q}\left( t,x \right) = {\mathfrak{W}_0}\left( t,x \right) + \sum\limits_{j = 1}^q {\frac{{\varsigma _j}}{\sqrt 2 \pi }{\mathfrak{B}_j}\left( t \right) \delta _{j-1}^{\alpha} \delta _{j+1}^{\frac{1}{2}} \cos \left( {{\lambda _j}{x_1}} \right){{\mathbf{e}}_2}} , \ \  {\varsigma _q} \in \mathbb{R} \backslash \left\{ 0 \right\} ,
\end{equation}
and ${\mathfrak{B}_j}\left( t \right)$ are the real-valued, mutually independent Brownian motions on the complete probability space $\left( {\Omega ,\mathcal{F},\mathbb{P}} \right)$. Further define the terminal stopping time
\begin{equation*}
{\mathfrak{t}} \triangleq \inf \left\{ {{{\mathfrak{t}}_q}  \wedge  {\mathfrak{t}_\mathfrak{S}} :q \in {\mathbb{Z}^ + } \cup \left\{ 0 \right\}} \right\},
\end{equation*}
where the truncated stopping time ${\mathfrak{t}_\mathfrak{S}}$ is given by \eqref{IDJIFHISDF-ADS-1}.
\par
\begin{lemma}\label{LEM-ST-A-1-TT}
The terminal stopping time ${\mathfrak{t}}$ admits $\mathbb{P}\left( {0 < {\mathfrak{t}} \leqslant T} \right) = 1$ and $\left[ {0,\mathfrak{t}} \right] = \mathop  \cap \limits_{q \in {\mathbb{Z}^ + } \cup \left\{ 0 \right\}} \left[ {0,{\mathfrak{t}_j}} \right] $.
\end{lemma}
\par
\noindent{\textbf{Proof}}. 
Due to the definitions of stopping times, we have $\mathbb{P}\left( {{\mathfrak{t}} \leqslant T} \right) = 1$. Next, we verify that the stopping time $\mathfrak{t}$ is strictly positive.
\par
It follows from Chebyshev inequality and Doob maximal inequality that
\begin{align*}
\mathbb{P}\left\{ {\mathfrak{t}_{q + 1}^* \left( {\kappa,N}  \right)  \leqslant \mathfrak{t}_q^* \left( {\kappa,N}  \right) } \right\} = & \mathbb{P}\left\{ {\mathop {\sup }\limits_{s \in \left[ {0,t_q^*} \right]} {{\left\| {{\mathfrak{W}_{q + 1}}\left( s \right)} \right\|}_N} > C\kappa \sum\limits_{j = 0}^{q + 1} {\delta _{j + 1}^{\frac{1}{2}}\lambda _j^N} } \right\} \\
\leqslant & {\kappa ^{ - 2}}{\left( {C\sum\limits_{j = 0}^{q + 1} {\delta _{j + 1}^{\frac{1}{2}}\lambda _j^N} } \right)^{ - 2}}{\left( {\delta _1^{\frac{1}{2}}\lambda _0^N + \sum\limits_{j = 1}^{q + 1} {\delta _{j - 1}^\alpha \delta _{j + 1}^{\frac{1}{2}}\lambda _j^N} } \right)^2} .
\end{align*}
Denote ${A_q} \triangleq \sum\limits_{j = 0}^{q + 1} {\delta _{j + 1}^{\frac{1}{2}}\lambda _j^N} $ and ${B_q} \triangleq \delta _1^{\frac{1}{2}}\lambda _0^N + \sum\limits_{j = 1}^{q + 1} {\delta _{j - 1}^\alpha \delta _{j + 1}^{\frac{1}{2}}\lambda _j^N} $. We observe that $\frac{{{B_q}}}{{{A_q}}} \to 0$ with exponential rate as $q  \to \infty $. Then 
\begin{equation*}
\sum\limits_{q \in {\mathbb{Z}^ + } \cup \left\{ 0 \right\}} {\mathbb{P}\left\{ {\mathfrak{t}_{q + 1}^* \left( {\kappa,N}  \right)  \leqslant \mathfrak{t}_q^* \left( {\kappa,N}  \right) } \right\}}  \leqslant \sum\limits_{q \in {\mathbb{Z}^ + } \cup \left\{ 0 \right\}} {\frac{{B_q^2}}{{{\kappa ^2}A_q^2}}}  < \infty , \ \ \forall \, \kappa > 0 , \ N \in {\mathbb{Z}^ + } \cup \left\{ 0 \right\} .
\end{equation*}
Employing the Borel--Cantelli lemma and the definition \eqref{KMCNHSGFLDPF-1}, there exists a sufficiently large integer ${N^*}$ such that
\begin{equation}\label{NCGYTDHFU-F-2-1}
\mathbb{P}\left\{ {\mathfrak{t}_{q + 1} > \mathfrak{t}_q} \right\} = 1 , \ \ \forall \, q > {N^*}.
\end{equation}
\par
Using the same argument as that of \eqref{chshdjhfy-a-1}, we have
\begin{equation*}
\mathbb{P}\left( { {\mathfrak{t}_q} > 0} \right) = 1  , \ \ \forall \, q \leqslant {N^* + 1}.
\end{equation*}
Invoking \eqref{NCGYTDHFU-F-2-1}, we obtain
\begin{equation*}
\inf \left\{ {{\mathfrak{t}_q} \wedge {\mathfrak{t}_\mathfrak{S}}:q \in {\mathbb{Z}^ + } \cup \left\{ 0 \right\}} \right\} = \inf \left\{ {{\mathfrak{t}_q} \wedge {\mathfrak{t}_\mathfrak{S}}:q \leqslant {N^* + 1}} \right\}  ,
\end{equation*}
which implies that
\begin{equation*}
\mathbb{P}\left\{ {\mathfrak{t} > 0} \right\} = \mathbb{P}\left\{ {\inf \left\{ {{\mathfrak{t}_q} \wedge {\mathfrak{t}_\mathfrak{S}}:q \leqslant {N^* + 1}} \right\} > 0} \right\} = 1  .
\end{equation*}
Furthermore, since $\mathop  \cap \limits_{j \leqslant {N^* + 1}} \left[ {0,{\mathfrak{t}_j}} \right] \subset \left[ {0,{\mathfrak{t}_q}} \right]$ for any $ q > {N^*}$, we derive that
\begin{equation*}
\left[ {0,\mathfrak{t}} \right] = \mathop  \cap \limits_{j \leqslant {N^* + 1}} \left[ {0,{\mathfrak{t}_j}} \right] = \mathop  \cap \limits_{q \in {\mathbb{Z}^ + } \cup \left\{ 0 \right\}} \left[ {0,{\mathfrak{t}_q}} \right]  .
\end{equation*}
The proof is complete. \qed
\par
We establish the initial Euler--Reynolds system \eqref{ERNE-3A.1-initial} and provide some estimates \eqref{HIE-u-1.1-initial}-\eqref{MDHIE-R-1.1-initial}. Without loss of generality, for any $t \in  \mathop  \cap \limits_{j = 0}^q \left[ {0,{\mathfrak{t}_j}} \right]$, assume that the $q$-step stochastic Euler--Reynolds system takes the form
\begin{equation}\label{ERNE-3A.1}
\left\{ {\begin{array}{*{20}{l}}
{\partial _t}{v_q} + \left( {{u_q} \cdot \nabla } \right){u_q} + \nabla {\mathfrak{p}_q} = \mathrm{div}\,   {\mathring{R}_q},\\
{u_q} = {v_q} + \mathfrak{ W}_q ,\\
\mathrm{div}\,   {u_q} = 0 .
\end{array}} \right.  
\end{equation}
The key procedure of the iteration scheme is to establish inductive estimates for the approximate equations \eqref{ERNE-3A.1}. To this end, we assume that Eq.\eqref{ERNE-3A.1} permits $q$-step inductive estimates analogous to \eqref{HIE-u-1.1-initial}-\eqref{MDHIE-R-1.1-initial}. Then, we iterate Eq.\eqref{ERNE-3A.1} to obtain the $\left( {q + 1} \right)$-step counterpart, and show that the $\left( {q + 1} \right)$-step system satisfies the corresponding $\left( {q + 1} \right)$-step inductive estimates.
\par
The construction proceeds in two stages. The first step is to build the Newtonian perturbation by linearizing the Euler equations with temporal and spatial oscillations. Using this Newtonian perturbation, we transform the $q$-step stochastic Euler--Reynolds system \eqref{ERNE-3A.1} into a finite-step stochastic Newton--Euler--Reynolds system. Furthermore, we establish some necessary estimates for these Newtonian elements, which validates the Newton iteration scheme. Our discussion is formal, and rigorous details are presented in Sections \ref{NNIS-JIUYRFD}-\ref{ENTP}.
\par
Second, we employ the finite-step stochastic Newton--Euler--Reynolds system and the intermittent building block to construct the non-interacting Nash perturbation with designed energy gap. Combining the Nash perturbation and the Newtonian perturbation, we transform all components of the $q$-step stochastic Euler--Reynolds system \eqref{ERNE-3A.1} into the $\left( {q + 1} \right)$-step counterparts. We estimate the first and second order temporal oscillations for the $\left( {q + 1} \right)$-step Nash elements, which justifies the Nash iteration scheme. Full details of the Nash iteration scheme are presented in Section \ref{NISTYU-SKO-A} and Section \ref{ENHP}. Furthermore, we demonstrate that the $\left( {q + 1} \right)$-step Nash elements and the energy gap satisfy inductive estimates as $q$ is replaced by $q+1$ (see Propositions \ref{HNNI-NEHU-MAIN}-\ref{RSRI-ESI-GAP} in Section \ref{ENHP-RHDU}).
\par
In this paper, all statements and estimates involving the approximate processes $u_q$, $v_q$, ${\mathfrak{p}_q}$, ${\mathring{R}_q}$ and the stopping time ${\mathfrak{t}}$ hold $\mathbb{P}$-almost surely. The inductive estimates for the Nash elements $\left( {{v_q},{u_q},{\mathfrak{p}_q},{\mathring{R}_q}} \right)$ are presented as follows.
\par
\begin{proposition}\label{HNNI-MAIN}
For any $q \in {\mathbb{Z}^ + }$ and $t \in \mathop  \cap \limits_{j = 0}^q \left[ {0,{\mathfrak{t}_j}} \right]$, there exists a parameter $\alpha \in \left( {\frac{\beta }{{b}},\beta } \right)$ such that the Nash elements of the Euler–Reynolds system \eqref{ERNE-3A.1} satisfy
\begin{equation}\label{HIE-u-1.1-BUDHU-A-1}
{\left\| {{u_q} } \right\|_0} + {\left\| {{v_q}} \right\|_0} + {\left\| {{\mathfrak{p}_q}} \right\|_0} \leqslant C  , 
\end{equation}
\begin{equation}\label{HIE-u-1.1}
{\left\| {{u_q} } \right\|_N} + {\left\| {{v_q}} \right\|_N} + {\left\| {{\mathfrak{p}_q}} \right\|_N} \lesssim \delta _q^{\frac{1}{2}}\lambda _q^N  ,  \ \ \forall \ N \in \left\{ {1, \cdots ,10} \right\} ,
\end{equation}
\begin{equation}\label{HIE-R-1.1}
{\left\| {{\mathring{R}_q}} \right\|_N} \lesssim \delta _{q+1}\lambda _q^{N - \alpha}  ,   \ \ \forall \ N \in \left\{ {0 , 1, \cdots ,10} \right\} ,
\end{equation}
and
\begin{equation}\label{MDHIE-R-1.1}
{\left\| {{D_{t,q}}{\mathring{R}_q}} \right\|_N} \lesssim \delta _q^{\frac{1}{2}}{\delta _{q + 1}}\lambda _q^{N + 1 - \alpha} ,  \ \ \forall \ N \in \left\{ {0 , 1, \cdots ,10} \right\}  , 
\end{equation}
where ${D_{t,q}} \triangleq {\partial _t} + {v_q} \cdot \nabla $ denotes the material derivative corresponding to ${v_q}$.
\end{proposition}
\subsection{The proof of Theorem \ref{CTSE-O-C}}
The proof of Theorem \ref{CTSE-O-C} is divided into three steps. The first is to construct $\gamma$-H\"{o}lder-continuous martingale solutions. The second is to establish the temporal support of $\gamma$-H\"{o}lder-continuous martingale solutions. The third is to derive the energy inequality \eqref{SHUDJ-NENH-B-SIDIS}.
\par
\vspace{1em}
{\textbf{Step 1. The existence of H\"{o}lder-continuous martingale solutions}}. 
\par
\vspace{1em}
Using the inductive estimate \eqref{HIE-WWW-1.1}, we have
\begin{equation}\label{UIJHF-DI}
\mathop {\sup }\limits_{t \in \left[ {0,T} \right]} \left( {\mathbb{E}\left\| {{\mathfrak{W}_q}\left( t \right) - {\mathfrak{W}_{q-1}}\left( t \right)} \right\|_0^2 + \lambda _{q + 1}^{ - 2}\mathbb{E}\left\| {{\mathfrak{W}_q}\left( t \right) - {\mathfrak{W}_{q-1}}\left( t \right)} \right\|_1^2} \right) \lesssim {\delta _q} .
\end{equation}
We note that the total Newton--Nash perturbation ${w_{q + 1}}\left( {t,x} \right) $ is jointly defined by \eqref{SHDUFY-A-2}, \eqref{THENDM-DKF} and \eqref{NISTYU-SKO-A-PPA-1}. It follows from the definition \eqref{NISTYU-SKO-A-PPA-2}, Proposition \ref{WPD-NP-main} and Proposition \ref{WPD-NS-SD-P-main} that
\begin{equation}\label{UIJHF-DI-OKU-A}
{\left\| {{v_q} - {v_{q-1}}} \right\|_0} + \lambda _q^{ - 1}{\left\| {{v_q} - {v_{q - 1}}} \right\|_1} \lesssim \delta _q^{\frac{1}{2}} .
\end{equation}
Combining \eqref{UIJHF-DI}, \eqref{UIJHF-DI-OKU-A} and the fact that ${u_q}\left( {t,x} \right) = {v_q}\left( {t,x} \right) + {\mathfrak{W}_q}\left( {t,x} \right)$, we obtain
\begin{equation*}
{\left\| {{u_q} - {u_{q - 1}}} \right\|_0} + \lambda _q^{ - 1}{\left\| {{u_q} - {u_{q - 1}}} \right\|_1} \lesssim \delta _q^{\frac{1}{2}} ,\ \ \forall \, t \in \mathop  \cap \limits_{j = 0}^q \left[ {0,{\mathfrak{t}_j}} \right] .
\end{equation*}
Therefore, employing the interpolation inequality, we get
\begin{equation}\label{MDHIE-LOIU-R-1.1-2}
{\left\| {{u_{q + 1}} - {u_q}} \right\|_\gamma } \lesssim \left\| {{u_{q + 1}} - {u_q}} \right\|_0^{1 - \gamma }\left\| {{u_{q + 1}} - {u_q}} \right\|_1^\gamma  \lesssim \delta _{q + 1}^{\frac{1}{2}}\lambda _{q + 1}^\gamma  \lesssim \lambda _{q + 1}^{\gamma  - \beta },\ \ \forall \, t \in \mathop  \cap \limits_{j = 0}^q \left[ {0,{\mathfrak{t}_j}} \right] ,
\end{equation}
which implies that $\left\{ {{u_q}\left( t,\cdot  \right)} \right\}_{q \in \mathbb{Z}^+}$ is a Cauchy sequence in ${C^\gamma }\left( {{\mathbb{T}^2}} \right)$. Then, there exists a process ${u\left( t,x  \right)} \in {C^\gamma }\left( {{\mathbb{T}^2}} \right)$ such that
\begin{equation}\label{OKIUHBF-A}
{{u_q}\left( t,x  \right)} \to  {u\left( t,x  \right)},\ \ \mbox{in}\ {C^\gamma }\left( {{\mathbb{T}^2}} \right),\ \ \forall\, t \in \left[ {0,{\mathfrak{t}}} \right], \ \text{as}\ q \to  \infty ,
\end{equation}
and there exists a small scale $0 < {\tilde \rho} \ll 1$ such that
\begin{equation}\label{COMJHI-F-j-1}
\mathop {\sup }\limits_{t \in \left[ {0,{\mathfrak{t}_q}} \right]} \int_{{\mathbb{T}^2}} {{{\left| {{u_q}\left( {t,x + h} \right) - {u_q}\left( {t,x} \right)} \right|}^2}} {\rm{d}}x \lesssim {\left| h \right|^{2\gamma }} , \ \ \forall \, \left| h \right| \in \left[ {{\tilde \rho} , C} \right],\ \mbox{as}\ q \ \mbox{is sufficiently large} .
\end{equation}
Using the argument similar to that of Constantin and Vicol \cite[Theorem 3.1]{ConstantinP}, it follows from \eqref{COMJHI-F-j-1} that
\begin{align}\label{NLT-C-S-NSBE-u-1}
& \left| {\int_0^t {\left\langle {\mathrm{div} \left( {{u_q}\left( \tau  \right) \otimes {u_q}\left( \tau  \right)} \right) - \mathrm{div} \left( {u\left( \tau  \right) \otimes u\left( \tau  \right)} \right),\psi } \right\rangle _{\mathbb{T}^2}} {\rm{d}}\tau } \right| \nonumber \\
\to & 0, \ \ \forall \, \psi \left( x \right) \in {C^\infty }\left( {{\mathbb{T}^2}} \right) , \ \forall \, t \in \mathop  \cap \limits_{j = 0}^q \left[ {0,{\mathfrak{t}_j}} \right]   , \ \text{as}\ q \to  \infty,
\end{align}
where ${u\left( t,x  \right)}$ is the limit of $\left\{ {{u_q}\left( t,x  \right)} \right\}$ in \eqref{OKIUHBF-A}. Furthermore, invoking the inductive estimate \eqref{HIE-R-1.1} and the interpolation inequality, we obtain
\begin{equation}\label{MDHIE-LOIU-R-1.1-1}
{\left\| {{\mathring{R}_q}} \right\|_\gamma } \lesssim \left\| {{\mathring{R}_q}} \right\|_0^{1 - \gamma }\left\| {{\mathring{R}_q}} \right\|_1^\gamma  \lesssim {\delta _{q + 1}}\lambda _{q + 1}^{\gamma  - \alpha } \lesssim \lambda _{q + 1}^{\gamma  - 2\beta  - \alpha } ,\ \ \forall \, t \in \mathop  \cap \limits_{j = 0}^q \left[ {0,{\mathfrak{t}_j}} \right]  ,
\end{equation}
which implies that
\begin{equation*}
{\mathring{R}_q} \to 0 , \ \ \mbox{in}\ {C^\gamma }\left( {{\mathbb{T}^2}} \right) , \ \ \forall\, t \in \left[ {0,{\mathfrak{t}}} \right], \ \mbox{as}\ q \to  \infty . 
\end{equation*}
\par
It follows from the inductive estimates \eqref{HIE-u-1.1} and \eqref{HIE-R-1.1} that
\begin{align*}
& {\left\| {\mathrm{div}\left( {{\mathring{R}_q} - {\mathring{R}_{q - 1}}} \right) - \mathrm{div}\left( {{u_q} \otimes {u_q} - {u_{q - 1}} \otimes {u_{q - 1}}} \right)} \right\|_N} \\
\lesssim & {\left\| {{\mathring{R}_q}} \right\|_{N + 1}} + {\left\| {{\mathring{R}_{q - 1}}} \right\|_{N + 1}} + {\left\| {{u_q} - {u_{q - 1}}} \right\|_{N+1}}\left( {{{\left\| {{u_q}} \right\|}_0} + {{\left\| {{u_{q - 1}}} \right\|}_0}} \right) \\
& + {\left\| {{u_q} - {u_{q - 1}}} \right\|_0}\left( {{{\left\| {{u_q}} \right\|}_{N+1}} + {{\left\| {{u_{q - 1}}} \right\|}_{N+1}}} \right) \\
\lesssim & \delta _{q - 1}^{\frac{1}{2}}\delta _q^{\frac{1}{2}}\lambda _q^{N + 1} , \ \ \forall \, t \in \mathop  \cap \limits_{j = 0}^q \left[ {0,{\mathfrak{t}_j}} \right] .
\end{align*}
Using the Schauder estimate (see \cite{MR3374958}), the interpolation inequality and the fact that
\begin{equation*}
\Delta {\mathfrak{p}_q} = {\mathrm{div}}\,{\mathrm{div}}\left( { {\mathring{R}_q} - {u_q} \otimes {u_q}} \right)  ,
\end{equation*}
we deduce that
\begin{align}\label{djubh0sdf-s-A-1}
{\left\| {{\mathfrak{p}_{q + 1}} - {\mathfrak{p}_{q}}} \right\|_{1 + {\alpha ^*}}} \lesssim & {\left\| {\mathrm{div}\left( {{\mathring{R}_{q + 1}} - {\mathring{R}_{q }}} \right) - \mathrm{div} \left( {{u_{q + 1}} \otimes {u_{q + 1}} - {u_{q }} \otimes {u_{q }}} \right)} \right\|_{{\alpha ^*}}} \nonumber \\
\lesssim & \left\| {\mathrm{div}\left( {{\mathring{R}_{q + 1}} - {\mathring{R}_{q }}} \right) - \mathrm{div} \left( {{u_{q + 1}} \otimes {u_{q + 1}} - {u_{q }} \otimes {u_{q }}} \right)} \right\|_0^{1 - {\alpha ^*}}    \nonumber \\
& \cdot \left\| {\mathrm{div}\left( {{\mathring{R}_{q + 1}} - {\mathring{R}_{q }}} \right) - \mathrm{div} \left( {{u_{q + 1}} \otimes {u_{q + 1}} - {u_{q }} \otimes {u_{q }}} \right)} \right\|_1^{{\alpha ^*}} \nonumber \\
\lesssim & \delta _{q }^{\frac{1}{2}}\delta _{q + 1}^{\frac{1}{2}}\lambda _{q + 1}^{1 + {\alpha ^*}} , \ \ \forall \, t \in \mathop  \cap \limits_{j = 0}^q \left[ {0,{\mathfrak{t}_j}} \right]  , \ \forall \,  {\alpha ^*} \in \left( {0,\beta} \right)  .
\end{align}
Thus, combining \eqref{HIE-u-1.1}, \eqref{djubh0sdf-s-A-1} and the definition \eqref{parameters-S-2}, we obtain
\begin{align}\label{LKLJFIJFSFKH}
{\left\| {{\mathfrak{p}_{q + 1}} - {\mathfrak{p}_{q }}} \right\|_\gamma } \lesssim & \left\| {{\mathfrak{p}_{q + 1}} - {\mathfrak{p}_{q }}} \right\|_0^{1 - \gamma }\left\| {{\mathfrak{p}_{q + 1}} - {\mathfrak{p}_{q }}} \right\|_{1 + {\alpha ^*}}^\gamma  \nonumber \\
\lesssim & \delta _{q }^{\frac{1}{2}}\delta _{q + 1}^{\frac{1}{2}\gamma }\lambda _{q + 1}^{\gamma  + \gamma {\alpha ^*}} \nonumber \\
\lesssim & {a^{\left( {b\gamma \left( {1 + {\alpha ^*} - \beta } \right) - \beta } \right){b^{q }}}}  , \ \ \forall \, t \in \mathop  \cap \limits_{j = 0}^q \left[ {0,{\mathfrak{t}_j}} \right] .
\end{align}
By the monotonicity of reciprocal functions, we choose parameters $b \in \left( {1.005,1.3} \right]$ and $\beta  \in \left( {\gamma,\frac{1}{3}} \right)$ such that $f\left( \gamma  \right) = \frac{\beta }{{b\gamma }} + \beta  - 1 > 0$ for any $\gamma  \in \left( {0,\beta } \right)$. Then, fixing the parameter ${\alpha ^*} \in \left( {0,\frac{\beta }{{b\gamma }} + \beta  - 1} \right)$, we have ${b\gamma \left( {1 + {\alpha ^*} - \beta } \right) - \beta } < 0$. It follows from \eqref{LKLJFIJFSFKH} that there exists a non-trivial pressure $\mathfrak{p}\left( t,x \right) $ such that
\begin{equation*}
{\mathfrak{p}_q} \left( t,x  \right) \to \mathfrak{p} \left( t,x  \right) , \ \ \mbox{in}\ {C^\gamma }\left( {{\mathbb{T}^2}} \right),\ \forall\, t \in \left[ {0,{\mathfrak{t}}} \right], \ \text{as}\ q \to  \infty ,
\end{equation*}
which implies that
\begin{equation}\label{NLT-C-S-NSBE-p-1}
\left\langle {\nabla {\mathfrak{p}_q} \left( t  \right)  - \nabla \mathfrak{p}\left( t  \right) ,\psi } \right\rangle _{\mathbb{T}^2}  \to 0, \ \ \forall \, \psi \left( x \right) \in {C^\infty }\left( {{\mathbb{T}^2}} \right) , \ \forall \, t \in \left[ {0,{\mathfrak{t}}} \right], \  \text{as}\ q \to  \infty  .
\end{equation}
\par
For each step size $q \in \mathbb{Z}^+$, we truncate a degenerate segment of the Wiener process ${\mathfrak{W}}\left( {t} \right)$ specified in Appendix $B$ to construct the $q$-step Wiener process ${\mathfrak{W}_q}\left( {t} \right)$. Then, we have
\begin{equation*}
\mathbb{E}\left\| {{\mathfrak{W}_q}\left( t \right)} \right\|_\gamma ^2 \lesssim \mathrm{Tr} \left( {Q + {Q^{\frac{1}{2}}}{{\left( { - \Delta } \right)}^{\frac{1 + \gamma }{2} + \rho}}{{\left( {{Q^{\frac{1}{2}}}} \right)}^*}} \right) <  \infty  ,
\end{equation*}
which implies ${\mathfrak{W}_q}\left( t \right) \in {L^2}\left( {\Omega ;{C^\gamma }\left( {{\mathbb{T}^2}} \right)} \right)$ for any $q \in \mathbb{Z}^+$. Using \eqref{UIJHF-DI} and the interpolation inequality once more, there exists a Wiener process $ {\mathfrak{W}}\left( {t,x} \right) \in \mathscr{W}$ such that
\begin{equation}\label{CW-SJU-1}
{\mathfrak{W}_q}\left( {t,x} \right) \to {\mathfrak{W}}\left( {t,x} \right) , \ \ \mbox{in}\ {L^2}\left( {\Omega ;{C^\gamma }\left( {{\mathbb{T}^2}} \right)} \right),\ \forall\, t \in \left[ {0,{\mathfrak{t}}} \right], \ \text{as}\ q \to  \infty .
\end{equation}
Consequently, in view of \eqref{OKIUHBF-A}, \eqref{NLT-C-S-NSBE-u-1} and \eqref{NLT-C-S-NSBE-p-1}, we infer that the limit $\left( {u ,\mathfrak{p},\mathfrak{W}} \right)$ fulfills the two-dimensional stochastic Euler equations \eqref{2DEEUGE1} in distribution.
\par
\vspace{1em}
{\textbf{Step 2. Temporal support}}. 
\par
\vspace{1em}
By the definition of $\mathfrak{h}\left( t \right)$, we have
\begin{equation*}
{\mathrm{supp}_\mathrm{t}} \, \left( {{\partial _t}\mathfrak{h}} \right) = \left[ {0,\frac{3}{4}T} \right]\backslash \left[ {0,\frac{1}{4}T} \right] = \left( {\frac{1}{4}T,\frac{3}{4}T} \right] ,
\end{equation*}
which implies that the initial Reynolds stress satisfies
\begin{equation*}
{\mathrm{supp}_\mathrm{t}} \, {\mathring{R}_0} = \left( {\frac{1}{4}T,\frac{3}{4}T} \right] .
\end{equation*}
Without loss of generality, assume that the temporal support of $q$-step Reynolds stress is
\begin{equation}\label{RRR-SUPPOI}
{\mathrm{supp}_\mathrm{t}} \, {\mathring{R}_q} = \mathop  \cap \limits_{j = 0}^q \left[ {0,{\mathfrak{t}_j}} \right] \cap \left( {\frac{1}{4}T ,\frac{3}{4}T - 2\delta _q^{ - \frac{1}{2}}\lambda _q^{ - 1} + 4\Gamma {\tau _q} } \right] ,\ \ \mbox{for}\ q \in \mathbb{Z}^+ ,
\end{equation}
where $\Gamma$ denotes the number of Newtonian iteration steps, and $\tau _q$ is the temporal scale (see Appendix C for their definitions). It follows from the support property \eqref{LIJIFNI-DKFJ-R-ss3} that the $\left( {q+1} \right)$-step Reynolds stress ${\mathring{R}_{q+1}}$ admits the temporal support \eqref{RRR-SUPPOI} as $q$ is replaced by $q+1$, which completes the iterative loop for temporal support.
\par
Since the inductive temporal support of Newtonian perturbations and Nash perturbations are provided by \eqref{OPIUY-A-1}, \eqref{THENDM-DKF-SUPP-1} and \eqref{THENDM-DKF-SUPP-2}, we have
\begin{equation*}
{\mathrm{supp}_\mathrm{t}} \, {w_{q + 1}} \subset \mathop  \cap \limits_{j = 0}^q \left[ {0,{\mathfrak{t}_j}} \right] \cap \left( {\frac{1}{4}T ,\frac{3}{4}T - 2\delta _{q + 1}^{ - \frac{1}{2}}\lambda _{q + 1}^{ - 1} + 4\Gamma {\tau _{q+1}} } \right] .
\end{equation*}
Consequently,
\begin{equation}\label{RRR-SUPPOI-pp-1}
{\mathrm{supp}_\mathrm{t}} \, \left( {{u_{q + 1}} - {u_q}} \right) \subset \mathop  \cap \limits_{j = 0}^q \left[ {0,{\mathfrak{t}_j}} \right] \cap \left( {\frac{1}{4}T ,\frac{3}{4}T - 2\delta _{q + 1}^{ - \frac{1}{2}}\lambda _{q + 1}^{ - 1}+ 4\Gamma {\tau _{q+1}}} \right] .
\end{equation}
According to \eqref{RRR-SUPPOI} and \eqref{RRR-SUPPOI-pp-1}, the limit velocity ${u}\left( {t,x} \right)$ permits
\begin{equation*}
{\mathrm{supp}_\mathrm{t}} \, u \subset \mathop  \cup \limits_{q \in {\mathbb{Z}^ + }} \left( {{\mathrm{supp}_\mathrm{t}}\, {w_q} \cup {\mathrm{supp}_\mathrm{t}}\, {\mathring{R}_q} \cup {\mathrm{supp}_\mathrm{t}}\, {u_0}  } \right) \subset \left( {\left[ {0, {\mathfrak{t}} } \right] \cap  \left[ {0,\frac{3}{4} T} \right]} \right) .
\end{equation*}
\par
\vspace{1em}
{\textbf{Step 3. Dissipative behavior}}. 
\par
\vspace{1em}
Using the same argument as in \eqref{OKIUHBF-A}, we deduce that
\begin{equation}\label{OKIUHBF-A-UIJ-AY}
{{v_{q + 1}}\left( t,x  \right)} \to  {v\left( t,x  \right)},\ \ \mbox{in}\ {C^\gamma }\left( {{\mathbb{T}^2}} \right),\ \ \forall\, t \in \left[ {0, {\mathfrak{t}} } \right], \ \text{as}\ q \to  \infty ,
\end{equation}
where ${{v}\left( t,x  \right)}  = {{u}\left( t,x  \right)}  - {\mathfrak{W}}\left( {t,x} \right) $. Denote $\mathcal{\tilde E}_{q} \left( t \right) \triangleq \left\| {v_{q}\left( t \right)} \right\|_{{L^2}}^2 $. Then
\begin{align*}
{\mathcal{\tilde E}_{q + 1}}\left( t \right) - {\mathcal{\tilde E}_q}\left( t \right) = & \int_{{\mathbb{T}^2}} {{{\left| {{v_{q + 1}}\left( {t,x} \right)} \right|}^2} - {{\left| {{v_q}\left( {t,x} \right)} \right|}^2}}\ \mathrm{d}x \\
= & \int_{{\mathbb{T}^2}} {{{\left| {{w_{q + 1}}\left( {t,x} \right)} \right|}^2}}\ \mathrm{d}x + 2\int_{{\mathbb{T}^2}} {{w_{q + 1}}\left( {t,x} \right) \cdot {v_q}\left( {t,x} \right)}\ \mathrm{d}x ,
\end{align*}
which implies that
\begin{align}\label{xiuh-IJHYU-1}
\left| {{\mathcal{\tilde E}_{q + 1}}\left( t \right) - {\mathcal{\tilde E}_q}\left( t \right)} \right| \lesssim & \int_{{\mathbb{T}^2}} {{{\left| {w_{q + 1}^{\left( p \right)}\left( {t,x} \right)} \right|}^2} + {{\left| {w_{q + 1}^{\left( c \right)}\left( {t,x} \right)} \right|}^2} + {{\left| {w_{q + 1}^{\left( i \right)}\left( {t,x} \right)} \right|}^2}}\ \mathrm{d}x \nonumber \\
& + \int_{{\mathbb{T}^2}} {{{\left| {w_{q + 1}^{\left( \Gamma \right)}\left( {t,x} \right)} \right|}^2} + \left| {{w_{q + 1}}\left( {t,x} \right)} \right|\left| {{v_q}\left( {t,x} \right)} \right|}\ \mathrm{d}x .
\end{align}
Here, the components ${w_{q + 1}^{\left( p \right)}\left( {t,x} \right)}$, ${w_{q + 1}^{\left( c \right)}\left( {t,x} \right)}$, ${w_{q + 1}^{\left( i \right)}\left( {t,x} \right)}$ and ${w_{q + 1}^{\left( \Gamma \right)}\left( {t,x} \right)}$ are defined by \eqref{Nashsteps-A-P-3}, \eqref{Nashsteps-A-P-4}, \eqref{THENDM-DKF}, \eqref{Newtonsteps-A-4} and \eqref{UOIYH-LKO-DHU}, respectively.
\par
According to Proposition \ref{WPD-NS-SD-P-main} and \eqref{WPD-NP-main-BTO-1}, we infer that for any $t \in \mathop  \cap \limits_{j = 0}^q \left[ {0,{\mathfrak{t}_j}} \right]$,
\begin{equation*}
\int_{{\mathbb{T}^2}} {{{\left| {w_{q + 1}^{\left( p \right)}\left( {t,x} \right)} \right|}^2} + {{\left| {w_{q + 1}^{\left( c \right)}\left( {t,x} \right)} \right|}^2} + {{\left| {w_{q + 1}^{\left( i \right)}\left( {t,x} \right)} \right|}^2}}\ \mathrm{d}x \lesssim {\delta _{q + 1}}  ,
\end{equation*}
and
\begin{equation*}
\int_{{\mathbb{T}^2}} {{{\left| {w_{q + 1}^{\left( \Gamma \right)}\left( {t,x} \right)} \right|}^2}}\ \mathrm{d}x \lesssim \mu _{q + 1}^{ - 2}\delta _{q + 1}^2\lambda _q^2l_q^{ - 2\alpha } \lesssim {\delta _{q + 1}} . 
\end{equation*}
Moreover, we have
\begin{equation*}
\int_{{\mathbb{T}^2}} {\left| {{w_{q + 1}}\left( {t,x} \right)} \right|\left| {{v_q}\left( {t,x} \right)} \right|}\ \mathrm{d}x \lesssim {\left\| {{w_{q + 1}}\left( t \right)} \right\|_0}{\left\| {{v_q}\left( t \right)} \right\|_0} \lesssim \delta _{q + 1}^{\frac{1}{2}} , \ \ \forall \, t \in \mathop  \cap \limits_{j = 1}^{q + 1} \left[ {0,{\mathfrak{t}_j}} \right]  .
\end{equation*}
It follows from \eqref{xiuh-IJHYU-1} that
\begin{equation*}
\left| {{\mathcal{\tilde E}_{q + 1}}\left( t \right) - {\mathcal{\tilde E}_q}\left( t \right)} \right| \lesssim \delta _{q + 1}^{\frac{1}{2}} , \ \ \forall \, t \in \mathop  \cap \limits_{j = 1}^{q + 1} \left[ {0,{\mathfrak{t}_j}} \right] . 
\end{equation*}
Hence, combining \eqref{OKIUHBF-A-UIJ-AY}, \eqref{SJKFJIFJSKJGG-ASD-1}, the compact embedding ${C^\gamma }\left( {{\mathbb{T}^2}} \right) \subset {L^2}\left( {{\mathbb{T}^2}} \right)$ and the fact that $\left| {\mathfrak{h}\left( t \right)} \right|  \leqslant   \left| {\mathfrak{h}\left( 0 \right)} \right|$ for any $t \in \left( {0,T} \right]$, there exists a function $\mathcal{\tilde E}\left( t \right) = \left\| {v\left( t \right)} \right\|_{{L^2}}^2$ such that
\begin{equation*}
\mathfrak{e}\left( t \right)  \lesssim  {\mathcal{\tilde E}_q}\left( t \right) + {\delta _{q + 1}} \to \mathcal{\tilde E}\left( t \right) \leqslant \mathcal{E}\left( 0 \right) , \ \ \forall \, t \in \left( {0, {\mathfrak{t}} } \right] ,\   \mbox{as}\ q \to  \infty . 
\end{equation*}
This shows that the function $\mathfrak{e}\left( t \right)$ is well-defined as $q \to  \infty$.
\par
We next control the stochastic integral term and the It\^{o} drift correction. Denote by $\hat f \left( x \right)$ the mollification of $ f \left( x \right) $ (see Section \ref{NISTYU-SKO-A} for the definition of $\hat f  \left( x \right)$). Employing It\^{o} isometry, we obtain
\begin{align}\label{xiuh-IJHYU-2-ITO-D-A-1}
& \mathbb{E}{\left| {\int_0^t {{{\left\langle {v,\mathrm{d}\mathfrak{W}} \right\rangle }_{{\mathbb{T}^2}}}}  - \int_0^t {{{\left\langle {{{\hat v}_q},\mathrm{d}{\mathfrak{W}_q}} \right\rangle }_{{\mathbb{T}^2}}}} } \right|^2} \nonumber \\
\lesssim & \mathbb{E}{\left| {\int_0^t {{{\left\langle {v - {v_q},\mathrm{d}\mathfrak{W}} \right\rangle }_{{\mathbb{T}^2}}}} } \right|^2} + \mathbb{E}{\left| {\int_0^t {{{\left\langle {{v_q} - {{\hat v}_q},\mathrm{d}\mathfrak{W}} \right\rangle }_{{\mathbb{T}^2}}}} } \right|^2} + \mathbb{E}{\left| {\int_0^t {{{\left\langle {{{\hat v}_q},\mathrm{d}\left( {\mathfrak{W} - {\mathfrak{W}_q}} \right)} \right\rangle }_{{\mathbb{T}^2}}}} } \right|^2} \nonumber \\
\lesssim & \mathbb{E}\int_0^t {\left\| {v - {v_q}} \right\|_{{Q^{\frac{1}{2}}}\mathcal{H}}^2} \mathrm{d}s + \mathbb{E}\int_0^t {\left\| {{v_q} - {{\hat v}_q}} \right\|_{{Q^{\frac{1}{2}}}\mathcal{H}}^2} \mathrm{d}s + \mathbb{E}\int_0^t {\left\| {\hat v} \right\|_{{{\left( {Q - {Q_q}} \right)}^{\frac{1}{2}}}\mathcal{H}}^2} \mathrm{d}s  \nonumber  \\
\lesssim & \left\| {v - {v_q}} \right\|_{{L^2}}^2 \cdot \mathrm{Tr} \, Q + \hat l_q^2\left\| {{v_q}} \right\|_1^2 \cdot \mathrm{Tr} \, Q + \left\| v \right\|_{{L^2}}^2 \cdot \mathrm{Tr} \left( {Q - {Q_q}} \right) \nonumber \\
\to & 0 , \ \ \mbox{as}\ q \to  \infty  .
\end{align}
It follows from the It\^{o}'s formula that
\begin{align}\label{xiuh-IJHYU-2-ITO-D-A-2}
& \mathbb{E}{\left| {2\int_0^t {{{\left\langle {\mathfrak{W},\mathrm{d}\mathfrak{W}} \right\rangle }_{{\mathbb{T}^2}}}}  + t \cdot \mathrm{Tr} \, Q - \left\| {{\mathfrak{\hat W}_q}} \right\|_{{L^2}}^2} \right|^2} \nonumber \\
= & \mathbb{E}{\left| {\left\| \mathfrak{W} \right\|_{{L^2}}^2 - \left\| {{\mathfrak{\hat W}_q}} \right\|_{{L^2}}^2} \right|^2} \nonumber \\
\leqslant & \mathbb{E}\left\| {\mathfrak{W} - {\mathfrak{\hat W}_q}} \right\|_{{L^2}}^2 \mathbb{E}\left\| {\mathfrak{W} + {\mathfrak{\hat W}_q}} \right\|_{{L^2}}^2 \nonumber \\
\leqslant & \left( {\mathbb{E}\left\| {\mathfrak{W} - {\mathfrak{W}_q}} \right\|_{{L^2}}^2 + \hat l_q^2\left\| {{\mathfrak{W}_q}} \right\|_0^2} \right)\mathbb{E}\left\| {\mathfrak{W} + {\mathfrak{\hat W}_q}} \right\|_{{L^2}}^2 \nonumber \\
\to & 0 , \ \ \mbox{as}\ q \to  \infty  .
\end{align}
\par
We recall that the energy gap is defined by \eqref{LOKJI-A-FGH-A-1}. Owing to the inductive estimate of energy gap in Proposition \ref{RSRI-ESI-GAP}, we have
\begin{equation}\label{xiuh-IJHYU-2-ITO-D-A-2-ksi-1}
\left| {{\mathfrak{e}\left( t \right)} - \left\| {{{\hat u}_q}} \right\|_{{L^2}}^2 - \frac{1}{2}{\delta _{q + 2}} + \left\| {{\mathfrak{\hat W}_q}} \right\|_{{L^2}}^2 + 2\int_0^t {{{\left\langle {{{\hat v}_q},{\text{d}}{\mathfrak{W}_q}} \right\rangle }_{{\mathbb{T}^2}}}} } \right| \lesssim {\delta _{q + 1}}   \to  0 , \ \ \mbox{as}\ q \to  \infty  .
\end{equation}
Combining \eqref{OKIUHBF-A}, \eqref{xiuh-IJHYU-2-ITO-D-A-1}-\eqref{xiuh-IJHYU-2-ITO-D-A-2-ksi-1} and the continuity argument in the temporal interval $\left[ {0,\mathfrak{t}} \right]$, we derive that
\begin{align*}
& \left| {{\mathfrak{e}\left( t \right)} - \left\| u \right\|_{{L^2}}^2 + 2\int_0^t {{{\left\langle {u,\mathrm{d}\mathfrak{W}} \right\rangle }_{{\mathbb{T}^2}}} + t \cdot \mathrm{Tr} \, Q} } \right|  \\
\leqslant & \left| {{\mathfrak{e}\left( t \right)} - \left\| {{{\hat u}_q}} \right\|_{{L^2}}^2 - \frac{1}{2}{\delta _{q + 2}} + \left\| {{\mathfrak{\hat W}_q}} \right\|_{{L^2}}^2 + 2\int_0^t {{{\left\langle {{{\hat v}_q},{\text{d}}{\mathfrak{W}_q}} \right\rangle }_{{\mathbb{T}^2}}}} } \right| \\
& + \left| {2\int_0^t {{{\left\langle {W,\mathrm{d}\mathfrak{W}} \right\rangle }_{{\mathbb{T}^2}}}}  + t \cdot \mathrm{Tr} \, Q - \left\| {{\mathfrak{\hat W}_q}} \right\|_{{L^2}}^2} \right| + \left| {\left\| u \right\|_{{L^2}}^2 - \left\| {{{\hat u}_q}} \right\|_{{L^2}}^2} \right|  \\
& + 2\left| {\int_0^t {{{\left\langle {v,\mathrm{d}\mathfrak{W}} \right\rangle }_{{\mathbb{T}^2}}}}  - \int_0^t {{{\left\langle {{{\hat v}_q},\mathrm{d}{\mathfrak{W}_q}} \right\rangle }_{{\mathbb{T}^2}}}} } \right| + \frac{1}{2}{\delta _{q + 2}} \\
\to & 0  , \ \ \forall \ t \in \left[ {0,\mathfrak{t}} \right] , \ \mbox{as}\ q \to  \infty  ,
\end{align*}
which implies that 
\begin{equation}\label{OKIJDU=SKDLO-Akij-3}
{{\mathfrak{e}\left( t \right)} - \left\| u\left( t \right) \right\|_{{L^2}}^2 + 2\int_0^t {{{\left\langle {u\left( s \right),\mathrm{d}\mathfrak{W}\left( s \right)} \right\rangle }_{{\mathbb{T}^2}}} + t \cdot \mathrm{Tr} \, Q} }  = 0 , \ \ \forall \ t \in \left[ {0,\mathfrak{t}} \right] .
\end{equation}
Consequently, it follows from \eqref{OKIJDU=SKDLO-Akij-3} that the energy inequality \eqref{SHUDJ-NENH-B-SIDIS} holds. The proof of Theorem \ref{CTSE-O-C} is complete. \qed
\subsection{The proof of Theorem \ref{CTSE-O-C-1B}}\label{TPT-A-4}
Since the $\gamma$-H\"{o}lder martingale solution $u\left( {t,x} \right)$ contains Fourier components at infinitely many frequency scales arising from the iterative construction, we employ the Littlewood--Paley projector to formulate its intermittency at a dominant component. The parameter condition $p > \frac{{2M}}{{\gamma M + M - \beta  + \alpha }}$ guarantees that the intermittency intensity of the high-frequency tail does not exceed that of the dominant component, and an appropriate integral scale $l_I$ is chosen to control the low-frequency oscillations. In particular, we introduce a Lagrangian Littlewood--Paley localization through a change of variables, which enables us to estimate the amplitude leakage error and the flow deformation error.
\par
\vspace{1em}
{\textbf{Step 1. Lagrangian localization of Littlewood--Paley decomposition}}. 
\vspace{1em}
\par
Recall that the Dirichlet kernel ${\mathfrak{D}_r}\left( x \right) $ is defined by \eqref{XINHUI-SJI-A}. Inspired by the argument in Buckmaster and Vicol \cite{Zbl1412.35215}, we take the frequency parameter of ${\mathfrak{D}_r}\left( x \right) $ as
\begin{equation}\label{IPO-AH}
r_q \triangleq r = {\lambda _q^M} , \ \ \mbox{for}\ M \in \left( {\frac{1}{2},\frac{b}{2}} \right) .
\end{equation}
Let ${{\tilde \Phi }_k}\left( {t,x} \right)$ be the solution to Eq.\eqref{Nashsteps-A-P-1}. We perform the change of variables
\begin{equation}
y=\widetilde\Phi_k(t,x)   , \ \ \ \ \ \ x=\widetilde\Psi_k(t,y) = \widetilde\Phi_k^{-1}(t,y) . \label{JDIJKFHG-F-1}
\end{equation}
According to the definitions \eqref{Nashsteps-A-P-3}, the principal part of Nash perturbation is rewritten as
\begin{equation}
w_{q + 1}^{\left( p \right)}\left( {t,y} \right) \triangleq \sum\limits_{n = 0}^\Gamma  {\sum\limits_{k \in {\mathbb{Z}_{q,n}}} {\sum\limits_{\xi  \in \Lambda } {{A_{\xi ,k,n,{q+1}}}\left( {t,y} \right){\mathbb{W}_\xi }\left( {{\lambda _{q + 1}}y} \right)} } } , \label{eq:principal-perturbation}
\end{equation}
where the amplitude reads
\begin{equation*}
{A_{\xi ,k,n,{q+1}}}\left( {t,y} \right) = {g_{\xi ,k,n + 1}}\left( {{\mu _{q + 1}}t} \right){{\tilde \eta }_{\xi ,k,n}}\left( {{{\tilde \Psi }_k}(t,y)} \right)  {{\left( {\nabla {{\tilde \Phi }_k}} \right)}^{ - 1}} \left( {{{\tilde \Psi }_k}(t,y)} \right)  .
\end{equation*}
By the definitions \eqref{XINHUI-SJI-A}, \eqref{IBW-GJI-A-1} and \eqref{IBW-GJI-A-2}, we observe that the Fourier modes of $\mathbb{W}_\xi(t,\lambda_{q+1}y)$ are contained in
\begin{equation*}
\Xi = \left\{ {{\lambda _{q + 1}}\left( { \pm {\xi ^ \bot } + \sigma \left( {{k_1}\tilde \xi  + {k_2}{{\tilde \xi }^ \bot }} \right)} \right):k = \left( {{k_1},{k_2}} \right) \in {\mathfrak{H}_{{r_q}}}} \right\}  .
\end{equation*}
In the inductive estimates of the intermittent Dirichlet kernel ${\mathfrak{D}_{\tilde \mu }}$, we take the strength parameter $\sigma=C(\Lambda)\lambda_{q+1}^{-\frac12}$ (see Lemma \ref{WPD-NHS-HUD-BB-2}). Then, owing to the parameter condition $M < \frac{b}{2}$, we have
\begin{equation*}
\sigma {r_q} \lesssim \lambda _q^{M - \frac{b}{2}} \to 0 , \ \ \mbox{as}\ q \to \infty .
\end{equation*}
Therefore, there exist constants $0 < c_\Lambda < C_\Lambda < \infty$ such that
\begin{equation}
\Xi \subset \left\{ {\tilde \xi  \in {\mathbb{Z}^2}: {c_\Lambda } {\lambda _{q + 1}} \leqslant  \left| {\tilde \xi } \right| \leqslant  {C_\Lambda }{\lambda _{q + 1}}} \right\}  . \label{eq:annulus-location}
\end{equation}
\par
Let $\chi_\Lambda  \left( s \right)$ be the differentiable function with compact support satisfying
\begin{equation*}
{\chi _\Lambda }(s) = 1, \ \ \forall \, s \in [{c_\Lambda },{C_\Lambda }]  .
\end{equation*}
Choosing the Littlewood--Paley scale $R_{q+1}:=\lambda_{q+1}$, the homogeneous Littlewood--Paley projector is defined by
\begin{equation}
{{\dot \Delta }_{{R_{q + 1}}}} f = {\mathfrak{F}^{ - 1}}\left( {{\chi _\Lambda }\left( {\frac{{\left| {\tilde \xi } \right|}}{{{\lambda _{q + 1}}}}} \right) \mathfrak{F}f\left( {\tilde \xi } \right)} \right) , \label{eq:LP-projector}
\end{equation}
where ${\mathfrak{F}}$ is the Fourier transform, and ${\mathfrak{F}^{ - 1}}$ means its inverse. It is obvious that
\begin{equation*}
\operatorname{supp} \mathfrak{F}\left( {{{\dot \Delta }_{{R_{q + 1}}}} f} \right) \subset \left\{ {\tilde \xi :  \frac{1}{2}{c_\Lambda }{\lambda _{q + 1}} \leqslant |\tilde \xi | \leqslant 2{C_\Lambda }{\lambda _{q + 1}}} \right\}  .
\end{equation*}
We apply the homogeneous Littlewood--Paley projector ${{\dot \Delta }_{{R_{q + 1}}}}$ in the Lagrangian coordinate $y$ to the principal perturbation $w_{q + 1}^{\left( p \right)}\left( {t,y} \right)$ after the change of variables, that is,
\begin{equation}
{{\dot \Delta }_{{R_{q + 1}}}}w_{q + 1}^{\left( p \right)}\left( y \right) = {\mathfrak{F}^{ - 1}}\left( {{\chi _\Lambda }\left( {\frac{{\left| {\tilde \xi } \right|}}{{{\lambda _{q + 1}}}}} \right)\mathfrak{F}w_{q + 1}^{\left( p \right)} \left( {\tilde \xi } \right)} \right)  \left( y \right).  \label{SMDNJDHIG-SCA-1}
\end{equation}
To track the new errors generated by the Lagrangian transformation in \eqref{SMDNJDHIG-SCA-1}, we further define the Lagrangian conjugation of ${{\dot \Delta }_{{R_{q + 1}}}}$ by
\begin{equation}
\left( {\dot \Delta _{{R_{q + 1}}}^{Lag}w_{q + 1}^{\left( p \right)}} \right)\left( x \right) \triangleq \left[ {{{\dot \Delta }_{{R_{q + 1}}}}\left( {w_{q + 1}^{\left( p \right)} \circ {{\tilde \Psi }_k}} \right)} \right]\left( y \right)  .   \label{KDIJFHG=VDF-1}
\end{equation}
\par
\begin{remark}
\rm{Since the Lagrangian transformation \eqref{JDIJKFHG-F-1} preserves volume invariance
\begin{equation}
\det \nabla {{\tilde \Phi }_k}(t,x) = \det \nabla {{\tilde \Psi }_k}(t,y) = 1 , \label{volume-preserving-flow}
\end{equation}
the operator ${\dot \Delta _{R_{q + 1}}^{Lag}}$ represents the standard Littlewood--Paley projection performed in the Lagrangian coordinates and transported back through the volume-preserving flow. Therefore, the operator ${\dot \Delta _{R_{q + 1}}^{Lag}}$ further identifies the amplitude leakage error of ${A_{\xi ,k,n,{q+1}}}$ and the flow deformation error of $\widetilde\Phi_k(t,x)$.}
\end{remark}
\par
\begin{lemma}\label{lem:Lag-LP-stability}
Under the stochastic and intermittent Newton--Nash iteration framework, the operator ${\dot \Delta _{R_{q + 1}}^{Lag}}$ satisfies
\begin{equation}
{\left\| {\dot \Delta _{{R_{q + 1}}}^{Lag}w_{q + 1}^{\left( p \right)} - w_{q + 1}^{\left( p \right)}} \right\|_{{L^p}}} \lesssim \left( {\lambda _{q + 1}^{ - \alpha } + \frac{{{\lambda _q}}}{{{\lambda _{q + 1}}}}} \right)\delta _{q + 1}^{\frac{1}{2}}\lambda _{q + 1}^{ - \alpha }r_q^{1 - \frac{2}{p}}   .  \label{lem:Lag-LP-stability-SLKDAOFJHG-X-1}
\end{equation}
\end{lemma}
\par
\noindent{\textbf{Proof}}. 
In view of Eq.\eqref{Nashsteps-A-P-1}, we have
\begin{equation*}
\left\{ {\begin{array}{*{20}{l}}
{\partial _t}\nabla {{\tilde \Phi }_k} + \left( {{{\bar v}_{q,\Gamma }} \cdot \nabla } \right)\nabla {{\tilde \Phi }_k} =  - \left( {\nabla {{\bar v}_{q,\Gamma }}} \right)\nabla {{\tilde \Phi }_k}, \\ 
\nabla {{\tilde \Phi }_k}\left( {{t_k}} \right) = \mathrm{Id}  .
\end{array}} \right.
\end{equation*}
Then ${{\bar D}_{t,\Gamma }}\nabla {{\tilde \Phi }_k} =  - \left( {\nabla {{\bar v}_{q,\Gamma }}} \right)\nabla {{\tilde \Phi }_k}$, which implies that
\begin{equation*}
{{\bar D}_{t,\Gamma }}\left( {\nabla {{\tilde \Phi }_k} - \mathrm{Id}} \right) = {{\bar D}_{t,\Gamma }}\nabla {{\tilde \Phi }_k} =  - \left( {\nabla {{\bar v}_{q,\Gamma }}} \right)\left( {\nabla {{\tilde \Phi }_k} - \mathrm{Id}} \right) - \nabla {{\bar v}_{q,\Gamma }}  .
\end{equation*}
Employing the argument along characteristics of the Lagrangian flow \eqref{KDKJFIG-DNGJU-1-1}, we obtain that
\begin{equation*}
{\left\| {\nabla {{\tilde \Phi }_k}\left( t \right) - \mathrm{Id}} \right\|_0} \lesssim \int_{{t_k}}^t {{{\left\| {{{\bar v}_{q,\Gamma }}} \right\|}_1}\left( s \right){{\left\| {\nabla {{\tilde \Phi }_k}\left( s \right) - \mathrm{Id}} \right\|}_0} + {{\left\| {{{\bar v}_{q,\Gamma }}\left( s \right)} \right\|}_1}} \, \mathrm{d}s .
\end{equation*}
It follows from \eqref{WPD-NHS-MAIN-AA-P-3} and Gronwall's inequality that
\begin{align}
{\left\| {\nabla {{\tilde \Phi }_k} - \mathrm{Id}} \right\|_0} \lesssim & \exp \left( {\left( {t - {t_k}} \right){{\left\| {{{\bar v}_{q,\Gamma }}} \right\|}_1}} \right) - 1 \nonumber \\
\lesssim & {\tau _q} \exp \left( 1 \right){\left\| {{{\bar v}_{q,\Gamma }}} \right\|_1} \nonumber \\
\lesssim & \lambda _q^{ - \alpha }  .  \label{NDHFUYTS-XD-F-1}
\end{align}
\par
For every fixed $k$, $n$ and $\xi$, we denote $w_{\xi ,k,n}^{\left( p \right)}\left( {t,y\left( x \right)} \right) \triangleq {A_{\xi ,k,n,{q+1}}}\left( {t,y\left( x \right)} \right){\mathbb{W}_\xi }\left( {{\lambda _{q + 1}}y\left( x \right)} \right)$. Let ${K_{{\lambda _{q + 1}}}}$ be the periodic convolution kernel corresponding to the Fourier multiplier ${{\chi _\Lambda }\left( {\frac{{\left| {\tilde \xi } \right|}}{{{\lambda _{q + 1}}}}} \right)}$ such that
\begin{equation*}
{{\dot \Delta }_{{R_{q + 1}}}}w_{q + 1}^{\left( p \right)}\left( {y\left( x \right)} \right) = \int_{{\mathbb{T}^2}} {{K_{{\lambda _{q + 1}}}}} (y\left( x \right) - y\left( z \right)) \cdot  w_{q + 1}^{\left( p \right)}\left( {y\left( z \right)} \right)\, \mathrm{d}z .
\end{equation*}
Then, using the projector \eqref{KDIJFHG=VDF-1}, we have
\begin{equation*}
\dot \Delta _{{R_{q + 1}}}^{Lag}   w_{\xi ,k,n}^{(p)}(x) = \int_{{\mathbb{T}^2}} {{K_{{\lambda _{q + 1}}}}} \left( {{{\tilde \Phi }_k}(x) - {{\tilde \Phi }_k}(z)} \right)w_{\xi ,k,n}^{(p)}(z) \, \mathrm{d}z  , 
\end{equation*}
which implies that
\begin{equation*}
\left( {\dot \Delta _{{R_{q + 1}}}^{Lag} - {{\dot \Delta }_{{R_{q + 1}}}}} \right)w_{\xi ,k,n}^{(p)}(x) = \int_{{\mathbb{T}^2}} {\left[ {{K_{{\lambda _{q + 1}}}}\left( {{{\tilde \Phi }_k}(x) - {{\tilde \Phi }_k}(z)} \right) - {K_{q + 1}}(x - z)} \right]} w_{\xi ,k,n}^{(p)}(z) \, \mathrm{d}z . 
\end{equation*}
By the fundamental theorem of calculus and \eqref{NDHFUYTS-XD-F-1}, we obtain
\begin{equation*}
{{{\tilde \Phi }_k}(x) - {{\tilde \Phi }_k}(z) - x + z} =  {\int_0^1 {\left[ {\nabla {{\tilde \Phi }_k}\left( {z + s\left( {x - z} \right)} \right) - {\text{Id}}} \right]\left( {x - z} \right)} \, \mathrm{d}s}  ,
\end{equation*}
which implies that
\begin{equation}
\left( {1 - \lambda _{q + 1}^{ - \alpha }} \right)\left| {x - z} \right| \lesssim \left| {{{\tilde \Phi }_k}\left( x \right) - {{\tilde \Phi }_k}\left( z \right) - x + z} \right| \lesssim \lambda _{q + 1}^{ - \alpha }\left| {x - z} \right|  .\label{FLOW-DISTORTION-DIRECT}
\end{equation}
\par
Since $\chi_\Lambda$ is smooth and compactly supported on an annulus, the convolution kernel ${K_{{\lambda _{q + 1}}}}$ is smooth on ${\mathbb T^2}$ and satisfies (see \cite{Zbl1227.35004})
\begin{equation*}
\left| {\nabla {K_{q + 1}}\left( z \right)} \right| \lesssim \lambda _{q + 1}^3{\left( {1 + {\lambda _{q + 1}}\left| z \right|} \right)^{ - {N^*}}}, \ \ \forall \ N^ * \in \mathbb{Z}^+ .
\end{equation*}
Applying the mean value theorem and \eqref{FLOW-DISTORTION-DIRECT}, we infer that
\begin{align*}
& \left| {{K_{{\lambda _{q + 1}}}}\left( {{{\tilde \Phi }_k}\left( x \right) - {{\tilde \Phi }_k}\left( z \right)} \right) - {K_{q + 1}}\left( {x - z} \right)} \right|  \\
\lesssim & \left| {{{\tilde \Phi }_k}\left( x \right) - {{\tilde \Phi }_k}\left( z \right) - x + z} \right|\mathop {\sup }\limits_{s \in \left[ {0,1} \right]} \left| {\nabla {K_{q + 1}}\left( {x - z + s\left[ {{{\tilde \Phi }_k}\left( x \right) - {{\tilde \Phi }_k}\left( z \right) - x + z} \right]} \right)} \right|  \\
\lesssim & \lambda _{q + 1}^{ - \alpha }\lambda _{q + 1}^3\left| {x - z} \right|{\left( {1 + {\lambda _{q + 1}}\left| {x - z} \right|} \right)^{ - {N^*}}}  .
\end{align*}
Taking $ N^ *  > 3$, we have $\mathop {\sup }\limits_{x \in {\mathbb{T}^2}} \left| {\int_{{\mathbb{T}^2}} {\lambda _{q + 1}^3\left| {x - z} \right|{{\left( {1 + {\lambda _{q + 1}}\left| {x - z} \right|} \right)}^{ - {N^*}}}} \, \mathrm{d}z} \right| \lesssim 1$. Consequently,
\begin{equation*}
\mathop {\sup }\limits_{x \in {\mathbb{T}^2}} \int_{{\mathbb{T}^2}} {\left| {{K_{{\lambda _{q + 1}}}}\left( {{{\tilde \Phi }_k}\left( x \right) - {{\tilde \Phi }_k}\left( z \right)} \right) - {K_{q + 1}}\left( {x - z} \right)} \right|} {\mkern 1mu} {\text{d}}z \lesssim \lambda _{q + 1}^{ - \alpha } .
\label{LP-KERNEL-L1-DIRECT}
\end{equation*}
It follows from H\"{o}lder inequality and Fubini theorem that
\begin{equation}
{\left\| {\dot \Delta _{R_{q + 1}}^{Lag}w_{\xi ,k,n}^{\left( p \right)} - {{\dot \Delta }_{{R_{q + 1}}}}w_{\xi ,k,n}^{\left( p \right)}} \right\|_{{L^p}}} \lesssim \lambda _{q + 1}^{ - \alpha }{\left\| {w_{\xi ,k,n}^{\left( p \right)}} \right\|_{{L^p}}}  .  \label{SBHDUFHJHF-Z-1}
\end{equation}
\par
In view of \eqref{volume-preserving-flow}, the change of variables $y=\widetilde\Phi_k(t,x)$ preserves the Lebesgue measure. Therefore,
\begin{equation}
{\left\| {{\mathbb{W}_\xi }\left( {{\lambda _{q + 1}}{{\tilde \Phi }_k}} \right)} \right\|_{L_x^p}} = {\left\| {{\mathbb{W}_\xi }\left( {{\lambda _{q + 1}}y} \right)} \right\|_{L_y^p}}. \label{eq:Lp-preservation}
\end{equation}
Owing to \eqref{WIEJDUHFJ-DFIJDJ-A-2}, we have 
\begin{equation}
{\left\| {{\mathbb{W}_\xi }\left( {{\lambda _{q + 1}}{{\tilde \Phi }_k}} \right)} \right\|_{L^p}} \sim r_q^{1 - \frac{2}{p}} .   \label{eq:Lp-preservation-SKJKJFI}
\end{equation}
Using \eqref{eq:Lp-preservation-SKJKJFI} and the fact that ${\left\| {{A_{\xi ,k,n,{q+1}}}} \right\|_N} \lesssim \delta _{q + 1}^{\frac{1}{2}}\lambda _{q + 1}^{ - \alpha }\lambda _q^N$, we derive that
\begin{equation}
{\left\| {w_{\xi ,k,n}^{\left( p \right)}} \right\|_{{L^p}}} \lesssim {\left\| {{A_{\xi ,k,n,{q+1}}}} \right\|_0}{\left\| {{\mathbb{W}_\xi }\left( {{\lambda _{q + 1}}y} \right)} \right\|_{L_y^p}} \lesssim \delta _{q + 1}^{\frac{1}{2}}\lambda _{q + 1}^{ - \alpha }r_q^{1  - \frac{2}{p}} . \label{LP-FLOW-COMPONENT-FINAL}
\end{equation}
Therefore, summing over the finitely many active indices $(\xi,k,n)$ and combining \eqref{SBHDUFHJHF-Z-1} and \eqref{LP-FLOW-COMPONENT-FINAL}, we infer that
\begin{equation}
{\left\| {\dot \Delta _{R_{q + 1}}^{Lag}w_{q + 1}^{\left( p \right)} - {{\dot \Delta }_{{R_{q + 1}}}}w_{q + 1}^{\left( p \right)}} \right\|_{{L^p}}} \lesssim \delta _{q + 1}^{\frac{1}{2}}\lambda _{q + 1}^{ - 2\alpha  }r_q^{1  - \frac{2}{p}} . \label{lem:Lag-LP-stability-thdkfi-a}
\end{equation}
\par
Since ${{\dot \Delta }_{{R_{q + 1}}}}{\mathbb{W}_\xi }\left( {{\lambda _{q + 1}}y} \right) = {\mathbb{W}_\xi }\left( {{\lambda _{q + 1}}y} \right)$, we have
\begin{align*}
{{\dot \Delta }_{{R_{q + 1}}}}w_{\xi ,k,n}^{\left( p \right)}\left( {t,y\left( x \right)} \right) - w_{\xi ,k,n}^{\left( p \right)}\left( {t,y\left( x \right)} \right) = & {{\dot \Delta }_{{R_{q + 1}}}}v - {A_{\xi ,k,n,{q+1}}}{{\dot \Delta }_{{R_{q + 1}}}}{\mathbb{W}_\xi } \\
\triangleq & \left[ {{{\dot \Delta }_{{R_{q + 1}}}},{A_{\xi ,k,n,{q+1}}}} \right]{\mathbb{W}_\xi }  .
\end{align*}
Employing the standard Littlewood--Paley commutator estimate (see \cite{Zbl1227.35004}), we obtain
\begin{equation}
{\left\| {{{\dot \Delta }_{{R_{q + 1}}}}w_{q + 1}^{\left( p \right)} - w_{q + 1}^{\left( p \right)}} \right\|_{{L^p}}} \lesssim \lambda _{q + 1}^{ - 1}{\left\| {{A_{\xi ,k,n,{q+1}}}} \right\|_1}{\left\| {{\mathbb{W}_\xi }} \right\|_{{L^p}}} \lesssim \delta _{q + 1}^{\frac{1}{2}}{\lambda _q}\lambda _{q + 1}^{ - \alpha  - 1}r_q^{1 - \frac{2}{p}}  .  \label{JKGISDNBVCXGSG-A-1}
\end{equation}
Therefore, it follows from \eqref{lem:Lag-LP-stability-thdkfi-a} and \eqref{JKGISDNBVCXGSG-A-1} that \eqref{lem:Lag-LP-stability-SLKDAOFJHG-X-1} holds.
\par
\vspace{1em}
{\textbf{Step 2. Frequency-localized intermittency of martingale solutions}}. 
\vspace{1em}
\par
\begin{lemma}\label{thm:frequency_intermittency}
Let $u\left( {t,x} \right)$ be the martingale solution to Eq.\eqref{2DEEUGE1} obtained by Theorem \ref{CTSE-O-C}. Then, for any $1 < p_2 < p_1$ and $t \in \left[ {0,\mathfrak{t}} \right]$, one has
\begin{equation*}
\frac{{{{\left\| { {{{\dot \Delta }_{{R_{q + 1}}}}} u\left( t \right)} \right\|}_{{L^{{p_1}}}}}}}{{{{\left\| { {{{\dot \Delta }_{{R_{q + 1}}}}} u\left( t \right)} \right\|}_{{L^{{p_2}}}}}}} \gtrsim r_q^{2\left( {\frac{1}{{{p_2}}} - \frac{1}{{{p_1}}}} \right)} \to \infty , \ \  \mbox{as}\ q \to \infty ,
\end{equation*}
and there exists an inertial range $\left[ {{l_D},{l_I}} \right]$ with the integral scale ${l_I}\left( {q} \right) \to 0$ as $q \to \infty$ such that
\begin{equation*}
\frac{{{{\left\| { {{{\dot \Delta }_{{R_{q + 1}}}}} \left( {u\left( {x + l\mathbf{n}} \right) - u\left( x \right)} \right)} \right\|}_{{L^{{p_1}}}}}}}{{{{\left\| {  {{{\dot \Delta }_{{R_{q + 1}}}}} \left( {u\left( {x + l\mathbf{n}} \right) - u\left( x \right)} \right)} \right\|}_{{L^{{p_2}}}}}}} \gtrsim r_q^{2(\frac{1}{{{p_2}}} - \frac{1}{{{p_1}}})} \to \infty  , \ \  \forall \, l \in \left[ {{l_D},{l_I}} \right] , \  \mbox{as}\ q \to \infty .
\end{equation*}
\end{lemma}
\par
\noindent{\textbf{Proof}}. 
From the Newton--Nash construction \eqref{NISTYU-SKO-A-PPA-2}, we have $u = u_0  + \sum\limits_{j \in {\mathbb{Z}^ + } } {\left( {{w_{j }} + \mathfrak{W}_{j }^{loc}} \right)} $. We first isolate the contribution of the principal intermittent perturbation $w_{q + 1}^{\left( p \right)}$. Using the operator ${\dot \Delta _{R_{q + 1}}^{Lag}}$ to decompose $u$, we have
\begin{equation}
{{{\dot \Delta }_{{R_{q + 1}}}}}  u = w_{q + 1}^{\left( p \right)} + E_{q + 1}^{flow} + E_{q + 1}^{amp} + E_{q + 1}^{osc} ,  \label{thm:frequency_intermittency-a-1-proof-6}
\end{equation}
where the flow deformation error is
\begin{equation*}
E_{q + 1}^{flow} =  {{\dot \Delta }_{{R_{q + 1}}}}w_{q + 1}^{\left( p \right)}   - \dot \Delta _{{R_{q + 1}}}^{Lag}w_{q + 1}^{\left( p \right)} ,
\end{equation*}
the amplitude leakage error is
\begin{equation*}
E_{q + 1}^{amp} = \dot \Delta _{{R_{q + 1}}}^{Lag}w_{q + 1}^{\left( p \right)} - w_{q + 1}^{\left( p \right)} ,
\end{equation*}
and the oscillation error is
\begin{equation*}
E_{q + 1}^{cor} =  {\dot \Delta _{R_{q + 1}}^{Lag}}  \left( {w_{q + 1}^{\left( c \right)} + w_{q + 1}^{\left( i \right)} + w_{q + 1}^{\left( \Gamma  \right)} + \mathfrak{W}_{q + 1}^{loc}} \right)  + {{{\dot \Delta }_{{R_{q + 1}}}}} \sum\limits_{j \ne q + 1} {{w_j}} . 
\end{equation*}
\par
Since ${\nabla {{\tilde \Phi }_k}\nabla \tilde \Phi _k^ \bot  - \delta _{q + 1,n}^{ - 1}\nabla {{\tilde \Phi }_k}{{\bar R}_{q,n}}\nabla \tilde \Phi _k^ \bot } \neq 0$, the smooth function ${\gamma _\xi }$ determined by Lemma \ref{GLTNNRS-A} satisfies
\begin{equation*}
{\gamma _\xi }\left( {\nabla {{\tilde \Phi }_k}\nabla \tilde \Phi _k^ \bot  - \delta _{q + 1,n}^{ - 1}\nabla {{\tilde \Phi }_k}{{\bar R}_{q,n}}\nabla \tilde \Phi _k^ \bot } \right)  \neq 0 .
\end{equation*}
In view of the definition \eqref{Nashsteps-A-P-2}, we have
\begin{equation}
\delta _{q + 1,\Gamma }^{\frac{1}{2}} \lesssim \left| {{{\tilde \eta }_{\xi ,k,n}}} \right| \lesssim \delta _{q + 1}^{\frac{1}{2}} ,   \label{VHUYTR-A-1}
\end{equation}
where $\delta _{q + 1,\Gamma }$ is defined by \eqref{parameters-S-4}. It follows from \eqref{VHUYTR-A-1} and the definition \eqref{Nashsteps-A-P-3} that
\begin{equation*}
\delta _{q + 1,\Gamma }^{\frac{1}{2}} {\left\| {\sum\limits_{\xi  \in \Lambda } {{\mathbb{W}_\xi }} } \right\|_{{L^p}}} \lesssim {\left\| {w_{q + 1}^{\left( p \right)}} \right\|_{{L^p}}} \lesssim \delta _{q + 1}^{\frac{1}{2}}{\left\| {\sum\limits_{\xi  \in \Lambda } {{\mathbb{W}_\xi }} } \right\|_{{L^p}}} .
\end{equation*}
Due to \eqref{eq:Lp-preservation-SKJKJFI}, we deduce that
\begin{equation}
\delta _{q + 1,\Gamma }^{\frac{1}{2}}  {r_q^{1 - \frac{2}{p}}} \lesssim {\left\| {w_{q + 1}^{\left( p \right)}} \right\|_{{L^p}}} \lesssim \delta _{q + 1}^{\frac{1}{2}}  {r_q^{1 - \frac{2}{p}}}  .    \label{VHUYTR-A-1-KLBHTF-A-1}
\end{equation}
According to \eqref{lem:Lag-LP-stability-SLKDAOFJHG-X-1}, \eqref{lem:Lag-LP-stability-thdkfi-a} and \eqref{VHUYTR-A-1-KLBHTF-A-1}, we obtain 
\begin{equation*}
\frac{{{{\left\| {\dot \Delta _{{R_{q + 1}}}^{Lag}w_{q + 1}^{\left( p \right)} - w_{q + 1}^{\left( p \right)}} \right\|}_{{L^p}}}}}{{{{\left\| {w_{q + 1}^{\left( p \right)}} \right\|}_{{L^p}}}}} \lesssim \left( {\lambda _{q + 1}^{ - \alpha } + \frac{{{\lambda _q}}}{{{\lambda _{q + 1}}}}} \right)\delta _{q + 1}^{\frac{1}{2}}\delta _{q + 1,\Gamma }^{ - \frac{1}{2}}\lambda _{q + 1}^{ - \alpha } \to 0 ,\ \ \mbox{as}\ q \to \infty  ,
\end{equation*}
and
\begin{equation*}
\frac{{{{\left\| {\dot \Delta _{{R_{q + 1}}}^{Lag}w_{q + 1}^{\left( p \right)} - {{\dot \Delta }_{{R_{q + 1}}}}w_{q + 1}^{\left( p \right)}} \right\|}_{{L^p}}}}}{{{{\left\| {w_{q + 1}^{\left( p \right)}} \right\|}_{{L^p}}}}} \lesssim \delta _{q + 1}^{\frac{1}{2}}\delta _{q + 1,\Gamma }^{ - \frac{1}{2}}\lambda _{q + 1}^{ - 2\alpha } \to 0   ,\ \ \mbox{as}\ q \to \infty  ,
\end{equation*}
which implies that the flow deformation error $E_{q + 1}^{flow}$ and the amplitude leakage error $E_{q + 1}^{amp}$ admit
\begin{equation}
{\left\| {E_{q + 1}^{flow}} \right\|_{{L^p}}} + {\left\| {E_{q + 1}^{amp}} \right\|_{{L^p}}} =    {o_{q \to \infty }}\left( {{{\left\| {w_{q + 1}^{\left( p \right)}} \right\|}_{{L^p}}}} \right)  .   \label{VHUYTR-A-1-KLBHTF-A-145}
\end{equation}
\par
It follows from \eqref{Nashsteps-A-P-4} and \eqref{WIEJDUHFJ-DFIJDJ-A-1} that
\begin{align}
{\left\| {w_{q + 1}^{\left( c \right)}} \right\|_{{L^p}}} \lesssim & \lambda _{q + 1}^{ - 1}\delta _{q + 1}^{\frac{1}{2}}\left( {{\lambda _q}{{\left\| {{\mathfrak{D}_{\tilde \mu }}} \right\|}_{{L^p}}} + {{\left\| {{\mathfrak{D}_{\tilde \mu }}} \right\|}_{{W^{1,p}}}}} \right) \nonumber \\
\lesssim & \lambda _{q + 1}^{  - 1}  \delta _{q + 1}^{1/2}r_q^{2 - \frac{2}{p}} . \label{thm:frequency_intermittency-a-1-proof-1}
\end{align}
Combining \eqref{VHUYTR-A-1-KLBHTF-A-1} and \eqref{thm:frequency_intermittency-a-1-proof-1}, we obtain
\begin{equation*}
\frac{{{{\left\| {w_{q + 1}^{\left( c \right)}} \right\|}_{{L^p}}}}}{{{{\left\| {w_{q + 1}^{\left( p \right)}} \right\|}_{{L^p}}}}} \lesssim {r_q}\lambda _q^{\left( {\beta  - \frac{1}{3}} \right)\Gamma }\lambda _{q + 1}^{\left( {\frac{1}{3} - \beta } \right)\Gamma  - 1} \lesssim \lambda _{q + 1}^{M - 1} .
\end{equation*}
Furthermore, according to \eqref{THENDM-DKF}, \eqref{NISTYU-SKO-A-PPA-FJU-sd} and \eqref{WPD-NP-main-BTO-1}, we have
\begin{equation}
{\left\| {w_{q + 1}^{\left( i \right)}} \right\|_{{L^p}}} + {\left\| {w_{q + 1}^{\left( \Gamma  \right)}} \right\|_{{L^p}}} + {\left\| {\mathfrak{W}_{q + 1}^{loc}} \right\|_{{L^p}}} \lesssim  {{\tilde \mu }^{ - 1}}{\delta _{q + 1}}r_q^{2 - \frac{2}{p}} + \mu _{q + 1}^{ - 1}{\delta _{q + 1}}{\lambda _q}l_q^{ - \alpha } + \delta _q^\alpha \delta _{q + 2}^{\frac{1}{2}}   . \label{thm:frequency_intermittency-a-1-proof-10}
\end{equation}
Since $p > \frac{{2M}}{{\gamma M + M - \beta  + \alpha }} \geqslant \frac{{2M}}{{2\alpha \beta  + M}}$, we have $M\left( {\frac{2}{p} - 1} \right) - \alpha < M\left( {\frac{2}{p} - 1} \right) - 2\alpha \beta < 0$, which implies that
\begin{equation*}
\mu _{q + 1}^{ - 1}{\delta _{q + 1}}{\lambda _q}l_q^{ - \alpha }\delta _{q + 1,\Gamma }^{ - \frac{1}{2}}r_q^{\frac{2}{p} - 1} \lesssim \lambda _q^{M\left( {\frac{2}{p} - 1} \right) - \alpha } \to 0   ,\ \ \mbox{as}\ q \to \infty ,
\end{equation*}
and
\begin{equation*}
\delta _q^\alpha \delta _{q + 2}^{\frac{1}{2}}\delta _{q + 1,\Gamma }^{ - \frac{1}{2}}r_q^{\frac{2}{p} - 1} \lesssim \lambda _q^{M\left( {\frac{2}{p} - 1} \right) - 2\alpha \beta }  \to 0   ,\ \ \mbox{as}\ q \to \infty .
\end{equation*}
Owing to
\begin{equation*}
{{\tilde \mu }^{ - 1}}{\delta _{q + 1}}\delta _{q + 1,\Gamma }^{ - \frac{1}{2}}{r_q} \lesssim \lambda _q^{M - 1} \to 0   ,\ \ \mbox{as}\ q \to \infty ,
\end{equation*}
we deduce that
\begin{equation}
{\left\| {w_{q + 1}^{\left( c \right)}} \right\|_{{L^p}}} + {\left\| {w_{q + 1}^{\left( i \right)}} \right\|_{{L^p}}} + {\left\| {w_{q + 1}^{\left( \Gamma  \right)}} \right\|_{{L^p}}} + {\left\| {\mathfrak{W}_{q + 1}^{loc}} \right\|_{{L^p}}} = {o_{q \to \infty }}\left( {{{\left\| {w_{q + 1}^{\left( p \right)}} \right\|}_{{L^p}}}} \right)  .  \label{thm:frequency_intermittency-a-1-proof-3}
\end{equation}
\par
For $j\neq q+1$, the frequency separation property implies ${{\dot \Delta }_{{R_{q + 1}}}}{\mathbb{W}_\xi }\left( {{\lambda _j} \cdot } \right) = 0$. We then have
\begin{align*}
{{\dot \Delta }_{{R_{q + 1}}}} w_{\xi ,k,n,j}^{(p)} = & {{\dot \Delta }_{{R_{q + 1}}}}\left( {{A_{\xi ,k,n,j}}{\mathbb{W}_\xi }\left( {{\lambda _j} \cdot } \right)} \right) \\
= & {{\dot \Delta }_{{R_{q + 1}}}}\left( {{A_{\xi ,k,n,j}}{\mathbb{W}_\xi }\left( {{\lambda _j} \cdot } \right)} \right) - {A_{\xi ,k,n,j}}{{\dot \Delta }_{{R_{q + 1}}}}{\mathbb{W}_\xi }\left( {{\lambda _j} \cdot } \right)  \\
\triangleq & \left[ {{{\dot \Delta }_{{R_{q + 1}}}},{A_{\xi ,k,n,j}}} \right]{\mathbb{W}_\xi }\left( {{\lambda _j} \cdot } \right)  ,
\end{align*}
where $w_{\xi ,k,n,j}^{(p)}$ represents the $j$-th component of $w_{\xi ,k,n}^{(p)}$. It follows from the standard Littlewood--Paley commutator estimate that
\begin{align*}
{\left\| {\left[ {{{\dot \Delta }_{{R_{q + 1}}}},{A_{\xi ,k,n,j}}} \right]{\mathbb{W}_\xi }\left( {{\lambda _j} \cdot } \right)} \right\|_{{L^p}}} \lesssim & \lambda _{q + 1}^{ - 1}{\left\| {{A_{\xi ,k,n,j}}} \right\|_1}{\left\| {{\mathbb{W}_\xi }\left( {{\lambda _j}y} \right)} \right\|_{{L^p}}}     \\
\lesssim & \lambda _{q + 1}^{ - 1}\delta _j^{\frac{1}{2}}{\lambda _{j - 1}}\lambda _j^{ - \alpha }r_{j - 1}^{1 - \frac{2}{p}}  , 
\end{align*}
which implies that
\begin{align*}
\frac{{{{\left\| {{{\dot \Delta }_{{R_{q + 1}}}} w_j^{\left( p \right)}} \right\|}_{{L^p}}}}}{{{{\left\| {w_{q + 1}^{\left( p \right)}} \right\|}_{{L^p}}}}} \lesssim & \delta _j^{\frac{1}{2}}{\lambda _{j - 1}}\lambda _j^{ - \alpha }r_{j - 1}^{1 - \frac{2}{p}}\delta _{q + 1,\Gamma }^{ - \frac{1}{2}}\lambda _{q + 1}^{ - 1}r_q^{\frac{2}{p} - 1}  \\
\lesssim & \lambda _j^{1 - \alpha  - \beta  + M\left( {1 - \frac{2}{p}} \right)}\lambda _q^{\beta  - 1 - M\left( {1 - \frac{2}{p}} \right)}  \\
\to & 0 ,\ \ \mbox{as}\ q \to \infty  , \ \forall \ j \ne q + 1.  
\end{align*}
Therefore,
\begin{equation}
\frac{{{{\left\| {{{\dot \Delta }_{{R_{q + 1}}}} \sum\limits_{j \ne q + 1} w_j^{\left( p \right)}} \right\|}_{{L^p}}}}}{{{{\left\| {w_{q + 1}^{\left( p \right)}} \right\|}_{{L^p}}}}} \to  0 ,\ \ \mbox{as}\ q \to \infty  .  \label{SMNMFGKJGHG-CDF-10}
\end{equation}
Notice that the corrector part $w_{j}^{\left( c \right)}$, the intermittency part $w_{j}^{\left( i \right)}$ and the Newton perturbation $w_{j}^{\left( \Gamma \right)}$ is controlled by the principal part $w_{j}^{\left( p \right)}$. Then, it follows from \eqref{thm:frequency_intermittency-a-1-proof-3}, \eqref{SMNMFGKJGHG-CDF-10} and H\"{o}lder inequality that the oscillation error $E_{q + 1}^{osc}$ satisfies
\begin{equation}
{\left\| {E_{q + 1}^{osc}} \right\|_{{L^p}}} =  {o_{q \to \infty }}\left( {{{\left\| {w_{q + 1}^{\left( p \right)}} \right\|}_{{L^p}}}} \right)  .  \label{SMNMFGKJGHG-CDF-11}
\end{equation}
\par
Combining \eqref{thm:frequency_intermittency-a-1-proof-6}, \eqref{VHUYTR-A-1-KLBHTF-A-145} and \eqref{SMNMFGKJGHG-CDF-11}, there exists a sequence ${\left\{ {{\varepsilon _q}} \right\}_{q \in {\mathbb{Z}^ + }}}$, satisfying $1 \gg {\varepsilon _q} \to 0$ as $q \to \infty $, such that
\begin{equation}
(1 - {\varepsilon _q}){\left\| {w_{q + 1}^{\left( p \right)}} \right\|_{{L^p}}} \lesssim {\left\| {{{\dot \Delta }_{{R_{q + 1}}}}u} \right\|_{{L^p}}} \lesssim (1 + {\varepsilon _q}){\left\| {w_{q + 1}^{\left( p \right)}} \right\|_{{L^p}}}  ,  \label{thm:frequency_intermittency-a-1-proof-20}
\end{equation}
which implies that
\begin{equation*}
\frac{{{{\left\| {{{\dot \Delta }_{{R_{q + 1}}}}u} \right\|}_{{L^{{p_1}}}}}}}{{{{\left\| {{{\dot \Delta }_{{R_{q + 1}}}}u} \right\|}_{{L^{{p_2}}}}}}} \gtrsim \frac{{1 - {\varepsilon _q}}}{{1 + {\varepsilon _q}}} \cdot \frac{{{{\left\| {w_{q + 1}^{\left( p \right)}} \right\|}_{{L^{{p_1}}}}}}}{{{{\left\| {w_{q + 1}^{\left( p \right)}} \right\|}_{{L^{{p_2}}}}}}} \gtrsim \frac{{1 - {\varepsilon _q}}}{{1 + {\varepsilon _q}}} \cdot r_q^{2(\frac{1}{{{p_2}}} - \frac{1}{{{p_1}}})} \to \infty ,\ \ \mbox{as}\ q \to \infty .
\end{equation*}
\par
For notational convenience, we denote ${\mathscr{D}_{l,n}}f \triangleq f\left( {x + l\mathbf{n}} \right) - f\left( x \right)$ and ${\Theta _{k,l,{\mathbf{n}}}}(y): = {{\tilde \Phi }_k}\left( {{{\tilde \Psi }_k}\left( y \right) + l{\mathbf{n}}} \right) - y$. According to \eqref{thm:frequency_intermittency-a-1-proof-6}, we have
\begin{equation*}
{\mathscr{D}_{l,n}}{{\dot \Delta }_{{R_{q + 1}}}}u = {\mathscr{D}_{l,n}}w_{q + 1}^{(p)} + {\mathscr{D}_{l,n}}E_{q + 1}^{flow} + {\mathscr{D}_{l,n}}E_{q + 1}^{amp} + {\mathscr{D}_{l,n}}E_{q + 1}^{osc}  .  
\end{equation*}
We write various components of the principal part ${\mathscr{D}_{l,n}}w_{q + 1}^{(p)}$ as
\begin{align*}
& {\mathscr{D}_{l,n}}\left( {{A_{\xi ,k,n,q + 1}}{\mathbb{W}_\xi }\left( {t,{\lambda _{q + 1}}{{\tilde \Phi }_k}\left( {t,x} \right)} \right)} \right)  \\
= & {\mathscr{D}_{l,n}}{A_{\xi ,k,n,q + 1}}\left( x \right){\mathbb{W}_\xi }\left( {t,{\lambda _{q + 1}}{{\tilde \Phi }_k}\left( {t,x} \right)} \right)  + {A_{\xi ,k,n,q + 1}}\left( {x + l{\mathbf{n}}} \right)\left( {{\mathbb{W}_\xi }\left( {t,{\lambda _{q + 1}}\left( {y + l{\mathbf{n}}} \right)} \right) - {\mathbb{W}_\xi }\left( {t,{\lambda _{q + 1}}y} \right)} \right) \\
& + {A_{\xi ,k,n,q + 1}}\left( {x + l{\mathbf{n}}} \right)\left( {{\mathbb{W}_\xi }\left( {t,{\lambda _{q + 1}}\left( {{\Theta _{k,l,{\mathbf{n}}}}(y) + y} \right)} \right) - {\mathbb{W}_\xi }\left( {t,{\lambda _{q + 1}}\left( {y + l{\mathbf{n}}} \right)} \right)} \right) \\
\triangleq & {I_{\xi ,k,n,q + 1}} + I{I_{\xi ,k,n,q + 1}}  .
\end{align*}
It follows from \eqref{eq:Lp-preservation-SKJKJFI} that
\begin{align}
{\left\| {{I_{\xi ,k,n,q + 1}}} \right\|_{{L^p}}} \lesssim & l{\left\| {{A_{\xi ,k,n,q + 1}}} \right\|_1}{\left\| {{\mathbb{W}_\xi }\left( {t,{\lambda _{q + 1}}y} \right)} \right\|_{L_y^p}} \nonumber \\
& + {\left\| {{A_{\xi ,k,n,q + 1}}} \right\|_0}{\left\| {{\mathbb{W}_\xi }\left( {t,{\lambda _{q + 1}}\left( {y + l{\mathbf{n}}} \right)} \right) - {\mathbb{W}_\xi }\left( {t,{\lambda _{q + 1}}y} \right)} \right\|_{L_y^p}}  \nonumber  \\
\lesssim & \delta _{q + 1}^{\frac{1}{2}}r_q^{1 - \frac{2}{p}} + l\delta _{q + 1}^{\frac{1}{2}}{\lambda _q}r_q^{1 - \frac{2}{p}} .  \label{thm:frequency_intermittency-a-1-proof-945}
\end{align}
Consider the second term $I{I_{\xi ,k,n,q + 1}}$. Since ${\Theta _{k,l,{\mathbf{n}}}}(y) - l{\mathbf{n}} = \int_0^l {\left( {\nabla {{\tilde \Phi }_k}\left( {{{\tilde \Psi }_k}\left( y \right) + s{\mathbf{n}}} \right) - \mathrm{Id}} \right) \cdot {\mathbf{n}}} \mathrm{d}s$, we have
\begin{equation*}
{\left\| {{\Theta _{k,l,{\mathbf{n}}}} - l{\mathbf{n}}} \right\|_0} \leqslant l{\left\| {\nabla {{\tilde \Phi }_k} - {\text{Id}}} \right\|_0} \lesssim l\lambda _{q + 1}^{ - \alpha } .
\end{equation*}
Consequently,
\begin{equation}
{\left\| {I{I_{\xi ,k,n,q + 1}}} \right\|_{{L^p}}} \lesssim {\lambda _{q + 1}}{\left\| {{A_{\xi ,k,n,q + 1}}} \right\|_0}{\left\| {{\Theta _{k,l,{\mathbf{n}}}} - l{\mathbf{n}}} \right\|_0}{\left\| {{\mathbb{W}_\xi }} \right\|_{{W^{1,p}}}}  \lesssim l\delta _{q + 1}^{\frac{1}{2}}\lambda _{q + 1}^{1 - \alpha }\sigma r_q^{2 - \frac{2}{p}}  . \label{thm:frequency_intermittency-a-1-proof-10945}
\end{equation}
\par
In view of the ${{\lambda _{q + 1}}}$-fold stretching transformation in the spatial variable $y$, the scale $l$ should satisfy $l \gtrsim \frac{1}{{{\lambda _{q + 1}}\sigma {r_q}}}$, that is, $l \gtrsim \lambda _{q + 1}^{ - \frac{1}{2} - M}$. Due to $M > \frac{1}{2}$, there exists an integral scale $l_I \left( q \right)$ decreasing with respect to $q$ such that
\begin{equation*}
l_I \delta _{q + 1}^{\frac{1}{2}}{\lambda _q} + l_I \delta _{q + 1}^{\frac{1}{2}}\lambda _{q + 1}^{1 - \alpha }\sigma {r_q} \leqslant c_0 \delta _{q + 1,\Gamma }^{\frac{1}{2}} , \ \ \mbox{for}\ c_0 \in \left( {0,1} \right) .
\end{equation*}
Therefore, it follows from \eqref{thm:frequency_intermittency-a-1-proof-945} and \eqref{thm:frequency_intermittency-a-1-proof-10945} that
\begin{equation*}
\delta _{q + 1,\Gamma }^{\frac{1}{2}}r_q^{1 - \frac{2}{p}} \lesssim {\left\| {{\mathscr{D}_{l,n}}\left( {{A_{\xi ,k,n,q + 1}}{\mathbb{W}_\xi }\left( {t,{\lambda _{q + 1}}{{\tilde \Phi }_k}\left( {t,x} \right)} \right)} \right)} \right\|_{{L^p}}} \lesssim \delta _{q + 1}^{\frac{1}{2}}r_q^{1 - \frac{2}{p}} ,
\end{equation*}
which implies that
\begin{equation}
\delta _{q + 1,\Gamma }^{\frac{1}{2}}{r_q^{1 - \frac{2}{p}}} \lesssim {\left\| {{\mathscr{D}_{l,n}}w_{q + 1}^{\left( p \right)}} \right\|_{{L^p}}} \lesssim \delta _{q + 1}^{\frac{1}{2}}{r_q^{1 - \frac{2}{p}}} .  \label{thm:frequency_intermittency-a-1-proof-9}
\end{equation}
\par
Similarly, using \eqref{thm:frequency_intermittency-a-1-proof-1} and \eqref{thm:frequency_intermittency-a-1-proof-10}, we derive that
\begin{align*}
& {\left\| {{\mathscr{D}_{l,n}}w_{q + 1}^{\left( c \right)}} \right\|_{{L^p}}} + {\left\| {{\mathscr{D}_{l,n}}w_{q + 1}^{\left( i \right)}} \right\|_{{L^p}}} + {\left\| {{\mathscr{D}_{l,n}}w_{q + 1}^{\left( \Gamma  \right)}} \right\|_{{L^p}}} + {\left\| {{\mathscr{D}_{l,n}}\mathfrak{W}_{q + 1}^{loc}} \right\|_{{L^p}}} \\
\lesssim & \lambda _{q + 1}^{ - 1}\delta _{q + 1}^{1/2}r_q^{2 - \frac{2}{p}} + {{\tilde \mu }^{ - 1}}{\delta _{q + 1}}r_q^{2 - \frac{2}{p}} + \mu _{q + 1}^{ - 1}{\delta _{q + 1}}{\lambda _q}l_q^{ - \alpha } + \delta _q^\alpha \delta _{q + 2}^{\frac{1}{2}} ,
\end{align*}
which implies that
\begin{equation*}
{\left\| {{\mathscr{D}_{l,n}}w_{q + 1}^{\left( c \right)}} \right\|_{{L^p}}} + {\left\| {{\mathscr{D}_{l,n}}w_{q + 1}^{\left( i \right)}} \right\|_{{L^p}}} + {\left\| {{\mathscr{D}_{l,n}}w_{q + 1}^{\left( \Gamma  \right)}} \right\|_{{L^p}}} + {\left\| {{\mathscr{D}_{l,n}}\mathfrak{W}_{q + 1}^{loc}} \right\|_{{L^p}}} = {o_{q \to \infty }}\left( {{{\left\| {{\mathscr{D}_{l,n}}w_{q + 1}^{\left( p \right)}} \right\|}_{{L^p}}}} \right)  .
\end{equation*}
Therefore, there exists a sequence ${\left\{ {{\varepsilon _q}} \right\}_{q \in {\mathbb{Z}^ + }}}$ with $1 \gg {\varepsilon _q} \to 0$, $q \to \infty $, such that
\begin{equation}
(1 - {\varepsilon _q}){\left\| {{\mathscr{D}_{l,n}}w_{q + 1}^{(p)}} \right\|_{{L^p}}} \lesssim {\left\| {{\mathscr{D}_{l,n}}{{\dot \Delta }_{{R_{q + 1}}}}u} \right\|_{{L^p}}} \lesssim (1 + {\varepsilon _q}){\left\| {{\mathscr{D}_{l,n}}w_{q + 1}^{(p)}} \right\|_{{L^p}}} . \label{thm:frequency_intermittency-a-1-proof-13}
\end{equation}
It follows from \eqref{thm:frequency_intermittency-a-1-proof-9} and \eqref{thm:frequency_intermittency-a-1-proof-13} that
\begin{equation*}
\frac{{{{\left\| {{\mathscr{D}_{l,n}}{{\dot \Delta }_{{R_{q + 1}}}}u} \right\|}_{{L^{{p_1}}}}}}}{{{{\left\| {{\mathscr{D}_{l,n}}{{\dot \Delta }_{{R_{q + 1}}}}u} \right\|}_{{L^{{p_2}}}}}}} \gtrsim \frac{{1 - {\varepsilon _q}}}{{1 + {\varepsilon _q}}} \cdot \frac{{{{\left\| {{\mathscr{D}_{l,n}}w_{q + 1}^{(p)}} \right\|}_{{L^{{p_1}}}}}}}{{{{\left\| {{\mathscr{D}_{l,n}}w_{q + 1}^{(p)}} \right\|}_{{L^{{p_2}}}}}}} \gtrsim \frac{{1 - {\varepsilon _q}}}{{1 + {\varepsilon _q}}} \cdot r_q^{2(\frac{1}{{{p_2}}} - \frac{1}{{{p_1}}})} \to \infty , \ \ \mbox{as}\ q \to \infty .
\end{equation*}
The proof is complete. \qed
\par
\vspace{1em}
{\textbf{Step 3. Propagation of intermittency to martingale solutions}}. 
\vspace{1em}
\par
Using the homogeneous Littlewood--Paley projector ${{\dot \Delta }_{{R_{q + 1}}}}$, we decompose $u\left( {t,x} \right)$ as
\begin{equation*}
u = {{\dot \Delta }_{{R_{q + 1}}}}u + {L_q} + {H_q}  ,
\end{equation*}
where the low-frequency component is
\begin{equation*}
{L_q} = \sum\limits_{j < q + 1} {{{\dot \Delta }_{{R_j}}}u} ,
\end{equation*}
and the high-frequency tail is
\begin{equation*}
{H_q} = \sum\limits_{j > q + 1} {{{\dot \Delta }_{{R_j}}}u} .
\end{equation*}
\par
It follows from the Littlewood--Paley characterization that
\begin{equation*}
{\left\| {{{\dot \Delta }_{{R_j}}}u} \right\|_{{L^p}}} \lesssim R_j^{ - \gamma }  .
\end{equation*}
Hence,
\begin{equation*}
{\left\| {{H_q}} \right\|_{{L^p}}} \lesssim \sum\limits_{j > q + 1} {R_j^{ - \gamma }}
\end{equation*}
We note that $R_j$ grows super-exponentially with respect to $j$. Therefore,
\begin{equation}
{\left\| {{H_q}} \right\|_{{L^p}}} \lesssim R_{q + 1}^{ - \gamma } \lesssim \lambda _q^{ - \gamma (b + M)}  .  \label{thm:frequency_intermittency-a-1-proof-15}
\end{equation}
Then, employing \eqref{VHUYTR-A-1-KLBHTF-A-1}, \eqref{thm:frequency_intermittency-a-1-proof-20} and \eqref{thm:frequency_intermittency-a-1-proof-15}, we infer that
\begin{align*}
\frac{{{{\left\| {{H_q}} \right\|}_{{L^p}}}}}{{{{\left\| {{{\dot \Delta }_{{R_{q + 1}}}}u} \right\|}_{{L^p}}}}} \lesssim & \frac{1}{{1 - {\varepsilon _q}}}\lambda _q^{ - \gamma (b + M)}\delta _{q + 1,\Gamma }^{ - \frac{1}{2}}r_q^{\frac{2}{p} - 1}  \\
\lesssim & \lambda _q^{M\left( {\frac{2}{p} - 1} \right) - \gamma (b + M) + \frac{1}{2}\left( {\beta  - \frac{1}{3}} \right)\Gamma }\lambda _{q + 1}^{\beta  - \frac{1}{2}\left( {\beta  - \frac{1}{3}} \right)\Gamma } .
\end{align*}
\par
Since $p > \frac{{2M}}{{\gamma M + M - \beta  + \alpha }}$, we have $M\left( {\frac{2}{p} - 1} \right) + \beta  - \gamma M <  - \alpha $. Thus, there exists a positive constant $b > 1.005$ such that
\begin{equation*}
\lambda _q^{M\left( {\frac{2}{p} - 1} \right) - \gamma (b + M) + \frac{1}{2}\left( {\beta  - \frac{1}{3}} \right)\Gamma }\lambda _{q + 1}^{\beta  - \frac{1}{2}\left( {\beta  - \frac{1}{3}} \right)\Gamma } \to 0, \ \ \mbox{as}\ q \to \infty  , 
\end{equation*}
which implies that
\begin{equation*}
{\left\| {{H_q}} \right\|_{{L^p}}}  = {o_{q \to \infty }}\left( {{\left\| {{{\dot \Delta }_{{R_{q + 1}}}}u} \right\|_{{L^p}}}} \right) . 
\end{equation*}
Then, there exists a sequence ${\left\{ {{\varepsilon _q}} \right\}_{q \in {\mathbb{Z}^ + }}}$, satisfying $1 \gg {\varepsilon _q} \to 0$ as $q \to \infty $, such that
\begin{equation}
\left( {1 - {\varepsilon _q}} \right){\left\| {{{\dot \Delta }_{{R_{q + 1}}}}u} \right\|_{{L^p}}} \lesssim {\left\| {u - {L_q}} \right\|_{{L^p}}} \sim {\left\| {\sum\limits_{j \geqslant q + 1} {{{\dot \Delta }_{{R_j}}}u} } \right\|_{{L^p}}} \lesssim \left( {1 + {\varepsilon _q}} \right){\left\| {{{\dot \Delta }_{{R_{q + 1}}}}u} \right\|_{{L^p}}} .  \label{thm:frequency_intermittency-a-1-proof-21}
\end{equation}
Combining Lemma \ref{thm:frequency_intermittency} and \eqref{thm:frequency_intermittency-a-1-proof-21}, we deduce that
\begin{equation*}
\frac{{{{\left\| {u - {L_q}} \right\|}_{{L^{{p_1}}}}}}}{{{{\left\| {u - {L_q}} \right\|}_{{L^{{p_2}}}}}}} \gtrsim \frac{{1 - {\varepsilon _q}}}{{1 + {\varepsilon _q}}} \cdot \frac{{{{\left\| {{{\dot \Delta }_{{R_{q + 1}}}}u} \right\|}_{{L^{{p_1}}}}}}}{{{{\left\| {{{\dot \Delta }_{{R_{q + 1}}}}u} \right\|}_{{L^{{p_2}}}}}}} \gtrsim \frac{{1 - {\varepsilon _q}}}{{1 + {\varepsilon _q}}} \cdot r_q^{2(\frac{1}{{{p_2}}} - \frac{1}{{{p_1}}})} \to \infty , \ \ \mbox{as}\  q \to \infty ,
\end{equation*}
which implies that the intermittency \eqref{Thes-sdgh-ASSDFG-1} holds.
\par
Using the homogeneous Littlewood--Paley projector ${{\dot \Delta }_{{R_{q + 1}}}}$ again, we decompose ${\mathscr{D}_{l,n}}u\left( {t,x} \right)$ as
\begin{equation*}
{\mathscr{D}_{l,n}}u = {\mathscr{D}_{l,n}}{{\dot \Delta }_{{R_{q + 1}}}}u + {\mathscr{D}_{l,n}}{L_q} + {\mathscr{D}_{l,n}}{H_q}  .
\end{equation*}
For the low-frequency component ${L_q}$, it follows from Bernstein's inequality that
\begin{equation*}
{\left\| {{\mathscr{D}_{l,n}}{{\dot \Delta }_{{R_j}}}u} \right\|_{{L^p}}} \lesssim l{R_j}{\left\| {{{\dot \Delta }_{{R_j}}}u} \right\|_{{L^p}}} .
\end{equation*}
Then, it follows from \eqref{VHUYTR-A-1-KLBHTF-A-1}, \eqref{thm:frequency_intermittency-a-1-proof-20}, \eqref{thm:frequency_intermittency-a-1-proof-9} and \eqref{thm:frequency_intermittency-a-1-proof-13} that
\begin{align*}
\frac{{{{\left\| {{\mathscr{D}_{l,n}}{L_q}} \right\|}_{{L^p}}}}}{{{{\left\| {{\mathscr{D}_{l,n}}{{\dot \Delta }_{{R_j}}}u} \right\|}_{{L^p}}}}} \lesssim & \frac{{l\sum\limits_{j < q + 1} {{R_j}{{\left\| {{{\dot \Delta }_{{R_j}}}u} \right\|}_{{L^p}}}} }}{{\left( {1 - {\varepsilon _q}} \right){{\left\| {{\mathscr{D}_{l,n}}w_{q + 1}^{(p)}} \right\|}_{{L^p}}}}}  \\
\lesssim & \frac{l}{{1 - {\varepsilon _q}}}\sum\limits_{j < q + 1} {\left( {1 + {\varepsilon _{j - 1}}} \right)\delta _j^{\frac{1}{2}}r_{j - 1}^{1 - \frac{2}{p}}{R_j}} \delta _{q + 1,\Gamma }^{ - \frac{1}{2}}r_q^{\frac{2}{p} - 1} .
\end{align*}
If $p \geqslant 2$, it is easy to check that
\begin{equation*}
\frac{l}{{1 - {\varepsilon _q}}}\sum\limits_{j < q + 1} {\left( {1 + {\varepsilon _{j - 1}}} \right)\delta _j^{\frac{1}{2}}r_{j - 1}^{1 - \frac{2}{p}}{R_j}} \delta _{q + 1,\Gamma }^{ - \frac{1}{2}}r_q^{\frac{2}{p} - 1} \to 0 ,  \ \ \mbox{as}\  q \to \infty .
\end{equation*}
If $1 < p < 2$, we have $\frac{{{R_j}\delta _q^\alpha }}{{{R_{q + 1}}}} {\delta _{q + 1,\Gamma }^{ - \frac{1}{2}}}  r_q^{\frac{2}{p} - 1} \to 0$ for any $ j < q +1 $ as $q \to \infty$. Then
\begin{align*}
\frac{l}{{1 - {\varepsilon _q}}}\sum\limits_{j < q + 1} {\left( {1 + {\varepsilon _{j - 1}}} \right)\delta _j^{\frac{1}{2}}r_{j - 1}^{1 - \frac{2}{p}}{R_j}} \delta _{q + 1,\Gamma }^{ - \frac{1}{2}}r_q^{\frac{2}{p} - 1} \lesssim & \left( {{l_I}{R_{q + 1}}\delta _q^{ - \alpha }\sum\limits_{j < q + 1} {\delta _j^{\frac{1}{2}}r_{j - 1}^{1 - \frac{2}{p}}} } \right) \cdot \frac{{{R_q}\delta _q^\alpha }}{{{R_{q + 1}}}}\delta _{q + 1,\Gamma }^{ - \frac{1}{2}}r_q^{\frac{2}{p} - 1} \\
\to & 0 , \ \ \mbox{as}\  q \to \infty ,
\end{align*}
where the sufficiently small integral scale $l_I$ restricts the growth rate of ${{R_{q + 1}}\delta _q^{ - \alpha }}$. Therefore,
\begin{equation}
{\left\| {{\mathscr{D}_{l,n}}{L_q}} \right\|_{{L^p}}}  = {o_{q \to \infty }}\left( {\left\| {{\mathscr{D}_{l,n}}{{\dot \Delta }_{{R_j}}}u} \right\|_{{L^p}}} \right) .  \label{thm:frequency_intermittency-a-1-proof-23}
\end{equation}
\par
For the high-frequency tail ${H_q}$, it follows from the Littlewood--Paley characterization that
\begin{align*}
{\left\| {{\mathscr{D}_{l,n}}{H_q}} \right\|_{{L^p}}} \lesssim & \sum\limits_{j > q + 1} {{{\left\| {{\mathscr{D}_{l,n}}{{\dot \Delta }_{{R_j}}}u} \right\|}_{{L^p}}}} \\
\lesssim & \sum\limits_{j > q + 1} {\min \left\{ {1,l{R_j}} \right\}{{\left\| {{{\dot \Delta }_{{R_j}}}u} \right\|}_{{L^p}}}}  \\
\lesssim & \sum\limits_{j > q + 1} {\min \left\{ {1,l{R_j}} \right\}R_j^{ - \gamma }}  .
\end{align*}
Owing to $l{R_j} \geqslant {l}{R_{q + 2}} \gtrsim \lambda _{q + 1}^{b - \frac{1}{2} - M} \geqslant 1$, we have
\begin{equation}
{\left\| {{\mathscr{D}_{l,n}}{H_q}} \right\|_{{L^p}}} \lesssim \sum\limits_{j > q + 1} {R_j^{ - \gamma }}  \lesssim R_{q + 1}^{ - \gamma } .  \label{thm:frequency_intermittency-a-1-proof-24}
\end{equation}
Employing \eqref{VHUYTR-A-1-KLBHTF-A-1}, \eqref{thm:frequency_intermittency-a-1-proof-20}, \eqref{thm:frequency_intermittency-a-1-proof-9}, \eqref{thm:frequency_intermittency-a-1-proof-13} and \eqref{thm:frequency_intermittency-a-1-proof-24}, we derive that
\begin{align*}
\frac{{{{\left\| {{\mathscr{D}_{l,n}}{H_q}} \right\|}_{{L^p}}}}}{{{{\left\| {{\mathscr{D}_{l,n}}{{\dot \Delta }_{{R_j}}}u} \right\|}_{{L^p}}}}} \lesssim & \frac{{R_{q + 1}^{ - \gamma }}}{{\left( {1 - {\varepsilon _q}} \right){{\left\| {{\mathscr{D}_{l,n}}w_{q + 1}^{(p)}} \right\|}_{{L^p}}}}} \\
\lesssim & \frac{1}{{1 - {\varepsilon _q}}}R_{q + 1}^{ - \gamma }\delta _{q + 1,\Gamma }^{ - \frac{1}{2}}r_q^{\frac{2}{p} - 1} \\
\lesssim & \lambda _q^{M\left( {\frac{2}{p} - 1} \right) - \gamma (b + M) + \frac{1}{2}\left( {\beta  - \frac{1}{3}} \right)\Gamma }\lambda _{q + 1}^{\beta  - \frac{1}{2}\left( {\beta  - \frac{1}{3}} \right)\Gamma }  .
\end{align*}
Invoking \eqref{thm:frequency_intermittency-a-1-proof-24}, we have
\begin{equation}
{\left\| {{\mathscr{D}_{l,n}}{H_q}} \right\|_{{L^p}}} = {o_{q \to \infty }}\left( {\left\| {{\mathscr{D}_{l,n}}{{\dot \Delta }_{{R_j}}}u} \right\|_{{L^p}}} \right)  . \label{thm:frequency_intermittency-a-1-proof-27}
\end{equation}
\par
Combining \eqref{thm:frequency_intermittency-a-1-proof-23} and \eqref{thm:frequency_intermittency-a-1-proof-27}, we infer that
\begin{equation}
\left( {1 - {\varepsilon _q}} \right){\left\| {{\mathscr{D}_{l,n}}{{\dot \Delta }_{{R_j}}}u} \right\|_{{L^p}}} \lesssim {\left\| {{\mathscr{D}_{l,n}}u} \right\|_{{L^p}}} \lesssim \left( {1 + {\varepsilon _q}} \right){\left\| {{\mathscr{D}_{l,n}}{{\dot \Delta }_{{R_j}}}u} \right\|_{{L^p}}}  . \label{thm:frequency_intermittency-a-1-proof-28}
\end{equation}
Then, it follows from Lemma \ref{thm:frequency_intermittency} and \eqref{thm:frequency_intermittency-a-1-proof-28} that
\begin{equation*}
\frac{{{{\left\| {{\mathscr{D}_{l,n}}u} \right\|}_{{L^{{p_1}}}}}}}{{{{\left\| {{\mathscr{D}_{l,n}}u} \right\|}_{{L^{{p_2}}}}}}} \geqslant \frac{{1 - {\varepsilon _q}}}{{1 + {\varepsilon _q}}} \cdot \frac{{{{\left\| {{\mathscr{D}_{l,n}}{{\dot \Delta }_{{R_j}}}u} \right\|}_{{L^{{p_1}}}}}}}{{{{\left\| {{\mathscr{D}_{l,n}}{{\dot \Delta }_{{R_j}}}u} \right\|}_{{L^{{p_2}}}}}}} \geqslant \frac{{1 - {\varepsilon _q}}}{{1 + {\varepsilon _q}}} \cdot r_q^{2(\frac{1}{{{p_2}}} - \frac{1}{{{p_1}}})} \to \infty , \ \ \mbox{as}\ q \to \infty  ,
\end{equation*}
which implies that the inertial-range intermittency \eqref{Thes-sdgh-ASSDFG-2} holds. The proof of Theorem \ref{CTSE-O-C-1B} is complete. \qed
\section{Newton iteration scheme}\label{NNIS}
In this section, we construct the Newtonian perturbation, which allows us to transform the $q$-step stochastic Euler--Reynolds system \eqref{ERNE-3A.1} into the $\left( {q , n} \right)$-step stochastic Newton--Euler--Reynolds system. To estimate the Newtonian perturbation, we mollify the $q$-th step velocity to construct Lagrangian flows and backward flows. Since the driving flow is progressively measurable with respect to the filtration generated by $\mathfrak{W}_q \left( {t} \right)$, the associated Lagrangian flows and backward flows remain adapted. Moreover, all estimates for the Newton iteration are restricted to the stopping time interval $\mathop  \cap \limits_{j = 0}^q \left[ {0,{\mathfrak{t}_j}} \right]$.
\subsection{Stochastic Newton--Euler--Reynolds system}\label{NNIS-JIUYRFD}
Define the inverse-divergence operator
\begin{equation*}
{\left\{ {\mathcal{R}u} \right\}_{ij}} = {\Delta ^{ - 1}}\left( {{\partial _{{x_i}}}{u^j} + {\partial _{{x_j}}}{u^i} - {\mathfrak{I}}_{ij} \mathrm{div}\,    u} \right), \ \  i,j \in \left\{ {1,2} \right\},
\end{equation*}
where ${u^i}$ denotes the $i$-th component of $u$, and ${\mathfrak{I}}_{ij}$ is the Kronecker delta
\begin{equation*}
{\mathfrak{I}}_{ij} = \left\{ {\begin{array}{*{20}{c}}
1, & i = j,\\
0, & i \ne j.
\end{array}} \right.
\end{equation*}
\par
\begin{lemma}[\cite{Zbl1556.35231}]\label{PO-IDO-A}
Suppose that $u$ is a smooth and mean-zero vector field. Then the two-tensor field ${\mathcal{R}u}$ is symmetric and satisfies $\mathrm{div}\,   {\mathcal{R}u} = u$.
\end{lemma}
\par
Let ${\varsigma _{{l_q}}}$ be the symmetric spatial mollifier at scale ${l_q}$, where ${l_q}$ is given by \eqref{parameters-S-9}. Define the initial Newtonian stress 
\begin{equation}\label{Newtonsteps-A-1}
{\mathring{R}_{q,0}} \triangleq {\mathring{R}_q}*{\varsigma _{{l_q}}} ,
\end{equation}
and the smooth auxiliary function
\begin{equation}\label{Newtonsteps-A-1-KIJH}
{{\bar v}_q} \triangleq {{ v}_q}*{\varsigma _{{l_q}}} ,
\end{equation}
where ${v_q}$ and ${\mathring{R}_q}$ are elements of the $q$-step stochastic Euler--Reynolds system \eqref{ERNE-3A.1}. It is clear that
\begin{equation*}
{\mathrm{supp}_\mathrm{t}} \, {\mathring{R}_{q,0}} = \left( {\frac{1}{4}T + \delta _q^{ - \frac{1}{2}}\lambda _q^{ - 1},\frac{3}{4}T - \delta _q^{ - \frac{1}{2}}\lambda _q^{ - 1}} \right]  .
\end{equation*}
Furthermore, it follows from the inductive estimates \eqref{HIE-u-1.1}, \eqref{HIE-R-1.1} and \eqref{MDHIE-R-1.1} that
\begin{equation}\label{lemma-GLTNNRS-B-M-1-sij-a}
{\left\| {{{\bar v}_q}} \right\|_0} \leqslant C  , 
\end{equation}
\begin{equation}\label{lemma-GLTNNRS-B-M-1}
{\left\| {{{\bar v}_q}} \right\|_N} \lesssim  \delta _q^{\frac{1}{2}}\lambda _q^N  , \ \ \forall \, N \in \left\{ {1, \cdots ,10} \right\}  ,
\end{equation}
\begin{equation}\label{lemma-GLTNNRS-B-M-2}
{\left\| {{\mathring{R}_{q,0}}} \right\|_N} \lesssim   {\delta _{q + 1}}\lambda _q^{N - \alpha }    ,  \ \ \forall \, N \in \left\{ {0,1, \cdots ,10} \right\}, 
\end{equation}
and
\begin{equation}\label{lemma-GLTNNRS-B-M-3}
{\left\| {{{\bar D}_{t,q}}{\mathring{R}_{q,0}}} \right\|_N} \lesssim  \delta _q^{\frac{1}{2}}{\delta _{q + 1}}\lambda _q^{N + 1 - \alpha }  ,  \ \ \forall \, \left\{ {0,1, \cdots ,10} \right\}, 
\end{equation}
where ${{\bar D}_{t,q}} \triangleq {\partial _t} + {{\bar v}_q} \cdot \nabla $ is the material derivative corresponding to ${{\bar v}_q}$.
\par
In fact, using Lemma \ref{MMD-A-H}, we obtain
\begin{equation*}
{\left\| {{{\bar v}_q}} \right\|_0} \leqslant {\left\| {{v_q}} \right\|_0} + l_q^2{\left\| {{v_q}} \right\|_{2}} \leqslant C  ,
\end{equation*}
\begin{equation*}
{\left\| {{{\bar v}_q}} \right\|_N} \leqslant {\left\| {{v_q}} \right\|_N} + l_q^2{\left\| {{v_q}} \right\|_{N + 2}} \lesssim \delta _q^{\frac{1}{2}}\lambda _q^N , \ \ \forall \, N \in \left\{ {1, \cdots ,10} \right\}  ,
\end{equation*}
and
\begin{equation*}
{\left\| {{\mathring{R}_{q,0}}} \right\|_N} \leqslant {\left\| {{\mathring{R}_q}} \right\|_N} + l_q^2{\left\| {{\mathring{R}_q}} \right\|_{N + 2}} \lesssim {\delta _{q + 1}}\lambda _q^{N - \alpha } , \ \ \forall \, N \in \left\{ {0, 1, \cdots ,10} \right\} .
\end{equation*}
Thus \eqref{lemma-GLTNNRS-B-M-1} and \eqref{lemma-GLTNNRS-B-M-2} hold. Notice that
\begin{equation*}
{{{\bar D}_{t,q}}{\mathring{R}_{q,0}}}  = {\left( {{D_{t,q}}{\mathring{R}_q}} \right)*{\varsigma _{{l_q}}}} + {\left( {{{\bar v}_q} \cdot \nabla } \right){\mathring{R}_{q,0}} - {\left( {\left( {{v_q} \cdot \nabla } \right){\mathring{R}_{q}}} \right)*{\varsigma _{{l_q}}}}} .
\end{equation*}
It follows from Lemma \ref{MMD-A-H} that
\begin{align*}
{\left\| {\left( {{D_{t,q}}{\mathring{R}_q}} \right)*{\varsigma _{{l_q}}}} \right\|_N} \leqslant & {\left\| {{D_{t,q}}{\mathring{R}_q}} \right\|_N} + l_q^2{\left\| {{D_{t,q}}{\mathring{R}_q}} \right\|_{N + 2}} \\
\lesssim & \delta _q^{\frac{1}{2}}{\delta _{q + 1}}\lambda _q^{N + 1 - \alpha } , \ \ \forall \, N \in \left\{ {0, 1, \cdots ,10} \right\} ,
\end{align*}
and
\begin{align*}
{\left\| {\left( {{{\bar v}_q} \cdot \nabla } \right){\mathring{R}_{q,0}} - {\left( {\left( {{v_q} \cdot \nabla } \right){\mathring{R}_{q}}} \right)*{\varsigma _{{l_q}}}}} \right\|_N}  \leqslant & {l^2}\left( {{{\left\| v_q \right\|}_{N + 1}}{{\left\| {{\mathring{R}_{q}}} \right\|}_2} + {{\left\| v_q \right\|}_1}{{\left\| {{\mathring{R}_{q}}} \right\|}_{N + 2}}} \right) \\
\lesssim & \delta _q^{\frac{1}{2}}{\delta _{q + 1}}\lambda _q^{N + 1 - \alpha }\frac{{{\lambda _q}}}{{{\lambda _{q + 1}}}}, \ \ \forall \, N \in \left\{ {0, 1, \cdots ,10} \right\}.
\end{align*}
Hence \eqref{lemma-GLTNNRS-B-M-3} holds.
\par
Define the initial Newtonian elements
\begin{equation*}
{u_{q,0}} = {u_q},\ \ {v_{q,0}} = {v_q} ,\ \ {\mathfrak{p}_{q,0}} = {\mathfrak{p}_q} ,
\end{equation*}
and
\begin{equation*}
{S_{q,0}} = 0,\ \ {\mathfrak{P}_{q + 1,0}} = {\mathring{R}_q} - {\mathring{R}_{q,0}} .
\end{equation*}
These Newtonian elements fulfill the initial stochastic Newton--Euler--Reynolds system
\begin{equation}\label{Newtonsteps-A-2}
\left\{ {\begin{array}{*{20}{l}}
{\partial _t}{v_{q,0}} + \mathrm{div} \left( {{u_{q,0}} \otimes {u_{q,0}}} \right) + \nabla {\mathfrak{p}_{q,0}} = \mathrm{div} \left( {{\mathring{R}_{q,0}} + {S_{q,0}} + {\mathfrak{P}_{q + 1,0}}} \right),\\
{u_{q,0}} = {v_{q,0}} + \mathfrak{W}_q ,\\
\mathrm{div}\,   {u_{q,0}} = 0 .
\end{array}} \right.    
\end{equation}
It follows from \eqref{lemma-GLTNNRS-B-M-1}-\eqref{lemma-GLTNNRS-B-M-3} and the inductive estimate \eqref{HIE-u-1.1} that
\begin{equation}\label{IJDI-FJIJD-A-1-KSIJDD-A}
\left\| {{\mathfrak{p}_{q,0}}} \right\|_0 \leqslant C    ,
\end{equation}
\begin{equation}\label{IJDI-FJIJD-A-1}
\left\| {{\mathfrak{p}_{q,0}}} \right\|_N \lesssim \delta _q^{\frac{1}{2}}\lambda _q^N , \ \ \forall \, N \in \left\{ {1, \cdots ,10} \right\}     ,
\end{equation}
\begin{equation}\label{IJDI-FJIJD-A-2}
{\left\| {{\mathring{R}_{q,0}}} \right\|_N} \lesssim {\delta _{q + 1,0}}\lambda _q^{N - \alpha }   , \ \ \forall \, N \in \left\{ {0,1, \cdots ,10} \right\}   ,
\end{equation}
and
\begin{equation}\label{IJDI-FJIJD-A-3}
{\left\| {{{\bar D}_{t,q}}{\mathring{R}_{q,0}}} \right\|_N} \lesssim {\delta _{q + 1,0}}\tau _q^{ - 1}\lambda _q^{N - \alpha } , \ \ \forall \, N \in \left\{ {0,1, \cdots ,10} \right\}   .
\end{equation}
\par
Without loss of generality, we assume that $\left( {{u_{q,n}},{v_{q,n}},{\mathfrak{p}_{q,n}},{\mathring{R}_{q,n}},{S_{q,n}},{\mathfrak{P}_{q + 1,n}}} \right)$ solves the $n$-step stochastic Newton--Euler--Reynolds system
\begin{equation}\label{Newtonsteps-A-3}
\left\{ {\begin{array}{*{20}{l}}
{\partial _t}{v_{q,n}} + \mathrm{div} \left( {{u_{q,n}} \otimes {u_{q,n}}} \right) + \nabla {\mathfrak{p}_{q,n}} = \mathrm{div} \left( {{\mathring{R}_{q,n}} + {S_{q,n}} + {\mathfrak{P}_{q + 1,n}}} \right),\\
{u_{q,n}} = {v_{q,n}} + \mathfrak{W}_q,\\
\mathrm{div}\,   {u_{q,n}} = 0,
\end{array}} \right.    
\end{equation}
and
\begin{equation}\label{SUPP-RSPP}
{\mathrm{supp}_\mathrm{t}} \, {\mathring{R}_{q,n}} \cup {\mathrm{supp}_\mathrm{t}} \, {S_{q,n}} \cup {\mathrm{supp}_\mathrm{t}} \, {\mathfrak{P}_{q + 1,n}} \subset \left( {\frac{1}{4}T + \delta _q^{ - \frac{1}{2}}\lambda _q^{ - 1} - 2n{\tau _q},\frac{3}{4}T - \delta _q^{ - \frac{1}{2}}\lambda _q^{ - 1} + 2n{\tau _q}} \right] \cap \mathop  \cap \limits_{j = 0}^q \left[ {0,{\mathfrak{t}_j}} \right] ,
\end{equation}
where $n$ is a finite positive integer. Drawing on \eqref{IJDI-FJIJD-A-1}-\eqref{IJDI-FJIJD-A-3}, assume inductively that the $n$-step Newtonian stress ${\mathring{R}_{q,n}}$ and pressure ${\mathfrak{p}_{q,n}}$ satisfy the inductive estimates
\begin{equation}\label{HIDJJF-3-A-1-A-1-KSIJDD-A}
\left\| {{\mathfrak{p}_{q,n}}} \right\|_0 \leqslant C ,
\end{equation}
\begin{equation}\label{HIDJJF-3-A-1}
\left\| {{\mathfrak{p}_{q,n}}} \right\|_N \lesssim \delta _q^{\frac{1}{2}}\lambda _q^N , \ \ \forall \, N \in \left\{ {1, \cdots ,10} \right\}   ,
\end{equation}
\begin{equation}\label{HIDJJF-3-A-2}
{\left\| {{\mathring{R}_{q,n}}} \right\|_N} \lesssim {\delta _{q + 1,n}}\lambda _q^{N - \alpha }  , \ \ \forall \, N \in \left\{ {0,1, \cdots ,10} \right\}   ,
\end{equation}
and
\begin{equation}\label{HIDJJF-3-A-3}
{\left\| {{{\bar D}_{t,q}}{\mathring{R}_{q,n}}} \right\|_N} \lesssim {\delta _{q + 1,n}}\tau _q^{ - 1}\lambda _q^{N - \alpha } , \ \ \forall \, N \in \left\{ {0,1, \cdots ,10} \right\}   .
\end{equation}

\subsection{Newtonian perturbation and iteration scheme}\label{NHHDKFI-DIJ}
Define the temporal scale ${t_k} \triangleq k{\tau _q}$ with $k \in \mathbb{Z} $. To construct a partition of unity in time, define the cut-off function ${\chi _k}\left( t \right)$ such that
\begin{equation}\label{HUDHUFK-AJSI-A-1}
\sum\limits_{k \in \mathbb{Z}^+  \cup \left\{ 0 \right\}} {\chi _k^2\left( t \right)}  = 1 , \ \ \forall \, t \in \mathrm{supp}_t \ {\chi _k} \subset \left( {{t_k} ,{t_k} + \frac{4}{3}{\tau _q}} \right) ,
\end{equation}
\begin{equation*}
\mathrm{supp}_t \ {\chi _{k - 1}} \cap \mathrm{supp}  \ {\chi _{k + 1}} = \emptyset ,
\end{equation*}
and
\begin{equation*}
\mathop {\sup }\limits_{t \in \mathrm{supp}_t \ {\chi _k}} \left| {\partial _t^N{\chi _k}} \right| \lesssim \tau _q^{ - N},\ \ \forall \, N \in \left\{ {0,1, \cdots ,10} \right\}   .
\end{equation*}
Moreover, define another cut-off function $ {{\tilde \chi }_k}\left( t \right) $ such that
\begin{equation*}
{\tilde \chi _k \left( t \right)}  = 1 , \ \ \forall \, t \in \left( {{t_k} ,{t_k} + \frac{4}{3}{\tau _q}} \right) ,
\end{equation*}
\begin{equation*}
\mathrm{supp}_t \ {\tilde \chi _k} \subset \left( {{t_k} ,{t_k} + 2{\tau _q}} \right),
\end{equation*}
and
\begin{equation*}
\mathop {\sup }\limits_{t \in \mathrm{supp}_t \ {\tilde \chi _k}} \left| {\partial _t^N{\tilde \chi _k}} \right| \lesssim \tau _q^{ - N},\ \ \forall \, N \in \left\{ {0,1, \cdots ,10} \right\}    .
\end{equation*}
\par
\begin{lemma}[\cite{Zbl1556.35231}]\label{GLTNNRS-A}
Let ${\mathbb{B}_{\frac{1}{2}}}\left( {\mathrm{Id}} \right)$ be the metric ball centered at $\mathrm{Id}$ in the space of symmetric $2 \times 2$ matrices, where $\mathrm{Id}$ is the identity matrix. Then, there exist finite set $\Lambda  \subset {\mathbb{Z}^2}$ and smooth function ${\gamma _\xi }:\ {\mathbb{B}_{\frac{1}{2}}}\left( {\mathrm{Id}} \right) \longrightarrow \mathbb{R}$ for each $\xi \in \Lambda$ such that
\begin{equation*}
R = \sum\limits_{\xi  \in \Lambda } {\gamma _\xi ^2\left( R \right)\xi  \otimes \xi },\ \ \forall \, R \in {\mathbb{B}_{\frac{1}{2}}}\left( {\mathrm{Id}} \right) .
\end{equation*}
Furthermore, there exist smooth $1$-periodic functions ${g_{\xi ,e,n}} {\left( t \right)}:\ \mathbb{R} \longrightarrow \mathbb{R}$ such that
\begin{equation*}
\int_0^1 {g_{\xi ,\iota ,n}^2\left( t \right)} \ \mathrm{d}t = 1,\ \ \forall \,  \left( {\xi ,\iota,n} \right) \in \Lambda  \times \left\{ {e,o} \right\} \times \left\{ {1,2, \cdots ,\Gamma } \right\},
\end{equation*}
and
\begin{equation*}
\mathrm{supp}_t \ {g_{{\xi _1},{\iota_1},{n_1}}} \cap \mathrm{supp}_t \ {g_{{\xi _2},{\iota_2},{n_2}}} = \emptyset , \ \ \mbox{as}\ \left( {{\xi _1},{\iota_1},{n_1}} \right) \ne \left( {{\xi _2},{\iota_2},{n_2}} \right) \in \Lambda  \times \left\{ {e,o} \right\} \times \left\{ {1,2, \cdots ,\Gamma } \right\} ,
\end{equation*}
where the step size $\Gamma$ is specified in Appendix C.
\end{lemma}
\par
Let ${\Phi _k}\left( {t,x} \right)$ be the backward flow of ${{{\bar v}_q}} \left( {t,x} \right)$, that is,
\begin{equation}\label{Newtonsteps-A-B-5}
\left\{ {\begin{array}{*{20}{l}}
{\partial _t}{\Phi _k} + \left( {{{\bar v}_q} \cdot \nabla } \right){\Phi _k} = 0,\\
{\Phi _k}\left( {{t_k}} \right) = x .
\end{array}} \right. 
\end{equation}
Define the amplitude function
\begin{equation}\label{Newtonsteps-A-B-9}
{\eta _{\xi ,k,{n+1}}} \triangleq \delta _{q + 1,n}^{\frac{1}{2}}{\chi _k}{\gamma _\xi }\left( {\nabla {\Phi _k}\nabla \Phi _k^{\mathrm{T}} - \delta _{q + 1,n}^{ - 1}\nabla {\Phi _k}{\mathring{R}_{q,n}}\nabla \Phi _k^{\mathrm{T}}} \right)  .
\end{equation}
In view of Lemma \ref{GLTNNRS-A}, choose the oscillation
\begin{equation}\label{Newtonsteps-A-B-12}
{A_{\xi ,k,{n+1}}} = \eta _{\xi ,k,{n+1}}^2{\left( {\nabla {\Phi _k}} \right)^{ - 1}}\xi  \otimes \xi {\left( {\nabla {\Phi _k}} \right)^{ - \mathrm{T}}} ,
\end{equation}
and the temporal oscillatory function
\begin{equation}\label{HUDHUFK-AJSI-A-2}
{f_{\xi ,k,n + 1}}\left( t \right) = 1 - g_{\xi ,k,n + 1}^2\left( t \right),\ \ f_{\xi ,k,n + 1}^{\left[ 1 \right]}\left( t \right) = \int_0^t {{f_{\xi ,k,n + 1}}\left( s \right)} \ \mathrm{d}s ,
\end{equation}
where
\begin{equation*}
{g_{\xi ,k,n + 1}}\left( t \right) = \left\{ {\begin{array}{*{20}{l}}
{g_{\xi ,e,n + 1}}\left( t \right),\hspace*{1em} & k = 2m - 1 , \\
{g_{\xi ,o,n + 1}}\left( t \right), & k = 2m ,
\end{array}} \right.  \ m \in \mathbb{Z} , \ n \in \left\{ {1,2, \cdots ,\Gamma -1 } \right\} .
\end{equation*}
\par
Denote the Leray projector operator by $\Pi \triangleq \mathrm{Id} - \nabla {\Delta ^{ - 1}}\mathrm{div}$. Suppose that ${w_{k,n + 1}^{\left( l \right)}}\left( {t,x} \right)$ is the unique mean-zero, divergence-free solution to the equation as follows
\begin{equation}\label{Newtonsteps-A-B-22}
\left\{ {\begin{array}{*{20}{l}}
\begin{array}{l}
{\partial _t}{w_{k,n + 1}^{\left( l \right)}}  + \left( {{{\bar v}_q} \cdot \nabla } \right){w_{k,n + 1}^{\left( l \right)}} + \left( {{w_{k,n + 1}^{\left( l \right)}} \cdot \nabla } \right){{\bar v}_q} + \nabla {\mathfrak{p}_{k,n + 1}} = \sum\limits_{\xi  \in \Lambda } {{f_{\xi ,k,n + 1}}\left( {{\mu _{q + 1}}t} \right)\Pi \mathrm{div}\,   {A_{\xi ,k,{n+1}}}} ,
\end{array}\\
\mathrm{div}\,   {w_{k,n + 1}^{\left( l \right)}} = 0,\\
{w_{k,n + 1}^{\left( l \right)}}\left( {{t_k}} \right) = \mu _{q + 1}^{ - 1}\sum\limits_{\xi  \in \Lambda } {f_{\xi ,k,n + 1}^{\left[ 1 \right]}\left( {{\mu _{q + 1}}{t_k}} \right)\Pi \mathrm{div}\,   {A_{\xi ,k,{n+1}}}} ,
\end{array}} \right.
\end{equation}
where the Newtonian pressure reads
\begin{equation*}
{\mathfrak{p}_{k,n + 1}} =  - 2{\Delta ^{ - 1}}\mathrm{div} \left( {\left( {w_{k,n + 1}^{\left( l \right)} \cdot \nabla } \right){{\bar v}_q}} \right) ,
\end{equation*}
the oscillation parameter ${\mu _{q + 1}}$ is given by \eqref{PASJD-MU-1}, and the time $t$ is restricted to the interval $\mathop  \cap \limits_{j = 0}^q \left[ {0,{\mathfrak{t}_j}} \right] \cap \left[ {{t_k},{t_{k + 1}}} \right)$. Similar to the argument in Giri and Radu \cite[Appendix E]{Zbl1556.35231}, the existence of solutions for Eq.\eqref{Newtonsteps-A-B-22} is established by the Cauchy--Lipschitz theory (also see Cheskidov and Luo \cite{Zbl1504.35221,Zbl1531.35213}). Since all associated random quantities in Eq.\eqref{Newtonsteps-A-B-22} are progressively measurable with respect to the filtration generated by $\mathfrak{W}_q \left( {t} \right)$, the solution ${w_{k,n + 1}^{\left( l \right)}} \left( {t } \right)$ remains adapted to the filtration generated by ${\mathfrak{W}_q}\left( {t} \right)$. We define the Newtonian perturbation by the temporally localized Newtonian perturbations
\begin{equation}\label{Newtonsteps-A-4}
w_{q + 1,n + 1}^{(n)}\left( {t,x} \right) \triangleq \sum\limits_{k \in {\mathbb{Z}_{q,n}}} {{{\tilde \chi }_k}\left( t \right){w_{k,n + 1}^{\left( l \right)}}\left( {t,x} \right)} ,
\end{equation}
where ${\mathbb{Z}_{q,n}} \triangleq \left\{ {k \in \mathbb{Z}  \ \mid\ k{\tau _q} \in {\mathcal{N}_{{\tau _q}}}\left( {\mathrm{supp}_{t}\ {\mathring{R}_{q,n}}} \right)} \right\}$, and ${\mathcal{N}_{{\tau _q}}}\left( {\cdot} \right)$ is the neighbourhood with size $\tau _q$. 
\par
Define the total Newtonian perturbation
\begin{equation}\label{OIUYYDGH-DJFU-1}
w_{q + 1}^{\left( n \right)} \triangleq \sum\limits_{j = 1}^{n + 1} {w_{q + 1,j}^{\left( n \right)}} .
\end{equation}
It is straightforward to observe that
\begin{equation}\label{Nashsteps-A-P-7-sjdisdaf}
\mathrm{div}\,   w_{q + 1}^{\left( n \right)} = 0 .
\end{equation}
Take the $\left( {n + 1} \right)$-step Newtonian elements
\begin{equation}\label{Newtonsteps-A-5}
{u_{q,n + 1}} = {u_{q,n}} + w_{q + 1,n + 1}^{(n)} = {u_q} + w_{q + 1}^{\left( n \right)} ,
\end{equation}
\begin{equation}\label{Newtonsteps-A-5v}
{v_{q,n + 1}} = {v_{q,n}} + w_{q + 1,n + 1}^{(n)} = {v_q} + w_{q + 1}^{\left( n \right)} ,
\end{equation}
\begin{equation}\label{Newtonsteps-A-7}
{\mathring{R}_{q,n + 1}} = \mathcal{R}\sum\limits_{k \in {\mathbb{Z}_{q,n}}} {\left( {{\partial _t}{{\tilde \chi }_k}} \right){w_{k,n + 1}^{\left( l \right)}}} ,
\end{equation}
\begin{equation}\label{Newtonsteps-A-8}
{S_{q,n + 1}} = {S_{q,n}} - \sum\limits_{k \in {\mathbb{Z}_{q,n}}} {\sum\limits_{\xi  \in \Lambda } {g_{\xi ,k,n + 1}^2 {\left( {{\mu _{q + 1}}t} \right)}{A_{\xi ,k,{n+1}}}} }  ,
\end{equation}
and
\begin{equation}\label{Newtonsteps-A-9}
{\mathfrak{P}_{q + 1,n + 1}} = {\mathfrak{P}_{q + 1,n}} - \sum\limits_{k \in {\mathbb{Z}_{q,n}}} {\sum\limits_{\xi  \in \Lambda } {\left( {1 - {f_{\xi ,k,n + 1}}{\left( {{\mu _{q + 1}}t} \right)}} \right)\Pi {A_{\xi ,k,{n+1}}}} }  .
\end{equation}
Further define the $\left( {n + 1} \right)$-step Newtonian pressure
\begin{align}\label{Newtonsteps-A-10}
{\mathfrak{p}_{q,n + 1}} \triangleq & {\mathfrak{p}_{q,n}} + \mathfrak{p}_{q + 1,n + 1}^{(n)} - \left\langle {w_{q + 1,n + 1}^{(n)},{u_q} - {{\bar v}_q}} \right\rangle  - \frac{1}{2}{\left| {w_{q + 1,n + 1}^{(n)}} \right|^2} - \left\langle {w_{q + 1,n + 1}^{(n)}, w_{q + 1}^{\left( n \right)} } \right\rangle \nonumber \\
& - {\Delta ^{ - 1}}\mathrm{div}  \left( { \left( {\mathrm{Id} - \Pi } \right) \mathrm{div}\,   {\mathring{R}_{q,n}} + \sum\limits_{k \in {\mathbb{Z}_{q,n}}} {\sum\limits_{\xi  \in \Lambda } {g_{\xi ,k,n + 1}^2 {\left( {{\mu _{q + 1}}t} \right)}\mathrm{div} \, {A_{\xi ,k,{n+1}}}} } } \right)  ,
\end{align}
and
\begin{equation}\label{Newtonsteps-A-6}
\mathfrak{p}_{q + 1,n + 1}^{(n)} \triangleq \sum\limits_{k \in {\mathbb{Z}_{q,n}}} {{{\tilde \chi }_k}{\mathfrak{p}_{k,n + 1}}}  ,
\end{equation}
where $\left\langle { \cdot , \cdot } \right\rangle $ means the inner product on $\mathbb{R}^2$. Using Eq.\eqref{Newtonsteps-A-B-22}, \eqref{Newtonsteps-A-4} and \eqref{Newtonsteps-A-6}, we obtain
\begin{align}\label{N+1s-NE-PIT-P-1}
& {\partial _t}w_{q + 1,n + 1}^{\left( n \right)} + \left( {{{\bar v}_q} \cdot \nabla } \right)w_{q + 1,n + 1}^{\left( n \right)} + \left( {w_{q + 1,n + 1}^{\left( n \right)} \cdot \nabla } \right){{\bar v}_q} + \nabla \mathfrak{p}_{q + 1,n + 1}^{(n)} \nonumber \\
= & \sum\limits_{k \in {\mathbb{Z}_{q,n}}} {\sum\limits_{\xi  \in \Lambda } {{{\tilde \chi }_k}{f_{\xi ,k,n + 1}}\left( {{\mu _{q + 1}}t} \right)\Pi {\mathrm{div}}\, {A_{\xi ,k,{n+1}}}} }  + \sum\limits_{k \in {\mathbb{Z}_{q,n}}} {w_{k,n + 1}^{\left( l \right)}{\partial _t}{{\tilde \chi }_k}} . 
\end{align}
\par
\begin{lemma}\label{N+1s-NE-PIT}
For any $t \in \mathop  \cap \limits_{j = 0}^q \left[ {0,{\mathfrak{t}_j}} \right] \cap \left[ {{t_k},{t_{k + 1}}} \right)$, the combination
\begin{equation*}
\left( {{u_{q,n + 1}},{v_{q,n + 1}},{\mathfrak{p}_{q,n + 1}},{\mathring{R}_{q,n + 1}},{S_{q,n + 1}},{\mathfrak{P}_{q + 1,n + 1}}} \right)
\end{equation*}
fulfills the $\left( {n + 1} \right)$-step stochastic Newton--Euler--Reynolds system
\begin{equation*}
\left\{ {\begin{array}{*{20}{l}}
{\partial _t}{v_{q,n + 1}} + \mathrm{div} \left( {{u_{q,n + 1}} \otimes {u_{q,n + 1}}} \right) + \nabla {\mathfrak{p}_{q,n + 1}} = \mathrm{div} \left( {{\mathring{R}_{q,n + 1}} + {S_{q,n + 1}} + {\mathfrak{P}_{q + 1,n + 1}}} \right),\\
{u_{q,n + 1}} = {v_{q,n + 1}} + \mathfrak{W}_q,\\
\mathrm{div}\,   {u_{q,n + 1}} = 0 .
\end{array}} \right.  
\end{equation*}
\end{lemma}
\par
\noindent{\textbf{Proof}}. 
Substituting \eqref{Newtonsteps-A-5}-\eqref{Newtonsteps-A-6} into the $n$-step stochastic Newton--Euler--Reynolds system \eqref{Newtonsteps-A-3}, we deduce that
\begin{align}\label{N+1s-NE-PIT-P-3}
&{\partial _t}\left( {{v_{q,n + 1}} - w_{q + 1}^{\left( n \right)}} \right) + \mathrm{div}\left( {\left( {{u_{q,n + 1}} - w_{q + 1}^{\left( n \right)}} \right) \otimes \left( {{u_{q,n + 1}} - w_{q + 1}^{\left( n \right)}} \right)} \right)  \nonumber \\
& + \nabla {\mathfrak{p}_{q,n + 1}} - \nabla \left( {\mathfrak{p}_{q + 1,n + 1}^{(n)} - \left\langle {w_{q + 1,n + 1}^{(n)},{u_q} - {{\bar u}_q}} \right\rangle  - \frac{1}{2}{{\left| {w_{q + 1,n + 1}^{(n)}} \right|}^2}} \right)  \nonumber \\
& + \nabla \left\langle {w_{q + 1,n + 1}^{(n)},\sum\limits_{m = 1}^n {w_{q + 1,m}^{(n)}} } \right\rangle  + \nabla {\Delta ^{ - 1}}\mathrm{div}  \left( {\mathrm{div}\,   {\mathring{R}_{q,n}} + \sum\limits_{k \in {\mathbb{Z}_{q,n}}} {\sum\limits_{\xi  \in \Lambda } {g_{\xi ,k,n + 1}^2 {\left( {{\mu _{q + 1}}t} \right)}\mathrm{div} \, {A_{\xi ,k,{n+1}}}} } } \right)  \nonumber \\
= & \mathrm{div}\left( {{\mathring{R}_{q,n}} + {S_{q,n + 1}} + {\mathfrak{P}_{q + 1,n + 1}}} \right) + \sum\limits_{k \in {\mathbb{Z}_{q,n}}} {\sum\limits_{\xi  \in \Lambda } {g_{\xi ,k,n + 1}^2 {\left( {{\mu _{q + 1}}t} \right)}\mathrm{div} \, {A_{\xi ,k,{n+1}}}} }   \nonumber \\
& + \Pi  \mathrm{div}\,   {\mathring{R}_{q,n}} + \sum\limits_{k \in {\mathbb{Z}_{q,n}}} {\sum\limits_{\xi  \in \Lambda } {\left( {1 - {f_{\xi ,k,n + 1}}{\left( {{\mu _{q + 1}}t} \right)}} \right)\Pi \mathrm{div}\,   {A_{\xi ,k,{n+1}}}} }   .
\end{align}
Owing to \eqref{Newtonsteps-A-7}, \eqref{N+1s-NE-PIT-P-3} and
\begin{align*}
& \mathrm{div}\left( {\left( {{u_{q,n + 1}} - w_{q + 1}^{\left( n \right)}} \right) \otimes \left( {{u_{q,n + 1}} - w_{q + 1}^{\left( n \right)}} \right)} \right) + \frac{1}{2}{\left| {w_{q + 1,n + 1}^{(n)}} \right|^2} \\
= & \mathrm{div}\left( {{u_{q,n + 1}} \otimes {u_{q,n + 1}}} \right) - \nabla \left( {\left\langle {w_{q + 1,n + 1}^{(n)},{u_q}} \right\rangle  - \left\langle {w_{q + 1,n + 1}^{(n)},\sum\limits_{m = 1}^n {w_{q + 1,m}^{(n)}} } \right\rangle } \right) ,
\end{align*}
we obtain
\begin{align}\label{N+1s-NE-PIT-P-2}
& {\partial _t}{v_{q,n + 1}} + \mathrm{div}\left( {{u_{q,n + 1}} \otimes {u_{q,n + 1}}} \right) + \nabla {\mathfrak{p}_{q,n + 1}} - \mathrm{div}\left( {{\mathring{R}_{q,n + 1}} + {S_{q,n + 1}} + {\mathfrak{P}_{q + 1,n + 1}}} \right) \nonumber \\
= & {\partial _t}w_{q + 1}^{\left( n \right)} + \left( {{{\bar u}_q} \cdot \nabla } \right)w_{q + 1,n + 1}^{\left( n \right)} + \left( {w_{q + 1,n + 1}^{\left( n \right)} \cdot \nabla } \right){\bar u _q} + \nabla \mathfrak{p}_{q + 1,n + 1}^{(n)}  \nonumber \\
& - \sum\limits_{k \in {\mathbb{Z}_{q,n}}} {w_{k,n + 1}^{\left( l \right)}{\partial _t}{{\tilde \chi }_k}}  + \Pi \mathrm{div}\,   {\mathring{R}_{q,n}} + \sum\limits_{k \in {\mathbb{Z}_{q,n}}} {\sum\limits_{\xi  \in \Lambda } {\left( {1 - {f_{\xi ,k,n + 1}}{\left( {{\mu _{q + 1}}t} \right)}} \right)\Pi \mathrm{div}\,   {A_{\xi ,k,{n+1}}}} } .
\end{align}
\par
Applying Lemma \ref{TSP-E-L-A-1} to Eq.\eqref{Newtonsteps-A-B-5}, it follows from \eqref{lemma-GLTNNRS-B-M-1} and \eqref{parameters-S-6} that
\begin{equation}\label{Newtonsteps-A-B-11}
{\left\| {D{\Phi _k}} \right\|_N} \lesssim {\tau _q} {\left\| {{{\bar v}_q}} \right\|_{N +1 }} \lesssim {\tau _q} \delta _q^{\frac{1}{2}}\lambda _q^{N + 1} \lesssim \lambda _q^N \lambda _{q + 1}^{ - \alpha } .
\end{equation}
Employing \eqref{HIDJJF-3-A-2} and \eqref{Newtonsteps-A-B-11}, we find
\begin{equation*}
{\left\| {\mathrm{Id} - \left( {\nabla {\Phi _k}\nabla \Phi _k^T - \delta _{q + 1,n}^{ - 1}\nabla {\Phi _k}{\mathring{R}_{q,n}}\nabla \Phi _k^T} \right)} \right\|_0} \lesssim \lambda _q^{ - \alpha } .
\end{equation*}
Then, there exists a parameter $\alpha > 0$ such that
\begin{equation*}
{\left\| {\mathrm{Id} - \left( {\nabla {\Phi _k}\nabla \Phi _k^T - \delta _{q + 1,n}^{ - 1}\nabla {\Phi _k}{\mathring{R}_{q,n}}\nabla \Phi _k^T} \right)} \right\|_0} \lesssim \frac{1}{2}  ,
\end{equation*}
which implies that $\nabla {\Phi _k}\nabla \Phi _k^{\mathrm{T}} - \delta _{q + 1,n}^{ - 1}\nabla {\Phi _k}{\mathring{R}_{q,n}}\nabla \Phi _k^{\mathrm{T}} \in {\mathbb{B}_{\frac{1}{2}}}\left( {\mathrm{Id}} \right)$. Invoking Lemma \ref{GLTNNRS-A}, we have
\begin{equation*}
\nabla {\Phi _k}\nabla \Phi _k^T - \delta _{q + 1,n}^{ - 1}\nabla {\Phi _k}{\mathring{R}_{q,n}}\nabla \Phi _k^T =  \sum\limits_{\xi  \in \Lambda } {\gamma _\xi ^2\left( {\nabla {\Phi _k}\nabla \Phi _k^T - \delta _{q + 1,n}^{ - 1}\nabla {\Phi _k}{\mathring{R}_{q,n}}\nabla \Phi _k^T} \right)\xi  \otimes \xi }  .
\end{equation*}
Therefore, it follows from \eqref{HUDHUFK-AJSI-A-1}, \eqref{Newtonsteps-A-B-9} and \eqref{Newtonsteps-A-B-12} that
\begin{align}\label{N+1s-NE-PIT-P-7}
& \sum\limits_{k \in {\mathbb{Z}_{q,n}}} {\sum\limits_{\xi  \in \Lambda } {\mathrm{div}\,   {A_{\xi ,k,{n+1}}}} } \nonumber \\
= & {\delta _{q + 1,n}}\mathrm{div} \sum\limits_{k \in {\mathbb{Z}_{q,n}}} {\chi _k^2{{\left( {\nabla {\Phi _k}} \right)}^{ - 1}}\left( {\nabla {\Phi _k}\nabla \Phi _k^T - \delta _{q + 1,n}^{ - 1}\nabla {\Phi _k}{\mathring{R}_{q,n}}\nabla \Phi _k^T} \right){{\left( {\nabla {\Phi _k}} \right)}^{ - {\rm{T}}}}}  \nonumber  \\
= & \mathrm{div}\sum\limits_{k \in {\mathbb{Z}_{q,n}}} {\chi _k^2\left( {{\delta _{q + 1,n}}\mathrm{Id} - {\mathring{R}_{q,n}}} \right)}   \nonumber  \\
= & - \mathrm{div}\,   {\mathring{R}_{q,n}} .
\end{align}
Recall that ${{\tilde \chi }_k}{A_{\xi ,k,{n+1}}} = {A_{\xi ,k,{n+1}}}$ for any ${k \in {\mathbb{Z}_{q,n}}}$. Using \eqref{HUDHUFK-AJSI-A-2} and \eqref{N+1s-NE-PIT-P-7}, we infer that
\begin{align}\label{N+1s-NE-PIT-P-8}
& \sum\limits_{k \in {\mathbb{Z}_{q,n}}} {\sum\limits_{\xi  \in \Lambda } {{{\tilde \chi }_k}{f_{\xi ,k,n + 1}}\left( {{\mu _{q + 1}}t} \right)\Pi \mathrm{div}\,   {A_{\xi ,k,{n+1}}}} }   \nonumber \\
= & \sum\limits_{k \in {\mathbb{Z}_{q,n}}} {\sum\limits_{\xi  \in \Lambda } {{{\tilde \chi }_k}\left( {1 - g_{\xi ,k,n + 1}^2\left( {{\mu _{q + 1}}t} \right)} \right)\Pi \mathrm{div}\,   {A_{\xi ,k,{n+1}}}} }  \nonumber \\
= & \Pi \sum\limits_{k \in {\mathbb{Z}_{q,n}}} {\sum\limits_{\xi  \in \Lambda } {\mathrm{div}\,   {A_{\xi ,k,{n+1}}}} }  - \sum\limits_{k \in {\mathbb{Z}_{q,n}}} {\sum\limits_{\xi  \in \Lambda } {{{\tilde \chi }_k}\left( {1 - {f_{\xi ,k,n + 1}}\left( {{\mu _{q + 1}}t} \right)} \right)\Pi \mathrm{div}\,   {A_{\xi ,k,{n+1}}}} }  \nonumber \\
= & - \Pi \mathrm{div}\,   {\mathring{R}_{q,n}} - \sum\limits_{k \in {\mathbb{Z}_{q,n}}} {\sum\limits_{\xi  \in \Lambda } {{{\tilde \chi }_k}\left( {1 - {f_{\xi ,k,n + 1}}\left( {{\mu _{q + 1}}t} \right)} \right)\Pi \mathrm{div}\,   {A_{\xi ,k,{n+1}}}} }  .
\end{align}
\par
Combining \eqref{N+1s-NE-PIT-P-1}, \eqref{N+1s-NE-PIT-P-2} and \eqref{N+1s-NE-PIT-P-8}, we obtain
\begin{equation*}
{\partial _t}{v_{q,n + 1}} + \mathrm{div}\left( {{u_{q,n + 1}} \otimes {u_{q,n + 1}}} \right) + \nabla {\mathfrak{p}_{q,n + 1}} - \mathrm{div}\left( {{\mathring{R}_{q,n + 1}} + {S_{q,n + 1}} + {\mathfrak{P}_{q + 1,n + 1}}} \right) = 0 .
\end{equation*}
In addition, it follows from the definition \eqref{Newtonsteps-A-5} that
\begin{equation*}
\mathrm{div}\,   {u_{q,n + 1}} = \mathrm{div}\,   {u_{q,n}} + \mathrm{div}\,   w_{q + 1,n + 1}^{(n)} = \mathrm{div}\,   {u_{q,n}} + \sum\limits_{k \in {\mathbb{Z}_{q,n}}} {{{\tilde \chi }_k}\mathrm{div} \, w_{k,n + 1}^{\left( l \right)}}  = 0 .
\end{equation*}
The proof is complete. \qed
\par
In view of Lemma \ref{N+1s-NE-PIT}, the total stochastic Newton--Euler--Reynolds system can be constructed after $\Gamma $-step Newton iteration, that is,
\begin{equation}\label{N+1s-NE-PIT-P-9}
\left\{ {\begin{array}{*{20}{l}}
{\partial _t}{v_{q,\Gamma }} + \mathrm{div}\left( {{u_{q,\Gamma }} \otimes {u_{q,\Gamma }}} \right) + \nabla {p_{q,\Gamma }} = \mathrm{div}\left( {{\mathring{R}_{q,\Gamma }} + {S_{q,\Gamma }} + {\mathfrak{P}_{q + 1,\Gamma }}} \right),\\
{u_{q,\Gamma }} = {v_{q,\Gamma }} +  \mathfrak{W}_q,\\
\mathrm{div}\,   {u_{q,\Gamma }} = 0 .
\end{array}} \right.   
\end{equation}
Here the $\Gamma $-step Newtonian elements read
\begin{align}\label{SHDUFY-A-3}
{\mathring{R}_{q,\Gamma }} = {\mathcal R}\sum\limits_{k \in {\mathbb{Z}_{q,\Gamma }}} {\left( {{\partial _t}{{\tilde \chi }_k}} \right)w_{k,\Gamma }^{\left( l \right)}} ,
\end{align}
\begin{equation}\label{SHDUFY-A-1}
{S_{q,\Gamma }} =  {S_{q,\Gamma  - 1}} - \sum\limits_{k \in {\mathbb{Z}_{q,\Gamma  - 1}}} {\sum\limits_{\xi  \in \Lambda } {g_{\xi ,k,\Gamma }^2\left( {{\mu _{q + 1}}t} \right){A_{\xi ,k,\Gamma  }}} } =  \sum\limits_{n = 0}^{\Gamma  - 1} {\sum\limits_{k \in {\mathbb{Z}_{q,n}}} {\sum\limits_{\xi  \in \Lambda } {g_{\xi ,k,n + 1}^2\left( {{\mu _{q + 1}}t} \right){A_{\xi ,k,{n+1}}}} } } ,
\end{equation}
\begin{align}\label{SHDUFY-A-2}
{\mathfrak{P}_{q + 1,\Gamma }} = & {\mathfrak{P}_{q + 1,\Gamma  - 1}}  - \sum\limits_{k \in {\mathbb{Z}_{q,\Gamma  - 1}}} {\sum\limits_{\xi  \in \Lambda } {\left( {1 - {f_{\xi ,k,\Gamma }}\left( {{\mu _{q + 1}}t} \right)} \right)\Pi {A_{\xi ,k,\Gamma  }}} }  \nonumber \\
= & {\mathring{R}_q} - {\mathring{R}_{q,0}}  - \sum\limits_{n = 0}^{\Gamma  - 1} {\sum\limits_{k \in {\mathbb{Z}_{q,\Gamma  - 1}}} {\sum\limits_{\xi  \in \Lambda } {\left( {1 - {f_{\xi ,k,n + 1}}\left( {{\mu _{q + 1}}t} \right)} \right)\Pi {A_{\xi ,k,n+1  }}} } } ,
\end{align}
\begin{align}\label{SHDUFY-A-2dJD}
{\mathfrak{p}_{q,\Gamma }} = & {\mathfrak{p}_q} + \sum\limits_{n = 0}^{\Gamma  - 1} {\left( {\sum\limits_{k \in {\mathbb{Z}_{q,\Gamma  - 1}}} {{{\tilde \chi }_k}{\mathfrak{p}_{k,n + 1}}}  - \left\langle {w_{q + 1,n + 1}^{(n)},{u_q} - {{\bar v}_q}} \right\rangle  - \frac{1}{2}\left| {w_{q + 1,n + 1}^{(n)}} \right|^2 - {{\left\langle {w_{q + 1,n + 1}^{(n)},w_{q + 1}^{\left( \Gamma \right)}} \right\rangle }}} \right)}  \nonumber  \\
& - {\Delta ^{ - 1}}\mathrm{div}  \left( {\left( {\mathrm{Id} - \Pi } \right)  \mathrm{div}\sum\limits_{n = 0}^{\Gamma  - 1} {{\mathring{R}_{q,n}}}  + \sum\limits_{n = 0}^{\Gamma  - 1} {\sum\limits_{k \in {\mathbb{Z}_{q,\Gamma  - 1}}} {\sum\limits_{\xi  \in \Lambda } {g_{\xi ,k,n + 1}^2\left( {{\mu _{q + 1}}t} \right)\mathrm{div}\,   {A_{\xi ,k,{n+1}}}} } } } \right) ,
\end{align}
and
\begin{equation}\label{UOIYH-LKO-DHU}
w_{q + 1}^{\left( \Gamma \right)}  =  \sum\limits_{n = 1}^{\Gamma } {w_{q + 1,n}^{\left( n \right)}} .
\end{equation}
\par
It follows from \eqref{N+1s-NE-PIT-P-7} that ${\mathrm{supp}_\mathrm{t}} \, {A_{\xi ,k,{n+1}}} = {\mathrm{supp}_\mathrm{t}} \, {\mathring{R}_{q,n}}$, which implies that
\begin{equation*}
{\mathrm{supp}_\mathrm{t}} \, w_{k,n + 1}^{\left( l \right)} \subset {\mathrm{supp}_\mathrm{t}} \, {A_{\xi ,k,{n+1}}} = {\mathrm{supp}_\mathrm{t}} \, {\mathring{R}_{q,n}} .
\end{equation*}
Owing to \eqref{Newtonsteps-A-4} and \eqref{OIUYYDGH-DJFU-1}, we have
\begin{equation}\label{OPIUY-A-1}
{\mathrm{supp}_\mathrm{t}} \, w_{q + 1}^{\left( \Gamma \right)} \subset {\mathrm{supp}_\mathrm{t}} \, {\mathring{R}_{q,n}} \subset \left( {\frac{1}{4}T ,\frac{3}{4}T - 2\delta _q^{ - \frac{1}{2}}\lambda _q^{ - 1} + 4\Gamma {\tau _q}} \right] \cap  \mathop  \cap \limits_{j = 0}^q \left[ {0,{\mathfrak{t}_j}} \right] .
\end{equation}
Similarly, it follows from \eqref{Newtonsteps-A-7}-\eqref{Newtonsteps-A-9} and \eqref{SHDUFY-A-3}-\eqref{SHDUFY-A-2} that
\begin{equation}\label{OPIUY-A-2}
{\mathrm{supp}_\mathrm{t}} \, {R_{q ,\Gamma}} \cup {\mathrm{supp}_\mathrm{t}} \, {S_{q ,\Gamma}} \cup {\mathrm{supp}_\mathrm{t}} \, {\mathfrak{P}_{q + 1, \Gamma}} \subset \left( {\frac{1}{4}T ,\frac{3}{4}T - 2\delta _q^{ - \frac{1}{2}}\lambda _q^{ - 1} + 4\Gamma {\tau _q}} \right] \cap   \mathop  \cap \limits_{j = 0}^q \left[ {0,{\mathfrak{t}_j}} \right] .
\end{equation}
\subsection{Estimates of the Newtonian perturbation}\label{ENTP}
\begin{lemma}\label{WPD-NP}
For any $N \in \left\{ {0,1, \cdots ,10} \right\} $, $n \in {\mathbb{Z}^ + }$ and $t \in \mathop  \cap \limits_{j = 0}^q \left[ {0,{\mathfrak{t}_j}} \right] \cap \left[ {{t_k},{t_{k + 1}}} \right)$, the amplitude function ${\eta _{\xi ,k,n}}$ and the oscillation ${A_{\xi ,k,{n+1}}}$ admit
\begin{equation}\label{Newtonsteps-A-B-15}
{\left\| {{\eta _{\xi ,k,n}}} \right\|_N} \lesssim \delta _{q + 1,n}^{\frac{1}{2}}\lambda _q^N    ,
\end{equation}
\begin{equation}\label{Newtonsteps-A-B-16}
{\left\| {{{\bar D}_{t,q}}{\eta _{\xi ,k,n}}} \right\|_N} \lesssim \delta _{q + 1,n}^{\frac{1}{2}}\tau _q^{ - 1}\lambda _q^N    ,
\end{equation}
\begin{equation}\label{Newtonsteps-A-B-19}
{\left\| {{A_{\xi ,k,n,{q+1}}}} \right\|_N} \lesssim {\delta _q} {\delta _{q + 1,n}}\lambda _q^N , 
\end{equation}
and
\begin{equation}\label{Newtonsteps-A-B-20}
{\left\| {{{\bar D}_{t,q}}{A_{\xi ,k,n,{q+1}}}} \right\|_N} \lesssim {\delta _q}  {\delta _{q + 1,n}}\tau _q^{ - 1}\lambda _q^N   .
\end{equation}
\end{lemma}
\par
\noindent{\textbf{Proof}}.
Let ${X_t}$ be the Lagrangian flow satisfying
\begin{equation*}
\left\{ {\begin{array}{*{20}{l}}
{\partial _s}{X_k}\left( {x,t} \right) = {{\bar v}_q}\left( {{X_k}\left( {x,t} \right),t} \right),\\
{X_k}\left( {x,t} \right) = x .
\end{array}} \right.
\end{equation*}
Using Lemma \ref{ONCS}, we obtain that for $N \leqslant  1$,
\begin{equation*}
{\left\| {{\partial _s} {D^N}{X_k}} \right\|_0} \lesssim {\left\| {{D^N}{{\bar v}_q}\left( {{X_k}} \right)} \right\|_0} \lesssim {\left\| {D{{\bar v}_q}} \right\|_0}{\left\| {D{X_k}} \right\|_{N - 1}} + {\left\| {{{\bar v}_q}} \right\|_N}\left\| {D{X_k}} \right\|_0^N .
\end{equation*}
Invoking \eqref{parameters-S-6}, we have
\begin{equation}\label{Newtonsteps-A-B-6}
{\left\| {D{X_k}} \right\|_0} \leqslant \exp \left( {{\tau _q} \delta _q^{\frac{1}{2}}{\lambda _q}} \right) \leqslant C .
\end{equation}
Then
\begin{equation*}
{\partial _s} {\left\| {D{X_k}} \right\|_N} \lesssim {\left\| {{{\bar v}_q}} \right\|_1}{\left\| {D{X_k}} \right\|_N} + {\left\| {{{\bar v}_q}} \right\|_{N + 1}} .
\end{equation*}
Employing Gronwall's inequality and \eqref{parameters-S-6}, we derive that for any $N \in \left\{ {0,1, \cdots ,10} \right\}$,
\begin{equation}\label{Newtonsteps-A-B-7}
{\left\| {D{X_k}} \right\|_N} \leqslant 1 + {\tau _q} {\left\| {{{\bar v}_q}} \right\|_{N + 1}}\exp \left( {{\tau _q} \delta _q^{\frac{1}{2}}{\lambda _q}} \right) \lesssim \max \left\{ {1,\lambda _q^N \lambda _{q + 1}^{ - \alpha }} \right\} .
\end{equation}
\par
Since ${\left( {\nabla {\Phi _k}} \right)^{ - 1}}\left( {t,x} \right) = D{X_k}\left( {{\Phi _k}\left( {t,x} \right),t} \right)$ (see \cite[Lemma $3.2$]{Zbl1556.35231}), it follows from \eqref{Newtonsteps-A-B-11}, \eqref{Newtonsteps-A-B-6} and \eqref{Newtonsteps-A-B-7} that
\begin{equation}\label{Newtonsteps-A-B-10}
{\left\| {{{\left( {\nabla {\Phi _k}} \right)}^{ - 1}}} \right\|_N} \lesssim  {\left\| {{X_k}\left( {{\Phi _k}} \right)} \right\|_{N + 1}} \lesssim {\left\| {D{X_k}} \right\|_0}{\left\| {D{\Phi _k}} \right\|_N} + {\left\| {D{X_k}} \right\|_N}\left\| {D{\Phi _k}} \right\|_0^{N + 1} \lesssim \lambda _q^N  \lambda _{q + 1}^{ -  \alpha } .
\end{equation}
Owing to
\begin{align*}
{{\bar D}_{t,q}}\left( {\nabla {\Phi _k}} \right) = & {\partial _t}\nabla {\Phi _k} + \left( {{{\bar v}_q} \cdot \nabla } \right)\nabla {\Phi _k} \\
= & \left( {{{\bar v}_q} \cdot \nabla } \right)\nabla {\Phi _k} - \nabla \left[ {\left( {{{\bar v}_q} \cdot \nabla } \right){\Phi _k}} \right] \\
= & - \left( {\nabla {{\bar v}_q} \cdot \nabla } \right){\Phi _k} ,
\end{align*}
and
\begin{align*}
{{\bar D}_{t,q}}{\left( {\nabla {\Phi _k}} \right)^{ - 1}} = & {{\bar D}_{t,q}}D{X_k}\left( {{\Phi _k}} \right) \\
= & \frac{d}{{dt}}D{X_k}\left( {{\Phi _k}} \right) + \left( {{{\bar v}_q} \cdot \nabla } \right)D{X_k}\left( {{\Phi _k}} \right) \\
= & D{{\bar v}_q}\left( {{X_k}\left( {{\Phi _k}} \right)} \right) + \left( {{{\bar v}_q} \cdot \nabla } \right)D{X_k}\left( {{\Phi _k}} \right) \\
= & D{{\bar v}_q}{\left( {\nabla {\Phi _k}} \right)^{ - 1}} + \left( {{{\bar v}_q} \cdot \nabla } \right){\left( {\nabla {\Phi _k}} \right)^{ - 1}} ,
\end{align*}
it follows from \eqref{lemma-GLTNNRS-B-M-1}, \eqref{Newtonsteps-A-B-11}, \eqref{Newtonsteps-A-B-10} and \eqref{RSRI-main-PROVE-2-JIJSHU-posksij-a-1} that
\begin{align}\label{Newtonsteps-A-B-13}
{\left\| {{{\bar D}_{t,q}}\left( {\nabla {\Phi _k}} \right)} \right\|_N} =  & {\left\| {\left( {\nabla {{\bar v}_q} \cdot \nabla } \right){\Phi _k}} \right\|_N} \nonumber \\
\lesssim & {\left\| {{{\bar v}_q}} \right\|_{N + 1}}{\left\| {D{\Phi _k}} \right\|_0} + {\left\| {{{\bar v}_q}} \right\|_1}{\left\| {D{\Phi _k}} \right\|_N} \nonumber \\
\lesssim & \delta _q^{\frac{1}{2}}\lambda _q^{N + 1} ,
\end{align}
and
\begin{align}\label{Newtonsteps-A-B-14}
{\left\| {{{\bar D}_{t,q}}{{\left( {\nabla {\Phi _k}} \right)}^{ - 1}}} \right\|_N} \leqslant & {\left\| {D{{\bar v}_q}{{\left( {\nabla {\Phi _k}} \right)}^{ - 1}}} \right\|_N} + {\left\| {\left( {{{\bar v}_q} \cdot \nabla } \right){{\left( {\nabla {\Phi _k}} \right)}^{ - 1}}} \right\|_N}  \nonumber \\
\lesssim & {\left\| {{{\bar v}_q}} \right\|_{N + 1}}{\left\| {{{\left( {\nabla {\Phi _k}} \right)}^{ - 1}}} \right\|_0} + {\left\| {{{\bar v}_q}} \right\|_1}{\left\| {{{\left( {\nabla {\Phi _k}} \right)}^{ - 1}}} \right\|_N} \nonumber \\
& + {\left\| {{{\bar v}_q}} \right\|_N}{\left\| {{{\left( {\nabla {\Phi _k}} \right)}^{ - 1}}} \right\|_1} + {\left\| {{{\bar v}_q}} \right\|_0}{\left\| {{{\left( {\nabla {\Phi _k}} \right)}^{ - 1}}} \right\|_{N + 1}} \nonumber \\
\lesssim & \delta _q^{\frac{1}{2}}\lambda _q^{N + 1} .
\end{align}
\par
According to \eqref{HIDJJF-3-A-2}, \eqref{Newtonsteps-A-B-11} and the definition \eqref{Newtonsteps-A-B-9}, we deduce that
\begin{align*}
{\left\| {{\eta _{\xi ,k,n}}} \right\|_N} \lesssim & \delta _{q + 1,n}^{\frac{1}{2}}{\left\| {\nabla {\Phi _k}\nabla \Phi _k^T} \right\|_N} + \delta _{q + 1,n}^{ - \frac{1}{2}}{\left\| {\nabla {\Phi _k}{\mathring{R}_{q,n}}\nabla \Phi _k^T} \right\|_N} \nonumber \\
\lesssim & \delta _{q + 1,n}^{\frac{1}{2}}{\left\| {D{\Phi _k}} \right\|_N}{\left\| {D{\Phi _k}} \right\|_0} + \delta _{q + 1,n}^{ - \frac{1}{2}}\left( {{{\left\| {D{\Phi _k}} \right\|}_N}{{\left\| {{\mathring{R}_{q,n}}} \right\|}_0}{{\left\| {D{\Phi _k}} \right\|}_0} + \left\| {D{\Phi _k}} \right\|_0^2{{\left\| {{\mathring{R}_{q,n}}} \right\|}_N}} \right)  \nonumber \\
\lesssim &  \delta _{q + 1,n}^{\frac{1}{2}}\lambda _q^N   .
\end{align*}
In addition, since
\begin{align*}
{\left\| {{{\bar D}_{t,q}}{\eta _{\xi ,k,n}}} \right\|_N} \lesssim & \mathop {\sup }\limits_{t \in \mathrm{supp}_t {\chi _k}} \left| {{\partial _t}{\chi _k}} \right|\left( {\delta _{q + 1,n}^{\frac{1}{2}}{{\left\| {\nabla {\Phi _k}\nabla \Phi _k^T} \right\|}_N} + \delta _{q + 1,n}^{ - \frac{1}{2}}{{\left\| {\nabla {\Phi _k}{\mathring{R}_{q,n}}\nabla \Phi _k^T} \right\|}_N}} \right) \\
& + \delta _{q + 1,n}^{\frac{1}{2}}{\left\| {{{\bar D}_{t,q}}\left( {\nabla {\Phi _k}\nabla \Phi _k^T} \right)} \right\|_N} + \delta _{q + 1,n}^{ - \frac{1}{2}}{\left\| {{{\bar D}_{t,q}}\left( {\nabla {\Phi _k}{\mathring{R}_{q,n}}\nabla \Phi _k^T} \right)} \right\|_N} \\
\lesssim & \mathop {\sup }\limits_{t \in \mathrm{supp}_t {\chi _k}} \left| {{\partial _t}{\chi _k}} \right|\left( {\delta _{q + 1,n}^{\frac{1}{2}}{{\left\| {\nabla {\Phi _k}\nabla \Phi _k^T} \right\|}_N} + \delta _{q + 1,n}^{ - \frac{1}{2}}{{\left\| {\nabla {\Phi _k}{\mathring{R}_{q,n}}\nabla \Phi _k^T} \right\|}_N}} \right) \\
& + \delta _{q + 1,n}^{\frac{1}{2}}\left( {{{\left\| {{{\bar D}_{t,q}}\nabla {\Phi _k}} \right\|}_N}{{\left\| {\nabla {\Phi _k}} \right\|}_0} + {{\left\| {{{\bar D}_{t,q}}\nabla {\Phi _k}} \right\|}_0}{{\left\| {\nabla {\Phi _k}} \right\|}_N}} \right) \\
& + \delta _{q + 1,n}^{ - \frac{1}{2}}\left( {{{\left\| {{{\bar D}_{t,q}}\nabla {\Phi _k}} \right\|}_N}{{\left\| {{\mathring{R}_{q,n}}} \right\|}_0}{{\left\| {\nabla {\Phi _k}} \right\|}_0} + {{\left\| {{{\bar D}_{t,q}}\nabla {\Phi _k}} \right\|}_0}{{\left\| {{\mathring{R}_{q,n}}} \right\|}_N}{{\left\| {\nabla {\Phi _k}} \right\|}_0}} \right) \\
& + \delta _{q + 1,n}^{ - \frac{1}{2}}\left( {{{\left\| {\nabla {\Phi _k}} \right\|}_N}{{\left\| {{{\bar D}_{t,q}}{\mathring{R}_{q,n}}} \right\|}_0}{{\left\| {\nabla {\Phi _k}} \right\|}_0} + \left\| {\nabla {\Phi _k}} \right\|_0^2{{\left\| {{{\bar D}_{t,q}}{\mathring{R}_{q,n}}} \right\|}_N}} \right) ,
\end{align*}
it follows from \eqref{HIDJJF-3-A-2}, \eqref{HIDJJF-3-A-3}, \eqref{Newtonsteps-A-B-11}, \eqref{Newtonsteps-A-B-13} and \eqref{parameters-S-6} that
\begin{align*}
{\left\| {{{\bar D}_{t,q}}{\eta _{\xi ,k,n}}} \right\|_N} \lesssim & \delta _{q + 1,n}^{\frac{1}{2}}\tau _q^{ - 1}\lambda _q^N + \delta _q^{\frac{1}{2}}\delta _{q + 1,n}^{\frac{1}{2}}\lambda _q^{N + 1}  + \delta _q^{\frac{1}{2}}\delta _{q + 1,n}^{\frac{1}{2}}\lambda _q^{N + 1 - \alpha } + \delta _{q + 1,n}^{\frac{1}{2}}\tau _q^{ - 1}\lambda _q^{N - \alpha } \\
\lesssim & \delta _{q + 1,n}^{\frac{1}{2}}\tau _q^{ - 1}\lambda _q^N   .
\end{align*}
Hence, the estimates \eqref{Newtonsteps-A-B-15} and \eqref{Newtonsteps-A-B-16} hold.
\par
Using \eqref{Newtonsteps-A-B-15}, \eqref{Newtonsteps-A-B-10} and the definition \eqref{Newtonsteps-A-B-12}, we get
\begin{align*}
{\left\| {{A_{\xi ,k,n,{q+1}}}} \right\|_N} \lesssim & {\left\| {\eta _{\xi ,k,n}^2{{\left( {\nabla {\Phi _k}} \right)}^{ - 1}}\xi  \otimes \xi {{\left( {\nabla {\Phi _k}} \right)}^{ - {\rm{T}}}}} \right\|_N} \\
\lesssim & {\left\| {{\eta _{\xi ,k,n}}} \right\|_N}{\left\| {{\eta _{\xi ,k,n}}} \right\|_0}\left\| {{{\left( {\nabla {\Phi _k}} \right)}^{ - 1}}} \right\|_0^2 + \left\| {{\eta _{\xi ,k,n}}} \right\|_0^2{\left\| {{{\left( {\nabla {\Phi _k}} \right)}^{ - 1}}} \right\|_N}{\left\| {{{\left( {\nabla {\Phi _k}} \right)}^{ - 1}}} \right\|_0} \\
\lesssim & {\delta _q} {\delta _{q + 1,n}}\lambda _q^N   .
\end{align*}
Notice that
\begin{align*}
{\left\| {{{\bar D}_{t,q}}{A_{\xi ,k,n,{q+1}}}} \right\|_N} \lesssim & {\left\| {{{\bar D}_{t,q}}\left( {\eta _{\xi ,k,n}^2{{\left( {\nabla {\Phi _k}} \right)}^{ - 1}}\xi  \otimes \xi {{\left( {\nabla {\Phi _k}} \right)}^{ - {\rm{T}}}}} \right)} \right\|_N} \\
\lesssim & {\left\| {{{\bar D}_{t,q}}{\eta _{\xi ,k,n}}} \right\|_N}{\left\| {{\eta _{\xi ,k,n}}} \right\|_0}\left\| {{{\left( {\nabla {\Phi _k}} \right)}^{ - 1}}} \right\|_0^2 + {\left\| {{{\bar D}_{t,q}}{\eta _{\xi ,k,n}}} \right\|_0}{\left\| {{\eta _{\xi ,k,n}}} \right\|_N}\left\| {{{\left( {\nabla {\Phi _k}} \right)}^{ - 1}}} \right\|_0^2 \\
& + {\left\| {{{\bar D}_{t,q}}{\eta _{\xi ,k,n}}} \right\|_0}{\left\| {{\eta _{\xi ,k,n}}} \right\|_0}{\left\| {{{\left( {\nabla {\Phi _k}} \right)}^{ - 1}}} \right\|_N}{\left\| {{{\left( {\nabla {\Phi _k}} \right)}^{ - 1}}} \right\|_0} \\
& + \left\| {{\eta _{\xi ,k,n}}} \right\|_0^2{\left\| {{{\bar D}_{t,q}}{{\left( {\nabla {\Phi _k}} \right)}^{ - 1}}} \right\|_N}{\left\| {{{\left( {\nabla {\Phi _k}} \right)}^{ - 1}}} \right\|_0} + \left\| {{\eta _{\xi ,k,n}}} \right\|_0^2{\left\| {{{\bar D}_{t,q}}{{\left( {\nabla {\Phi _k}} \right)}^{ - 1}}} \right\|_0}{\left\| {{{\left( {\nabla {\Phi _k}} \right)}^{ - 1}}} \right\|_N} \\
&  + {\left\| {{\eta _{\xi ,k,n}}} \right\|_N}{\left\| {{\eta _{\xi ,k,n}}} \right\|_0}{\left\| {{{\bar D}_{t,q}}{{\left( {\nabla {\Phi _k}} \right)}^{ - 1}}} \right\|_0}{\left\| {{{\left( {\nabla {\Phi _k}} \right)}^{ - 1}}} \right\|_0} .
\end{align*}
It then follows from \eqref{Newtonsteps-A-B-15}, \eqref{Newtonsteps-A-B-16}, \eqref{Newtonsteps-A-B-10}, \eqref{Newtonsteps-A-B-14} and \eqref{parameters-S-6} that
\begin{equation*}
{\left\| {{{\bar D}_{t,q}}{A_{\xi ,k,n,{q+1}}}} \right\|_N}  \lesssim {\delta _q} {\delta _{q + 1,n}}\tau _q^{ - 1}\lambda _q^N   .
\end{equation*}
Thus, the estimates \eqref{Newtonsteps-A-B-19} and \eqref{Newtonsteps-A-B-20} hold. The proof is complete. \qed
\par
\begin{proposition}\label{WPD-NP-main}
For any $N \in \left\{ {0,1, \cdots ,10} \right\} $, $n \in {\mathbb{Z}^ + } \cup \left\{ 0 \right\}$ and $t \in \mathrm{supp}_t\ w_{q + 1,n + 1}^{(n)}$, there exists a parameter $\alpha \in \left( {\frac{\beta }{{b}},\beta } \right)$ such that the $\left( {n + 1} \right)$-step Newtonian perturbation $w_{q + 1,n + 1}^{(n)}\left( {t,x} \right)$ satisfies
\begin{equation}\label{WPD-NP-main-A-1}
{\left\| {w_{q + 1,n + 1}^{(n)}} \right\|_N} \lesssim \mu _{q + 1}^{ - 1}{\delta _{q + 1}}\lambda _q^{N + 1}l_q^{ - \alpha }  ,
\end{equation}
and
\begin{equation}\label{WPD-NP-main-A-2}
{\left\| {{{\bar D}_{t,q}}w_{q + 1,n + 1}^{(n)}} \right\|_N} \lesssim {\delta _{q + 1}}\lambda _q^{N + 1}l_q^{ - \alpha }  .
\end{equation}
\end{proposition}
\par
\noindent{\textbf{Proof}}. Take the auxiliary functions $\Psi$, $\bar \Psi$ and $\tilde \Psi$ satisfying
\begin{equation}\label{Newtonsteps-A-B-24}
{{\bar D}_{t,q}}\Psi  = \left( {w_{k,n + 1}^{\left( l \right)} \cdot \nabla } \right){{\bar v}_q},
\end{equation}
\begin{equation}\label{Newtonsteps-A-B-25}
{{\bar D}_{t,q}}\bar \Psi  = \sum\limits_{\xi  \in \Lambda } {{f_{\xi ,k,n + 1}}\left( {{\mu _{q + 1}}t} \right)\Pi {\mathrm{div}}\, {A_{\xi ,k,{n+1}}}}  ,
\end{equation}
and
\begin{equation}\label{Newtonsteps-A-B-26}
\tilde \Psi  = \mu _{q + 1}^{ - 1}\sum\limits_{\xi  \in \Lambda } {f_{\xi ,k,n + 1}^{\left[ 1 \right]}\left( {{\mu _{q + 1}}{t_k}} \right)\Pi {\mathrm{div}}\, {A_{\xi ,k,{n+1}}}} .
\end{equation}
Then, Eq.\eqref{Newtonsteps-A-B-22} is rewritten as
\begin{equation}\label{Newtonsteps-A-B-23}
\left\{ {\begin{array}{*{20}{l}}
w_{k,n + 1}^{\left( l \right)} = \Psi  + \bar \Psi , \ \ t \in \left[ {{t_k},{t_{k + 1}}} \right),\\
w_{k,n + 1}^{\left( l \right)}\left( {{t_k}} \right) = \bar \Psi \left( {{t_k}} \right) = \tilde \Psi ,\\
\Psi \left( {{t_k}} \right) = 0.
\end{array}} \right.
\end{equation}
\par
We first estimate ${\left\| {w_{k,n + 1}^{\left( l \right)}} \right\|_0}$ and ${\left\| {{{\bar D}_{t,q}}w_{k,n + 1}^{\left( l \right)}} \right\|_0}$. Invoking \eqref{Newtonsteps-A-B-24}, Lemma \ref{STN-HCS-A} and
\begin{equation*}
{\left\| {\left( {w_{k,n + 1}^{\left( l \right)} \cdot \nabla } \right){{\bar v}_q}} \right\|_0} \lesssim {\left\| {w_{k,n + 1}^{\left( l \right)}} \right\|_0}{\left\| {{{\bar v}_q}} \right\|_1} ,
\end{equation*}
we obtain
\begin{equation}\label{Newtonsteps-A-B-27}
{\left\| \Psi  \right\|_0} \lesssim \int_{{t_k}}^{{t_{k + 1}}}  {{{\left\| {\left( {w_{k,n + 1}^{\left( l \right)}\left( s \right) \cdot \nabla } \right){{\bar v}_q}\left( s \right)} \right\|}_0}}\ \mathrm{d}s \lesssim  \delta _q^{\frac{1}{2}}{\lambda _q} \int_{{t_k}}^{{t_{k + 1}}}  {{{\left\| {w_{k,n + 1}^{\left( l \right)}\left( s \right)} \right\|}_0}}\ \mathrm{d}s .
\end{equation}
Using the fact that
\begin{equation*}
{\left\| {\tilde \Psi } \right\|_0} \leqslant C\left( \Lambda  \right)\mu _{q + 1}^{ - 1}\mathop {\sup }\limits_{t \in \left[ {{t_k},{t_{k + 1}}} \right)} \left| {f_{\xi ,k,n + 1}^{\left[ 1 \right]}\left( t \right)} \right|{\left\| {{A_{\xi ,k,{n+1}}}} \right\|_1} \lesssim \mu _{q + 1}^{ - 1}{\delta _{q + 1,n}}  {\lambda _q} ,
\end{equation*}
and
\begin{align*}
& \int_{{t_k}}^{{t_{k + 1}}}  {{{\left\| {\sum\limits_{\xi  \in \Lambda } {{f_{\xi ,k,n + 1}}\left( {{\mu _{q + 1}}s} \right)\Pi {\mathrm{div}}\, {A_{\xi ,k,{n+1}}}\left( s \right)} } \right\|}_0}}\ \mathrm{d}s \\
\leqslant & C\left( \Lambda  \right)\mu _{q + 1}^{ - 1}\mathop {\sup }\limits_{t \in \left[ {{t_k},{t_{k + 1}}} \right)} \left( {\left| {f_{\xi ,k,n + 1}^{\left[ 1 \right]}\left( t \right)} \right|{{\left\| {{A_{\xi ,k,{n+1}}}\left( t \right)} \right\|}_1}} \right) \\
\lesssim & \mu _{q + 1}^{ - 1}{\delta _{q + 1,n}}  {\lambda _q} ,
\end{align*}
we have
\begin{equation}\label{Newtonsteps-A-B-28}
{\left\| {\bar \Psi } \right\|_0} \lesssim {\left\| {\tilde \Psi } \right\|_0} + \int_{{t_k}}^{{t_{k + 1}}}  {{{\left\| {\sum\limits_{\xi  \in \Lambda } {{f_{\xi ,k,n + 1}}\left( {{\mu _{q + 1}}s} \right)\Pi {\mathrm{div}}\, {A_{\xi ,k,{n+1}}}\left( s \right)} } \right\|}_0}}\ \mathrm{d}s \lesssim \mu _{q + 1}^{ - 1}{\delta _{q + 1,n}}  {\lambda _q} .
\end{equation}
\par
Applying \eqref{Newtonsteps-A-B-27} and \eqref{Newtonsteps-A-B-28} to Eq.\eqref{Newtonsteps-A-B-23}, we obtain
\begin{equation*}
{\left\| {w_{k,n + 1}^{\left( l \right)}} \right\|_0} \lesssim  \delta _q^{\frac{1}{2}}{\lambda _q} \int_{{t_k}}^{{t_{k + 1}}}  {{{\left\| {w_{k,n + 1}^{\left( l \right)}\left( s \right)} \right\|}_0}}\ \mathrm{d}s + \mu _{q + 1}^{ - 1}{\delta _{q + 1,n}  {\lambda _q}}  .
\end{equation*}
Then, it follows from \eqref{parameters-S-30} and Gronwall's inequality that
\begin{equation}\label{Newtonsteps-A-B-29}
{\left\| {w_{k,n + 1}^{\left( l \right)}} \right\|_0} \leqslant \left( {{{\left\| {\tilde \Psi } \right\|}_0} + \mu _{q + 1}^{ - 1}{\delta _{q + 1,n}}} \right)\exp \left( {{\tau _q} \delta _q^{\frac{1}{2}}{\lambda _q}} \right) \lesssim \mu _{q + 1}^{ - 1}{\delta _{q + 1,n}}  {\lambda _q} .
\end{equation}
Employing \eqref{Newtonsteps-A-B-24}, \eqref{Newtonsteps-A-B-25}, \eqref{Newtonsteps-A-B-29}, \eqref{parameters-S-9} and \eqref{parameters-S-8}, we further deduce that
\begin{align}\label{Newtonsteps-A-B-40}
{\left\| {{{\bar D}_{t,q}}w_{k,n + 1}^{\left( l \right)}} \right\|_0} \leqslant & {\left\| {{{\bar D}_{t,q}}\Psi } \right\|_0} + {\left\| {{{\bar D}_{t,q}}\bar \Psi } \right\|_0} \nonumber  \\
\leqslant & \delta _q^{\frac{1}{2}}{\lambda _q}{\left\| {w_{k,n + 1}^{\left( l \right)}} \right\|_0} + C\left( \Lambda  \right){\left\| {{f_{\xi ,k,n + 1}}\left( {{\mu _{q + 1}}t} \right)\Pi {\mathrm{div}}\, {A_{\xi ,k,{n+1}}}} \right\|_0}  \nonumber  \\
\lesssim & \tau _q^{ - 1}{\left\| {w_{k,n + 1}^{\left( l \right)}} \right\|_0} + {\left\| {{A_{\xi ,k,{n+1}}}} \right\|_1}  \nonumber  \\
\lesssim & {\delta _{q + 1}}{\lambda _q}l_q^{ - \alpha } .
\end{align}
\par
Next, we examine ${\left\| {w_{k,n + 1}^{\left( l \right)}} \right\|_N}$ and ${\left\| {{{\bar D}_{t,q}}w_{k,n + 1}^{\left( l \right)}} \right\|_N}$ for $N \geqslant 1$. Using Lemma \ref{STN-HCS-A}, we have
\begin{equation*}
{\left\| {\left( {w_{k,n + 1}^{\left( l \right)} \cdot \nabla } \right){{\bar v}_q}} \right\|_N} \lesssim {\left\| {w_{k,n + 1}^{\left( l \right)}} \right\|_N}{\left\| {{{\bar v}_q}} \right\|_1} + {\left\| {w_{k,n + 1}^{\left( l \right)}} \right\|_0}{\left\| {{{\bar v}_q}} \right\|_{N + 1}} .
\end{equation*}
It then follows from \eqref{Newtonsteps-A-B-24}, \eqref{Newtonsteps-A-B-25}, \eqref{Newtonsteps-A-B-29} and Lemma \ref{TSP-E-L-A-1} that
\begin{align}\label{Newtonsteps-A-B-31}
{\left\| \Psi  \right\|_N} \lesssim & \int_{{t_k}}^{{t_{k + 1}}}  {{{\left\| {\left( {w_{k,n + 1}^{\left( l \right)}\left( s \right) \cdot \nabla } \right){{\bar v}_q}\left( s \right)} \right\|}_N} + \left| {t_{k+1} - s} \right|{{\left\| {{{\bar v}_q}} \right\|}_N}{{\left\| {\left( {w_{k,n + 1}^{\left( l \right)}\left( s \right) \cdot \nabla } \right){{\bar v}_q}\left( s \right)} \right\|}_1}}\ \mathrm{d}s \nonumber \\
\lesssim & \int_{{t_k}}^{{t_{k + 1}}}  {{{\left\| {w_{k,n + 1}^{\left( l \right)}} \right\|}_N}{{\left\| {{{\bar v}_q}} \right\|}_1} + \mu _{q + 1}^{ - 1}{\delta _{q + 1,n}}{\lambda _q}{{\left\| {{{\bar v}_q}} \right\|}_{N + 1}}}\ \mathrm{d}s \nonumber \\
& + \int_{{t_k}}^{{t_{k + 1}}}  {\left| {t_{k+1} - s} \right|{{\left\| {{{\bar v}_q}} \right\|}_N}\left( {{{\left\| {w_{k,n + 1}^{\left( l \right)}} \right\|}_1}{{\left\| {{{\bar v}_q}} \right\|}_1} + \mu _{q + 1}^{ - 1}{\delta _{q + 1,n}}{\lambda _q}{{\left\| {{{\bar v}_q}} \right\|}_2}} \right)}\ \mathrm{d}s ,
\end{align}
and
\begin{align}\label{Newtonsteps-A-B-33}
{\left\| {\bar \Psi } \right\|_N} \lesssim & {\left\| {\tilde \Psi } \right\|_N} + {\tau _q} {\left\| {{{\bar v}_q}} \right\|_N}{\left\| {\tilde \Psi } \right\|_1} + \int_{{t_k}}^{{t_{k + 1}}}  {{{\left\| {\sum\limits_{\xi  \in \Lambda } {{f_{\xi ,k,n + 1}}\left( {{\mu _{q + 1}}s} \right)\Pi {\mathrm{div}}\, {A_{\xi ,k,{n+1}}}\left( s \right)} } \right\|}_N}}\ \mathrm{d}s \nonumber \\
& + \int_{{t_k}}^{{t_{k + 1}}}  {\left| {t_{k+1} - s} \right| {{\left\| {{{\bar v}_q}} \right\|}_N}{{\left\| {\sum\limits_{\xi  \in \Lambda } {{f_{\xi ,k,n + 1}}\left( {{\mu _{q + 1}}s} \right)\Pi {\mathrm{div}}\, {A_{\xi ,k,{n+1}}}\left( s \right)} } \right\|}_1}}\ \mathrm{d}s .
\end{align}
Since
\begin{equation*}
{\left\| {\tilde \Psi } \right\|_N} \lesssim C\left( \Lambda  \right) {\tau _q}  \mu _{q + 1}^{ - 1}\mathop {\sup }\limits_{t \in \left[ {{t_k},{t_{k + 1}}} \right)} \left| {f_{\xi ,k,n + 1}^{\left[ 1 \right]}\left( t \right)} \right|{\left\| {{A_{\xi ,k,{n+1}}}} \right\|_{N + 1}} \lesssim {\tau _q} \mu _{q + 1}^{ - 1}{\left\| {{A_{\xi ,k,{n+1}}}} \right\|_{N + 1}} ,
\end{equation*}
we infer from \eqref{Newtonsteps-A-B-31} and \eqref{Newtonsteps-A-B-33} that
\begin{align*}
{\left\| {w_{k,n + 1}^{\left( l \right)}} \right\|_N} \lesssim & {\tau _q} \mu _{q + 1}^{ - 1}{\left\| {{A_{\xi ,k,{n+1}}}} \right\|_{N + 1}} + \tau _q^2 \mu _{q + 1}^{ - 1}{\left\| {{{\bar v}_q}} \right\|_N}{\left\| {{A_{\xi ,k,{n+1}}}} \right\|_2} \\
& + \int_{{t_k}}^{{t_{k + 1}}}  {{{\left\| {w_{k,n + 1}^{\left( l \right)}} \right\|}_N}{{\left\| {{{\bar v}_q}} \right\|}_1} + \mu _{q + 1}^{ - 1}{\delta _{q + 1,n}}{\lambda _q}{{\left\| {{{\bar v}_q}} \right\|}_{N + 1}}}\ \mathrm{d}s \\
& + \int_{{t_k}}^{{t_{k + 1}}}  {\left| {t_{k+1} - s} \right|{{\left\| {{{\bar v}_q}} \right\|}_N}\left( {{{\left\| {w_{k,n + 1}^{\left( l \right)}} \right\|}_1}{{\left\| {{{\bar v}_q}} \right\|}_1} + \mu _{q + 1}^{ - 1}{\delta _{q + 1,n}}{\lambda _q}{{\left\| {{{\bar v}_q}} \right\|}_2}} \right)}\ \mathrm{d}s \\
& + \int_{{t_k}}^{{t_{k + 1}}}  {{{\left\| {\sum\limits_{\xi  \in \Lambda } {{f_{\xi ,k,n + 1}}\left( {{\mu _{q + 1}}s} \right)\Pi {\mathrm{div}}\, {A_{\xi ,k,{n+1}}}\left( s \right)} } \right\|}_N}}\ \mathrm{d}s \\
& + \int_{{t_k}}^{{t_{k + 1}}}  {\left| {t_{k+1} - s} \right| {{\left\| {{{\bar v}_q}} \right\|}_N}{{\left\| {\sum\limits_{\xi  \in \Lambda } {{f_{\xi ,k,n + 1}}\left( {{\mu _{q + 1}}s} \right)\Pi {\mathrm{div}}\, {A_{\xi ,k,{n+1}}}\left( s \right)} } \right\|}_1}}\ \mathrm{d}s .
\end{align*}
Therefore, it follows from \eqref{lemma-GLTNNRS-B-M-1} and \eqref{Newtonsteps-A-B-19} that
\begin{align}\label{Newtonsteps-A-B-34}
{\left\| {w_{k,n + 1}^{\left( l \right)}} \right\|_N} \lesssim & \mu _{q + 1}^{ - 1}{\delta _{q + 1,n}}\lambda _q^{N + 1} {\tau _q}  \left( {1 + {\tau _q} \delta _q^{\frac{1}{2}}{\lambda _q}} \right)  + \delta _q^{\frac{1}{2}}{\lambda _q}\int_{{t_k}}^{{t_{k + 1}}}  {{{\left\| {w_{k,n + 1}^{\left( l \right)}} \right\|}_N}}\ \mathrm{d}s  \nonumber \\
& + {\tau _q} {\delta _q}\lambda _q^{N + 1}\int_{{t_k}}^{{t_{k + 1}}}  {{{\left\| {w_{k,n + 1}^{\left( l \right)}} \right\|}_1}}\ \mathrm{d}s .
\end{align}
\par
Taking $N = 1$ in \eqref{Newtonsteps-A-B-34}, we find
\begin{equation*}
{\left\| {w_{k,n + 1}^{\left( l \right)}} \right\|_1} \lesssim \mu _{q + 1}^{ - 1}{\delta _{q + 1,n}}\lambda _q^2 {\tau _q} \left( {1 + {\tau _q} \delta _q^{\frac{1}{2}}{\lambda _q} } \right) + \left( {\delta _q^{\frac{1}{2}}{\lambda _q} + {\tau _q}{\delta _q}\lambda _q^2} \right)\int_{{t_k}}^{{t_{k + 1}}}  {{{\left\| {w_{k,n + 1}^{\left( l \right)}} \right\|}_1}}\ \mathrm{d}s ,
\end{equation*}
which implies that
\begin{align}\label{Newtonsteps-A-B-35}
{\left\| {w_{k,n + 1}^{\left( l \right)}} \right\|_1} \lesssim & \mu _{q + 1}^{ - 1}{\delta _{q + 1,n}}\lambda _q^2 {\tau _q} \left( {1 + {\tau _q} \delta _q^{\frac{1}{2}}{\lambda _q} } \right)\exp \left( {{\tau _q} \left( {\delta _q^{\frac{1}{2}}{\lambda _q} + {\tau _q} {\delta _q}\lambda _q^2} \right)} \right)  \nonumber \\
\lesssim & {\tau _q} \mu _{q + 1}^{ - 1}{\delta _{{q + 1},n}}\lambda _q^2l_q^{ - \alpha } .
\end{align}
Invoking \eqref{Newtonsteps-A-B-34}, \eqref{Newtonsteps-A-B-35} and \eqref{parameters-S-10}, we obtain
\begin{equation*}
{\left\| {w_{k,n + 1}^{\left( l \right)}} \right\|_N} \lesssim \tau _q^2 \mu _{q + 1}^{ - 1}{\delta _q}{\delta _{{q + 1},n}}\lambda _q^{N + 3}l_q^{ - \alpha } +  \delta _q^{\frac{1}{2}}{\lambda _q}\int_{{t_k}}^{{t_{k + 1}}}  {{{\left\| {w_{k,n + 1}^{\left( l \right)}} \right\|}_N}}\ \mathrm{d}s   ,
\end{equation*}
which implies that
\begin{align}\label{Newtonsteps-A-B-36}
{\left\| {w_{k,n + 1}^{\left( l \right)}} \right\|_N} \lesssim & \tau _q^2 \mu _{q + 1}^{ - 1}{\delta _q}{\delta _{{q + 1},n}}\lambda _q^{N + 3}l_q^{ - \alpha }\exp \left( {{\tau _q} \delta _q^{\frac{1}{2}}{\lambda _q}} \right)  \nonumber \\
\lesssim & \mu _{q + 1}^{ - 1}{\delta _{{q + 1},n}}\lambda _q^{N + 1}l_q^{ - \alpha }     .
\end{align}
Furthermore, it follows from \eqref{lemma-GLTNNRS-B-M-1}, \eqref{Newtonsteps-A-B-29} and \eqref{Newtonsteps-A-B-36} that
\begin{equation*}
{\left\| {{{\bar D}_{t,q}}\Psi } \right\|_N} \lesssim {\left\| {w_{k,n + 1}^{\left( l \right)}} \right\|_N}{\left\| {{{\bar v}_q}} \right\|_1} + {\left\| {w_{k,n + 1}^{\left( l \right)}} \right\|_0}{\left\| {{{\bar v}_q}} \right\|_{N + 1}} \lesssim {\delta _{{q + 1},n}}\lambda _q^{N + 1}l_q^{ - \alpha }   .
\end{equation*}
Employing the similar argument in \eqref{Newtonsteps-A-B-40}, we deduce that
\begin{align}\label{Newtonsteps-A-B-41}
{\left\| {{{\bar D}_{t,q}}w_{k,n + 1}^{\left( l \right)}} \right\|_N} \lesssim & {\left\| {{{\bar D}_{t,q}}\Psi } \right\|_N} + {\left\| {{{\bar D}_{t,q}}\bar \Psi } \right\|_N}  \nonumber \\
\lesssim & {\left\| {{{\bar D}_{t,q}}\Psi } \right\|_N} + {\left\| {\sum\limits_{\xi  \in \Lambda } {{f_{\xi ,k,n + 1}}\left( {{\mu _{q + 1}}t} \right)\Pi {\mathrm{div}}\, {A_{\xi ,k,{n+1}}}} } \right\|_N}  \nonumber \\
\lesssim & {\delta _{{q + 1}}}\lambda _q^{N + 1}l_q^{ - \alpha }    .
\end{align}
\par
Consequently, we derive from \eqref{Newtonsteps-A-B-29}, \eqref{Newtonsteps-A-B-36} and the definition \eqref{Newtonsteps-A-4} that
\begin{equation*}
{\left\| {w_{q + 1,n + 1}^{\left( n \right)}} \right\|_N} \lesssim \mathop {\sup }\limits_{k \in {\mathbb{Z}_{q,n}}} {\left\| {w_{k,n + 1}^{\left( l \right)}} \right\|_N}  \lesssim  \mu _{q + 1}^{ - 1}{\delta _{q + 1}}\lambda _q^{N + 1}l_q^{ - \alpha }   .
\end{equation*}
Combining \eqref{Newtonsteps-A-B-40}, \eqref{Newtonsteps-A-B-41} and the definition \eqref{Newtonsteps-A-4}, we obtain
\begin{equation*}
{\left\| {{{\bar D}_{t,q}}w_{q + 1,n + 1}^{\left( n \right)}} \right\|_N} \lesssim  C\left( n \right)\mathop {\sup }\limits_{k \in {\mathbb{Z}_{q,n}}} {\left\| {{{\bar D}_{t,q}}w_{k,n + 1}^{\left( l \right)}} \right\|_N} \lesssim  {\delta _{q + 1}}\lambda _q^{N + 1}l_q^{ - \alpha }  . 
\end{equation*}
The proof is complete. \qed
\section{Nash iteration scheme}\label{NSAH-US}
This section constructs the Nash perturbation with the prescribed energy gap to iterate the $q$-step stochastic Euler--Reynolds system \eqref{ERNE-3A.1} into its $\left( {q + 1} \right)$-step counterpart. The Nash perturbation consists of principal perturbations, divergence correctors and intermittency perturbations, which generate additional stochastic pressure and random Reynolds stress. The Reynolds stress incorporates the linear error, the intermittent oscillation error and the random oscillation error. Since progressive measurability is preserved throughout the Newton iteration layer, all Newton--Nash elements constructed in this section remain progressively measurable with respect to the underlying filtration. Furthermore, all estimates for the Nash iteration are restricted to the stopping time interval $\mathop  \cap \limits_{j = 0}^{q + 1} \left[ {0,{\mathfrak{t}_j}} \right]$.
\subsection{Nash perturbation and iteration scheme}\label{NISTYU-SKO-A}
Denote ${{\bar v}_{q,\Gamma }}\left( {t,x} \right) \triangleq {{\bar v}_q}\left( {t,x} \right) + w_{q + 1}^{\left( \Gamma \right)} \left( {t,x} \right)$. Let ${{\tilde X}_t}$ be the Lagrangian flow satisfying
\begin{equation}\label{KDKJFIG-DNGJU-1-1}
\left\{ {\begin{array}{*{20}{l}}
{\partial _s} {{\tilde X}_k}\left( {x,t} \right) = {{\bar v}_{q, \Gamma }}\left( {{{\tilde X}_k}\left( {x,t} \right),t} \right),\\
{{\tilde X}_k}\left( {x,t} \right) = x ,
\end{array}} \right.
\end{equation}
and let ${{\tilde \Phi }_k}\left( {t,x} \right)$ be the backward flow of ${{\bar v}_{q, \Gamma }} \left( {t,x} \right)$ such that
\begin{equation}\label{Nashsteps-A-P-1}
\left\{ {\begin{array}{*{20}{l}}
{\partial _t}{{\tilde \Phi }_k} + \left( {{{\bar v}_{q, \Gamma }} \cdot \nabla } \right){{\tilde \Phi }_k} = 0,\\
{{\tilde \Phi }_k}\left( {{t_k}, x} \right) = x .
\end{array}} \right. 
\end{equation}
Choose the temporal mollifier ${{\tilde \varsigma }_{{{\tilde l}_{q}}}}$ and the spatial mollifier ${{\hat \varsigma }_{{{\hat l}_q}}}$, where the mollification scales ${{\tilde l}_{q}}$ and ${{\hat l}_q}$ are defined by \eqref{parameters-S-11} and \eqref{parameters-S-11-A-1}, respectively. Denote $\hat f \triangleq f * {{\hat \varsigma }_{{{\hat l}_q}}}$ and 
\begin{equation}\label{Nashsteps-A-AIJS-P-2}
{{\tilde R}_{q,n}}\left( {t,x} \right) \triangleq \int_{ - {{\tilde l}_{q}}}^{{{\tilde l}_{q}}} {{{ \mathring{R}}_{q,n}}\left( {t + s,{{\tilde X}_t}\left( {x,t + s} \right)} \right){{\tilde \varsigma }_{{{\tilde l}_{q}}}}\left( s \right)}\ \mathrm{d}s .
\end{equation}
\par
Let $\mathfrak{e}\left( t \right)$ be a strictly decreasing function and satisfy
\begin{equation}
0 \leqslant \mathfrak{e}\left( t \right) - \left\| {{v_q}} \right\|_{{L^2}}^2 \lesssim {\delta _{q + 1}}  .  \label{SJKFJIFJSKJGG-ASD-1}
\end{equation}
We choose the amplitude function
\begin{equation}\label{Nashsteps-A-P-2}
{{\tilde \eta }_{\xi ,k,n}} = {\mathfrak{E}_{q,\xi ,k,n}^{\frac{1}{2}}}{\chi _k}{\gamma _\xi }\left( {\nabla {{\tilde \Phi }_k}\nabla \tilde \Phi _k^ \bot  - {\mathfrak{E}_{q,\xi ,k,n}^{ - 1}}\nabla {{\tilde \Phi }_k}{{\tilde R}_{q,n}}\nabla \tilde \Phi _k^ \bot } \right) ,
\end{equation}
where 
\begin{equation*}
{\mathfrak{E}_{q,\xi ,k,n}} \triangleq \frac{1}{{8a_W^2\left| \Lambda  \right|\left| {{\mathbb{Z}_{q,n}}} \right|\left( {\Gamma  + 1} \right)}}{\mathscr{\hat M}^{ - 1}}\left( {t,r} \right)g_{\xi ,k,n + 1}^{ - 2} \left( {{\mu _{q + 1}}t} \right) \chi _k^{ - 2}{\mathfrak{E}_q} ,
\end{equation*}
and the energy gap reads
\begin{equation}\label{LOKJI-A-FGH-A-1}
{\mathfrak{E}_q} = \mathfrak{e}\left( t \right) - \left\| {{{\hat u}_q}} \right\|_{{L^2}}^2 - \frac{1}{2}{\delta _{q + 2}} + \left\| {{\mathfrak{\hat W}_q}} \right\|_{{L^2}}^2 + 2{\int_0^t {\left\langle {{{\hat v}_q},\mathrm{d}{\mathfrak{W}_q}} \right\rangle } _{{\mathbb{T}^2}}} .
\end{equation}
In addition, the function $\mathscr{\hat M}\left( {t,r} \right) = \int_{{\mathbb{T}^2}} {{{\cos }^2}\left( {{\xi ^ \bot } \cdot {\lambda _{q + 1}}{{\tilde \Phi }_k}} \right)\mathfrak{D}_{\tilde \mu }^2\left( {t,{\lambda _{q + 1}}{{\tilde \Phi }_k}} \right)} \mathrm{d}x$ satisfies ${C_1} \leqslant \left| { \mathscr{\hat M}\left( {t,r} \right)} \right| \leqslant {C_2}$ for any $t \in \left[ {0,T} \right]$ and $r \in \mathbb{Z}^+$ (see Lemma \ref{IBW-P-A-1} for more details), where the intermittent Dirichlet kernel ${\mathfrak{D}_{\tilde \mu }}\left( {t,x} \right)$ is defined by \eqref{IBW-GJI-A-1}. It is straightforward to check that
\begin{equation*}
\left| {{\mathfrak{E}_{q }}} \right| \gtrsim {\delta _{q }} - {\delta _{q + 2}} \gtrsim {\delta _{q + 1,\Gamma }}  ,
\end{equation*}
and
\begin{equation*}
{\mathrm{supp}_\mathrm{t}} \, {{\tilde \eta }_{\xi ,k,n}} \subset {\mathrm{supp}_\mathrm{t}} \, {{\tilde R}_{q,n}} \subset {\mathrm{supp}_\mathrm{t}} \, {\mathring{R}_{q,n}} .
\end{equation*}
Since $\left| {{\mathfrak{E}_{0,\xi ,k,n}}} \right| \lesssim \left| {{\mathfrak{E}_0}} \right| \lesssim {\delta _1}$, without loss of generality, we assume that
\begin{equation*}
\left| {{\mathfrak{E}_q}} \right| \lesssim {\delta _{q + 1}} .
\end{equation*}
We verify the inductive estimate of ${\mathfrak{E}_q}$ in Proposition \ref{RSRI-ESI-GAP}.
\par
Define the principal part of Nash perturbation
\begin{equation}\label{Nashsteps-A-P-3}
w_{q + 1}^{\left( p \right)}\left( {t,x} \right) \triangleq \sum\limits_{n = 0}^\Gamma  {\sum\limits_{k \in {\mathbb{Z}_{q,n}}} {\sum\limits_{\xi  \in \Lambda } {{g_{\xi ,k,n + 1}}\left( {{\mu _{q + 1}}t} \right){{\tilde \eta }_{\xi ,k,n}}{{\left( {\nabla {{\tilde \Phi }_k}} \right)}^{ - 1}}{\mathbb{W}_\xi }\left( {t,{\lambda _{q + 1}}{{\tilde \Phi }_k}} \right)} } }  ,
\end{equation}
where ${\mathbb{W}_\xi }\left( {t,\cdot} \right)$ is the intermittent building block specified by \eqref{IBW-GJI-A-2}. To construct the divergence-free Nash perturbation, we take the corrector part
\begin{align}\label{Nashsteps-A-P-4}
w_{q + 1}^{\left( c \right)}\left( {t,x} \right) \triangleq  \lambda _{q + 1}^{ - 1}\sum\limits_{n = 0}^\Gamma  \sum\limits_{k \in {\mathbb{Z}_{q,n}}} \sum\limits_{\xi  \in \Lambda } & {{g_{\xi ,k,n + 1}}\left( {{\mu _{q + 1}}t} \right){\mathfrak{S}_\xi }\left( {{\lambda _{q + 1}}{{\tilde \Phi }_k}} \right)}   \nonumber  \\
&  {\nabla ^ \bot }\left( {{{\tilde \eta }_{\xi ,k,n}}{{\left( {\nabla {{\tilde \Phi }_k}} \right)}^{ - 1}}{{\mathfrak{D}_{\tilde \mu }}}\left( {t,{\lambda _{q + 1}}{{\tilde \Phi }_k}} \right)\left( {\nabla {{\tilde \Phi }_k}} \right)} \right) ,
\end{align}
where ${\mathfrak{S}_\xi }\left( { \cdot } \right)$ is the stream function given by \eqref{IUHYU-DJ}, ${{\mathfrak{D}_{\tilde \mu }}}\left( {t, \cdot } \right)$ is the intermittent Dirichlet kernel defined by \eqref{XINHUI-SJI-A}, and ${\nabla ^ \bot } \triangleq \left( { - {\partial _{{x_2}}},{\partial _{{x_1}}}} \right)$. Owing to \eqref{BWsdfgg-P-A-1}, we have
\begin{equation*}
{W_\xi }\left( {{\lambda _{q + 1}}{{\tilde \Phi }_k}} \right) = \lambda _{q + 1}^{ - 1}\left( {\nabla {{\tilde \Phi }_k}} \right) \cdot {\nabla ^ \bot }\mathfrak{S}_{\xi} \left( {{\lambda _{q + 1}}{{\tilde \Phi }_k}} \right) ,
\end{equation*}
which implies that
\begin{equation}\label{Nashsteps-A-P-5}
{\mathbb{W}_\xi }\left( {t,{\lambda _{q + 1}}{{\tilde \Phi }_k}} \right) = \lambda _{q + 1}^{ - 1}{\mathfrak{D}_{\tilde \mu }}\left( {t,{\lambda _{q + 1}}{{\tilde \Phi }_k}} \right)\left( {\nabla {{\tilde \Phi }_k}} \right) \cdot {\nabla ^ \bot }\mathfrak{S}_{\xi} \left( {{\lambda _{q + 1}}{{\tilde \Phi }_k}} \right) .
\end{equation}
Therefore, it follows from \eqref{Nashsteps-A-P-3}-\eqref{Nashsteps-A-P-5} that
\begin{align}\label{Nashsteps-A-P-6}
\mathrm{div}\left( {w_{q + 1}^{\left( p \right)} + w_{q + 1}^{\left( c \right)}} \right) = & \lambda _{q + 1}^{ - 1}\mathrm{div} \sum\limits_{n = 0}^\Gamma  {\sum\limits_{k \in {\mathbb{Z}_{q,n}}} {\sum\limits_{\xi  \in \Lambda } {{g_{\xi ,k,n + 1}}\left( {{\mu _{q + 1}}t} \right){{\tilde \eta }_{\xi ,k,n}}{{\left( {\nabla {{\tilde \Phi }_k}} \right)}^{ - 1}}} } }  \nonumber \\
& \hspace*{9em}  {\mathfrak{D}_{\tilde \mu }}\left( {t,{\lambda _{q + 1}}{{\tilde \Phi }_k}} \right)\left( {\nabla {{\tilde \Phi }_k}} \right) \cdot {\nabla ^ \bot }\mathfrak{S}_{\xi} \left( {{\lambda _{q + 1}}{{\tilde \Phi }_k}} \right) \nonumber \\
& + \lambda _{q + 1}^{ - 1}\mathrm{div} \sum\limits_{n = 0}^\Gamma  {\sum\limits_{k \in {\mathbb{Z}_{q,n}}} {\sum\limits_{\xi  \in \Lambda } {{g_{\xi ,k,n + 1}}\left( {{\mu _{q + 1}}t} \right){\mathfrak{S}_\xi }\left( {{\lambda _{q + 1}}{{\tilde \Phi }_k}} \right)} } }   \nonumber \\
&  \hspace*{9em}  {\nabla ^ \bot }\left( {{{\tilde \eta }_{\xi ,k,n}}{{\left( {\nabla {{\tilde \Phi }_k}} \right)}^{ - 1}}{{\mathfrak{D}_{\tilde \mu }}}\left( {t,{\lambda _{q + 1}}{{\tilde \Phi }_k}} \right)\left( {\nabla {{\tilde \Phi }_k}} \right)} \right)   \nonumber \\ 
= & \lambda _{q + 1}^{ - 1}\mathrm{div} \, {\nabla ^ \bot }\sum\limits_{n = 0}^\Gamma  {\sum\limits_{k \in {\mathbb{Z}_{q,n}}} {\sum\limits_{\xi  \in \Lambda } {{g_{\xi ,k,n + 1}}\left( {{\mu _{q + 1}}t} \right){{\tilde \eta }_{\xi ,k,n}}{{\left( {\nabla {{\tilde \Phi }_k}} \right)}^{ - 1}}} } }  \nonumber \\ 
&  \hspace*{9em}  {\mathfrak{D}_{\tilde \mu }}\left( {t,{\lambda _{q + 1}}{{\tilde \Phi }_k}} \right)\left( {\nabla {{\tilde \Phi }_k}} \right) \cdot \mathfrak{S}_{\xi} \left( {{\lambda _{q + 1}}{{\tilde \Phi }_k}} \right)  \nonumber \\ 
= & 0 .
\end{align}
Moreover,
\begin{equation}\label{THENDM-DKF-SUPP-1}
{\mathrm{supp}_\mathrm{t}} \, \left( {w_{q + 1}^{\left( p \right)} + w_{q + 1}^{\left( c \right)}} \right) \subset {\mathrm{supp}_\mathrm{t}} \, {{\tilde \eta }_{\xi ,k,n}} .
\end{equation}
\par
In view of \eqref{IBW-P-A-FGHY}, the intermittent building block ${\mathbb{W}_{\xi }}\left( {t,x} \right)$ exhibits the temporal oscillation term $4{{\tilde \mu }^{ - 1}}a_W^2\xi {\cos ^2}\left( {\xi  \cdot x} \right){\partial _t}\left( {\mathfrak{D}_{\tilde \mu }^2} \right)$. To quantify its intermittency intensity, we define the intermittency perturbation
\begin{equation}\label{THENDM-DKF}
w_{q + 1}^{\left( i \right)}\left( {t,x} \right) \triangleq {{\tilde \mu }^{ - 1}}\sum\limits_{n = 0}^\Gamma  {\sum\limits_{k \in {\mathbb{Z}_{q,n}}} {\sum\limits_{\xi  \in \Lambda } {\Pi {\Pi _{ \ne 0}}{g_{\xi ,k,n + 1}}\left( {{\mu _{q + 1}}t} \right)\xi ^ \bot \tilde \eta _{\xi ,k,n}^2\mathfrak{D}_{\tilde \mu }^2\left( {t,{\lambda _{q + 1}}{{\tilde \Phi }_k}} \right)} } } ,
\end{equation}
where ${\Pi _{ \ne 0}} \triangleq \mathrm{Id} - \int_{{\mathbb{T}^2}}  \cdot  \mathrm{d}x$ is the frequency operator of $\Pi$. We observe that $\mathrm{div} \, w_{q + 1}^{\left( i \right)} = 0$ and
\begin{equation}\label{THENDM-DKF-SUPP-2}
{\mathrm{supp}_\mathrm{t}} \, w_{q + 1}^{\left( i \right)} \subset {\mathrm{supp}_\mathrm{t}} \, {{\tilde \eta }_{\xi ,k,n}} .
\end{equation}
Then, the total Nash perturbation reads
\begin{equation}\label{Nashsteps-A-P-7}
w_{q + 1}^{\left( s \right)}\left( {t,x} \right) = w_{q + 1}^{\left( p \right)}\left( {t,x} \right) + w_{q + 1}^{\left( c \right)}\left( {t,x} \right) + w_{q + 1}^{\left( i \right)}\left( {t,x} \right) ,
\end{equation}
which satisfies ${\mathrm{supp}_\mathrm{t}} \, w_{q + 1}^{\left( s \right)} \subset {\mathrm{supp}_\mathrm{t}} \, {{\tilde \eta }_{\xi ,k,n}}$ and
\begin{equation}\label{Nashsteps-A-P-7-sjdi}
\mathrm{div}\,   w_{q + 1}^{\left( s \right)} = 0 .
\end{equation}
\par
Denote the total Newton--Nash perturbation
\begin{equation}\label{NISTYU-SKO-A-PPA-1}
{w_{q + 1}}\left( {t,x} \right) \triangleq w_{q + 1}^{\left( \Gamma \right)}\left( {t,x} \right) + w_{q + 1}^{\left( s \right)}\left( {t,x} \right) .
\end{equation}
We choose the $\left( {q + 1} \right)$-step Nash elements
\begin{equation}\label{NISTYU-SKO-A-PPA-2}
{u_{q + 1}} \left( {t,x} \right) = {u_q} \left( {t,x} \right) + {w_{q + 1}} \left( {t,x} \right) + \mathfrak{W}_{q + 1}^{loc} \left( {t,x} \right) ,\ \ {v_{q + 1}} \left( {t,x} \right) = {v_q} \left( {t,x} \right) +  {w_{q + 1}} \left( {t,x} \right) , 
\end{equation}
\begin{equation}\label{NISTYU-SKO-A-PWWWPA-2}
{\mathfrak{W}_{q + 1}} \left( {t,x} \right) = {\mathfrak{W}_q} \left( {t,x} \right) + \mathfrak{W}_{q + 1}^{loc} \left( {t,x} \right) , 
\end{equation}
\begin{equation}\label{NISTYU-SKO-A-PPA-3}
{\mathfrak{p}_{q + 1}} \left( {t,x} \right) = {\mathfrak{p}_{q,\Gamma }} \left( {t,x} \right) - \mathfrak{p}_{q + 1}^s\left( {t,x} \right) ,
\end{equation}
and the $\left( {q + 1} \right)$-step Reynolds stress
\begin{equation}\label{NISTYU-SKO-A-PPA-4}
{\mathring{R}_{q + 1}} \left( {t,x} \right) = \mathring{R}_{q + 1}^{\left( l \right)} \left( {t,x} \right) + \mathring{R}_{q + 1}^{\left( i \right)} \left( {t,x} \right) + \mathring{R}_{q + 1}^{\left( r \right)} \left( {t,x} \right) ,
\end{equation}
where
\begin{equation}\label{NISTYU-SKO-A-PPA-FJU-sd}
\mathfrak{W}_{q + 1}^{loc} \left( {t,x} \right)  \triangleq  \frac{{\varsigma _{q+1}}}{\sqrt 2 \pi }{\mathfrak{B}_{q+1}}\left( t \right) \delta _{q}^{\alpha}  \delta _{q+2}^{\frac{1}{2}}\cos \left( {{\lambda _{q+1}}{x_1}} \right){{\mathbf{e}}_2} ,
\end{equation}
and the stochastic pressure reads
\begin{equation}\label{LIJIFNI-DKFJ-L-56}
\mathfrak{p}_{q + 1}^s\left( {t,x} \right) = {\nabla ^{ - 1}} \mathrm{div}\left( {{u_{q + 1}} \otimes \mathfrak{W}_{q + 1}^{loc} + \mathfrak{W}_{q + 1}^{loc} \otimes {u_{q + 1}} - \mathfrak{W}_{q + 1}^{loc} \otimes \mathfrak{W}_{q + 1}^{loc} + w_{q + 1}^{\left( s \right)} \otimes {\mathfrak{W}_q} + {\mathfrak{W}_q} \otimes w_{q + 1}^{\left( s \right)}  } \right) ,
\end{equation}
the linear error is
\begin{equation}\label{LIJIFNI-DKFJ-L-1}
\mathring{R}_{q + 1}^{\left( l \right)} = \mathcal{R} \left( {{{\bar D}_{t,\Gamma }}\left( {w_{q + 1}^{\left( p \right)} + w_{q + 1}^{\left( c \right)}} \right) - \left( {{{\bar v}_{q,\Gamma }} \cdot \nabla } \right){w_{q + 1}^{\left( s \right)}}} \right) ,
\end{equation}
the intermittent oscillation error is
\begin{equation}\label{LIJIFNI-DKFJ-O-2}
\mathring{R}_{q + 1}^{\left( i \right)} = \mathcal{R} \left( {\mathrm{div}\left( {{S_{q,\Gamma }} + w_{q + 1}^{\left( s \right)} \otimes w_{q + 1}^{\left( s \right)} } \right) + {\partial _t}w_{q + 1}^{\left( i \right)}} \right) ,
\end{equation}
and the random oscillation error is
\begin{equation}\label{LIJIFNI-DKFJ-R-3}
\mathring{R}_{q + 1}^{\left( r \right)} = {\mathring{R}_{q,\Gamma }} + {\mathfrak{P}_{q + 1,\Gamma }} + w_{q + 1}^{\left( s \right)} \otimes {v_{q,\Gamma }}  + {v_{q,\Gamma }} \otimes w_{q + 1}^{\left( s \right)} + {{\bar v}_{q,\Gamma }} \otimes w_{q + 1}^{\left( i \right)} .
\end{equation}
Here, we point out that ${u_{q,\Gamma }} = {u_q} + w_{q + 1}^{\left( \Gamma \right)}$, ${v_{q,\Gamma }} = {v_q} + w_{q + 1}^{\left( \Gamma \right)}$ and
\begin{equation}\label{HIE-WWW-1.1}
{\mathop {\sup }\limits_{t \in \left[ {0,{T}} \right]} }\mathbb{E}\left\| {{\mathfrak{W}_{q + 1}}\left( t \right) - {\mathfrak{W}_{q}}\left( t \right)} \right\|_N^2 \leqslant \frac{T{\varsigma _0 ^2}}{\pi ^2 }  {\delta _{q + 1}}\lambda _{q + 1}^{2N} , \ \ \forall \, N \in {\mathbb{Z}^ + } \cup \left\{ 0 \right\} ,
\end{equation}
\par
According to the temporal supports of Newtonian elements ${\mathring{R}_{q,\Gamma }}$, ${S_{q,\Gamma }}$ and ${\mathfrak{P}_{q + 1,\Gamma }}$ in \eqref{OPIUY-A-1} and \eqref{OPIUY-A-2}, we have
\begin{equation*}
{\mathrm{supp}_\mathrm{t}} \, {\mathring{R}_{q + 1}} \subset  {\mathrm{supp}_\mathrm{t}} \, {{\tilde \eta }_{\xi ,k,\Gamma}} \cup {\mathrm{supp}_\mathrm{t}} \, {\mathring{R}_{q + 1,\Gamma}} \cup {\mathrm{supp}_\mathrm{t}} \, {S_{q + 1,\Gamma}} \cup {\mathrm{supp}_\mathrm{t}} \, {\mathfrak{P}_{q + 1,\Gamma}}  ,
\end{equation*}
which implies that
\begin{equation}\label{LIJIFNI-DKFJ-R-ss3}
{\mathrm{supp}_\mathrm{t}} \, {\mathring{R}_{q + 1}} \subset \mathop  \cap \limits_{j = 0}^{q + 1} \left[ {0,{\mathfrak{t}_j}} \right] \cap \left( {\frac{1}{4}T ,\frac{3}{4}T - 2\delta _{q+1}^{ - \frac{1}{2}}\lambda _{q+1}^{ - 1} + 4\Gamma {\tau _{q+1}} } \right] .
\end{equation}
\par
\begin{lemma}\label{N+1s-NASH-PIT}
For any $t \in \mathop  \cap \limits_{j = 0}^{q + 1} \left[ {0,{\mathfrak{t}_j}} \right]$, the combination $\left( {{u_{q + 1}},{v_{q + 1}},{\mathfrak{p}_{q + 1}},{\mathring{R}_{q + 1}}} \right)$ fulfills the $\left( {q + 1} \right)$-step stochastic Euler--Reynolds system
\begin{equation}\label{KIU-EU-N-1}
\left\{ {\begin{array}{*{20}{l}}
{\partial _t}{v_{q + 1}} + \mathrm{div} \left( {{u_{q + 1}} \otimes {u_{q + 1}}} \right) + \nabla {\mathfrak{p}_{q + 1}} = \mathrm{div}\,   {\mathring{R}_{q + 1}},\\
{u_{q + 1}} = {v_{q + 1}} + \mathfrak{W} _{q + 1},\\
\mathrm{div}\,   {u_{q + 1}} = 0 .
\end{array}} \right.  
\end{equation}
\end{lemma}
\par
\noindent{\textbf{Proof}}. 
Employing Eq.\eqref{ERNE-3A.1}, \eqref{NISTYU-SKO-A-PPA-1}, \eqref{NISTYU-SKO-A-PPA-2}, \eqref{NISTYU-SKO-A-PPA-3} and \eqref{NISTYU-SKO-A-PPA-4}, we obtain
\begin{align*}
&{\partial _t}\left( {{v_{q + 1}} - {w_{q + 1}}} \right) + \mathrm{div}\left( {\left( {{u_{q + 1}} - {w_{q + 1}}} \right) \otimes \left( {{u_{q + 1}} - {w_{q + 1}}} \right)} \right) + \nabla \left( {{\mathfrak{p}_{q + 1}} - {\mathfrak{p}_{q,\Gamma }} + {\mathfrak{p}_q} + \mathfrak{p}_{q + 1}^s} \right) \\
= & \mathrm{div}\left( {{\mathring{R}_{q + 1}} - \mathring{R}_{q + 1}^{\left( l \right)} - \mathring{R}_{q + 1}^{\left( i \right)} - \mathring{R}_{q + 1}^{\left( r \right)} + {\mathring{R}_{q}}} \right) ,
\end{align*}
which implies that
\begin{equation}\label{N+1s-NASH-PIT-PP-1}
{\partial _t}{v_{q + 1}} + \mathrm{div}\left( {{u_{q + 1}} \otimes {u_{q + 1}}} \right) + \nabla {\mathfrak{p}_{q + 1}} - \mathrm{div}\,   {\mathring{R}_{q + 1}} = \mathfrak{G} ,
\end{equation}
with
\begin{align*}
\mathfrak{G} \triangleq & \mathrm{div}\left( {{u_{q + 1}} \otimes {w_{q + 1}} + {w_{q + 1}} \otimes {u_{q + 1}} - {w_{q + 1}} \otimes {w_{q + 1}}} \right) \\
& + {\partial _t}{w_{q + 1}} + \nabla {\mathfrak{p}_{q,\Gamma }} - \nabla {\mathfrak{p}_q} - \nabla \mathfrak{p}_{q + 1}^s + \mathrm{div}\,   \mathring{R}_{q} - \mathrm{div}\left( {\mathring{R}_{q + 1}^{\left( l \right)} + \mathring{R}_{q + 1}^{\left( i \right)} + \mathring{R}_{q + 1}^{\left( r \right)}} \right) \\
& + \mathrm{div}\left( {{u_{q + 1}} \otimes \mathfrak{W}_{q + 1}^{loc} + \mathfrak{W}_{q + 1}^{loc} \otimes {u_{q + 1}} - \mathfrak{W}_{q + 1}^{loc} \otimes \mathfrak{W}_{q + 1}^{loc}} \right) \\
&  - \mathrm{div} \left( {{w_{q + 1}} \otimes \mathfrak{W}_{q + 1}^{loc} + \mathfrak{W}_{q + 1}^{loc} \otimes {w_{q + 1}}  + w_{q + 1}^{\left( s \right)} \otimes {\mathfrak{W}_q} + {\mathfrak{W}_q} \otimes w_{q + 1}^{\left( s \right)} } \right).
\end{align*}
\par
Owing to
\begin{align*}
& {u_{q + 1}} \otimes {w_{q + 1}} + {w_{q + 1}} \otimes {u_{q + 1}} - {w_{q + 1}} \otimes {w_{q + 1}} \\
= & \left( {w_{q + 1}^{\left( s \right)} + {u_{q,\Gamma }} + \mathfrak{W}_{q + 1}^{loc}} \right) \otimes \left( {w_{q + 1}^{\left( s \right)} + {u_{q,\Gamma }} - {u_q}} \right)  + \left( {w_{q + 1}^{\left( s \right)} + {u_{q,\Gamma }} - {u_q}} \right) \otimes \left( {w_{q + 1}^{\left( s \right)} + {u_{q,\Gamma }} + \mathfrak{W}_{q + 1}^{loc}} \right) \\
& - \left( {w_{q + 1}^{\left( s \right)} + {u_{q,\Gamma }} - {u_q}} \right) \otimes \left( {w_{q + 1}^{\left( s \right)} + {u_{q,\Gamma }} - {u_q}} \right) \\
= & w_{q + 1}^{\left( s \right)} \otimes w_{q + 1}^{\left( s \right)} + w_{q + 1}^{\left( s \right)} \otimes {v_{q,\Gamma }} + {v_{q,\Gamma }} \otimes w_{q + 1}^{\left( s \right)} + {u_{q,\Gamma }} \otimes {u_{q,\Gamma }} \\
& + w_{q + 1}^{\left( s \right)} \otimes {\mathfrak{W}_q} + {\mathfrak{W}_q} \otimes w_{q + 1}^{\left( s \right)} - {u_q} \otimes {u_q} + \mathfrak{W}_{q + 1}^{loc} \otimes  {w_{q + 1}} + {w_{q + 1}} \otimes \mathfrak{W}_{q + 1}^{loc} ,
\end{align*}
we have
\begin{align*}
\mathfrak{G} = & \left( {{\partial _t}{v_{q,\Gamma }} + \mathrm{div}\left( {{u_{q,\Gamma }} \otimes {u_{q,\Gamma }}} \right) + \nabla {\mathfrak{p}_{q,\Gamma }} - \mathrm{div}\left( {{\mathring{R}_{q,\Gamma }} + {\mathfrak{P}_{q + 1,\Gamma }} + {S_{q,\Gamma }}} \right)} \right) \\
&  - \left( {{\partial _t}{v_q} + \mathrm{div}\left( {{u_q} \otimes {u_q}} \right) + \nabla {\mathfrak{p}_q} - \mathrm{div}\,   {\mathring{R}_q}} \right)  + {\partial _t}w_{q + 1}^{\left( s \right)} - {\partial _t}w_{q + 1}^{\left( i \right)} - {\partial _t}w_{q + 1}^{\left( p \right)} - {\partial _t}w_{q + 1}^{\left( c \right)}  \\
& + {\left( {{{\bar v}_{q,\Gamma }} \cdot \nabla } \right)w_{q + 1}^{\left( s \right)}} - \mathrm{div}\left( {{{\bar v}_{q,\Gamma }} \otimes \left( {w_{q + 1}^{\left( p \right)} + w_{q + 1}^{\left( c \right)} + w_{q + 1}^{\left( i \right)}} \right)} \right) .
\end{align*}
Therefore, it follows from Eq.\eqref{ERNE-3A.1} and Eq.\eqref{N+1s-NE-PIT-P-9} that
\begin{equation}\label{N+1s-NASH-PIT-PP-2}
\mathfrak{G} = 0   .
\end{equation}
Combining \eqref{N+1s-NASH-PIT-PP-1} and \eqref{N+1s-NASH-PIT-PP-2}, we get
\begin{equation}
{\partial _t}{v_{q + 1}} + \mathrm{div}\left( {{u_{q + 1}} \otimes {u_{q + 1}}} \right) + \nabla {\mathfrak{p}_{q + 1}} = \mathrm{div}\,   {\mathring{R}_{q + 1}} . 
\end{equation}
Furthermore, due to \eqref{Nashsteps-A-P-7-sjdisdaf}, \eqref{Nashsteps-A-P-7-sjdi} and \eqref{NISTYU-SKO-A-PPA-1}, we obtain
\begin{equation*}
\mathrm{div}\,   {u_{q + 1}} = \mathrm{div}\left( {{u_q} + {w_{q + 1}}  + \mathfrak{W}_{q + 1}^{loc} } \right) = 0 .
\end{equation*}
The proof is complete. \qed
\subsection{Estimates of the Nash perturbation}\label{ENHP}
\begin{lemma}\label{WPD-NHS-AA}
For any $N \in \left\{ {0,1, \cdots ,10} \right\}$ and $t \in \mathop  \cap \limits_{j = 0}^{q} \left[ {0,{\mathfrak{t}_j}} \right] \cap \left[ {{t_k},{t_{k + 1}}} \right)$, there exists a parameter $\alpha \in \left( {\frac{\beta }{{b}},\beta } \right)$ such that the backward flow ${{\tilde \Phi }_k}\left( {t,x} \right)$ and the Lagrangian flow ${{\tilde X}_t}$ satisfy
\begin{equation}\label{WPD-NHS-MAIN-AA-1}
{\left\| {\nabla {{\tilde \Phi }_k}} \right\|_N} + {\left\| {{{\left( {\nabla {{\tilde \Phi }_k}} \right)}^{ - 1}}} \right\|_N}  \lesssim \lambda _q^N \lambda _{q + 1}^{ - \alpha }   ,
\end{equation}
\begin{equation}\label{WPD-NHS-MAIN-AA-1-JUIOJH-1-1}
{\left\| {D{{\tilde X}_k}} \right\|_N} \lesssim  \max \left\{ {1,\lambda _q^N \lambda _{q + 1}^{ - \alpha }} \right\}  ,
\end{equation}
\begin{equation}\label{WPD-NHS-MAIN-AA-2}
{\left\| {{{\bar D}_{t,\Gamma }}\nabla {{\tilde \Phi }_k}} \right\|_N} + {\left\| {{{\bar D}_{t,\Gamma }}{{\left( {\nabla {{\tilde \Phi }_k}} \right)}^{ - 1}}} \right\|_N} \lesssim \delta _q^{\frac{1}{2}}\lambda _q^{N + 1}   ,
\end{equation}
\begin{equation}\label{WPD-NHS-MAIN-AA-4}
{\left\| {{\partial _{tt}}{{\tilde \Phi }_k}} \right\|_N} \lesssim \delta _q^{\frac{1}{2}}\lambda _q^{N + 1}  ,
\end{equation}
and
\begin{equation}\label{WPD-NHS-MAIN-AA-sjdiJJ-4}
{\left\| {{\partial _{tt}}{{\left( {\nabla {{\tilde \Phi }_k}} \right)}^{ - 1}}} \right\|_N} \lesssim \delta _q^{\frac{1}{2}} \lambda _q^{N + 3}   ,
\end{equation}
where ${{\bar D}_{t,\Gamma }} \triangleq {\partial _t} + {{\bar v}_{q,\Gamma }} \cdot \nabla $ is the material derivative corresponding to ${{\bar v}_{q,\Gamma }}$.
\end{lemma}
\par
\noindent{\textbf{Proof}}. 
According to \eqref{OIUYYDGH-DJFU-1} and Proposition \ref{WPD-NP-main}, we obtain
\begin{equation}\label{WPD-NP-main-BTO-1}
{\left\| {w_{q + 1}^{\left( \Gamma \right)}} \right\|_N} \lesssim \mu _{q + 1}^{ - 1}{\delta _{q + 1}}\lambda _q^{N + 1}l_q^{ - \alpha }   .
\end{equation}
It follows from \eqref{lemma-GLTNNRS-B-M-1}, \eqref{WPD-NP-main-BTO-1} and \eqref{Phsudj-2} that
\begin{equation}\label{WPD-NHS-MAIN-AA-P-3-JISHFI-1}
{\left\| {{{\bar v}_{q,\Gamma }}} \right\|_0} \lesssim {\left\| {{{\bar v}_q}} \right\|_0} + {\left\| {w_{q + 1}^{\left( \Gamma \right)}} \right\|_0} \leqslant C  .
\end{equation}
and
\begin{equation}\label{WPD-NHS-MAIN-AA-P-3}
{\left\| {{{\bar v}_{q,\Gamma }}} \right\|_N} \lesssim {\left\| {{{\bar v}_q}} \right\|_N} + {\left\| {w_{q + 1}^{\left( \Gamma \right)}} \right\|_N} \lesssim \delta _q^{\frac{1}{2}}\lambda _q^N , \ \ N \in \left\{ {1, \cdots ,10} \right\}  .
\end{equation}
Hence, using the similar argument in \eqref{Newtonsteps-A-B-11} and \eqref{Newtonsteps-A-B-7}-\eqref{Newtonsteps-A-B-14}, we verify that \eqref{WPD-NHS-MAIN-AA-1}, \eqref{WPD-NHS-MAIN-AA-1-JUIOJH-1-1} and \eqref{WPD-NHS-MAIN-AA-2} hold.
\par
We infer from \eqref{WPD-NHS-MAIN-AA-1}, \eqref{WPD-NHS-MAIN-AA-P-3} and Lemma \ref{STN-HCS-A} that
\begin{equation}\label{WPD-NHS-MAIN-XINSH-1}
{\left\| {{\partial _t}{{\tilde \Phi }_k}} \right\|_N} \lesssim {\left\| {{{\bar v}_{q,\Gamma }}} \right\|_N}{\left\| {\nabla {{\tilde \Phi }_k}} \right\|_0} + {\left\| {{{\bar v}_{q,\Gamma }}} \right\|_0}{\left\| {\nabla {{\tilde \Phi }_k}} \right\|_N} \lesssim \delta _q^{\frac{1}{2}}\lambda _q^N   .
\end{equation}
Since
\begin{equation}\label{RSRI-main-PROVE-15FAKFOG}
{\left\| {{\partial _t}{{\bar v}_q}} \right\|_N} \lesssim  {\left\| {\mathrm{div}\left( {{u_q} \otimes {u_q}} \right)*{\varsigma _{{l_q}}}} \right\|_N} + {\left\| {\nabla {\mathfrak{p}_q}*{\varsigma _{{l_q}}}} \right\|_N} + {\left\| {{\mathrm{div}}\, {\mathring{R}_q}*{\varsigma _{{l_q}}}} \right\|_N} \lesssim \delta _q^{\frac{1}{2}}\lambda _q^{N + 1}   ,
\end{equation}
we have
\begin{equation}\label{RSRI-main-PROVE-15}
{\left\| {{\partial _t}{{\bar v}_{q,\Gamma }}} \right\|_N} \lesssim {\left\| {{\partial _t}{{\bar v}_q}} \right\|_N} + {\left\| {{\partial _t}w_{q + 1}^{\left( \Gamma \right)}} \right\|_N} \lesssim \delta _q^{\frac{1}{2}}\lambda _q^{N + 1}    .
\end{equation}
Employing \eqref{WPD-NHS-MAIN-AA-1} and \eqref{WPD-NHS-MAIN-AA-P-3}-\eqref{RSRI-main-PROVE-15}, we find that
\begin{align*}
{\left\| {{\partial _{tt}}{{\tilde \Phi }_k}} \right\|_N} \lesssim & {\left\| {{\partial _t}{{\bar v}_{q,\Gamma }}} \right\|_N}{\left\| {\nabla {{\tilde \Phi }_k}} \right\|_0} + {\left\| {{\partial _t}{{\bar v}_{q,\Gamma }}} \right\|_0}{\left\| {\nabla {{\tilde \Phi }_k}} \right\|_N}   + {\left\| {{{\bar v}_{q,\Gamma }}} \right\|_N}{\left\| {{\partial _t}{{\tilde \Phi }_k}} \right\|_1} + {\left\| {{{\bar v}_{q,\Gamma }}} \right\|_0}{\left\| {{\partial _t}{{\tilde \Phi }_k}} \right\|_{N + 1}} \\
\lesssim & \delta _q^{\frac{1}{2}}\lambda _q^{N + 1}   .
\end{align*}
\par
It follows from \eqref{WPD-NHS-MAIN-AA-1}, \eqref{WPD-NHS-MAIN-AA-1-JUIOJH-1-1} and \eqref{WPD-NHS-MAIN-AA-P-3} that
\begin{align}\label{RSRI-main-PR-sjdhfOVE-15}
{\left\| {{\partial _t}{\tilde X_k}\left( {{\tilde \Phi _k}} \right)} \right\|_{N + 1}} \lesssim & {\left\| {{{\bar v}_{q,\Gamma }}\left( {{{\tilde X}_k}\left( {{{\tilde \Phi }_k}} \right)} \right)} \right\|_{N + 1}} \nonumber \\
\lesssim & {\left\| {{{\bar v}_{q,\Gamma }}} \right\|_1}\left( {{{\left\| {D{{\tilde X}_k}} \right\|}_0}{{\left\| {D{{\tilde \Phi }_k}} \right\|}_N} + {{\left\| {D{{\tilde X}_k}} \right\|}_N}\left\| {D{{\tilde \Phi }_k}} \right\|_0^{N + 1}} \right) \nonumber \\
& + {\left\| {{{\bar v}_{q,\Gamma }}} \right\|_{N + 1}}\left\| {D{{\tilde X}_k}} \right\|_0^{N + 1}\left\| {D{{\tilde \Phi }_k}} \right\|_0^{N + 1} \nonumber  \\ 
\lesssim & \delta _q^{\frac{1}{2}}\lambda _q^{N + 1}  \lambda _{q + 1}^{ - \alpha } .
\end{align}
Owing to ${\left( {\nabla {\tilde \Phi _k}} \right)^{ - 1}}\left( {t,x} \right) = D{\tilde X_k}\left( {{\tilde \Phi _k}\left( {t,x} \right),t} \right)$, we derive from \eqref{RSRI-main-PROVE-15} and \eqref{RSRI-main-PR-sjdhfOVE-15} that
\begin{align*}
{\left\| {{\partial _{tt}}{{\left( {\nabla {{\tilde \Phi }_k}} \right)}^{ - 1}}} \right\|_N} \lesssim & {\left\| {{\partial _t}{{\bar v}_{q,\Gamma }}\left( {{{\tilde X}_k}} \right)} \right\|_{N + 1}} \\
\lesssim & {\left\| {{\partial _t}D{{\bar v}_{q,\Gamma }}\left(  \cdot  \right){\partial _t}{{\tilde X}_k}} \right\|_{N + 1}} \\
\lesssim & {\left\| {{\partial _t}{{\bar v}_{q,\Gamma }}} \right\|_{N + 2}}{\left\| {D{{\tilde X}_k}} \right\|_0} {\left\| {{{\bar v}_{q,\Gamma }} } \right\|_0} + {\left\| {{\partial _t}{{\bar v}_{q,\Gamma }}} \right\|_1}{\left\| {{\partial _t}{{\tilde X}_k}} \right\|_{N + 1}} \\
\lesssim &  \delta _q^{\frac{1}{2}}  \lambda _q^{N + 3} .
\end{align*}
The proof is complete. \qed
\par
\begin{lemma}\label{WPD-NHS-BB}
For any $N \in \left\{ {0,1, \cdots ,10} \right\} $, $n \in {\mathbb{Z}^ + }$ and $t \in \mathop  \cap \limits_{j = 0}^{q} \left[ {0,{\mathfrak{t}_j}} \right] \cap \left[ {{t_k},{t_{k + 1}}} \right)$, the amplitude function ${{\tilde \eta }_{\xi ,k,n}} \left( {t,x} \right)$ satisfies
\begin{equation}\label{WPD-NHS-MAIN-1}
{\left\| {{{\tilde \eta }_{\xi ,k,n}}} \right\|_N} \lesssim  \delta _{q + 1}^{\frac{1}{2}}\lambda _q^N     ,
\end{equation}
\begin{equation}\label{WPD-NHS-MAIN-2}
{\left\| {{{\bar D}_{t,\Gamma }}{{\tilde \eta }_{\xi ,k,n}}} \right\|_N}  \lesssim \delta _{q + 1}^{\frac{1}{2}}\tau _q^{ - 1}\lambda _q^N    ,
\end{equation}
and there exists a parameter $\alpha \in \left( {\frac{\beta }{{b}},\beta } \right)$ such that
\begin{equation}\label{WPD-NHS-MAIN-3}
{\left\| {{\partial _{tt}}{{\tilde \eta }_{\xi ,k,n}}} \right\|_N}  \lesssim \delta _q^{\frac{1}{2}}\delta _{q + 1}^{\frac{1}{2}}\lambda _q^{N + \frac{1}{3} - \alpha }   .
\end{equation}
\end{lemma}
\par
\noindent{\textbf{Proof}}.
\textit{Case $1$}. If $N=0$, it follows from \eqref{HIDJJF-3-A-2}, \eqref{HIDJJF-3-A-3}, \eqref{WPD-NHS-MAIN-AA-1-JUIOJH-1-1} and \eqref{WPD-NP-main-BTO-1} that
\begin{equation}\label{JMNH-DKO-A-1-A-2}
{\left\| {{{\tilde R}_{q,n}}} \right\|_0} \lesssim {\left\| {{\mathring{R}_{q,n}}} \right\|_0}{\left\| {D{{\tilde X}_k}} \right\|_0} \lesssim {\delta _{q + 1,n}}\lambda _q^{ - \alpha }  ,
\end{equation}
and
\begin{align*}
{\left\| {{{\bar D}_{t,\Gamma }}{{\tilde R}_{q,n}}} \right\|_0} \lesssim & {\left\| {\left( {{{\bar D}_{t,\Gamma }}{\mathring{R}_{q,n}}} \right) \circ {{\tilde X}_k}} \right\|_0} \\
\lesssim & {\left\| {{{\bar D}_{t,q}}{\mathring{R}_{q,n}}} \right\|_0}{\left\| {D{{\tilde X}_k}} \right\|_0} \\
\lesssim & {{\left\| {{{\bar D}_{t,q}}{\mathring{R}_{q,n}}} \right\|}_0} + {{\left\| {w_{q + 1}^{\left( \Gamma \right)}} \right\|}_0}{{\left\| {{\mathring{R}_{q,n}}} \right\|}_1} \\
\lesssim & {\delta _{q + 1,n}}\tau _q^{ - 1}\lambda _q^{ - \alpha }   .
\end{align*}
Consequently,
\begin{equation}\label{WPD-NHS-MAIN-PPP-1}
{\left\| {{{\tilde \eta }_{\xi ,k,n}}} \right\|_0} \lesssim \delta _{q + 1}^{\frac{1}{2}} + \delta _{q + 1}^{ - \frac{1}{2}}{\left\| {{{\tilde R}_{q,n}}} \right\|_0} \lesssim \delta _{q + 1}^{\frac{1}{2}}    ,
\end{equation}
and
\begin{align}\label{WPD-NHS-MAIN-PPP-2}
{\left\| {{{\bar D}_{t,\Gamma }}{{\tilde \eta }_{\xi ,k,n}}} \right\|_0} \lesssim & \delta _{q + 1}^{\frac{1}{2}}\left( {{{\left\| {{{\bar D}_{t,\Gamma }}\nabla {{\tilde \Phi }_k}} \right\|}_0} + {{\left\| {{{\bar D}_{t,\Gamma }}\nabla \tilde \Phi _k^ \bot } \right\|}_0}} \right) \nonumber \\
& + \delta _{q + 1}^{ - \frac{1}{2}}{\left\| {{{\bar D}_{t,\Gamma }}\nabla {{\tilde \Phi }_k}} \right\|_0}{\left\| {{{\tilde R}_{q,n}}} \right\|_0} + \delta _{q + 1}^{ - \frac{1}{2}}{\left\| {{{\bar D}_{t,\Gamma }}{{\tilde R}_{q,n}}} \right\|_0} \nonumber \\
\lesssim & \delta _{q + 1}^{\frac{1}{2}}\tau _q^{ - 1}   .
\end{align}
\par
\textit{Case $2$}. If $N \geqslant 1$, by virtue of \eqref{HIDJJF-3-A-2}, \eqref{HIDJJF-3-A-3}, \eqref{WPD-NHS-MAIN-AA-1} and \eqref{WPD-NP-main-BTO-1}, we arrive at
\begin{align}\label{WPD-NHS-MAIN-PPP-XINH-3}
{\left\| {{{\tilde R}_{q,n}}} \right\|_N} \lesssim & {\left\| {{\mathring{R}_{q,n}} \circ {{\tilde X}_k}} \right\|_N} \nonumber \\
\lesssim & {\left\| {{\mathring{R}_{q,n}}} \right\|_1}{\left\| {D{{\tilde X}_k}} \right\|_{N - 1}} + {\left\| {{\mathring{R}_{q,n}}} \right\|_N}\left\| {D{{\tilde X}_k}} \right\|_0^N \nonumber \\
\lesssim & {\delta _{q + 1,n}}\lambda _q^{N - \alpha }    ,
\end{align}
and
\begin{align}\label{WPD-NHS-MAIN-PPP-XINH-4}
{\left\| {{{\bar D}_{t,\Gamma }}{{\tilde R}_{q,n}}} \right\|_N} \lesssim & {\left\| {{{\bar D}_{t,q}}{\mathring{R}_{q,n}}} \right\|_1}{\left\| {D{{\tilde X}_k}} \right\|_{N - 1}} + {\left\| {{{\bar D}_{t,q}}{\mathring{R}_{q,n}}} \right\|_N}\left\| {D{{\tilde X}_k}} \right\|_0^N  \nonumber  \\
\lesssim & {\delta _{q + 1,n}}\tau _q^{ - 1}\lambda _q^{N - \alpha }    .
\end{align}
We note that
\begin{align*}
{\left\| {{{\tilde \eta }_{\xi ,k,n}}} \right\|_N} \lesssim & \delta _{q + 1}^{\frac{1}{2}}\left( {{{\left\| {\nabla {{\tilde \Phi }_k}\nabla \tilde \Phi _k^ \bot } \right\|}_N} + \delta _{q + 1}^{ - 1}{{\left\| {\nabla {{\tilde \Phi }_k}{{\tilde R}_{q,n}}\nabla \tilde \Phi _k^ \bot } \right\|}_N}} \right)  \\
& + \delta _{q + 1}^{\frac{1}{2}}\left( {\left\| {\nabla {{\tilde \Phi }_k}\nabla \tilde \Phi _k^ \bot } \right\|_1^N + \delta _{q + 1}^{ - 1}\left\| {\nabla {{\tilde \Phi }_k}{{\tilde R}_{q,n}}\nabla \tilde \Phi _k^ \bot } \right\|_1^N} \right)   \\
\lesssim & \delta _{q + 1}^{\frac{1}{2}}{\left\| {\nabla {{\tilde \Phi }_k}} \right\|_N} + \delta _{q + 1}^{ - \frac{1}{2}}{\left\| {\nabla {{\tilde \Phi }_k}} \right\|_N}{\left\| {{{\tilde R}_{q,n}}} \right\|_0} + \delta _{q + 1}^{ - \frac{1}{2}}{\left\| {{{\tilde R}_{q,n}}} \right\|_N}   \\
& + \delta _{q + 1}^{\frac{1}{2}}\left\| {\nabla {{\tilde \Phi }_k}} \right\|_1^N + \delta _{q + 1}^{ - \frac{1}{2}}\left\| {\nabla {{\tilde \Phi }_k}} \right\|_1^N\left\| {{{\tilde R}_{q,n}}} \right\|_0^N + \delta _{q + 1}^{ - \frac{1}{2}}\left\| {{{\tilde R}_{q,n}}} \right\|_1^N  ,
\end{align*}
and
\begin{align*}
& {\left\| {{{\bar D}_{t,\Gamma }}{{\tilde \eta }_{\xi ,k,n}}} \right\|_N} \\
\lesssim & \delta _{q + 1}^{\frac{1}{2}}\left( {{{\left\| {{{\bar D}_{t,\Gamma }}\left( {\nabla {{\tilde \Phi }_k}\nabla \tilde \Phi _k^ \bot } \right)} \right\|}_N} + \delta _{q + 1}^{ - 1}{{\left\| {{{\bar D}_{t,\Gamma }}\left( {\nabla {{\tilde \Phi }_k}{{\tilde R}_{q,n}}\nabla \tilde \Phi _k^ \bot } \right)} \right\|}_N}} \right) \\
& + \delta _{q + 1}^{\frac{1}{2}}\left( {\left\| {{{\bar D}_{t,\Gamma }}\left( {\nabla {{\tilde \Phi }_k}\nabla \tilde \Phi _k^ \bot } \right)} \right\|_1^N + \delta _{q + 1}^{ - 1}\left\| {{{\bar D}_{t,\Gamma }}\left( {\nabla {{\tilde \Phi }_k}{{\tilde R}_{q,n}}\nabla \tilde \Phi _k^ \bot } \right)} \right\|_1^N} \right) \\
\lesssim & \delta _{q + 1}^{\frac{1}{2}}\left( {{{\left\| {{{\bar D}_{t,\Gamma }}\nabla {{\tilde \Phi }_k}} \right\|}_N} + {{\left\| {{{\bar D}_{t,\Gamma }}\nabla {{\tilde \Phi }_k}} \right\|}_0}{{\left\| {\nabla \tilde \Phi _k^ \bot } \right\|}_N}} \right) \\
& + \delta _{q + 1}^{ - \frac{1}{2}}\left( {{{\left\| {{{\bar D}_{t,\Gamma }}\nabla {{\tilde \Phi }_k}} \right\|}_N}{{\left\| {{{\tilde R}_{q,n}}} \right\|}_0} + {{\left\| {{{\bar D}_{t,\Gamma }}\nabla {{\tilde \Phi }_k}} \right\|}_0}{{\left\| {{{\tilde R}_{q,n}}} \right\|}_N} + {{\left\| {{{\bar D}_{t,\Gamma }}\nabla {{\tilde \Phi }_k}} \right\|}_0}{{\left\| {{{\tilde R}_{q,n}}} \right\|}_0}{{\left\| {\nabla \tilde \Phi _k^ \bot } \right\|}_N}} \right) \\
& + \delta _{q + 1}^{ - \frac{1}{2}}\left( {{{\left\| {\nabla {{\tilde \Phi }_k} + \nabla \tilde \Phi _k^ \bot } \right\|}_N}{{\left\| {{{\bar D}_{t,\Gamma }}{{\tilde R}_{q,n}}} \right\|}_0} + {{\left\| {{{\bar D}_{t,\Gamma }}{{\tilde R}_{q,n}}} \right\|}_N}} \right) \\
& + \delta _{q + 1}^{\frac{1}{2}}{\left( {{{\left\| {{{\bar D}_{t,\Gamma }}\nabla {{\tilde \Phi }_k}} \right\|}_1} + {{\left\| {{{\bar D}_{t,\Gamma }}\nabla {{\tilde \Phi }_k}} \right\|}_0}{{\left\| {\nabla \tilde \Phi _k^ \bot } \right\|}_1}} \right)^N} \\
& + \delta _{q + 1}^{ - \frac{1}{2}}{\left( {{{\left\| {{{\bar D}_{t,\Gamma }}\nabla {{\tilde \Phi }_k}} \right\|}_1}{{\left\| {{{\tilde R}_{q,n}}} \right\|}_0} + {{\left\| {{{\bar D}_{t,\Gamma }}\nabla {{\tilde \Phi }_k}} \right\|}_0}{{\left\| {{{\tilde R}_{q,n}}} \right\|}_1} + {{\left\| {{{\bar D}_{t,\Gamma }}\nabla {{\tilde \Phi }_k}} \right\|}_0}{{\left\| {{{\tilde R}_{q,n}}} \right\|}_0}{{\left\| {\nabla \tilde \Phi _k^ \bot } \right\|}_1}} \right)^N} \\
& + \delta _{q + 1}^{ - \frac{1}{2}}{\left( {{{\left\| {\nabla {{\tilde \Phi }_k} + \nabla \tilde \Phi _k^ \bot } \right\|}_1}{{\left\| {{{\bar D}_{t,\Gamma }}{{\tilde R}_{q,n}}} \right\|}_0} + {{\left\| {{{\bar D}_{t,\Gamma }}{{\tilde R}_{q,n}}} \right\|}_1}} \right)^N} .
\end{align*}
Then, invoking Lemma \ref{WPD-NHS-AA}, \eqref{WPD-NHS-MAIN-PPP-XINH-3} and \eqref{WPD-NHS-MAIN-PPP-XINH-4}, we obtain
\begin{equation*}
{\left\| {{{\tilde \eta }_{\xi ,k,n}}} \right\|_N} \lesssim \delta _{q + 1}^{\frac{1}{2}}\lambda _q^N   ,
\end{equation*}
and
\begin{align*}
{\left\| {{{\bar D}_{t,\Gamma }}{{\tilde \eta }_{\xi ,k,n}}} \right\|_N} \lesssim & \delta _{q + 1}^{\frac{1}{2}}\delta _q^{\frac{1}{2}}\lambda _q^{N + 1} + \delta _{q + 1}^{\frac{1}{2}}\tau _q^{ - 1}\lambda _q^{N - \alpha } + \delta _{q + 1}^{\frac{1}{2}}\delta _q^{\frac{1}{2}N}\lambda _q^{2N} + \delta _{q + 1}^{N - \frac{1}{2}}\delta _q^{\frac{1}{2}}\lambda _q^{2N}\lambda _q^{ - N\alpha } \\
\lesssim & \delta _{q + 1}^{\frac{1}{2}}\tau _q^{ - 1}\lambda _q^N   .
\end{align*}
\par
According to the definition \eqref{Nashsteps-A-AIJS-P-2}, we have (also see \cite{Zbl1556.35231})
\begin{equation*}
{\partial _{tt}}{{\tilde R}_{q,n}} =  - l_{t,q}^{ - 1}\int_{ - {{\tilde l}_{q}}}^{{{\tilde l}_{q}}} {{\partial _t}{\mathring{R}_{q,n}}\left( {t + s,{{\tilde X}_k}\left( {t + s,x} \right)} \right){\partial _s}{{\tilde \varsigma }_{{{\tilde l}_{q}}}}\left( s \right)}\ \mathrm{d}s ,
\end{equation*}
which implies that
\begin{align}\label{WPD-NHS-MAIN-PPP-XINH-5}
{\left\| {{\partial _{tt}}{{\tilde R}_{q,n}}} \right\|_N} \lesssim & l_{t,q}^{ - 1}\left( {{{\left\| {{{\bar D}_{q,\Gamma }}{\mathring{R}_{q,n}}} \right\|}_1}{{\left\| {D{{\tilde X}_k}} \right\|}_{N - 1}} + {{\left\| {{{\bar D}_{q,\Gamma }}{\mathring{R}_{q,n}}} \right\|}_N}\left\| {D{{\tilde X}_k}} \right\|_0^N} \right) \nonumber \\
\lesssim & l_{t,q}^{ - 1}\tau _q^{ - 1}{\delta _{q + 1,n}}\lambda _q^{N - \alpha }    .
\end{align}
Therefore, it follows from \eqref{WPD-NHS-MAIN-AA-4} and \eqref{WPD-NHS-MAIN-PPP-XINH-5} that
\begin{align*}
{\left\| {{\partial _{tt}}{{\tilde \eta }_{\xi ,k,n}}} \right\|_N} \lesssim & \delta _{q + 1}^{\frac{1}{2}}{\left\| {\nabla {{\tilde \Phi }_k}\nabla \tilde \Phi _k^ \bot  - \delta _{q + 1,n}^{ - 1}\nabla {{\tilde \Phi }_k}{{\tilde R}_{q,n}}\nabla \tilde \Phi _k^ \bot } \right\|_N} \\
& + \delta _{q + 1}^{\frac{1}{2}}{\left\| {{\partial _t}\left( {\nabla {{\tilde \Phi }_k}\nabla \tilde \Phi _k^ \bot  - \delta _{q + 1}^{ - 1}\nabla {{\tilde \Phi }_k}{{\tilde R}_{q,n}}\nabla \tilde \Phi _k^ \bot } \right)} \right\|_N} \\
& + \delta _{q + 1}^{\frac{1}{2}}{\left\| {{\partial _{tt}}\left( {\nabla {{\tilde \Phi }_k}\nabla \tilde \Phi _k^ \bot  - \delta _{q + 1}^{ - 1}\nabla {{\tilde \Phi }_k}{{\tilde R}_{q,n}}\nabla \tilde \Phi _k^ \bot } \right)} \right\|_N} \\
\lesssim & \delta _{q + 1}^{\frac{1}{2}}\lambda _q^N + \tau _q^{ - 1}\delta _q^{\frac{1}{2}}\delta _{q + 1}^{\frac{1}{2}}\lambda _q^{N - \alpha } + \delta _q^{\frac{1}{2}}\delta _{q + 1}^{\frac{1}{2}}\lambda _q^{N + 1} + l_{t,q}^{ - 1}\tau _q^{ - 1}\delta _{q + 1}^{\frac{3}{2}}\lambda _q^{N - \alpha } \\
\lesssim & \delta _q^{\frac{1}{2}}\delta _{q + 1}^{\frac{1}{2}}\lambda _q^{N + \frac{1}{3} - \alpha }    .
\end{align*}
The proof is complete. \qed
\par
\begin{lemma}\label{WPD-NHS-HUD-BB}
For any $N \in \left\{ {0,1, \cdots ,10} \right\} $ and $t \in \mathop  \cap \limits_{j = 0}^{q} \left[ {0,{\mathfrak{t}_j}} \right] \cap \left[ {{t_k},{t_{k + 1}}} \right)$, the stream function ${\mathfrak{S}_\xi }\left( \cdot \right)$ defined by \eqref{IUHYU-DJ} admits
\begin{equation}\label{WPD-NHS-HUD-BB-MAIN-5}
{\left\| {{\mathfrak{S}_\xi }\left( {{\lambda _{q + 1}}{{\tilde \Phi }_k}} \right)} \right\|_N} \lesssim  \max \left\{ {1, \delta _q^{\frac{1}{2}} \lambda _{q + 1}^N} \right\}     ,
\end{equation}
\begin{equation}\label{WPD-NHS-HUD-BB-MAIN-6}
{\left\| {{{\bar D}_{t,\Gamma }}{\mathfrak{S}_\xi }\left( {{\lambda _{q + 1}}{{\tilde \Phi }_k}} \right)} \right\|_N} \lesssim  \delta _q^{\frac{1}{2}}\lambda _{q + 1}^{N + 1}  ,
\end{equation}
and
\begin{equation}\label{WPD-NHS-HUD-BB-MAIN-XINhs-6}
{\left\| {{\partial _{tt}}{\mathfrak{S}_\xi } \left( {{\lambda _{q + 1}}{{\tilde \Phi }_k}} \right)} \right\|_N} \lesssim  \delta _q^{\frac{1}{2}}\lambda _q^{N + 1}\lambda _{q + 1}^2   .
\end{equation}
\end{lemma}
\par
\noindent{\textbf{Proof}}. When $ N = 0$, it is straightforward to verify that \eqref{WPD-NHS-HUD-BB-MAIN-5}, \eqref{WPD-NHS-HUD-BB-MAIN-6} and \eqref{WPD-NHS-HUD-BB-MAIN-XINhs-6} hold. We next consider the case $N \geqslant 1$. Employing Lemma \ref{ONCS}, we obtain
\begin{equation*}
{\left\| {{\mathfrak{S}_\xi }\left( {{\lambda _{q + 1}}{{\tilde \Phi }_k}} \right)} \right\|_N} \lesssim  {\left\| {{\mathfrak{S}_\xi }\left( {{\lambda _{q + 1}} \cdot } \right)} \right\|_1}{\left\| {\nabla {{\tilde \Phi }_k}} \right\|_{N - 1}} + {\left\| {{\mathfrak{S}_\xi }\left( {{\lambda _{q + 1}} \cdot } \right)} \right\|_N} \left\| {\nabla {{\tilde \Phi }_k}} \right\|_0 ^ N \lesssim  \delta _q^{\frac{1}{2}}\lambda _{q + 1}^{N } .
\end{equation*}
Owing to ${\left\| { {\partial _t}{\mathfrak{S}_\xi }  \left( {{\lambda _{q + 1}}{{\tilde \Phi }_k}} \right)} \right\|_N} \lesssim {\lambda _{q + 1}}{\left\| {{\partial _t}{{\tilde \Phi }_k}} \right\|_{N}}$, it follows from \eqref{WPD-NHS-MAIN-AA-1} and \eqref{WPD-NHS-MAIN-AA-P-3} that
\begin{align*}
{\left\| {{{\bar D}_{t,\Gamma }}{\mathfrak{S}_\xi }\left( {{\lambda _{q + 1}}{{\tilde \Phi }_k}} \right)} \right\|_N} \lesssim & {\left\| { {\partial _t}{\mathfrak{S}_\xi }  \left( {{\lambda _{q + 1}}{{\tilde \Phi }_k}} \right)} \right\|_N} + {\left\| {\left( {{{\bar v}_{q,\Gamma }} \cdot \nabla } \right){\mathfrak{S}_\xi }\left( {{\lambda _{q + 1}}{{\tilde \Phi }_k}} \right)} \right\|_N} \\
\lesssim & {\left\| { {\partial _t}{\mathfrak{S}_\xi }  \left( {{\lambda _{q + 1}}{{\tilde \Phi }_k}} \right)} \right\|_N} + {\left\| {{{\bar v}_{q,\Gamma }}} \right\|_N}{\left\| {{\mathfrak{S}_\xi }\left( {{\lambda _{q + 1}}{{\tilde \Phi }_k}} \right)} \right\|_1} + {\left\| {{{\bar v}_{q,\Gamma }}} \right\|_0}{\left\| {{\mathfrak{S}_\xi }\left( {{\lambda _{q + 1}}{{\tilde \Phi }_k}} \right)} \right\|_{N + 1}} \\
\lesssim & \delta _q^{\frac{1}{2}}\lambda _{q + 1}^{N +1}   .
\end{align*}
Furthermore, from \eqref{WPD-NHS-MAIN-AA-4} and \eqref{WPD-NHS-MAIN-XINSH-1}, we deduce that
\begin{equation*}
{\left\| { {\partial _{tt}}{\mathfrak{S}_\xi }  \left( {{\lambda _{q + 1}}{{\tilde \Phi }_k}} \right)} \right\|_N} \lesssim \lambda _{q + 1}^2{\left\| {{\partial _{tt}}{{\tilde \Phi }_k}} \right\|_{N }}  \lesssim \delta _q^{\frac{1}{2}}\lambda _q^{N + 1}\lambda _{q + 1}^2   .
\end{equation*}
The proof is complete. \qed
\par
\begin{lemma}\label{WPD-NHS-HUD-BB-2}
Choose the intensity parameter $ \sigma  = C\left( \Lambda  \right)\lambda _{q + 1}^{ - \frac{1}{2}}$ in the intermittent Dirichlet kernel ${\mathfrak{D}_{\tilde \mu }}\left( {t, \cdot } \right)$ specified by \eqref{IBW-GJI-A-1}. Then, for any $N \in \left\{ {0,1, \cdots ,10} \right\} $ and $t \in \mathop  \cap \limits_{j = 0}^{q} \left[ {0,{\mathfrak{t}_j}} \right] \cap \left[ {{t_k},{t_{k + 1}}} \right)$, we have
\begin{equation}\label{WPD-NHS-HUD-BB-MAIN-1}
{\left\| {{\mathfrak{D}_{\tilde \mu }}\left( {{\lambda _{q + 1}}{{\tilde \Phi }_k}} \right)} \right\|_N} \lesssim  \max \left\{ {1, \delta _q^{\frac{1}{2}} \lambda _{q + 1}^{\frac{1}{2}N}\left( {\lambda _q^{\frac{1}{2}\left( {N - 1} \right)} + 1} \right)} \right\} ,
\end{equation}
\begin{equation}\label{WPD-NHS-HUD-BB-MAIN-2}
{\left\| {{{\bar D}_{t,\Gamma }}{\mathfrak{D}_{\tilde \mu }}\left( {{\lambda _{q + 1}}{{\tilde \Phi }_k}} \right)} \right\|_N} \lesssim   \delta _q^{\frac{1}{2}}\lambda _q^{\frac{1}{2}N}\lambda _{q + 1}^{\frac{1}{2}N + \frac{1}{2}}    ,
\end{equation}
and
\begin{equation}\label{WPD-NHS-HUD-BB-MAIN-XINhs-4565}
{\left\| {{\partial _{tt}}{\mathfrak{D}_{\tilde \mu }}\left( {t,{\lambda _{q + 1}}{{\tilde \Phi }_k}} \right)} \right\|_N}  \lesssim   {{\tilde \mu }^2}\lambda _{q + 1}^N   .
\end{equation}
\end{lemma}
\par
\noindent{\textbf{Proof}}. When $ N = 0$, It is straightforward to verify that \eqref{WPD-NHS-HUD-BB-MAIN-1}, \eqref{WPD-NHS-HUD-BB-MAIN-2} and \eqref{WPD-NHS-HUD-BB-MAIN-XINhs-4565} hold. We now consider the case $N \geqslant 1$. It follows from \eqref{WPD-NHS-MAIN-AA-1} that
\begin{align*}
{\left\| {{\mathfrak{D}_{\tilde \mu }}\left( { \sigma {\lambda _{q + 1}}{{\tilde \Phi }_k}} \right)} \right\|_N} \lesssim & {\left\| {{\mathfrak{D}_{\tilde \mu }}\left( { \sigma {\lambda _{q + 1}} \cdot } \right)} \right\|_1}{\left\| {D{{\tilde \Phi }_k}} \right\|_{N - 1}} + {\left\| {{\mathfrak{D}_{\tilde \mu }}\left( { \sigma {\lambda _{q + 1}} \cdot } \right)} \right\|_N} {\left\| {D{{\tilde \Phi }_k}} \right\|_{0}^N} \\
\lesssim  & \delta _q^{\frac{1}{2}} \lambda _{q + 1}^{\frac{1}{2}N}\left( {\lambda _q^{\frac{1}{2}\left( {N - 1} \right)} + 1} \right)   .
\end{align*}
Since
\begin{align*}
{\partial _t}{\mathfrak{D}_{\tilde \mu }}\left( {t,{\lambda _{q + 1}}{{\tilde \Phi }_k}} \right) = & {\left( {2r + 1} \right)^{ - 1}}\mathrm{i}\tilde \mu  \sigma \sum\limits_{\tilde \xi  \in {\mathfrak{H}_r}} {{{\tilde \xi }_1}\exp \left( {\mathrm{i} \sigma \tilde \xi  \cdot \left( {\tilde \xi  \cdot {\lambda _{q + 1}}{{\tilde \Phi }_k} + \tilde \mu t,{{\tilde \xi }^ \bot } \cdot {\lambda _{q + 1}}{{\tilde \Phi }_k}} \right)} \right)}  \\
& + {\left( {2r + 1} \right)^{ - 1}}\mathrm{i} \sigma \sum\limits_{\tilde \xi  \in {\mathfrak{H}_r}} {\tilde \xi  \cdot \left( {\tilde \xi  \cdot {\lambda _{q + 1}}{\partial _t}{{\tilde \Phi }_k},{{\tilde \xi }^ \bot } \cdot {\lambda _{q + 1}}{\partial _t}{{\tilde \Phi }_k}} \right)}  \\
& \hspace*{8em} \exp \left( {\mathrm{i} \sigma \tilde \xi  \cdot \left( {\tilde \xi  \cdot {\lambda _{q + 1}}{{\tilde \Phi }_k} + \tilde \mu t,{{\tilde \xi }^ \bot } \cdot {\lambda _{q + 1}}{{\tilde \Phi }_k}} \right)} \right) ,
\end{align*}
we deduce that
\begin{align*}
{\left\| {{\partial _t}{\mathfrak{D}_{\tilde \mu }}\left( {t,{\lambda _{q + 1}}{{\tilde \Phi }_k}} \right)} \right\|_N} \lesssim  & \tilde \mu \left( {{\lambda _{q + 1}}{\sigma} ^2{{\left\| {D{{\tilde \Phi }_k}} \right\|}_{N - 1}} + \lambda _{q + 1}^N{\sigma} ^{N + 1}} \right) +  \sigma {\lambda _{q + 1}}{\left\| {{\partial _t}{{\tilde \Phi }_k}} \right\|_N} \\
& + {\left\| {{\partial _t}{{\tilde \Phi }_k}} \right\|_0}\left( {\lambda _{q + 1}^2{\sigma} ^2{{\left\| {D{{\tilde \Phi }_k}} \right\|}_{N - 1}} + \lambda _{q + 1}^{N + 1}{\sigma} ^{N + 1}} \right) .
\end{align*}
Then, it follows from \eqref{WPD-NHS-MAIN-AA-1} and \eqref{WPD-NHS-MAIN-XINSH-1} that
\begin{equation}\label{RSRI-main-PROVE-8}
{\left\| {{\partial _t}{\mathfrak{D}_{\tilde \mu }}\left( {t,{\lambda _{q + 1}}{{\tilde \Phi }_k}} \right)} \right\|_N} \lesssim   \delta _q^{\frac{1}{2}}\lambda _q^{\frac{1}{2}N}\lambda _{q + 1}^{\frac{1}{2}N + \frac{1}{2}}   .
\end{equation}
Invoking \eqref{WPD-NHS-MAIN-AA-P-3}, \eqref{WPD-NHS-HUD-BB-MAIN-1} and \eqref{RSRI-main-PROVE-8}, we obtain
\begin{align*}
{\left\| {{{\bar D}_{t,\Gamma }}{\mathfrak{D}_{\tilde \mu }}\left( {t,{\lambda _{q + 1}}{{\tilde \Phi }_k}} \right)} \right\|_N} \lesssim & {\left\| {{\partial _t}{\mathfrak{D}_{\tilde \mu }}\left( {t,{\lambda _{q + 1}}{{\tilde \Phi }_k}} \right)} \right\|_N} + {\left\| {\left( {{{\bar v}_{q,\Gamma }} \cdot \nabla } \right){\mathfrak{D}_{\tilde \mu }}\left( {{\lambda _{q + 1}}{{\tilde \Phi }_k}} \right)} \right\|_N} \\
\lesssim & {\left\| {{\partial _t}{\mathfrak{D}_{\tilde \mu }}\left( {t,{\lambda _{q + 1}}{{\tilde \Phi }_k}} \right)} \right\|_N} + {\left\| {{{\bar v}_{q,\Gamma }}} \right\|_N}{\left\| {{\mathfrak{D}_{\tilde \mu }}\left( {{\lambda _{q + 1}}{{\tilde \Phi }_k}} \right)} \right\|_1} \\
&  + {\left\| {{{\bar v}_{q,\Gamma }}} \right\|_0}{\left\| {{\mathfrak{D}_{\tilde \mu }}\left( {{\lambda _{q + 1}}{{\tilde \Phi }_k}} \right)} \right\|_{N + 1}} \\
\lesssim  &  \delta _q^{\frac{1}{2}}\lambda _q^{\frac{1}{2}N}\lambda _{q + 1}^{\frac{1}{2}N + \frac{1}{2}}   .
\end{align*}
\par
Owing to
\begin{align*}
{\partial _{tt}}{\mathfrak{D}_{\tilde \mu }}\left( {t,{\lambda _{q + 1}}{{\tilde \Phi }_k}} \right) = & - {\left( {2r + 1} \right)^{ - 1}}{{\tilde \mu }^2}{\sigma} ^2\sum\limits_{\tilde \xi  \in {\mathfrak{H}_r}} {\tilde \xi _1^2\exp \left( {\mathrm{i} \sigma \tilde \xi  \cdot \left( {\tilde \xi  \cdot {\lambda _{q + 1}}{{\tilde \Phi }_k} + \tilde \mu t,{{\tilde \xi }^ \bot } \cdot {\lambda _{q + 1}}{{\tilde \Phi }_k}} \right)} \right)}  \\
& - 2{\left( {2r + 1} \right)^{ - 1}}\tilde \mu {\sigma} ^2\sum\limits_{\tilde \xi  \in {\mathfrak{H}_r}} {{{\tilde \xi }_1}\tilde \xi  \cdot \left( {\tilde \xi  \cdot {\lambda _{q + 1}}{\partial _t}{{\tilde \Phi }_k},{{\tilde \xi }^ \bot } \cdot {\lambda _{q + 1}}{\partial _t}{{\tilde \Phi }_k}} \right)}  \\
& \hspace*{8em} \exp \left( {\mathrm{i} \sigma \tilde \xi  \cdot \left( {\tilde \xi  \cdot {\lambda _{q + 1}}{{\tilde \Phi }_k} + \tilde \mu t,{{\tilde \xi }^ \bot } \cdot {\lambda _{q + 1}}{{\tilde \Phi }_k}} \right)} \right) \\
&  - {\left( {2r + 1} \right)^{ - 1}}i{\sigma} ^2\sum\limits_{\tilde \xi  \in {\mathfrak{H}_r}} {{{\left( {\tilde \xi  \cdot \left( {\tilde \xi  \cdot {\lambda _{q + 1}}{\partial _t}{{\tilde \Phi }_k},{{\tilde \xi }^ \bot } \cdot {\lambda _{q + 1}}{\partial _t}{{\tilde \Phi }_k}} \right)} \right)}^2}}  \\
& \hspace*{8em} \exp \left( {\mathrm{i} \sigma \tilde \xi  \cdot \left( {\tilde \xi  \cdot {\lambda _{q + 1}}{{\tilde \Phi }_k} + \tilde \mu t,{{\tilde \xi }^ \bot } \cdot {\lambda _{q + 1}}{{\tilde \Phi }_k}} \right)} \right) \\
&  + {\left( {2r + 1} \right)^{ - 1}} \mathrm{i} \sigma \sum\limits_{\tilde \xi  \in {\mathfrak{H}_r}} {\tilde \xi  \cdot \left( {\tilde \xi  \cdot {\lambda _{q + 1}}{\partial _{tt}}{{\tilde \Phi }_k},{{\tilde \xi }^ \bot } \cdot {\lambda _{q + 1}}{\partial _{tt}}{{\tilde \Phi }_k}} \right)}  \\
& \hspace*{8em}  \exp \left( {\mathrm{i} \sigma \tilde \xi  \cdot \left( {\tilde \xi  \cdot {\lambda _{q + 1}}{{\tilde \Phi }_k} + \tilde \mu t,{{\tilde \xi }^ \bot } \cdot {\lambda _{q + 1}}{{\tilde \Phi }_k}} \right)} \right) ,
\end{align*}
we get
\begin{align*}
{\left\| {{\partial _{tt}}{\mathfrak{D}_{\tilde \mu }}\left( {t,{\lambda _{q + 1}}{{\tilde \Phi }_k}} \right)} \right\|_N} \lesssim  & {{\tilde \mu }^2}\left( {{\lambda _{q + 1}}{\sigma} ^3{{\left\| {D{{\tilde \Phi }_k}} \right\|}_{N - 1}} + \lambda _{q + 1}^N{\sigma} ^{N + 2}} \right) \\
& + \tilde \mu {\left\| {{\partial _t}{{\tilde \Phi }_k}} \right\|_0}\left( {{\lambda _{q + 1}}{\sigma} ^3{{\left\| {D{{\tilde \Phi }_k}} \right\|}_{N - 1}} + \lambda _{q + 1}^N{\sigma} ^{N + 2}} \right)   \\
& +  \sigma \left( { \sigma \left( {\tilde \mu {{\left\| {{\partial _t}{{\tilde \Phi }_k}} \right\|}_N} + {{\left\| {{{\left( {{\partial _t}{{\tilde \Phi }_k}} \right)}^2}} \right\|}_N}} \right) + {{\left\| {{\partial _{tt}}{{\tilde \Phi }_k}} \right\|}_N}} \right) \\
& + {\left\| {{{\left( {{\partial _t}{{\tilde \Phi }_k}} \right)}^2}} \right\|_0}\left( {{\lambda _{q + 1}}{\sigma} ^3{{\left\| {D{{\tilde \Phi }_k}} \right\|}_{N - 1}} + \lambda _{q + 1}^N{\sigma} ^{N + 2}} \right) \\
& + {\left\| {{\partial _{tt}}{{\tilde \Phi }_k}} \right\|_0}\left( {\lambda _{q + 1}^2{\sigma} ^2{{\left\| {D{{\tilde \Phi }_k}} \right\|}_{N - 1}} + \lambda _{q + 1}^{N + 1}{\sigma} ^{N + 1}} \right) .
\end{align*}
Then, we infer from Lemma \ref{WPD-NHS-AA} that
\begin{equation}\label{RSRI-main-PROVE-XIJDHF-1}
{\left\| {{\partial _{tt}}{\mathfrak{D}_{\tilde \mu }}\left( {t,{\lambda _{q + 1}}{{\tilde \Phi }_k}} \right)} \right\|_N} \lesssim   {{\tilde \mu }^2}\lambda _{q + 1}^N     .
\end{equation}
The proof is complete. \qed
\par
\begin{lemma}\label{WPD-NHS-HUD-BB-3}
For any $N \in \left\{ {0,1, \cdots ,10} \right\} $ and $t \in \mathop  \cap \limits_{j = 0}^{q} \left[ {0,{\mathfrak{t}_j}} \right] \cap \left[ {{t_k},{t_{k + 1}}} \right)$, the intermittent building block ${\mathbb{W}_\xi } \left( {t, \cdot } \right)$ defined by \eqref{IBW-GJI-A-2} satisfies
\begin{equation}\label{WPD-NHS-HUD-BB-MAIN-3}
{\left\| {{\mathbb{W}_\xi }\left( {{\lambda _{q + 1}}{{\tilde \Phi }_k}} \right)} \right\|_N} \lesssim \max \left\{ {1, \delta _q^{\frac{1}{2}} \lambda _{q + 1}^N} \right\}    ,
\end{equation}
and
\begin{equation}\label{WPD-NHS-HUD-BB-MAIN-4}
{\left\| {{{\bar D}_{t,\Gamma }}{\mathbb{W}_\xi }\left( {{\lambda _{q + 1}}{{\tilde \Phi }_k}} \right)} \right\|_N} \lesssim   \delta _q^{\frac{1}{2}}\lambda _{q + 1}^{N + 1}   .
\end{equation}
\end{lemma}
\par
\noindent{\textbf{Proof}}. In view of \eqref{IBW-GJI-A-1}, \eqref{IBW-GJI-A-2} and Lemma \ref{WPD-NHS-HUD-BB-2}, it is easy to verify that \eqref{WPD-NHS-HUD-BB-MAIN-3} and \eqref{WPD-NHS-HUD-BB-MAIN-4} hold as $ N = 0$. Subsequently, we consider the case $N \geqslant 1$. Notice that
\begin{equation*}
{\left\| {{W_\xi }\left( {{\lambda _{q + 1}}{{\tilde \Phi }_k}} \right)} \right\|_N} \lesssim {\left\| {{W_\xi }\left( {{\lambda _{q + 1}} \cdot } \right)} \right\|_1}{\left\| {D{{\tilde \Phi }_k}} \right\|_{N - 1}} + {\left\| {{W_\xi }\left( {{\lambda _{q + 1}} \cdot } \right)} \right\|_N} {\left\| {D{{\tilde \Phi }_k}} \right\|_{0}^N} \lesssim   \delta _q^{\frac{1}{2}}\lambda _{q + 1}^{N }  ,
\end{equation*}
where ${W_\xi }\left( \cdot \right)$ is the simplified building block defined by \eqref{UJIJY-DJI}. Then, it follows from \eqref{WPD-NHS-MAIN-XINSH-1}, \eqref{WPD-NHS-HUD-BB-MAIN-1} and \eqref{WPD-NHS-HUD-BB-MAIN-2} that
\begin{align*}
{\left\| {{\mathbb{W}_\xi }\left( {{\lambda _{q + 1}}{{\tilde \Phi }_k}} \right)} \right\|_N} \lesssim & {\left\| {{\mathfrak{D}_{\tilde \mu }}\left( {{\lambda _{q + 1}}{{\tilde \Phi }_k}} \right){W_\xi }\left( {{\lambda _{q + 1}}{{\tilde \Phi }_k}} \right)} \right\|_N} \\
\lesssim & {\left\| {{\mathfrak{D}_{\tilde \mu }}\left( {{\lambda _{q + 1}}{{\tilde \Phi }_k}} \right)} \right\|_N} + {\left\| {{W_\xi }\left( {{\lambda _{q + 1}}{{\tilde \Phi }_k}} \right)} \right\|_N} {\left\| {{\mathfrak{D}_{\tilde \mu }}\left( {{\lambda _{q + 1}}{{\tilde \Phi }_k}} \right)} \right\|_0} \\
\lesssim  & \delta _q^{\frac{1}{2}} \lambda _{q + 1}^N   ,
\end{align*}
and
\begin{align*}
{\left\| {{\partial _t}{W_\xi }\left( {{\lambda _{q + 1}}{{\tilde \Phi }_k}} \right)} \right\|_N} \lesssim & {\lambda _{q + 1}}\left( {{{\left\| {{W_\xi }\left( {{\lambda _{q + 1}}{{\tilde \Phi }_k}} \right)} \right\|}_0}{{\left\| {{\partial _t}{{\tilde \Phi }_k}} \right\|}_N} + {{\left\| {{W_\xi }\left( {{\lambda _{q + 1}}{{\tilde \Phi }_k}} \right)} \right\|}_N}{{\left\| {{\partial _t}{{\tilde \Phi }_k}} \right\|}_0}} \right) \\
\lesssim & \delta _q^{\frac{1}{2}}\lambda _{q + 1}^{N + 1}   .
\end{align*}
Therefore,
\begin{align*}
{\left\| {{{\bar D}_{t,\Gamma }}{\mathbb{W}_\xi }\left( {{\lambda _{q + 1}}{{\tilde \Phi }_k}} \right)} \right\|_N} \lesssim & {\left\| {{\partial _t}{\mathbb{W}_\xi }\left( {{\lambda _{q + 1}}{{\tilde \Phi }_k}} \right)} \right\|_N} + {\left\| {\left( {{{\bar v}_{q,\Gamma }} \cdot \nabla } \right){\mathbb{W}_\xi }\left( {{\lambda _{q + 1}}{{\tilde \Phi }_k}} \right)} \right\|_N} \\
\lesssim & {\left\| {{\partial _t}{\mathfrak{D}_{\tilde \mu }}\left( {t,{\lambda _{q + 1}}{{\tilde \Phi }_k}} \right)} \right\|_N} + {\left\| {{\partial _t}{\mathfrak{D}_{\tilde \mu }}\left( {t,{\lambda _{q + 1}}{{\tilde \Phi }_k}} \right)} \right\|_0}{\left\| {{W_\xi }\left( {{\lambda _{q + 1}}{{\tilde \Phi }_k}} \right)} \right\|_N} \\
& + {\left\| {{\mathfrak{D}_{\tilde \mu }}\left( {t,{\lambda _{q + 1}}{{\tilde \Phi }_k}} \right)} \right\|_N}{\left\| {{\partial _t}{W_\xi }\left( {{\lambda _{q + 1}}{{\tilde \Phi }_k}} \right)} \right\|_0} + {\left\| {{{\bar v}_{q,\Gamma }}} \right\|_N}{\left\| {{\mathbb{W}_\xi }\left( {{\lambda _{q + 1}}{{\tilde \Phi }_k}} \right)} \right\|_1} \\
& + {\left\| {{\mathfrak{D}_{\tilde \mu }}\left( {t,{\lambda _{q + 1}}{{\tilde \Phi }_k}} \right)} \right\|_0} {\left\| {{\partial _t}{W_\xi }\left( {{\lambda _{q + 1}}{{\tilde \Phi }_k}} \right)} \right\|_N} + {\left\| {{{\bar v}_{q,\Gamma }}} \right\|_0}{\left\| {{\mathbb{W}_\xi }\left( {{\lambda _{q + 1}}{{\tilde \Phi }_k}} \right)} \right\|_{N + 1}} \\
\lesssim  &  \delta _q^{\frac{1}{2}}\lambda _{q + 1}^{N + 1}   .
\end{align*}
The proof is complete. \qed
\par
\begin{proposition}\label{WPD-NS-SD-P-main}
For any $N \in \left\{ {0,1, \cdots ,10} \right\} $ and $t \in \mathrm{supp}_t\ w_{q + 1}^{\left( s \right)}$, various components of the Nash perturbation $w_{q + 1}^{\left( s \right)}\left( {t,x} \right)$ defined by \eqref{Nashsteps-A-P-7} satisfy
\begin{equation}\label{WPD-NS-SD-P-mainsd-AA-1}
{\left\| {w_{q + 1}^{\left( p \right)}} \right\|_N} \lesssim   \delta _{q + 1}^{\frac{1}{2}}\lambda _{q + 1}^N    ,
\end{equation}
\begin{equation}\label{WPD-NS-SD-P-mainsd-AA-2}
{\left\| {{{\bar D}_{t,\Gamma }}w_{q + 1}^{\left( p \right)}} \right\|_N}  \lesssim   \delta _{q + 1}^{\frac{1}{2}}\lambda _{q + 1}^{N + 1}   ,
\end{equation}
\begin{equation}\label{WPD-NS-SD-P-mainsd-AA-KIJHH-1}
{\left\| {w_{q + 1}^{\left( c \right)}} \right\|_N} \lesssim \delta _{q + 1}^{\frac{1}{2}}\lambda _{q + 1}^{N - \frac{1}{2}}   ,
\end{equation}
\begin{equation}\label{WPD-NS-SD-P-mainsd-KIJHH-AA-2}
 {\left\| {{{\bar D}_{t,\Gamma }}w_{q + 1}^{\left( c \right)}} \right\|_N} \lesssim {\mu _{q + 1}}\delta _q^{\frac{1}{2}}\delta _{q + 1}^{\frac{1}{2}}\lambda _{q + 1}^{N + \frac{1}{2}}  ,
\end{equation}
\begin{equation}\label{WPD-NS-SD-P-mainsd-AA-3}
{\left\| {w_{q + 1}^{\left( i \right)}} \right\|_N} \lesssim {{\tilde \mu }^{ - 1}} {\delta _{q + 1}}\lambda _{q + 1}^N   ,
\end{equation}
and
\begin{equation}\label{WPD-NS-SD-P-mainsd-AA-4}
{\left\| {{\partial _t}w_{q + 1}^{\left( i \right)}} \right\|_N} \lesssim   {{\tilde \mu }^{ - 1}}{\mu _{q + 1}}{\delta _{q + 1}}\lambda _{q + 1}^N .   
\end{equation}
\end{proposition}
\par
\noindent{\textbf{Proof}}. We first estimate the principal part $w_{q + 1}^{\left( p \right)}\left( {t,x} \right)$ defined by \eqref{Nashsteps-A-P-3}. Notice that
\begin{align*}
{\left\| {w_{q + 1}^{\left( p \right)}} \right\|_N} \lesssim & \mathop {\sup }\limits_{\xi ,k,n} \left( {{{\left\| {{{\tilde \eta }_{\xi ,k,n}}} \right\|}_N} + {{\left\| {{{\tilde \eta }_{\xi ,k,n}}} \right\|}_0}{{\left\| {{{\left( {\nabla {{\tilde \Phi }_k}} \right)}^{ - 1}}} \right\|}_N}} \right){\left\| {{\mathbb{W}_\xi }\left( {{\lambda _{q + 1}}{{\tilde \Phi }_k}} \right)} \right\|_0}  \\
& + \mathop {\sup }\limits_{\xi ,k,n} {\left\| {{{\tilde \eta }_{\xi ,k,n}}} \right\|_0}{\left\| {{\mathbb{W}_\xi }\left( {{\lambda _{q + 1}}{{\tilde \Phi }_k}} \right)} \right\|_N} ,
\end{align*}
and
\begin{align*}
{\left\| {{{\bar D}_{t,\Gamma }}w_{q + 1}^{\left( p \right)}} \right\|_N} \lesssim & \mathop {\sup }\limits_{\xi ,k,n} \left( {{{\left\| {{{\bar D}_{t,\Gamma }}{{\tilde \eta }_{\xi ,k,n}}} \right\|}_N} + {{\left\| {{{\tilde \eta }_{\xi ,k,n}}} \right\|}_N}{{\left\| {{{\bar D}_{t,\Gamma }}{{\left( {\nabla {{\tilde \Phi }_k}} \right)}^{ - 1}}} \right\|}_0}} \right) {{\left\| {{\mathbb{W}_\xi }\left( {{\lambda _{q + 1}}{{\tilde \Phi }_k}} \right)} \right\|}_0}  \\
& + \mathop {\sup }\limits_{\xi ,k,n}  \left( {{{\left\| {{{\tilde \eta }_{\xi ,k,n}}} \right\|}_N}{{\left\| {{{\bar D}_{t,\Gamma }}{\mathbb{W}_\xi }\left( {{\lambda _{q + 1}}{{\tilde \Phi }_k}} \right)} \right\|}_0} + {{\left\| {{{\bar D}_{t,\Gamma }}{{\tilde \eta }_{\xi ,k,n}}} \right\|}_0}{{\left\| {{\mathbb{W}_\xi }\left( {{\lambda _{q + 1}}{{\tilde \Phi }_k}} \right)} \right\|}_N}} \right) \\
& +  \mathop {\sup }\limits_{\xi ,k,n} {\left\| {{{\tilde \eta }_{\xi ,k,n}}} \right\|_0}{\left\| {{{\bar D}_{t,\Gamma }}{{\left( {\nabla {{\tilde \Phi }_k}} \right)}^{ - 1}}} \right\|_0}{\left\| {{W_\xi }\left( {{\lambda _{q + 1}}{{\tilde \Phi }_k}} \right)} \right\|_N} \\
& + \mathop {\sup }\limits_{\xi ,k,n} {\left\| {{{\tilde \eta }_{\xi ,k,n}}} \right\|_0}{\left\| {{{\bar D}_{t,\Gamma }}{{\left( {\nabla {{\tilde \Phi }_k}} \right)}^{ - 1}}} \right\|_N}{\left\| {{W_\xi }\left( {{\lambda _{q + 1}}{{\tilde \Phi }_k}} \right)} \right\|_0} \\
& + \mathop {\sup }\limits_{\xi ,k,n} {\left\| {{{\tilde \eta }_{\xi ,k,n}}} \right\|_0}\left( {{{\left\| {{{\left( {\nabla {{\tilde \Phi }_k}} \right)}^{ - 1}}} \right\|}_N}{{\left\| {{{\bar D}_{t,\Gamma }}{\mathbb{W}_\xi }\left( {{\lambda _{q + 1}}{{\tilde \Phi }_k}} \right)} \right\|}_0} + {{\left\| {{{\bar D}_{t,\Gamma }}{\mathbb{W}_\xi }\left( {{\lambda _{q + 1}}{{\tilde \Phi }_k}} \right)} \right\|}_N}} \right) \\
& + {\mu _{q + 1}}\mathop {\sup }\limits_{\xi ,k,n} \left( {{{\left\| {{{\tilde \eta }_{\xi ,k,n}}} \right\|}_N} + {{\left\| {{{\tilde \eta }_{\xi ,k,n}}} \right\|}_0}{{\left\| {{{\left( {\nabla {{\tilde \Phi }_k}} \right)}^{ - 1}}} \right\|}_N}} \right){\left\| {{W_\xi }\left( {{\lambda _{q + 1}}{{\tilde \Phi }_k}} \right)} \right\|_0}  \\
& + {\mu _{q + 1}}\mathop {\sup }\limits_{\xi ,k,n} {\left\| {{{\tilde \eta }_{\xi ,k,n}}} \right\|_0}{\left\| {{W_\xi }\left( {{\lambda _{q + 1}}{{\tilde \Phi }_k}} \right)} \right\|_N}  .
\end{align*}
It follows from Lemma \ref{WPD-NHS-AA}, Lemma \ref{WPD-NHS-BB}, \eqref{parameters-S-15} and \eqref{parameters-S-16} that
\begin{equation*}
{\left\| {w_{q + 1}^{\left( p \right)}} \right\|_N} \lesssim   \delta _{q + 1}^{\frac{1}{2}}\left( {\lambda _q^N + \lambda _{q + 1}^N} \right) \lesssim \delta _{q + 1}^{\frac{1}{2}} \lambda _{q + 1}^N    ,
\end{equation*}
and
\begin{align*}
{\left\| {{{\bar D}_{t,\Gamma }}w_{q + 1}^{\left( p \right)}} \right\|_N} \lesssim  &  \delta _{q + 1}^{\frac{1}{2}}\left( {\tau _q^{ - 1}\lambda _q^N + \tau _q^{ - 1}\lambda _{q + 1}^N + \delta _q^{\frac{1}{2}}\lambda _q^{N + 1} + \delta _q^{\frac{1}{2}}\lambda _q^N{\lambda _{q + 1}}} \right)  \\
&  + \delta _q^{\frac{1}{2}}\delta _{q + 1}^{\frac{1}{2}}\left( {{\lambda _q}\lambda _{q + 1}^N + \lambda _q^{N + 1} + \lambda _q^N{\lambda _{q + 1}} + \lambda _{q + 1}^{N + 1} + {\mu _{q + 1}}\lambda _{q + 1}^N} \right)  \\
\lesssim  &  \delta _{q + 1}^{\frac{1}{2}}\lambda _{q + 1}^{N + 1 }  . 
\end{align*}
\par
We next examine the corrector part $w_{q + 1}^{\left( c \right)}\left( {t,x} \right)$ defined by \eqref{Nashsteps-A-P-4}. Using Lemma \ref{WPD-NHS-AA}, Lemma \ref{WPD-NHS-BB} and Lemma \ref{WPD-NHS-HUD-BB-2}, we obtain
\begin{align}\label{WPD-NS-SD-P-main-P-3}
& {\left\| {{{\tilde \eta }_{\xi ,k,n}}{{\left( {\nabla {{\tilde \Phi }_k}} \right)}^{ - 1}}{\mathfrak{D}_{\tilde \mu }}\left( {{\lambda _{q + 1}}{{\tilde \Phi }_k}} \right)\left( {\nabla {{\tilde \Phi }_k}} \right)} \right\|_N} \nonumber \\
\lesssim & {\left\| {{{\tilde \eta }_{\xi ,k,n}}} \right\|_N} + {\left\| {{{\tilde \eta }_{\xi ,k,n}}} \right\|_0}\left( {{{\left\| {{{\left( {\nabla {{\tilde \Phi }_k}} \right)}^{ - 1}}} \right\|}_N} + {{\left\| {{\mathfrak{D}_{\tilde \mu }}\left( {{\lambda _{q + 1}}{{\tilde \Phi }_k}} \right)} \right\|}_N} + {{\left\| {\nabla {{\tilde \Phi }_k}} \right\|}_N}} \right)  \nonumber \\
\lesssim & \delta _{{q + 1},n}^{\frac{1}{2}}\lambda _{q + 1}^{N - \frac{1}{2}} .
\end{align}
Owing to
\begin{align*}
{\left\| {w_{q + 1}^{\left( c \right)}} \right\|_N} \lesssim & \lambda _{q + 1}^{ - 1}\mathop {\sup }\limits_{\xi ,k,n} {\left\| {{\mathfrak{S}_\xi }\left( {{\lambda _{q + 1}}{{\tilde \Phi }_k}} \right)} \right\|_N}  {\left\| {{{\tilde \eta }_{\xi ,k,n}}{{\left( {\nabla {{\tilde \Phi }_k}} \right)}^{ - 1}}{\mathfrak{D}_{\tilde \mu }}\left( {{\lambda _{q + 1}}{{\tilde \Phi }_k}} \right)\left( {\nabla {{\tilde \Phi }_k}} \right)} \right\|_1} \\
& + \lambda _{q + 1}^{ - 1}\mathop {\sup }\limits_{\xi ,k,n} {\left\| {{{\tilde \eta }_{\xi ,k,n}}{{\left( {\nabla {{\tilde \Phi }_k}} \right)}^{ - 1}}{\mathfrak{D}_{\tilde \mu }}\left( {{\lambda _{q + 1}}{{\tilde \Phi }_k}} \right)\left( {\nabla {{\tilde \Phi }_k}} \right)} \right\|_{N + 1}}  ,
\end{align*}
it follows from \eqref{WPD-NS-SD-P-main-P-3} and Lemma \ref{WPD-NHS-HUD-BB} that
\begin{equation*}
{\left\| {w_{q + 1}^{\left( c \right)}} \right\|_N} \lesssim \delta _{q + 1}^{\frac{1}{2}}\lambda _{q + 1}^{N - \frac{1}{2}}   .
\end{equation*}
Since
\begin{align}\label{WPD-NS-SD-P-main-P-4}
& {\left\| {{{\bar D}_{t,\Gamma }}\left( {{{\tilde \eta }_{\xi ,k,n}}{{\left( {\nabla {{\tilde \Phi }_k}} \right)}^{ - 1}}{\mathfrak{D}_{\tilde \mu }}\left( {{\lambda _{q + 1}}{{\tilde \Phi }_k}} \right)\left( {\nabla {{\tilde \Phi }_k}} \right)} \right)} \right\|_N} \nonumber \\
\lesssim & {\left\| {{{\bar D}_{t,\Gamma }}{{\tilde \eta }_{\xi ,k,n}}} \right\|_N} + {\left\| {{{\bar D}_{t,\Gamma }}{{\tilde \eta }_{\xi ,k,n}}} \right\|_0}{\left\| {{{\left( {\nabla {{\tilde \Phi }_k}} \right)}^{ - 1}} + \left( {\nabla {{\tilde \Phi }_k}} \right) + {\mathfrak{D}_{\tilde \mu }}\left( {{\lambda _{q + 1}}{{\tilde \Phi }_k}} \right)} \right\|_N} \nonumber \\
& + {\left\| {{{\tilde \eta }_{\xi ,k,n}}} \right\|_N}{\left\| {{{\bar D}_{t,\Gamma }}{{\left( {\nabla {{\tilde \Phi }_k}} \right)}^{ - 1}}} \right\|_0} + {\left\| {{{\tilde \eta }_{\xi ,k,n}}} \right\|_0}{\left\| {{{\bar D}_{t,\Gamma }}{{\left( {\nabla {{\tilde \Phi }_k}} \right)}^{ - 1}}} \right\|_N} \nonumber \\
& + {\left\| {{{\tilde \eta }_{\xi ,k,n}}} \right\|_0}{\left\| {{{\bar D}_{t,\Gamma }}{{\left( {\nabla {{\tilde \Phi }_k}} \right)}^{ - 1}}} \right\|_0}\left( {{{\left\| {{\mathfrak{D}_{\tilde \mu }}\left( {{\lambda _{q + 1}}{{\tilde \Phi }_k}} \right)} \right\|}_N} + {{\left\| {\nabla {{\tilde \Phi }_k}} \right\|}_N}} \right)  \nonumber \\
& + {\left\| {{{\tilde \eta }_{\xi ,k,n}}} \right\|_N}{\left\| {{{\bar D}_{t,\Gamma }}{\mathfrak{D}_{\tilde \mu }}\left( {{\lambda _{q + 1}}{{\tilde \Phi }_k}} \right)} \right\|_0} + {\left\| {{{\tilde \eta }_{\xi ,k,n}}} \right\|_0}{\left\| {{{\bar D}_{t,\Gamma }}{\mathfrak{D}_{\tilde \mu }}\left( {{\lambda _{q + 1}}{{\tilde \Phi }_k}} \right)} \right\|_N}  \nonumber \\
& + {\left\| {{{\tilde \eta }_{\xi ,k,n}}} \right\|_0}{\left\| {{{\bar D}_{t,\Gamma }}{\mathfrak{D}_{\tilde \mu }}\left( {{\lambda _{q + 1}}{{\tilde \Phi }_k}} \right)} \right\|_0}\left( {{{\left\| {{\mathfrak{D}_{\tilde \mu }}\left( {{\lambda _{q + 1}}{{\tilde \Phi }_k}} \right)} \right\|}_N} + {{\left\| {{{\left( {\nabla {{\tilde \Phi }_k}} \right)}^{ - 1}}} \right\|}_N}} \right)  \nonumber \\
\lesssim & {\mu _{q + 1}}\delta _{q + 1}^{\frac{1}{2}}\lambda _{q + 1}^N   ,
\end{align}
we deduce from \eqref{WPD-NS-SD-P-main-P-3}, \eqref{WPD-NS-SD-P-main-P-4} and Lemma \ref{WPD-NHS-HUD-BB} that
\begin{align*}
{\left\| {{{\bar D}_{t,\Gamma }}w_{q + 1}^{\left( c \right)}} \right\|_N} \lesssim & \lambda _{q + 1}^{ - 1}\mathop {\sup }\limits_{\xi ,k,n} {\left\| {{{\bar D}_{t,\Gamma }}{\mathfrak{S}_\xi }\left( {{\lambda _{q + 1}}{{\tilde \Phi }_k}} \right)} \right\|_N}{\left\| {{{\tilde \eta }_{\xi ,k,n}}{{\left( {\nabla {{\tilde \Phi }_k}} \right)}^{ - 1}}{\mathfrak{D}_{\tilde \mu }}\left( {{\lambda _{q + 1}}{{\tilde \Phi }_k}} \right)\left( {\nabla {{\tilde \Phi }_k}} \right)} \right\|_1}  \\
& + \lambda _{q + 1}^{ - 1}\mathop {\sup }\limits_{\xi ,k,n} {\left\| {{{\bar D}_{t,\Gamma }}{\mathfrak{S}_\xi }\left( {{\lambda _{q + 1}}{{\tilde \Phi }_k}} \right)} \right\|_0}{\left\| {{{\tilde \eta }_{\xi ,k,n}}{{\left( {\nabla {{\tilde \Phi }_k}} \right)}^{ - 1}}{\mathfrak{D}_{\tilde \mu }}\left( {{\lambda _{q + 1}}{{\tilde \Phi }_k}} \right)\left( {\nabla {{\tilde \Phi }_k}} \right)} \right\|_{N + 1}}  \\
& + \lambda _{q + 1}^{ - 1}\mathop {\sup }\limits_{\xi ,k,n} {\left\| {{\mathfrak{S}_\xi }\left( {{\lambda _{q + 1}}{{\tilde \Phi }_k}} \right)} \right\|_N}{\left\| {{{\bar D}_{t,\Gamma }}\left( {{{\tilde \eta }_{\xi ,k,n}}{{\left( {\nabla {{\tilde \Phi }_k}} \right)}^{ - 1}}{\mathfrak{D}_{\tilde \mu }}\left( {{\lambda _{q + 1}}{{\tilde \Phi }_k}} \right)\left( {\nabla {{\tilde \Phi }_k}} \right)} \right)} \right\|_1} \\
& + \lambda _{q + 1}^{ - 1}\mathop {\sup }\limits_{\xi ,k,n} {\left\| {{{\bar D}_{t,\Gamma }}\left( {{{\tilde \eta }_{\xi ,k,n}}{{\left( {\nabla {{\tilde \Phi }_k}} \right)}^{ - 1}}{\mathfrak{D}_{\tilde \mu }}\left( {{\lambda _{q + 1}}{{\tilde \Phi }_k}} \right)\left( {\nabla {{\tilde \Phi }_k}} \right)} \right)} \right\|_{N + 1}} \\
& + {\mu _{q + 1}}\lambda _{q + 1}^{ - 1}\mathop {\sup }\limits_{\xi ,k,n} {\left\| {{\mathfrak{S}_\xi }\left( {{\lambda _{q + 1}}{{\tilde \Phi }_k}} \right)} \right\|_N}{\left\| {{{\tilde \eta }_{\xi ,k,n}}{{\left( {\nabla {{\tilde \Phi }_k}} \right)}^{ - 1}}{\mathfrak{D}_{\tilde \mu }}\left( {{\lambda _{q + 1}}{{\tilde \Phi }_k}} \right)\left( {\nabla {{\tilde \Phi }_k}} \right)} \right\|_1} \\
& + {\mu _{q + 1}}\lambda _{q + 1}^{ - 1}\mathop {\sup }\limits_{\xi ,k,n} {\left\| {{{\tilde \eta }_{\xi ,k,n}}{{\left( {\nabla {{\tilde \Phi }_k}} \right)}^{ - 1}}{\mathfrak{D}_{\tilde \mu }}\left( {{\lambda _{q + 1}}{{\tilde \Phi }_k}} \right)\left( {\nabla {{\tilde \Phi }_k}} \right)} \right\|_{N + 1}} \\
\lesssim & {\mu _{q + 1}}\delta _q^{\frac{1}{2}}\delta _{q + 1}^{\frac{1}{2}}\lambda _{q + 1}^{N + \frac{1}{2}}   .
\end{align*}
\par
Consider the temporal corrector $w_{q + 1}^{\left( i \right)}\left( {t,x} \right)$ defined by \eqref{THENDM-DKF}. According to Lemma \ref{WPD-NHS-BB} and Lemma \ref{WPD-NHS-HUD-BB}, we have
\begin{align}\label{WPD-NS-SD-P-main-P-8}
{\left\| {w_{q + 1}^{\left( i \right)}} \right\|_N} \lesssim & {{\tilde \mu }^{ - 1}}\mathop {\sup }\limits_{\xi ,k,n} \left( {{{\left\| {{{\tilde \eta }_{\xi ,k,n}}} \right\|}_N}{{\left\| {{{\tilde \eta }_{\xi ,k,n}}} \right\|}_0} + \left\| {{{\tilde \eta }_{\xi ,k,n}}} \right\|_0^2{{\left\| {{\mathfrak{D}_{\tilde \mu }}\left( {{\lambda _{q + 1}}{{\tilde \Phi }_k}} \right)} \right\|}_N}} \right) \nonumber \\
\lesssim &  {{\tilde \mu }^{ - 1}} {\delta _{q + 1}}\lambda _{q + 1}^N    ,
\end{align}
and
\begin{align}\label{WPD-NS-SD-P-main-P-9}
{\left\| {{\partial _t}w_{q + 1}^{\left( i \right)}} \right\|_N} \lesssim & {{\tilde \mu }^{ - 1}}\mathop {\sup }\limits_{\xi ,k,n} {\left\| {{{\tilde \eta }_{\xi ,k,n}}} \right\|_N}{\left\| {{{\tilde \eta }_{\xi ,k,n}}} \right\|_0}{\left\| {{\partial _t}{\mathfrak{D}_{\tilde \mu }}\left( {t,{\lambda _{q + 1}}{{\tilde \Phi }_k}} \right)} \right\|_0} \nonumber \\
& + {{\tilde \mu }^{ - 1}}\mathop {\sup }\limits_{\xi ,k,n} \left\| {{{\tilde \eta }_{\xi ,k,n}}} \right\|_0^2{\left\| {{\partial _t}{\mathfrak{D}_{\tilde \mu }}\left( {t,{\lambda _{q + 1}}{{\tilde \Phi }_k}} \right)} \right\|_N} \nonumber \\
& + {{\tilde \mu }^{ - 1}}\mathop {\sup }\limits_{\xi ,k,n} \left\| {{{\tilde \eta }_{\xi ,k,n}}} \right\|_0^2{\left\| {{\partial _t}{\mathfrak{D}_{\tilde \mu }}\left( {t,{\lambda _{q + 1}}{{\tilde \Phi }_k}} \right)} \right\|_0}{\left\| {{\mathfrak{D}_{\tilde \mu }}\left( {t,{\lambda _{q + 1}}{{\tilde \Phi }_k}} \right)} \right\|_N} \nonumber \\
& + {{\tilde \mu }^{ - 1}}{\mu _{q + 1}}\mathop {\sup }\limits_{\xi ,k,n} {\left\| {{{\tilde \eta }_{\xi ,k,n}}} \right\|_0}\left( {{{\left\| {{{\tilde \eta }_{\xi ,k,n}}} \right\|}_N}{{\left\| {{\mathfrak{D}_{\tilde \mu }}\left( {t,{\lambda _{q + 1}}{{\tilde \Phi }_k}} \right)} \right\|}_0} + {{\left\| {{{\tilde \eta }_{\xi ,k,n}}} \right\|}_0}{{\left\| {{\mathfrak{D}_{\tilde \mu }}\left( {t,{\lambda _{q + 1}}{{\tilde \Phi }_k}} \right)} \right\|}_N}} \right) \nonumber \\
& + {{\tilde \mu }^{ - 1}}\mathop {\sup }\limits_{\xi ,k,n} \left( {{{\left\| {{\partial _t}{{\tilde \eta }_{\xi ,k,n}}} \right\|}_N}{{\left\| {{{\tilde \eta }_{\xi ,k,n}}} \right\|}_0} + {{\left\| {{\partial _t}{{\tilde \eta }_{\xi ,k,n}}} \right\|}_0}{{\left\| {{{\tilde \eta }_{\xi ,k,n}}} \right\|}_N}} \right){\left\| {{\mathfrak{D}_{\tilde \mu }}\left( {t,{\lambda _{q + 1}}{{\tilde \Phi }_k}} \right)} \right\|_0}  \nonumber \\
& + {{\tilde \mu }^{ - 1}}\mathop {\sup }\limits_{\xi ,k,n} {\left\| {{\partial _t}{{\tilde \eta }_{\xi ,k,n}}} \right\|_0}{\left\| {{{\tilde \eta }_{\xi ,k,n}}} \right\|_0}{\left\| {{\mathfrak{D}_{\tilde \mu }}\left( {t,{\lambda _{q + 1}}{{\tilde \Phi }_k}} \right)} \right\|_N} \nonumber \\
\lesssim  &  {{\tilde \mu }^{ - 1}}{\mu _{q + 1}}{\delta _{q + 1}}\lambda _{q + 1}^N   .
\end{align}
The proof is complete. \qed
\section{Inductive estimates}\label{ENHP-RHDU}
This section establishes the inductive estimates for the $\left( {q + 1} \right)$-step stochastic Euler--Reynolds system \eqref{KIU-EU-N-1}, which guarantees the continuation of the stochastic and intermittent Newton--Nash iteration scheme. To this end, we prove that \eqref{HIE-u-1.1}-\eqref{MDHIE-R-1.1} remain valid with $q$ replaced by $q+1$, and \eqref{HIDJJF-3-A-1}-\eqref{HIDJJF-3-A-3} remain valid with $n$ replaced by $n+1$. Furthermore, we verify the inductive estimates for energy gap \eqref{LOKJI-A-FGH-A-1}.
\par
\begin{proposition}\label{HNNI-NEHU-MAIN}
For any $t \in \mathrm{supp}_t\ w_{q + 1,n + 1}^{(n)}$, the Newtonian pressure ${\mathfrak{p}_{q,n + 1}}\left( {t,x} \right)$ satisfies
\begin{equation*}
\left\| {{\mathfrak{p}_{q,{n+1}}}} \right\| _0 \leqslant C ,
\end{equation*}
\begin{equation*}
\left\| {{\mathfrak{p}_{q,{n+1}}}} \right\| _N \lesssim \delta _q^{\frac{1}{2}}\lambda _q^N , \ \ \forall \, N \in \left\{ {1, \cdots ,10} \right\} ,
\end{equation*}
and there exists a parameter $\alpha \in \left( {\frac{\beta }{{b}},\beta } \right)$ such that the Newtonian stress ${\mathring{R}_{q,n + 1}}\left( {t,x} \right)$ satisfies
\begin{equation*}
{\left\| {{\mathring{R}_{q,{n+1}}}} \right\|_N} \lesssim {\delta _{q + 1,{n+1}}}\lambda _q^{N - \alpha }  , \ \ \forall \, N \in \left\{ {0,1, \cdots ,10} \right\}    ,
\end{equation*}
\begin{equation*}
{\left\| {{{\bar D}_{t,q}}{\mathring{R}_{q,{n+1}}}} \right\|_N} \lesssim {\delta _{q + 1,{n+1}}}\tau _q^{ - 1}\lambda _q^{N - \alpha } , \ \ \forall \, N \in \left\{ {0,1, \cdots ,10} \right\}   .
\end{equation*}
\end{proposition}
\par
\noindent{\textbf{Proof}}. 
Since
\begin{align*}
{\left\| {{\mathfrak{p}_{q,n + 1}}} \right\|_N} \lesssim & \mathop {\sup }\limits_{\xi ,k} {\left\| {{\mathfrak{p}_{q,n}} + p_{k,n + 1}^{(n)}} \right\|_N} + {\left\| {w_{q + 1,n + 1}^{(n)}} \right\|_N}\left( {{{\left\| {{u_q} - {{\bar v}_q}} \right\|}_0} + {{\left\| {w_{q + 1,n + 1}^{(n)}} \right\|}_0}} \right) \\
& + {\left\| {w_{q + 1}^{(n)}} \right\|_N}{\left\| {w_{q + 1,n + 1}^{(n)}} \right\|_0} + {\left\| {{\mathring{R}_{q,n}}} \right\|_N} + {\left\| {{A_{\xi ,k,{n+1}}}} \right\|_N}    ,
\end{align*}
it follows from \eqref{HIDJJF-3-A-1-A-1-KSIJDD-A}, \eqref{HIDJJF-3-A-1}, \eqref{HIDJJF-3-A-2}, \eqref{SHDUFY-A-2dJD}, \eqref{Newtonsteps-A-B-19} and \eqref{parameters-S-17} that\begin{equation*}
\left\| {{\mathfrak{p}_{q,{n+1}}}} \right\| _0 \leqslant C ,
\end{equation*}
and
\begin{equation*}
\left\| {{\mathfrak{p}_{q,{n+1}}}} \right\| _N \lesssim \delta _q^{\frac{1}{2}}\lambda _q^N + \mu _{q + 1}^{ - 1}{\delta _{q + 1,n}}\lambda _q^{N + 1}l_q^{ - \alpha } + {\delta _{q + 1,n}}\lambda _q^N \lesssim \delta _q^{\frac{1}{2}}\lambda _q^N , \ \ \forall \, N \in \left\{ {1, \cdots ,10} \right\} ,
\end{equation*}
Combining \eqref{Newtonsteps-A-7}, \eqref{Newtonsteps-A-B-36} and \eqref{Newtonsteps-A-B-41}, we deduce that
\begin{equation*}
{\left\| {{\mathring{R}_{q,n + 1}}} \right\|_N} \lesssim \lambda _q^{ - 1}{\left\| {w_{k,n + 1}^{\left( l \right)}} \right\|_N} \lesssim \mu _{q + 1}^{ - 1}{\delta _{q + 1,n}}\lambda _q^Nl_q^{ - \alpha } \lesssim {\delta _{q + 1,n + 1}}\lambda _q^{N - \alpha }   ,
\end{equation*}
and
\begin{equation*}
{\left\| {{{\bar D}_{t,q}}{\mathring{R}_{q,n + 1}}} \right\|_N} \lesssim \lambda _q^{ - 1}{\left\| {{{\bar D}_{t,q}}w_{k,n + 1}^{\left( l \right)}} \right\|_N} \lesssim {\delta _{q + 1,n}}\lambda _q^Nl_q^{ - \alpha } \lesssim {\delta _{q + 1,n + 1}}\tau _q^{ - 1}\lambda _q^{N - \alpha }   .
\end{equation*}
The proof is complete. \qed
\par
\begin{proposition}\label{RSRI-ujdh-main}
For any $t \in \mathrm{supp}_t\ w_{q + 1}^{\left( s \right)}$, the Wiener process ${\mathfrak{W}_{q+1}}\left( t \right)$, the velocity ${u_{q + 1}}\left( {t,x} \right)$ and the Nash pressure ${\mathfrak{p}_{q + 1}}\left( {x} \right)$ satisfy
\begin{equation}\label{RSRI-P-mainsd-AAsd-1-JUSJJ-A-1}
{\left\| {{u_{q+1}} } \right\|_0} + {\left\| {{v_{q + 1}}} \right\|_0} + {\left\| {{\mathfrak{p}_{q + 1}}} \right\|_0} \leqslant C  ,
\end{equation}
and
\begin{equation}\label{RSRI-P-mainsd-AAsd-1}
{\left\| {{u_{q+1}} } \right\|_N} + {\left\| {{v_{q + 1}}} \right\|_N} + {\left\| {{\mathfrak{p}_{q + 1}}} \right\|_N} \lesssim \delta _{q+1}^{\frac{1}{2}}\lambda _{q+1}^N , \ \ \forall \, N \in \left\{ {1, \cdots ,10} \right\}   . 
\end{equation}
\end{proposition}
\par
\noindent{\textbf{Proof}}. It is easy to check that \eqref{RSRI-P-mainsd-AAsd-1-JUSJJ-A-1} holds. When $N \geqslant  1$, it follows from \eqref{NISTYU-SKO-A-PPA-1} and \eqref{NISTYU-SKO-A-PPA-2} that
\begin{equation}\label{RSRI-P-PROOV-AA-2}
{\left\| {{v_{q + 1}}} \right\|_N} \lesssim {\left\| {{v_q}} \right\|_N} + {\left\| {w_{q + 1}^{\left( s \right)}} \right\|_N} + {\left\| {w_{q + 1}^{\left( \Gamma \right)}} \right\|_N} .
\end{equation}
Then, using \eqref{HIE-u-1.1}, \eqref{WPD-NP-main-BTO-1}, \eqref{RSRI-P-PROOV-AA-2}, \eqref{parameters-S-14} and Proposition \ref{WPD-NS-SD-P-main}, we infer that
\begin{equation}\label{RSRI-P-mainsd-AAsd-2}
{\left\| {{v_{q + 1}}} \right\|_N}  \lesssim \delta _{q + 1}^{\frac{1}{2}}\lambda _{q + 1}^N    .
\end{equation}
According to \eqref{NISTYU-SKO-A-PWWWPA-2} and \eqref{parameters-S-20-ENDH-A}, we have 
\begin{equation}\label{RSRI-P-HUSGDJ-AAsd-2}
\mathop {\sup }\limits_{t \in \left[ {0,T} \right]}  \mathbb{E} \left\| {{\mathfrak{W}_{q + 1}}} \right\|_N^2 \lesssim \sum\limits_{j = 0}^q {{\delta _j}\lambda _j^{2N}}  + {\delta _{q + 1}}\lambda _{q + 1}^{2N} \lesssim {\delta _{q + 1}}\lambda _{q + 1}^{2N} .
\end{equation}
Combining \eqref{RSRI-P-mainsd-AAsd-2} and \eqref{RSRI-P-HUSGDJ-AAsd-2}, we obtain
\begin{equation}\label{RSRI-P-mainsd-AAsd-1-KIJ}
{\left\| {{u_{q+1}} } \right\|_N} \leqslant {\left\| {{v_{q+1}}} \right\|_N} + {\left\| {{\mathfrak{W}_{q + 1}}} \right\|_N}  \lesssim \delta _{q+1}^{\frac{1}{2}}\lambda _{q+1}^N  .
\end{equation}
\par
We note that
\begin{equation}\label{JHKJFHSJKDHF-KSBDG=ASF-AF}
\mathop {\sup }\limits_{t \in \left[ {0,T} \right]} \mathbb{E}\left\| {\mathfrak{W}_{q + 1}^{loc}\left( t \right)} \right\|_N^m =  {\left( {\frac{{{\varsigma _{q + 1}}}}{\pi }} \right)^m} \delta _q^{\alpha m}  \delta _{q + 2}^{\frac{1}{2}m}\lambda _{q + 1}^{Nm} \mathop {\sup }\limits_{t \in \left[ {0,T} \right]} \mathbb{E}\left| {\mathfrak{B}_{q + 1}\left( t \right)} \right|^m \lesssim \delta _{q + 1}^{\frac{1}{2}m}\lambda _{q + 1}^{Nm} ,
\end{equation}
and
\begin{equation}\label{JHKJFHSJKDHF-SDAF}
\mathop {\sup }\limits_{t \in \left[ {0,T} \right]} \mathbb{E}\left\| {{\mathfrak{W}_q}\left( t \right)} \right\|_N^2 \lesssim \sum\limits_{j = 0}^q {{\delta _j}\lambda _j^{2N}} \lesssim {\delta _{q + 1}}\lambda _{q + 1}^{2N}  .
\end{equation}
Then, it follows from \eqref{LIJIFNI-DKFJ-L-56}, \eqref{RSRI-P-mainsd-AAsd-1-KIJ}-\eqref{JHKJFHSJKDHF-SDAF} and Proposition \ref{WPD-NS-SD-P-main} that
\begin{align}\label{RSRI-P-mainsd-AAsd-1-KIJ-a}
{\left\| {\mathfrak{p}_{q + 1}^s} \right\|_N} \lesssim & \left\| {{u_{q + 1}}} \right\|_N  \left\| {\mathfrak{W}_{q + 1}^{loc}} \right\|_0 + \left\| {{u_{q + 1}}} \right\|_0  \left\| {\mathfrak{W}_{q + 1}^{loc}} \right\|_N +  \left\| {\mathfrak{W}_{q + 1}^{loc}} \right\|_0  \left\| {\mathfrak{W}_{q + 1}^{loc}} \right\|_N \nonumber \\
& + {\left\| {w_{q + 1}^{\left( s \right)}} \right\|_0}{\left\| {{\mathfrak{W}_q}} \right\|_N} + {\left\| {w_{q + 1}^{\left( s \right)}} \right\|_N}{\left\| {{\mathfrak{W}_q}} \right\|_0} \nonumber \\
\lesssim & \delta _{q+1}^{\frac{1}{2}}\lambda _{q+1}^N  .
\end{align}
\par
Owing to \eqref{SHDUFY-A-2dJD}, we have
\begin{align*}
{\left\| {{\mathfrak{p}_{q,\Gamma }}} \right\|_N} \lesssim & {\left\| {{\mathfrak{p}_q}} \right\|_N} + \mathop {\sup }\limits_{k,n} {\left\| {{\mathfrak{p}_{k,n + 1}}} \right\|_N} + {\left\| {w_{q + 1,n + 1}^{(n)}} \right\|_N}{\left\| {{u_q} - {{\bar v}_q} + w_{q + 1}^{\left( \Gamma \right)}} \right\|_0}  \\
& + \mathop {\sup }\limits_n {\left\| {w_{q + 1,n + 1}^{(n)}} \right\|_0}{\left\| {{u_q} - {{\bar v}_q} + w_{q + 1}^{\left( \Gamma \right)}} \right\|_N} + {\left\| {w_{q + 1,n + 1}^{(n)}} \right\|_N}{\left\| {w_{q + 1,n + 1}^{(n)}} \right\|_0} \\
& + \mathop {\sup }\limits_{\xi ,k,n} \left( {{\left\| {{\mathring{R}_{q,n}}} \right\|_N} + {\left\| {{A_{\xi ,k,{n+1}}}} \right\|_N}} \right) .
\end{align*}
By virtue of \eqref{HIE-u-1.1}, \eqref{lemma-GLTNNRS-B-M-1}, \eqref{HIDJJF-3-A-2}, \eqref{Newtonsteps-A-B-19}, \eqref{WPD-NP-main-A-1}, \eqref{NISTYU-SKO-A-PPA-3}, \eqref{WPD-NP-main-BTO-1}, \eqref{RSRI-P-mainsd-AAsd-1-KIJ-a} and \eqref{parameters-S-13}, we deduce that
\begin{align}\label{RSRI-P-maSWQIJ-a}
{\left\| {{\mathfrak{p}_{q + 1}}} \right\|_N} \leqslant & {\left\| {{\mathfrak{p}_{q,\Gamma }}} \right\|_N} + {\left\| {\mathfrak{p}_{q + 1}^s} \right\|_N} \nonumber \\
\lesssim & \delta _q^{\frac{1}{2}}\lambda _q^N + \mu _{q + 1}^{ - 1}{\delta _{q + 1}}\lambda _q^{N + 1}l_q^{ - \alpha }\left( {C + \delta _q^{\frac{1}{2}} + \mu _{q + 1}^{ - 1}{\delta _{q + 1}}{\lambda _q}l_q^{ - \alpha }} \right) \nonumber \\
& + \mu _{q + 1}^{ - 2}\delta _{q + 1}^2\lambda _q^{N + 2}l_q^{ - 2\alpha } + {\delta _{q + 1}}\lambda _q^N  \nonumber \\
\lesssim & \delta _{q + 1}^{\frac{1}{2}}\lambda _{q + 1}^N .
\end{align}
Consequently, combining \eqref{RSRI-P-mainsd-AAsd-2}, \eqref{RSRI-P-mainsd-AAsd-1-KIJ} and \eqref{RSRI-P-maSWQIJ-a}, we conclude that \eqref{RSRI-P-mainsd-AAsd-1} holds. The proof is complete. \qed
\par
\begin{proposition}\label{RSRI-main}
For any $N \in \left\{ {0,1, \cdots ,10} \right\} $ and $t \in {\mathrm{supp}_\mathrm{t}} \, {\mathring{R}_{q + 1}}$, there exists a parameter $\alpha \in \left( {\frac{\beta }{{b}},\beta } \right)$ such that the Reynolds stress ${\mathring{R}_{q + 1}}\left( {t,x} \right)$ satisfies
\begin{equation}\label{RSRI-P-mainsd-AAJJS-1}
{\left\| { \mathring{R}_{q + 1} } \right\|_N} \lesssim {\delta _{q + 2}}\lambda _{q + 1}^{N - \alpha } ,
\end{equation}
and
\begin{equation}\label{RSRI-P-mainsd-AAJJS-2}
{\left\| { {{ D_{t,{q + 1} }}} \mathring{R}_{q + 1} } \right\|_N}  \lesssim \delta _{q + 1}^{\frac{1}{2}}{\delta _{q + 2}}\lambda _{q + 1}^{N + 1 - \alpha }  .  
\end{equation}
\end{proposition}
\par
\noindent{\textbf{Proof}}.
According to the definition \eqref{NISTYU-SKO-A-PPA-4}, the Reynolds stress ${\mathring{R}_{q + 1}}\left( {t,x} \right)$ is decomposed into the linear error $\mathring{R}_{q + 1}^{\left( l \right)}$ defined by \eqref{LIJIFNI-DKFJ-L-1}, the intermittent oscillation error $\mathring{R}_{q + 1}^{\left( i \right)}$  defined by \eqref{LIJIFNI-DKFJ-O-2}, and the random oscillation error $\mathring{R}_{q + 1}^{\left( r \right)}$  defined by \eqref{LIJIFNI-DKFJ-R-3}. In what follows, we examine these components one by one.
\par
\vspace{1em}
{\textbf{Step 1. Estimates of the linear error $\mathring{R}_{q + 1}^{\left( l \right)}$}}. 
\par
\vspace{1em}
Since $w_{q + 1}^{\left( s \right)}\left( {t,x} \right)$ is divergence-free, we have
\begin{align}\label{RSRI-main-PROVE-1}
{\left\| {\mathcal{R}\left( {{{\bar v}_{q,\Gamma }} \cdot \nabla } \right)w_{q + 1}^{\left( s \right)}} \right\|_N} = & {\left\| {\mathcal{R}\mathrm{div}\left( {w_{q + 1}^{\left( s \right)} \otimes {{\bar v}_{q,\Gamma }}} \right)} \right\|_N} \nonumber \\
\lesssim & {\left\| {\mathcal{R}\left( {w_{q + 1}^{\left( s \right)} \cdot \nabla } \right){{\bar v}_{q,\Gamma }}} \right\|_N} \nonumber \\
\lesssim &   {\left\| {\mathcal{R}\left( {w_{q + 1}^{\left( i \right)} \cdot \nabla } \right){{\bar v}_{q,\Gamma }}} \right\|_N} + {\left\| {\mathcal{R}\left( {\left( {w_{q + 1}^{\left( p \right)} + w_{q + 1}^{\left( c \right)}} \right) \cdot \nabla } \right){{\bar v}_{q,\Gamma }}} \right\|_N} .
\end{align}
Notice that
\begin{align*}
{\left\| {\mathcal{R}\left( {w_{q + 1}^{\left( i \right)} \cdot \nabla } \right){{\bar v}_{q,\Gamma }}} \right\|_N} = & {\left\| {\mathcal{R}\mathrm{div}\left( {w_{q + 1}^{\left( i \right)} \otimes {{\bar v}_{q,\Gamma }}} \right)} \right\|_N} \\
\lesssim & {\left\| {w_{q + 1}^{\left( i \right)}} \right\|_N}{\left\| {{{\bar v}_{q,\Gamma }}} \right\|_0} + {\left\| {w_{q + 1}^{\left( i \right)}} \right\|_0}{\left\| {{{\bar v}_{q,\Gamma }}} \right\|_N}  \\
\lesssim & {\delta _{q + 2}}\lambda _{q + 1}^{N - \alpha }{{\tilde \mu }^{ - 1}}\delta _{q + 1}^{1 - \frac{\alpha }{{2\beta }}}\lambda _{q + 2}^{2\beta }   .
\end{align*}
It follows from \eqref{RSRI-main-PROVE-2} that
\begin{equation}\label{RSRI-main-PROVE-3}
{\left\| {\mathcal{R}\left( {w_{q + 1}^{\left( i \right)} \cdot \nabla } \right){{\bar v}_{q,\Gamma }}} \right\|_N} \lesssim \delta _{q + 1}^{\frac{1}{2}} {\delta _{q + 2}}\lambda _{q + 1}^{N - \alpha }   .
\end{equation}
\par
Owing to
\begin{align*}
\mathcal{R}\left( {\left( {w_{q + 1}^{\left( p \right)} + w_{q + 1}^{\left( c \right)}} \right) \cdot \nabla } \right){{\bar v}_{q,\Gamma }} = & \lambda _{q + 1}^{ - 1}\mathcal{R}{\nabla ^ \bot }\sum\limits_{n = 0}^\Gamma  {\sum\limits_{k \in {\mathbb{Z}_{q,n}}} {\sum\limits_{\xi  \in \Lambda } {{g_{\xi ,k,n + 1}}\left( {{\mu _{q + 1}}t} \right){{\tilde \eta }_{\xi ,k,n}}{{\left( {\nabla {{\tilde \Phi }_k}} \right)}^{ - 1}}} } } \\
& \hspace*{6em} {\mathfrak{D}_{\tilde \mu }}\left( {t,{\lambda _{q + 1}}{{\tilde \Phi }_k}} \right)\left( {\nabla {{\tilde \Phi }_k}} \right) \cdot \mathfrak{S}_{\xi} \left( {{\lambda _{q + 1}}{{\tilde \Phi }_k}} \right) \cdot \nabla {{\bar v}_{q,\Gamma }}  ,
\end{align*}
we deduce that
\begin{align*}
& {\left\| {\mathcal{R}\left( {\left( {w_{q + 1}^{\left( p \right)} + w_{q + 1}^{\left( c \right)}} \right) \cdot \nabla } \right){{\bar v}_{q,\Gamma }}} \right\|_N}   \\
\lesssim & \lambda _{q + 1}^{ - 1}\mathop {\sup }\limits_{\xi ,k,n} {\left\| {{{\bar v}_{q,\Gamma }}} \right\|_1}\left( { {\delta _q} {{\left\| {{{\tilde \eta }_{\xi ,k,n}}} \right\|}_N} + \delta _q^{\frac{1}{2}}{{\left\| {{{\tilde \eta }_{\xi ,k,n}}} \right\|}_0}{{\left\| {\nabla {{\tilde \Phi }_k} + {{\left( {\nabla {{\tilde \Phi }_k}} \right)}^{ - 1}}} \right\|}_N} } \right)  \\
& + {\delta _q} \lambda _{q + 1}^{ - 1}\mathop {\sup }\limits_{\xi ,k,n} {\left\| {{{\tilde \eta }_{\xi ,k,n}}} \right\|_0}{\left\| {{{\bar v}_{q,\Gamma }}} \right\|_1}\left( {{{\left\| {{\mathfrak{D}_{\tilde \mu }}\left( {t,{\lambda _{q + 1}}{{\tilde \Phi }_k}} \right)} \right\|}_N} + {{\left\| {\mathfrak{S}_{\xi} \left( {{\lambda _{q + 1}}{{\tilde \Phi }_k}} \right)} \right\|}_N}} \right) \\
& +   \lambda _{q + 1}^{ - 1}\mathop {\sup }\limits_{\xi ,k,n} {\left\| {{{\tilde \eta }_{\xi ,k,n}}{{\left( {\nabla {{\tilde \Phi }_k}} \right)}^{ - 1}}{\mathfrak{D}_{\tilde \mu }}\left( {t,{\lambda _{q + 1}}{{\tilde \Phi }_k}} \right)\left( {\nabla {{\tilde \Phi }_k}} \right) \cdot \mathfrak{S}_{\xi} \left( {{\lambda _{q + 1}}{{\tilde \Phi }_k}} \right)} \right\|_0}{\left\| {{{\bar v}_{q,\Gamma }}} \right\|_{N + 1}}  \\
\lesssim & {\delta _{q + 2}}\lambda _{q + 1}^{N - \alpha }\delta _q^{\frac{3}{2}}\delta _{q + 1}^{\frac{1}{2} - \frac{{\alpha  - 1}}{{2\beta }}}{\lambda _q}\lambda _{q + 2}^{2\beta }   .
\end{align*}
Invoking \eqref{parameters-S-19}, we obtain
\begin{equation}\label{RSRI-main-PROVE-4}
{\left\| {\mathcal{R}\left( {\left( {w_{q + 1}^{\left( p \right)} + w_{q + 1}^{\left( c \right)}} \right) \cdot \nabla } \right){{\bar v}_{q,\Gamma }}} \right\|_N} \lesssim \delta _{q + 1}^{\frac{1}{2}} {\delta _{q + 2}}\lambda _{q + 1}^{N - \alpha }   .
\end{equation}
Thus, it follows from \eqref{RSRI-main-PROVE-1}-\eqref{RSRI-main-PROVE-4} that
\begin{equation}\label{RSRI-main-PROVE-5}
{\left\| {\mathcal{R}\left( {{{\bar v}_{q,\Gamma }} \cdot \nabla } \right)w_{q + 1}^{\left( s \right)}} \right\|_N} \lesssim \delta _{q + 1}^{\frac{1}{2}}  {\delta _{q + 2}}\lambda _{q + 1}^{N - \alpha }   .
\end{equation}
\par
For the remaining part $\mathcal{R} {{\bar D}_{t,\Gamma }}\left( {w_{q + 1}^{\left( p \right)} + w_{q + 1}^{\left( c \right)}} \right)$, direct computation gives
\begin{align}\label{RSRI-main-PROVE-6}
{\left\| {\mathcal{R}{{\bar D}_{t,\Gamma }}\left( {w_{q + 1}^{\left( p \right)} + w_{q + 1}^{\left( c \right)}} \right)} \right\|_N} \lesssim & {\left\| {\mathcal{R}{\partial _t}\left( {w_{q + 1}^{\left( p \right)} + w_{q + 1}^{\left( c \right)}} \right)} \right\|_N} + {\left\| {\mathcal{R}\left( {{{\bar v}_{q,\Gamma }} \cdot \nabla } \right)\left( {w_{q + 1}^{\left( p \right)} + w_{q + 1}^{\left( c \right)}} \right)} \right\|_N}  \nonumber \\
\lesssim & {\left\| {\mathcal{R}{\partial _t}\left( {w_{q + 1}^{\left( p \right)} + w_{q + 1}^{\left( c \right)}} \right)} \right\|_N} + {\left\| {\mathcal{R}\left( {\left( {w_{q + 1}^{\left( p \right)} + w_{q + 1}^{\left( c \right)}} \right) \cdot \nabla } \right){{\bar v}_{q,\Gamma }}} \right\|_N} .
\end{align}
Since
\begin{align*}
\mathcal{R}{\partial _t}\left( {w_{q + 1}^{\left( p \right)} + w_{q + 1}^{\left( c \right)}} \right) = & \lambda _{q + 1}^{ - 1}{\mu _{q + 1}}R{\nabla ^ \bot }\sum\limits_{n = 0}^\Gamma  {\sum\limits_{k \in {\mathbb{Z}_{q,n}}} {\sum\limits_{\xi  \in \Lambda } {D{g_{\xi ,k,n + 1}}\left( {{\mu _{q + 1}}t} \right){{\tilde \eta }_{\xi ,k,n}}{{\left( {\nabla {{\tilde \Phi }_k}} \right)}^{ - 1}}} } } \\
& \hspace*{6em} {\mathfrak{D}_{\tilde \mu }}\left( {t,{\lambda _{q + 1}}{{\tilde \Phi }_k}} \right)\left( {\nabla {{\tilde \Phi }_k}} \right) \cdot \mathfrak{S}_{\xi} \left( {{\lambda _{q + 1}}{{\tilde \Phi }_k}} \right) \\
& + \lambda _{q + 1}^{ - 1}R{\nabla ^ \bot }\sum\limits_{n = 0}^\Gamma  {\sum\limits_{k \in {\mathbb{Z}_{q,n}}} {\sum\limits_{\xi  \in \Lambda } {{g_{\xi ,k,n + 1}}\left( {{\mu _{q + 1}}t} \right){{\tilde \eta }_{\xi ,k,n}}{{\left( {\nabla {{\tilde \Phi }_k}} \right)}^{ - 1}}} } } \\
& \hspace*{6em} {\partial _t}{\mathfrak{D}_{\tilde \mu }}\left( {t,{\lambda _{q + 1}}{{\tilde \Phi }_k}} \right)\left( {\nabla {{\tilde \Phi }_k}} \right) \cdot \mathfrak{S}_{\xi} \left( {{\lambda _{q + 1}}{{\tilde \Phi }_k}} \right) \\
& + \lambda _{q + 1}^{ - 1}R{\nabla ^ \bot }\sum\limits_{n = 0}^\Gamma  {\sum\limits_{k \in {\mathbb{Z}_{q,n}}} {\sum\limits_{\xi  \in \Lambda } {{g_{\xi ,k,n + 1}}\left( {{\mu _{q + 1}}t} \right){\mathfrak{D}_{\tilde \mu }}\left( {t,{\lambda _{q + 1}}{{\tilde \Phi }_k}} \right)} } } \\
& \hspace*{6em}  {\partial _t}\left( {{{\left( {\nabla {{\tilde \Phi }_k}} \right)}^{ - 1}}{{\tilde \eta }_{\xi ,k,n}}\left( {\nabla {{\tilde \Phi }_k}} \right) \cdot \mathfrak{S}_{\xi} \left( {{\lambda _{q + 1}}{{\tilde \Phi }_k}} \right)} \right)  ,
\end{align*}
we have
\begin{align}\label{RSRI-main-PROVE-7}
& {\left\| {\mathcal{R}{\partial _t}\left( {w_{q + 1}^{\left( p \right)} + w_{q + 1}^{\left( c \right)}} \right)} \right\|_N} \nonumber \\
\lesssim & \lambda _{q + 1}^{ - 1}{\mu _{q + 1}}\mathop {\sup }\limits_{\xi ,k,n} \left( {{\delta _q}{{\left\| {{{\tilde \eta }_{\xi ,k,n}}} \right\|}_N} + \delta _q^{\frac{1}{2}}{{\left\| {{{\tilde \eta }_{\xi ,k,n}}} \right\|}_0}{{\left\| {\nabla {{\tilde \Phi }_k} + {{\left( {\nabla {{\tilde \Phi }_k}} \right)}^{ - 1}}} \right\|}_N}} \right) \nonumber \\
& + \delta _q^{\frac{1}{2}} \lambda _{q + 1}^{ - 1}{\mu _{q + 1}}\mathop {\sup }\limits_{\xi ,k,n} {\left\| {{{\tilde \eta }_{\xi ,k,n}}} \right\|_0}\left( {{{\left\| {{\mathfrak{D}_{\tilde \mu }}\left( {t,{\lambda _{q + 1}}{{\tilde \Phi }_k}} \right)} \right\|}_N} + {{\left\| {\mathfrak{S}_{\xi} \left( {{\lambda _{q + 1}}{{\tilde \Phi }_k}} \right)} \right\|}_N}} \right)  \nonumber \\
& + \delta _q^{\frac{1}{2}} \lambda _{q + 1}^{ - 1}\mathop {\sup }\limits_{\xi ,k,n} {\left\| {{\partial _t}{\mathfrak{D}_{\tilde \mu }}\left( {t,{\lambda _{q + 1}}{{\tilde \Phi }_k}} \right)} \right\|_0}\left( {{{\left\| {{{\tilde \eta }_{\xi ,k,n}}} \right\|}_N} + {{\left\| {{{\tilde \eta }_{\xi ,k,n}}} \right\|}_0}{{\left\| {\mathfrak{S}_{\xi} \left( {{\lambda _{q + 1}}{{\tilde \Phi }_k}} \right)} \right\|}_N}} \right)  \nonumber \\
& + \delta _q^{\frac{1}{2}} \lambda _{q + 1}^{ - 1}\mathop {\sup }\limits_{\xi ,k,n} {\left\| {{{\tilde \eta }_{\xi ,k,n}}} \right\|_0}{\left\| {\nabla {{\tilde \Phi }_k} + {{\left( {\nabla {{\tilde \Phi }_k}} \right)}^{ - 1}}} \right\|_N}{\left\| {{\partial _t}{\mathfrak{D}_{\tilde \mu }}\left( {t,{\lambda _{q + 1}}{{\tilde \Phi }_k}} \right)} \right\|_0}  \nonumber \\
& + \delta _q^{\frac{1}{2}} \lambda _{q + 1}^{ - 1}\mathop {\sup }\limits_{\xi ,k,n} {\left\| {{{\tilde \eta }_{\xi ,k,n}}} \right\|_0}{\left\| {{\partial _t}{\mathfrak{D}_{\tilde \mu }}\left( {t,{\lambda _{q + 1}}{{\tilde \Phi }_k}} \right)} \right\|_N}  \nonumber \\
& + \lambda _{q + 1}^{ - 1}\mathop {\sup }\limits_{\xi ,k,n} {\left\| {{\mathfrak{D}_{\tilde \mu }}\left( {t,{\lambda _{q + 1}}{{\tilde \Phi }_k}} \right)} \right\|_N}{\left\| {{\partial _t}\left( {{{\left( {\nabla {{\tilde \Phi }_k}} \right)}^{ - 1}}{{\tilde \eta }_{\xi ,k,n}}\left( {\nabla {{\tilde \Phi }_k}} \right) \cdot \mathfrak{S}_{\xi} \left( {{\lambda _{q + 1}}{{\tilde \Phi }_k}} \right)} \right)} \right\|_0} \nonumber \\
& + \delta _q^{\frac{1}{2}}  \lambda _{q + 1}^{ - 1}\mathop {\sup }\limits_{\xi ,k,n} {\left\| {{\partial _t}\left( {{{\left( {\nabla {{\tilde \Phi }_k}} \right)}^{ - 1}}{{\tilde \eta }_{\xi ,k,n}}\left( {\nabla {{\tilde \Phi }_k}} \right) \cdot \mathfrak{S}_{\xi} \left( {{\lambda _{q + 1}}{{\tilde \Phi }_k}} \right)} \right)} \right\|_N} .
\end{align}
Notice that
\begin{align}\label{RSRI-main-PROVE-8768}
& {\left\| {{\partial _t}\left( {{{\left( {\nabla {{\tilde \Phi }_k}} \right)}^{ - 1}}{{\tilde \eta }_{\xi ,k,n}}\left( {\nabla {{\tilde \Phi }_k}} \right) \cdot \mathfrak{S}_{\xi} \left( {{\lambda _{q + 1}}{{\tilde \Phi }_k}} \right)} \right)} \right\|_N} \nonumber \\
\lesssim & \delta _q^{\frac{3}{2}} {\left\| {{{\bar D}_{t,\Gamma }}{{\tilde \eta }_{\xi ,k,n}}} \right\|_N} + {\left\| {{{\tilde \eta }_{\xi ,k,n}}} \right\|_0}{\left\| {{\partial _t}\left( {\nabla {{\tilde \Phi }_k} + \mathfrak{S}_{\xi} \left( {{\lambda _{q + 1}}{{\tilde \Phi }_k}} \right)} \right)} \right\|_N}  \nonumber   \\
& + \delta _q^{\frac{1}{2}}  {\left\| {{{\tilde \eta }_{\xi ,k,n}}} \right\|_N}{\left\| {{\partial _t}\left( {\nabla {{\tilde \Phi }_k} + \mathfrak{S}_{\xi} \left( {{\lambda _{q + 1}}{{\tilde \Phi }_k}} \right)} \right)} \right\|_0} \nonumber \\
& + {\left\| {{{\tilde \eta }_{\xi ,k,n}}} \right\|_0}{\left\| {{{\bar D}_{t,\Gamma }}\left( {\nabla {{\tilde \Phi }_k}} \right)} \right\|_0}{\left\| {{{\left( {\nabla {{\tilde \Phi }_k}} \right)}^{ - 1}} + \mathfrak{S}_{\xi} \left( {{\lambda _{q + 1}}{{\tilde \Phi }_k}} \right)} \right\|_N}  \nonumber  \\
& + {\left\| {{{\bar D}_{t,\Gamma }}{{\tilde \eta }_{\xi ,k,n}}} \right\|_0}{\left\| {\nabla {{\tilde \Phi }_k} + {{\left( {\nabla {{\tilde \Phi }_k}} \right)}^{ - 1}} + \mathfrak{S}_{\xi} \left( {{\lambda _{q + 1}}{{\tilde \Phi }_k}} \right)} \right\|_N}  \nonumber  \\
& + {\left\| {{{\tilde \eta }_{\xi ,k,n}}} \right\|_0}{\left\| { {\partial _t}{\mathfrak{S}_\xi }  \left( {{\lambda _{q + 1}}{{\tilde \Phi }_k}} \right)} \right\|_0}{\left\| {\nabla {{\tilde \Phi }_k} + {{\left( {\nabla {{\tilde \Phi }_k}} \right)}^{ - 1}}} \right\|_N}  \nonumber  \\
\lesssim & \delta _q^{\frac{1}{2}}\delta _{q + 1}^{\frac{1}{2}}{\lambda _q}\lambda _{q + 1}^{N + \alpha }   .
\end{align}
Then, combining \eqref{WPD-NHS-HUD-BB-MAIN-1}, \eqref{RSRI-main-PROVE-8}, \eqref{RSRI-main-PROVE-7}, \eqref{RSRI-main-PROVE-8768}, \eqref{parameters-S-21-A-2}, Lemma \ref{WPD-NHS-AA} and Lemma \ref{WPD-NHS-BB}, we obtain
\begin{align}
{\left\| {\mathcal{R}{\partial _t}\left( {w_{q + 1}^{\left( p \right)} + w_{q + 1}^{\left( c \right)}} \right)} \right\|_N} \lesssim & {\mu _{q + 1}} \delta _q^{\frac{3}{2}}  \delta _{q + 1}^{\frac{1}{2}}\lambda _{q + 1}^{N - 1} + \tilde \mu  \delta _q^{\frac{3}{2}} \delta _{q + 1}^{\frac{1}{2}}\lambda _{q + 1}^{N - 1} \nonumber \\
\lesssim & \delta _{q + 1}^{\frac{1}{2}} {\delta _{q + 2}}\lambda _{q + 1}^{N - \alpha }  . \label{KNHYGF-S-D-A-1}
\end{align}
Employing \eqref{RSRI-main-PROVE-4}, \eqref{RSRI-main-PROVE-6} and \eqref{KNHYGF-S-D-A-1}, we infer that
\begin{equation}\label{RSRI-main-PROVE-10}
{\left\| {\mathcal{R}{{\bar D}_{t,\Gamma }}\left( {w_{q + 1}^{\left( p \right)} + w_{q + 1}^{\left( c \right)}} \right)} \right\|_N} \lesssim  \delta _{q + 1}^{\frac{1}{2}} {\delta _{q + 2}}\lambda _{q + 1}^{N - \alpha }   .
\end{equation}
Therefore, it follows from \eqref{LIJIFNI-DKFJ-L-1}, \eqref{RSRI-main-PROVE-5} and \eqref{RSRI-main-PROVE-10} that
\begin{equation}\label{RSRI-P-SAMMLGOAL-L-1}
{\left\| { \mathring{R}_{q + 1}^{\left( l \right)} } \right\|_N} \lesssim {\delta _{q + 2}}\lambda _{q + 1}^{N - \alpha }   .
\end{equation}
\par
Recall that
\begin{equation}\label{RSRI-main-PROVE-11}
{D_{t,q + 1}}\mathring{R}_{q + 1}^{\left( l \right)} = {\partial _t}\mathring{R}_{q + 1}^{\left( l \right)} + \left( {{v_{q + 1}} \cdot \nabla } \right)\mathring{R}_{q + 1}^{\left( l \right)} .
\end{equation}
Using \eqref{RSRI-main-PROVE-5}, \eqref{RSRI-main-PROVE-10}, \eqref{RSRI-P-SAMMLGOAL-L-1}, Proposition \ref{RSRI-ujdh-main} and Lemma \ref{STN-HCS-A}, we obtain
\begin{align}\label{RSRI-main-PROVE-12}
{\left\| {\left( {{v_{q + 1}} \cdot \nabla } \right)\mathring{R}_{q + 1}^{\left( l \right)}} \right\|_N} \lesssim & {\left\| {{v_{q + 1}}} \right\|_N}{\left\| {\mathring{R}_{q + 1}^{\left( l \right)}} \right\|_1} + {\left\| {\mathcal{R}\left( {{{\bar v}_{q,\Gamma }} \cdot \nabla } \right)w_{q + 1}^{\left( s \right)}} \right\|_{N + 1}} + {\left\| {\mathcal{R}{{\bar D}_{t,\Gamma }}\left( {w_{q + 1}^{\left( p \right)} + w_{q + 1}^{\left( c \right)}} \right)} \right\|_{N+1}} \nonumber \\
\lesssim & \delta _{q + 1}^{\frac{1}{2}}{\delta _{q + 2}}\lambda _{q + 1}^{N + 1 - \alpha }   .
\end{align}
Furthermore,
\begin{align}\label{RSRI-main-PROVE-13}
{\left\| {{\partial _t}\mathring{R}_{q + 1}^{\left( l \right)}} \right\|_N} \lesssim & {\left\| {\mathcal{R}{\partial _t}{{\bar D}_{t,\Gamma }}\left( {w_{q + 1}^{\left( p \right)} + w_{q + 1}^{\left( c \right)}} \right)} \right\|_N} + {\left\| {\mathcal{R}{\partial _t}\left( {\left( {{{\bar v}_{q,\Gamma }} \cdot \nabla } \right)w_{q + 1}^{\left( s \right)}} \right)} \right\|_N} \nonumber \\
\lesssim & {\left\| {\mathcal{R}{\partial _t}\left( {\left( {w_{q + 1}^{\left( p \right)} + w_{q + 1}^{\left( c \right)}} \right) \cdot \nabla } \right){{\bar v}_{q,\Gamma }}} \right\|_N} + {\left\| {\mathcal{R}{\partial _t}\left( {w_{q + 1}^{\left( i \right)} \cdot \nabla } \right){{\bar v}_{q,\Gamma }}} \right\|_N} \nonumber \\
& + {\left\| {\mathcal{R}{\partial _{tt}}\left( {w_{q + 1}^{\left( p \right)} + w_{q + 1}^{\left( c \right)}} \right)} \right\|_N} .
\end{align}
\par
For the first term on the right-hand side of \eqref{RSRI-main-PROVE-13}, using the fact that
\begin{align*}
{\partial _t}\mathcal{R}\left( {\left( {w_{q + 1}^{\left( p \right)} + w_{q + 1}^{\left( c \right)}} \right) \cdot \nabla } \right){{\bar v}_{q,\Gamma }} = & \lambda _{q + 1}^{ - 1}{\mu _{q + 1}}\mathcal{R}{\nabla ^ \bot }\sum\limits_{n = 0}^\Gamma  {\sum\limits_{k \in {\mathbb{Z}_{q,n}}} {\sum\limits_{\xi  \in \Lambda } {D{g_{\xi ,k,n + 1}}\left(  \cdot  \right){{\tilde \eta }_{\xi ,k,n}}{{\left( {\nabla {{\tilde \Phi }_k}} \right)}^{ - 1}}} } }  \\
& \hspace*{6em} {\mathfrak{D}_{\tilde \mu }}\left( {t,{\lambda _{q + 1}}{{\tilde \Phi }_k}} \right)\left( {\nabla {{\tilde \Phi }_k}} \right) \cdot \mathfrak{S}_{\xi} \left( {{\lambda _{q + 1}}{{\tilde \Phi }_k}} \right) \cdot \nabla {{\bar v}_{q,\Gamma }} \\
& + \lambda _{q + 1}^{ - 1}\mathcal{R}{\nabla ^ \bot }\sum\limits_{n = 0}^\Gamma  {\sum\limits_{k \in {\mathbb{Z}_{q,n}}} {\sum\limits_{\xi  \in \Lambda } {{g_{\xi ,k,n + 1}}\left( {{\mu _{q + 1}}t} \right){{\tilde \eta }_{\xi ,k,n}}{{\left( {\nabla {{\tilde \Phi }_k}} \right)}^{ - 1}}} } }  \\
& \hspace*{6em} {\partial _t}{\mathfrak{D}_{\tilde \mu }}\left( {t,{\lambda _{q + 1}}{{\tilde \Phi }_k}} \right)\left( {\nabla {{\tilde \Phi }_k}} \right) \cdot \mathfrak{S}_{\xi} \left( {{\lambda _{q + 1}}{{\tilde \Phi }_k}} \right) \cdot \nabla {{\bar v}_{q,\Gamma }} \\
& + \lambda _{q + 1}^{ - 1}\mathcal{R}{\nabla ^ \bot }\sum\limits_{n = 0}^\Gamma  {\sum\limits_{k \in {\mathbb{Z}_{q,n}}} {\sum\limits_{\xi  \in \Lambda } {{g_{\xi ,k,n + 1}}\left( {{\mu _{q + 1}}t} \right){{\tilde \eta }_{\xi ,k,n}}{{\left( {\nabla {{\tilde \Phi }_k}} \right)}^{ - 1}}} } }  \\
& \hspace*{6em} {\mathfrak{D}_{\tilde \mu }}\left( {t,{\lambda _{q + 1}}{{\tilde \Phi }_k}} \right)\left( {\nabla {{\tilde \Phi }_k}} \right) \cdot \mathfrak{S}_{\xi} \left( {{\lambda _{q + 1}}{{\tilde \Phi }_k}} \right) \cdot {\partial _t}\nabla {{\bar v}_{q,\Gamma }}  \\
& + \lambda _{q + 1}^{ - 1}\mathcal{R}{\nabla ^ \bot }\sum\limits_{n = 0}^\Gamma  {\sum\limits_{k \in {\mathbb{Z}_{q,n}}} {\sum\limits_{\xi  \in \Lambda } {{g_{\xi ,k,n + 1}}\left( {{\mu _{q + 1}}t} \right){\mathfrak{D}_{\tilde \mu }}\left( {t,{\lambda _{q + 1}}{{\tilde \Phi }_k}} \right)\nabla {{\bar v}_{q,\Gamma }}} } }  \\
& \hspace*{6em}  \cdot {\partial _t}\left( {{{\left( {\nabla {{\tilde \Phi }_k}} \right)}^{ - 1}}{{\tilde \eta }_{\xi ,k,n}}\left( {\nabla {{\tilde \Phi }_k}} \right) \cdot \mathfrak{S}_{\xi} \left( {{\lambda _{q + 1}}{{\tilde \Phi }_k}} \right)} \right) ,
\end{align*}
we derive that
\begin{align}\label{RSRI-main-PROVE-14}
& {\left\| {{\partial _t}\mathcal{R}\left( {\left( {w_{q + 1}^{\left( p \right)} + w_{q + 1}^{\left( c \right)}} \right) \cdot \nabla } \right){{\bar v}_{q,\Gamma }}} \right\|_N}  \nonumber  \\
\lesssim & \lambda _{q + 1}^{ - 1}{\mu _{q + 1}}\mathop {\sup }\limits_{\xi ,k,n} {\left\| {{{\tilde \eta }_{\xi ,k,n}}} \right\|_0}{\left\| {{{\bar v}_{q,\Gamma }}} \right\|_1}\left( {{{\left\| {\nabla {{\tilde \Phi }_k} + {{\left( {\nabla {{\tilde \Phi }_k}} \right)}^{ - 1}} + {\mathfrak{D}_{\tilde \mu }}\left( {t,{\lambda _{q + 1}}{{\tilde \Phi }_k}} \right) + \mathfrak{S}_{\xi} \left( {{\lambda _{q + 1}}{{\tilde \Phi }_k}} \right)} \right\|}_N}} \right)  \nonumber  \\
& + \lambda _{q + 1}^{ - 1}\mathop {\sup }\limits_{\xi ,k,n} \left( {{\mu _{q + 1}}\left( {{{\left\| {{{\tilde \eta }_{\xi ,k,n}}} \right\|}_N}{{\left\| {{{\bar v}_{q,\Gamma }}} \right\|}_1} + {{\left\| {{{\tilde \eta }_{\xi ,k,n}}} \right\|}_0}{{\left\| {{{\bar v}_{q,\Gamma }}} \right\|}_{N + 1}}} \right) + {{\left\| {{{\tilde \eta }_{\xi ,k,n}}} \right\|}_0}{{\left\| {{\partial _t}{{\bar v}_{q,\Gamma }}} \right\|}_{N + 1}}} \right)  \nonumber  \\
& + \lambda _{q + 1}^{ - 1}\mathop {\sup }\limits_{\xi ,k,n} {\left\| {{{\tilde \eta }_{\xi ,k,n}}} \right\|_0}{\left\| {{\partial _t}{\mathfrak{D}_{\tilde \mu }}\left( {t,{\lambda _{q + 1}}{{\tilde \Phi }_k}} \right)} \right\|_0}{\left\| {{{\bar v}_{q,\Gamma }}} \right\|_1}\left( {{{\left\| {\nabla {{\tilde \Phi }_k} + {{\left( {\nabla {{\tilde \Phi }_k}} \right)}^{ - 1}}} \right\|}_N} + {{\left\| {\mathfrak{S}_{\xi} \left( {{\lambda _{q + 1}}{{\tilde \Phi }_k}} \right)} \right\|}_N}} \right)  \nonumber  \\
& + \lambda _{q + 1}^{ - 1}\mathop {\sup }\limits_{\xi ,k,n} {\left\| {{{\tilde \eta }_{\xi ,k,n}}} \right\|_0}\left( {{{\left\| {{\partial _t}{\mathfrak{D}_{\tilde \mu }}\left( {t,{\lambda _{q + 1}}{{\tilde \Phi }_k}} \right)} \right\|}_N}{{\left\| {{{\bar v}_{q,\Gamma }}} \right\|}_1} + {{\left\| {{\partial _t}{\mathfrak{D}_{\tilde \mu }}\left( {t,{\lambda _{q + 1}}{{\tilde \Phi }_k}} \right)} \right\|}_0}{{\left\| {{{\bar v}_{q,\Gamma }}} \right\|}_{N + 1}}} \right)  \nonumber  \\
& + \lambda _{q + 1}^{ - 1}\mathop {\sup }\limits_{\xi ,k,n} {\left\| {{{\tilde \eta }_{\xi ,k,n}}} \right\|_0}{\left\| {{\partial _t}{{\bar v}_{q,\Gamma }}} \right\|_1}\left( {{{\left\| {\nabla {{\tilde \Phi }_k} + {{\left( {\nabla {{\tilde \Phi }_k}} \right)}^{ - 1}} + {\mathfrak{D}_{\tilde \mu }}\left( {t,{\lambda _{q + 1}}{{\tilde \Phi }_k}} \right) + \mathfrak{S}_{\xi} \left( {{\lambda _{q + 1}}{{\tilde \Phi }_k}} \right)} \right\|}_N}} \right)  \nonumber  \\
& + \lambda _{q + 1}^{ - 1}\mathop {\sup }\limits_{\xi ,k,n} {\left\| {{{\tilde \eta }_{\xi ,k,n}}} \right\|_N}\left( {{{\left\| {{\partial _t}{\mathfrak{D}_{\tilde \mu }}\left( {t,{\lambda _{q + 1}}{{\tilde \Phi }_k}} \right)} \right\|}_0}{{\left\| {{{\bar v}_{q,\Gamma }}} \right\|}_1} + {{\left\| {{\partial _t}{{\bar v}_{q,\Gamma }}} \right\|}_1}} \right)   \nonumber  \\
& + \lambda _{q + 1}^{ - 1}\mathop {\sup }\limits_{\xi ,k,n} {\left\| {{\mathfrak{D}_{\tilde \mu }}\left( {t,{\lambda _{q + 1}}{{\tilde \Phi }_k}} \right)\nabla {{\bar v}_{q,\Gamma }}} \right\|_N}{\left\| {{\partial _t}\left( {{{\left( {\nabla {{\tilde \Phi }_k}} \right)}^{ - 1}}{{\tilde \eta }_{\xi ,k,n}}\left( {\nabla {{\tilde \Phi }_k}} \right) \cdot \mathfrak{S}_{\xi} \left( {{\lambda _{q + 1}}{{\tilde \Phi }_k}} \right)} \right)} \right\|_0}  \nonumber  \\
& + \lambda _{q + 1}^{ - 1}\mathop {\sup }\limits_{\xi ,k,n} {\left\| {{{\bar v}_{q,\Gamma }}} \right\|_1}{\left\| {{\partial _t}\left( {{{\left( {\nabla {{\tilde \Phi }_k}} \right)}^{ - 1}}{{\tilde \eta }_{\xi ,k,n}}\left( {\nabla {{\tilde \Phi }_k}} \right) \cdot \mathfrak{S}_{\xi} \left( {{\lambda _{q + 1}}{{\tilde \Phi }_k}} \right)} \right)} \right\|_N} .
\end{align}
It follows from \eqref{WPD-NHS-MAIN-AA-P-3}, \eqref{RSRI-main-PROVE-15}, \eqref{RSRI-main-PROVE-8}, \eqref{RSRI-main-PROVE-8768}, \eqref{RSRI-main-PROVE-14}, Lemma \ref{WPD-NHS-AA}, Lemma \ref{WPD-NHS-BB}, Lemma \ref{WPD-NHS-HUD-BB} and Lemma \ref{WPD-NHS-HUD-BB-2} that
\begin{equation}\label{RSRI-main-PROVE-17}
{\left\| {{\partial _t}\mathcal{R}\left( {\left( {w_{q + 1}^{\left( p \right)} + w_{q + 1}^{\left( c \right)}} \right) \cdot \nabla } \right){{\bar v}_{q,\Gamma }}} \right\|_N} \lesssim \delta _{q + 1}^{\frac{1}{2}}{\delta _{q + 2}}\lambda _{q + 1}^{N + 1 - \alpha } .  
\end{equation}
\par
For the second term on the right-hand side of \eqref{RSRI-main-PROVE-13}, we deduce from \eqref{WPD-NHS-MAIN-AA-P-3}, \eqref{RSRI-main-PROVE-15} and Proposition \ref{WPD-NS-SD-P-main} that
\begin{align*}
{\left\| {\mathcal{R}{\partial _t}\left( {w_{q + 1}^{\left( i \right)} \cdot \nabla } \right){{\bar v}_{q,\Gamma }}} \right\|_N} \lesssim & {\left\| {{\partial _t}w_{q + 1}^{\left( i \right)}} \right\|_N} + {\left\| {{\partial _t}w_{q + 1}^{\left( i \right)}} \right\|_0}{\left\| {{{\bar v}_{q,\Gamma }}} \right\|_N}  \\
& + {\left\| {w_{q + 1}^{\left( i \right)}} \right\|_N}{\left\| {{\partial _t}{{\bar v}_{q,\Gamma }}} \right\|_0} + {\left\| {w_{q + 1}^{\left( i \right)}} \right\|_0}{\left\| {{\partial _t}{{\bar v}_{q,\Gamma }}} \right\|_N}  \\
\lesssim & {{\tilde \mu }^{ - 1}}{\mu _{q + 1}}{\delta _{q + 1}}\lambda _q^N + {{\tilde \mu }^{ - 1}}\delta _q^{\frac{1}{2}}{\delta _{q + 1}}\lambda _q^{N + 1}    ,
\end{align*}
which implies that 
\begin{equation}\label{RSRI-main-PROVE-18}
{\left\| {\mathcal{R}{\partial _t}\left( {w_{q + 1}^{\left( i \right)} \cdot \nabla } \right){{\bar v}_{q,\Gamma }}} \right\|_N} \lesssim \delta _{q + 1}^{\frac{1}{2}}{\delta _{q + 2}}\lambda _{q + 1}^{N + 1 - \alpha }   .
\end{equation}
\par
For the third term on the right-hand side of \eqref{RSRI-main-PROVE-13}, combining \eqref{WPD-NHS-MAIN-AA-4}, \eqref{WPD-NHS-MAIN-AA-sjdiJJ-4}, \eqref{WPD-NHS-MAIN-1}, \eqref{WPD-NHS-MAIN-3} and \eqref{WPD-NHS-HUD-BB-MAIN-XINhs-6}, we have
\begin{align}\label{RSRI-main-PROVE-19}
& {\left\| {{\partial _{tt}}\left( {{{\left( {\nabla {{\tilde \Phi }_k}} \right)}^{ - 1}}{{\tilde \eta }_{\xi ,k,n}}\left( {\nabla {{\tilde \Phi }_k}} \right) \cdot \mathfrak{S}_{\xi} \left( {{\lambda _{q + 1}}{{\tilde \Phi }_k}} \right)} \right)} \right\|_N} \nonumber  \\
\lesssim & {\delta _q} {\left\| {{{\tilde \eta }_{\xi ,k,n}}} \right\|_0}{\left\| {{\partial _{tt}}\left( {{{\left( {\nabla {{\tilde \Phi }_k}} \right)}^{ - 1}} + \mathfrak{S}_{\xi} \left( {{\lambda _{q + 1}}{{\tilde \Phi }_k}} \right) + \nabla {{\tilde \Phi }_k}} \right)} \right\|_N} + \delta _q^{\frac{3}{2}} {\left\| {{\partial _{tt}}{{\tilde \eta }_{\xi ,k,n}}} \right\|_N}  \nonumber  \\
\lesssim & {\delta _q ^2}\lambda _q^{N + 1}\lambda _{q + 1}^2    .
\end{align}
Notice that
\begin{align*}
\mathcal{R}{\partial _{tt}}\left( {w_{q + 1}^{\left( p \right)} + w_{q + 1}^{\left( c \right)}} \right) = & \lambda _{q + 1}^{ - 1}\mu _{q + 1}^2\mathcal{R}{\nabla ^ \bot }\sum\limits_{n = 0}^\Gamma  {\sum\limits_{k \in {\mathbb{Z}_{q,n}}} {\sum\limits_{\xi  \in \Lambda } {{D^2}{g_{\xi ,k,n + 1}}\left(  \cdot  \right){{\tilde \eta }_{\xi ,k,n}}{{\left( {\nabla {{\tilde \Phi }_k}} \right)}^{ - 1}}} } }  \\
& \hspace*{6em}  {\mathfrak{D}_{\tilde \mu }}\left( {t,{\lambda _{q + 1}}{{\tilde \Phi }_k}} \right)\left( {\nabla {{\tilde \Phi }_k}} \right) \cdot \mathfrak{S}_{\xi} \left( {{\lambda _{q + 1}}{{\tilde \Phi }_k}} \right)  \\
& + \lambda _{q + 1}^{ - 1}\mathcal{R}{\nabla ^ \bot }\sum\limits_{n = 0}^\Gamma  {\sum\limits_{k \in {\mathbb{Z}_{q,n}}} {\sum\limits_{\xi  \in \Lambda } {{g_{\xi ,k,n + 1}}\left( {{\mu _{q + 1}}t} \right){{\tilde \eta }_{\xi ,k,n}}{{\left( {\nabla {{\tilde \Phi }_k}} \right)}^{ - 1}}} } } \\
&  \hspace*{6em} {\partial _{tt}}{\mathfrak{D}_{\tilde \mu }}\left( {t,{\lambda _{q + 1}}{{\tilde \Phi }_k}} \right)\left( {\nabla {{\tilde \Phi }_k}} \right) \cdot \mathfrak{S}_{\xi} \left( {{\lambda _{q + 1}}{{\tilde \Phi }_k}} \right)  \\
& + \lambda _{q + 1}^{ - 1}\mathcal{R}{\nabla ^ \bot }\sum\limits_{n = 0}^\Gamma  {\sum\limits_{k \in {\mathbb{Z}_{q,n}}} {\sum\limits_{\xi  \in \Lambda } {{g_{\xi ,k,n + 1}}\left( {{\mu _{q + 1}}t} \right){\mathfrak{D}_{\tilde \mu }}\left( {t,{\lambda _{q + 1}}{{\tilde \Phi }_k}} \right)} } }  \\
&    \hspace*{6em} {\partial _{tt}}\left( {{{\left( {\nabla {{\tilde \Phi }_k}} \right)}^{ - 1}}{{\tilde \eta }_{\xi ,k,n}}\left( {\nabla {{\tilde \Phi }_k}} \right) \cdot \mathfrak{S}_{\xi} \left( {{\lambda _{q + 1}}{{\tilde \Phi }_k}} \right)} \right) \\
& + 2 \lambda _{q + 1}^{ - 1}{\mu _{q + 1}}\mathcal{R}{\nabla ^ \bot }\sum\limits_{n = 0}^\Gamma  {\sum\limits_{k \in {\mathbb{Z}_{q,n}}} {\sum\limits_{\xi  \in \Lambda } {D{g_{\xi ,k,n + 1}}\left(  \cdot  \right){{\tilde \eta }_{\xi ,k,n}}{{\left( {\nabla {{\tilde \Phi }_k}} \right)}^{ - 1}}} } }  \\
& \hspace*{6em} {\partial _t}{\mathfrak{D}_{\tilde \mu }}\left( {t,{\lambda _{q + 1}}{{\tilde \Phi }_k}} \right)\left( {\nabla {{\tilde \Phi }_k}} \right) \cdot \mathfrak{S}_{\xi} \left( {{\lambda _{q + 1}}{{\tilde \Phi }_k}} \right) \\
& + 2 \lambda _{q + 1}^{ - 1}{\mu _{q + 1}}\mathcal{R}{\nabla ^ \bot }\sum\limits_{n = 0}^\Gamma  {\sum\limits_{k \in {\mathbb{Z}_{q,n}}} {\sum\limits_{\xi  \in \Lambda } {D{g_{\xi ,k,n + 1}}\left(  \cdot  \right){\mathfrak{D}_{\tilde \mu }}\left( {t,{\lambda _{q + 1}}{{\tilde \Phi }_k}} \right)} } } \\
& \hspace*{6em} {\partial _t}\left( {{{\left( {\nabla {{\tilde \Phi }_k}} \right)}^{ - 1}}{{\tilde \eta }_{\xi ,k,n}}\left( {\nabla {{\tilde \Phi }_k}} \right) \cdot \mathfrak{S}_{\xi} \left( {{\lambda _{q + 1}}{{\tilde \Phi }_k}} \right)} \right)  \\
&  + 2 \lambda _{q + 1}^{ - 1}\mathcal{R}{\nabla ^ \bot }\sum\limits_{n = 0}^\Gamma  {\sum\limits_{k \in {\mathbb{Z}_{q,n}}} {\sum\limits_{\xi  \in \Lambda } {{g_{\xi ,k,n + 1}}\left( {{\mu _{q + 1}}t} \right){\partial _t}{\mathfrak{D}_{\tilde \mu }}\left( {t,{\lambda _{q + 1}}{{\tilde \Phi }_k}} \right)} } } \\
&  \hspace*{6em} {\partial _t}\left( {{{\left( {\nabla {{\tilde \Phi }_k}} \right)}^{ - 1}}{{\tilde \eta }_{\xi ,k,n}}\left( {\nabla {{\tilde \Phi }_k}} \right) \cdot \mathfrak{S}_{\xi} \left( {{\lambda _{q + 1}}{{\tilde \Phi }_k}} \right)} \right)  .
\end{align*}
Employing \eqref{RSRI-main-PROVE-17}, \eqref{RSRI-main-PROVE-19}, Lemma \ref{WPD-NHS-AA}, Lemma \ref{WPD-NHS-BB}, Lemma \ref{WPD-NHS-HUD-BB} and Lemma \ref{WPD-NHS-HUD-BB-2}, we deduce that
\begin{equation}\label{RSRI-main-PROVE-20}
{\left\| {\mathcal{R}{\partial _{tt}}\left( {w_{q + 1}^{\left( p \right)} + w_{q + 1}^{\left( c \right)}} \right)} \right\|_N} \lesssim \delta _{q + 1}^{\frac{1}{2}}{\delta _{q + 2}}\lambda _{q + 1}^{N + 1 - \alpha }   .
\end{equation}
Then, it follows from \eqref{RSRI-main-PROVE-11}, \eqref{RSRI-main-PROVE-13}, \eqref{RSRI-main-PROVE-17}, \eqref{RSRI-main-PROVE-18} and \eqref{RSRI-main-PROVE-20} that
\begin{equation}\label{RSRI-P-SAMMLGOAL-L-2}
{\left\| { {{ D_{t,{q + 1} }}} \mathring{R}_{q + 1}^{\left( l \right)} } \right\|_N}  \lesssim \delta _{q + 1}^{\frac{1}{2}}{\delta _{q + 2}}\lambda _{q + 1}^{N + 1 - \alpha }    .
\end{equation}
\par
\vspace{1em}
{\textbf{Step 2. Estimates of the intermittent oscillation error $\mathring{R}_{q + 1}^{\left( i \right)}$}}. 
\par
\vspace{1em}
Notice that
\begin{equation}\label{RSRI-P-PP-O-1}
{\left\| {\mathring{R}_{q + 1}^{\left( i \right)}} \right\|_N} \lesssim {\left\| {\mathcal{R}\mathrm{div}\,   {S_{q,\Gamma }}} \right\|_N} + {\left\| {{\mathcal{R}\mathrm{div}\left( {w_{q + 1}^{\left( s \right)} \otimes w_{q + 1}^{\left( s \right)}} \right)}} \right\|_N} + {\left\| {{\mathcal R}{\partial _t}w_{q + 1}^{\left( i \right)}} \right\|_N} ,
\end{equation}
and
\begin{equation}\label{RSRI-P-PP-O-2}
{\left\| {{D_{t,q + 1}}\mathring{R}_{q + 1}^{\left( i \right)}} \right\|_N} \lesssim {\left\| {{D_{t,q + 1}}{S_{q,\Gamma }}} \right\|_N} + {\left\| {{D_{t,q + 1}}\mathcal{R}\mathrm{div} \left( {w_{q + 1}^{\left( s \right)} \otimes w_{q + 1}^{\left( s \right)}} \right)} \right\|_N} + {\left\| {{\mathcal R}{D_{t,q + 1}}{\partial _t}w_{q + 1}^{\left( i \right)}} \right\|_N} .
\end{equation}
Employing \eqref{SHDUFY-A-1} and Calder\'{o}n-Zygmund inequality (see \cite[Proposition $C.2$]{Zbl1556.35231}), there exists a parameter $\tilde \alpha \in \left( {0,\beta } \right)$ such that
\begin{equation}\label{RSRI-P-PP-O-4}
{\left\| {\mathcal{R}\mathrm{div}\,   {S_{q,\Gamma }}} \right\|_0} \lesssim \mathop {\sup }\limits_{\xi ,k,n} {\left\| {\mathcal{R}\mathrm{div}\,   {A_{\xi ,k,{n+1}}}} \right\|_0} \lesssim {\left\| {{A_{\xi ,k,{n+1}}}} \right\|_{ \tilde \alpha }} .
\end{equation}
It follows from \eqref{Newtonsteps-A-B-19}, \eqref{parameters-S-20} and the interpolation inequality that
\begin{align*}
{\left\| {{A_{\xi ,k,{n+1}}}} \right\|_{\tilde \alpha }} \leqslant & \left\| {{A_{\xi ,k,{n+1}}}} \right\|_0^{1 - \tilde \alpha }\left\| {{A_{\xi ,k,{n+1}}}} \right\|_1^{\tilde \alpha } \\
\lesssim & {\delta _q} {\delta _{q + 1}} \lambda _q^{\tilde \alpha } {\left( {\frac{{{\lambda _q}}}{{\lambda _{q + 1}}}} \right)^{\left( {1 - \tilde \alpha } \right)\left( {\frac{1}{3} - \beta } \right)}} \\
\lesssim & {\delta _q} {\delta _{q + 1}} \lambda _{q + 1}^\alpha {\left( {\frac{{{\lambda _q}}}{{\lambda _{q + 1}}}} \right)^{\frac{1}{3} - \beta }} , \ \ \forall \, \alpha  \in \left[ {\frac{{3\tilde \alpha }}{{2b + 1}},\beta } \right)  .
\end{align*}
Thus, in view of \eqref{RSRI-P-PP-O-4} and \eqref{parameters-S-21}, we see that
\begin{equation}\label{RSRI-P-PP-O-4-5t}
{\left\| {\mathcal{R}\mathrm{div}\,   {S_{q,\Gamma }}} \right\|_0} \lesssim {\delta _{q + 2}}\lambda _{q + 1}^{ - \alpha }   .
\end{equation}
Similarly, for $\tilde \alpha \in \left( {0,\beta } \right)$, we have
\begin{align}\label{RSRI-P-PP-O-4-5t-A}
{\left\| {\mathcal{R}\mathrm{div}\,   {S_{q,\Gamma }}} \right\|_N} \lesssim & \left\| {{A_{\xi ,k,{n+1}}}} \right\|_N^{1 - \tilde \alpha }\left\| {{A_{\xi ,k,{n+1}}}} \right\|_{N + 1}^{\tilde \alpha } \nonumber \\
\lesssim & {\delta _q} {\delta _{q + 1}}  {\left( {\frac{{{\lambda _q}}}{{\lambda _{q + 1}}}} \right)^{\frac{1}{3} - \beta }}  \lambda _{q + 1}^\alpha \lambda _q^N \nonumber \\
\lesssim &  {\delta _{q + 2}}\lambda _{q + 1}^{N - \alpha }, \ \ \forall \, \alpha  \in \left[ {\frac{{3\tilde \alpha }}{{2b + 1}},\beta } \right)   .
\end{align}
\par
Employing \eqref{WPD-NP-main-BTO-1}, \eqref{RSRI-P-PP-O-4-5t-A}, Lemma \ref{WPD-NP} and Proposition \ref{WPD-NS-SD-P-main}, we arrive at
\begin{align*}
{\left\| {\left( {\left( {{u_q} - {{\bar u}_q}} \right) \cdot \nabla } \right){A_{\xi ,k,{n+1}}}} \right\|_N} \lesssim & {\left\| {{u_q} - {{\bar u}_q}} \right\|_N}{\left\| {{A_{\xi ,k,{n+1}}}} \right\|_1} + {\left\| {{u_q} - {{\bar u}_q}} \right\|_0}{\left\| {{A_{\xi ,k,{n+1}}}} \right\|_{N + 1}}   \\
\lesssim & \delta _q^{\frac{3}{2}}{\delta _{q + 1}} {\left( {\frac{{{\lambda _q}}}{{\lambda _{q + 1}}}} \right)^{\frac{1}{3} - \beta }}\lambda _{q + 1}^\alpha \lambda _q^{N + 1}    ,
\end{align*}
and
\begin{align*}
{\left\| {\left( {{w_{q + 1}} \cdot \nabla } \right){A_{\xi ,k,{n+1}}}} \right\|_N} \lesssim & {\left\| {{w_{q + 1}}} \right\|_N}{\left\| {{A_{\xi ,k,{n+1}}}} \right\|_1} + {\left\| {{w_{q + 1}}} \right\|_0}{\left\| {{A_{\xi ,k,{n+1}}}} \right\|_{N + 1}}  \\
\lesssim & {\delta _q} \delta _{q + 1}^{\frac{3}{2}} {\left( {\frac{{{\lambda _q}}}{{\lambda _{q + 1}}}} \right)^{\frac{1}{3} - \beta }}\lambda _{q + 1}^\alpha \lambda _{q + 1}^{N + 1}  .
\end{align*}
Thus, it follows from \eqref{SHDUFY-A-1} that
\begin{align}\label{RSRI-P-PP-O-3}
{\left\| {{D_{t,q + 1}}{S_{q,\Gamma }}} \right\|_N} \lesssim & \mathop {\sup }\limits_{\xi ,k,n} \left( {{{\left\| {{D_{t,q + 1}}{A_{\xi ,k,{n+1}}}} \right\|}_N} + {\mu _{q + 1}}\mathop {\sup }\limits_{\xi ,k,n} {{\left\| {{A_{\xi ,k,{n+1}}}} \right\|}_N}} \right) \nonumber  \\
\lesssim & \mathop {\sup }\limits_{\xi ,k,n} \left( {{{\left\| {{{\bar D}_{t,q}}{A_{\xi ,k,{n+1}}}} \right\|}_N} + {{\left\| {\left( {\left( {{u_q} - {{\bar u}_q}} \right) \cdot \nabla } \right){A_{\xi ,k,{n+1}}}} \right\|}_N}} \right)  \nonumber  \\
& + \mathop {\sup }\limits_{\xi ,k,n} \left( {{{\left\| {\left( {{w_{q + 1}} \cdot \nabla } \right){A_{\xi ,k,{n+1}}}} \right\|}_N} + {\mu _{q + 1}}{{\left\| {{A_{\xi ,k,{n+1}}}} \right\|}_N}} \right)  \nonumber  \\
\lesssim & {\delta _q} \delta _{q + 1}^{\frac{3}{2}} {\left( {\frac{{{\lambda _q}}}{{\lambda _{q + 1}}}} \right)^{\frac{1}{3} - \beta }}\lambda _{q + 1}^\alpha \lambda _{q + 1}^{N + 1} + \delta _q^{\frac{3}{2}}{\delta _{q + 1}} {\left( {\frac{{{\lambda _q}}}{{\lambda _{q + 1}}}} \right)^{\frac{1}{3} - \beta }}\lambda _{q + 1}^\alpha \lambda _q^{N + 1} \nonumber \\
& + \tau _q^{ - 1} {\delta _q} {\delta _{q + 1}}{\left( {\frac{{{\lambda _q}}}{{{\lambda _{q + 1}}}}} \right)^{\left( {\frac{1}{3} - \beta } \right)}}\lambda _q^N + {\mu _{q + 1}} {\delta _q} {\delta _{q + 1}}{\left( {\frac{{{\lambda _q}}}{{{\lambda _{q + 1}}}}} \right)^{\left( {\frac{1}{3} - \beta } \right)}}\lambda _q^N \nonumber \\
\lesssim & \delta _{q + 1}^{\frac{1}{2}}{\delta _{q + 2}}\lambda _{q + 1}^{N + 1 - \alpha }  . 
\end{align}
\par
Consider the intermittent oscillation ${\partial _t}w_{q + 1}^{\left( i \right)}$. According to the definition \eqref{THENDM-DKF}, we have
\begin{align*}
{\left\| {{\partial _{tt}}w_{q + 1}^{\left( i \right)}} \right\|_N} \lesssim & {{\tilde \mu }^{ - 1}}\mu _{q + 1}^2\mathop {\sup }\limits_{\xi ,k,n} \left\| {{{\tilde \eta }_{\xi ,k,n}}} \right\|_0^2 {\left\| {{\mathfrak{D}_{\tilde \mu }}\left( {t,{\lambda _{q + 1}}{{\tilde \Phi }_k}} \right)} \right\|_0} {\left\| {{\mathfrak{D}_{\tilde \mu }}\left( {t,{\lambda _{q + 1}}{{\tilde \Phi }_k}} \right)} \right\|_N}   \\
& + {{\tilde \mu }^{ - 1}}{\mu _{q + 1}}\mathop {\sup }\limits_{\xi ,k,n} {\left\| {{{\tilde \eta }_{\xi ,k,n}}} \right\|_0}{\left\| {{\partial _t}{{\tilde \eta }_{\xi ,k,n}}} \right\|_0} {\left\| {{\mathfrak{D}_{\tilde \mu }}\left( {t,{\lambda _{q + 1}}{{\tilde \Phi }_k}} \right)} \right\|_0}{\left\| {{\mathfrak{D}_{\tilde \mu }}\left( {t,{\lambda _{q + 1}}{{\tilde \Phi }_k}} \right)} \right\|_N} \\
& + {{\tilde \mu }^{ - 1}}{\mu _{q + 1}}\mathop {\sup }\limits_{\xi ,k,n} \left\| {{{\tilde \eta }_{\xi ,k,n}}} \right\|_0^2{\left\| {{\partial _t}{\mathfrak{D}_{\tilde \mu }}\left( {t,{\lambda _{q + 1}}{{\tilde \Phi }_k}} \right)} \right\|_0}{\left\| {{\partial _t}{\mathfrak{D}_{\tilde \mu }}\left( {t,{\lambda _{q + 1}}{{\tilde \Phi }_k}} \right)} \right\|_N} \\
& + {{\tilde \mu }^{ - 1}}\mathop {\sup }\limits_{\xi ,k,n} \left( {\left\| {{\partial _t}{{\tilde \eta }_{\xi ,k,n}}} \right\|_0^2 + {{\left\| {{{\tilde \eta }_{\xi ,k,n}}} \right\|}_0}{{\left\| {{\partial _{tt}}{{\tilde \eta }_{\xi ,k,n}}} \right\|}_0}} \right){\left\| {{\mathfrak{D}_{\tilde \mu }}\left( {t,{\lambda _{q + 1}}{{\tilde \Phi }_k}} \right)} \right\|_0}{\left\| {{\mathfrak{D}_{\tilde \mu }}\left( {t,{\lambda _{q + 1}}{{\tilde \Phi }_k}} \right)} \right\|_N} \\
& + {{\tilde \mu }^{ - 1}}\mathop {\sup }\limits_{\xi ,k,n} \left( {{{\left\| {{{\tilde \eta }_{\xi ,k,n}}} \right\|}_0} + {{\left\| {{\partial _t}{{\tilde \eta }_{\xi ,k,n}}} \right\|}_0}} \right){\left\| {{\partial _t}{\mathfrak{D}_{\tilde \mu }}\left( {t,{\lambda _{q + 1}}{{\tilde \Phi }_k}} \right)} \right\|_N} {\left\| {{\mathfrak{D}_{\tilde \mu }}\left( {t,{\lambda _{q + 1}}{{\tilde \Phi }_k}} \right)} \right\|_0} \\
& + {{\tilde \mu }^{ - 1}}\mathop {\sup }\limits_{\xi ,k,n} \left\| {{{\tilde \eta }_{\xi ,k,n}}} \right\|_0^2\left( {{{\left\| {{\partial _t}{\mathfrak{D}_{\tilde \mu }}\left( {t,{\lambda _{q + 1}}{{\tilde \Phi }_k}} \right)} \right\|}_0}{{\left\| {{\partial _t}{\mathfrak{D}_{\tilde \mu }}\left( {t,{\lambda _{q + 1}}{{\tilde \Phi }_k}} \right)} \right\|}_N}} \right)  \\
& + {{\tilde \mu }^{ - 1}}\mathop {\sup }\limits_{\xi ,k,n} \left\| {{{\tilde \eta }_{\xi ,k,n}}} \right\|_0^2\left( {{{\left\| {{\mathfrak{D}_{\tilde \mu }}\left( {t,{\lambda _{q + 1}}{{\tilde \Phi }_k}} \right)} \right\|}_0}{{\left\| {{\partial _{tt}}{\mathfrak{D}_{\tilde \mu }}\left( {t,{\lambda _{q + 1}}{{\tilde \Phi }_k}} \right)} \right\|}_N}} \right) ,
\end{align*}
which implies that
\begin{align}\label{RSRI-P-PP-O-5}
{\left\| {{\partial _{tt}}w_{q + 1}^{\left( i \right)}} \right\|_N} \lesssim & {\delta _q} {\delta _{q + 1}}\left( {{{\tilde \mu }^{ - 1}}\mu _{q + 1}^2\lambda _{q + 1}^{\frac{1}{2}N}\left( {\lambda _q^{\frac{1}{2}\left( {N - 1} \right)} + 1} \right) + \tilde \mu \lambda _{q + 1}^N} \right)  \nonumber \\
& + {{\tilde \mu }^{ - 1}}  {\delta _q} \lambda _q^{\frac{1}{2}N}\lambda _{q + 1}^{\frac{1}{2}N + 1}\left( {{\mu _{q + 1}}{\delta _{q + 1}}{\delta _q} + \left( {1 + \tau _q^{ - 1}} \right)\delta _{q + 1}^{\frac{1}{2}}\delta _q^{\frac{1}{2}}} \right)   .
\end{align}
It follows from \eqref{WPD-NS-SD-P-mainsd-AA-4}, \eqref{RSRI-P-PP-O-5} and Proposition \ref{RSRI-ujdh-main} that
\begin{equation}\label{RSRI-P-PP-O-6}
{\left\| {\mathcal{R}{\partial _t}w_{q + 1}^{\left( i \right)}} \right\|_0} \lesssim  \lambda _{q + 1}^{ - 1}  {{\tilde \mu }^{ - 1}}  {\delta _q}{\mu _{q + 1}}{\delta _{q + 1}} \lesssim {\delta _{q + 2}}\lambda _{q + 1}^{ - \alpha }    ,
\end{equation}
and
\begin{align}\label{RSRI-P-PP-O-7}
{\left\| {{\mathcal R}{D_{t,q + 1}}{\partial _t}w_{q + 1}^{\left( i \right)}} \right\|_0} \lesssim & {\left\| {{\mathcal R}{\partial _{tt}}w_{q + 1}^{\left( i \right)}} \right\|_0} + {\left\| {{v_{q + 1}} \otimes {\partial _t}w_{q + 1}^{\left( i \right)}} \right\|_0}  \nonumber  \\
\lesssim &  \lambda _{q + 1}^{ - 1} {{\tilde \mu }^{ - 1}}{\mu _{q + 1}}{\delta _{q + 1}}\left( {{\mu _{q + 1}} + {\delta _q}{\lambda _{q + 1}} + 1} \right) \nonumber \\
& + \lambda _{q + 1}^{ - 1} \left( {{{\tilde \mu }^{ - 1}}{\mu _{q + 1}}\delta _{q + 1}^{\frac{1}{2}}\delta _q^{\frac{1}{2}}\lambda _{q + 1}^{\frac{1}{2}} + \tilde \mu {\delta _{q + 1}}} \right) \nonumber  \\
\lesssim & \delta _{q + 1}^{\frac{1}{2}}{\delta _{q + 2}}\lambda _{q + 1}^{1 - \alpha }  . 
\end{align}
For the case $N \geqslant 1$, it also follows from \eqref{WPD-NS-SD-P-mainsd-AA-4}, \eqref{RSRI-P-PP-O-5} and Proposition \ref{RSRI-ujdh-main} that
\begin{equation}\label{RSRI-P-PP-O-8}
{\left\| {\mathcal{R}{\partial _t}w_{q + 1}^{\left( i \right)}} \right\|_N} \lesssim {{\tilde \mu }^{ - 1}}{\mu _{q + 1}}{\delta _{q + 1}}\lambda _{q + 1}^{N - 1} \lesssim {\delta _{q + 2}}\lambda _{q + 1}^{N - \alpha }    ,
\end{equation}
and
\begin{align}\label{RSRI-P-PP-O-9}
{\left\| {{\mathcal R}{D_{t,q + 1}}{\partial _t}w_{q + 1}^{\left( i \right)}} \right\|_N} \lesssim & {\left\| {{\mathcal R}{\partial _{tt}}w_{q + 1}^{\left( i \right)}} \right\|_N} + {\left\| {{v_{q + 1}} \otimes {\partial _t}w_{q + 1}^{\left( i \right)}} \right\|_N}  \nonumber  \\
\lesssim & {\delta _{q + 1}}\left( {{{\tilde \mu }^{ - 1}}\mu _{q + 1}^2\lambda _{q + 1}^{\frac{1}{2}N - 1}\left( {\lambda _q^{\frac{1}{2}\left( {N - 2} \right)} + 1} \right) + \tilde \mu \lambda _{q + 1}^{N - 1}} \right) \nonumber  \\
& + {{\tilde \mu }^{ - 1}}\lambda _q^{\frac{1}{2}N - 1}\lambda _{q + 1}^{\frac{1}{2}N}\left( {{\mu _{q + 1}}{\delta _{q + 1}}{\delta _q} + \left( {1 + \tau _q^{ - 1}} \right)\delta _{q + 1}^{\frac{1}{2}}\delta _q^{\frac{1}{2}}} \right) + {{\tilde \mu }^{ - 1}}{\mu _{q + 1}}\delta _{q + 1}^{\frac{3}{2}}\lambda _{q + 1}^N  \nonumber  \\
\lesssim & \delta _{q + 1}^{\frac{1}{2}}{\delta _{q + 2}}\lambda _{q + 1}^{N + 1 - \alpha } .  
\end{align}
\par
Considering the oscillation $w_{q + 1}^{\left( s \right)} \otimes w_{q + 1}^{\left( s \right)}$, we have
\begin{align*}
w_{q + 1}^{\left( s \right)} \otimes w_{q + 1}^{\left( s \right)} = & w_{q + 1}^{\left( p \right)} \otimes w_{q + 1}^{\left( p \right)} + w_{q + 1}^{\left( p \right)} \otimes \left( {w_{q + 1}^{\left( c \right)} + w_{q + 1}^{\left( i \right)}} \right) \\
& + \left( {w_{q + 1}^{\left( c \right)} + w_{q + 1}^{\left( i \right)}} \right) \otimes w_{q + 1}^{\left( p \right)} + \left( {w_{q + 1}^{\left( c \right)} + w_{q + 1}^{\left( i \right)}} \right) \otimes \left( {w_{q + 1}^{\left( c \right)} + w_{q + 1}^{\left( i \right)}} \right) .
\end{align*}
Notice that
\begin{align*}
w_{q + 1}^{\left( p \right)} \otimes w_{q + 1}^{\left( p \right)} = & \sum\limits_{n = 0}^\Gamma  {\sum\limits_{k \in {\mathbb{Z}_{q,n}}} {\sum\limits_{\xi  \in \Lambda } {g_{\xi ,k,n + 1}^2\left( {{\mu _{q + 1}}t} \right)\mathfrak{D}_{\tilde \mu }^2\left( {t,{\lambda _{q + 1}}{{\tilde \Phi }_k}} \right)\tilde \eta _{\xi ,k,n}^2{{\left( {\nabla {{\tilde \Phi }_k}} \right)}^{ - 1}}} } }  \\
&  \hspace*{6em} {W_\xi }\left( {{\lambda _{q + 1}}{{\tilde \Phi }_k}} \right) \otimes {W_\xi }\left( {{\lambda _{q + 1}}{{\tilde \Phi }_k}} \right){\left( {\nabla {{\tilde \Phi }_k}} \right)^{ - \mathrm{T}}}  \\
= & 4a_W^2\sum\limits_{n = 0}^\Gamma  {\sum\limits_{k \in {\mathbb{Z}_{q,n}}} {\sum\limits_{\xi  \in \Lambda } {g_{\xi ,k,n + 1}^2\left( {{\mu _{q + 1}}t} \right){{\cos }^2}\left( {\xi  \cdot {\lambda _{q + 1}}{{\tilde \Phi }_k}} \right){{\tilde A}_{\xi ,k,n}}} } }  ,
\end{align*}
and
\begin{equation*}
{\left( {\nabla {{\tilde \Phi }_k}} \right)^{ - 1}}{\xi ^ \bot } \otimes {\xi ^ \bot }{\left( {\nabla {{\tilde \Phi }_k}} \right)^{ - {\rm{T}}}} \cdot \nabla {\cos ^2}\left( {\xi  \cdot {\lambda _{q + 1}}{{\tilde \Phi }_k}} \right) = 0 ,
\end{equation*}
where
\begin{equation*}
{{\tilde A}_{\xi ,k,n}} = \mathfrak{D}_{\tilde \mu }^2\left( {t,{\lambda _{q + 1}}{{\tilde \Phi }_k}} \right)\tilde \eta _{\xi ,k,n}^2{\left( {\nabla {{\tilde \Phi }_k}} \right)^{ - 1}}{\xi ^ \bot } \otimes {\xi ^ \bot }{\left( {\nabla {{\tilde \Phi }_k}} \right)^{ - {\rm{T}}}} .
\end{equation*}
We then obtain
\begin{equation*}
\mathrm{div}\left( {w_{q + 1}^{\left( p \right)} \otimes w_{q + 1}^{\left( p \right)}} \right) = 4a_W^2\sum\limits_{n = 0}^\Gamma  {\sum\limits_{k \in {\mathbb{Z}_{q,n}}} {\sum\limits_{\xi  \in \Lambda } {g_{\xi ,k,n + 1}^2\left( {{\mu _{q + 1}}t} \right){{\cos }^2}\left( {\xi  \cdot {\lambda _{q + 1}}{{\tilde \Phi }_k}} \right)} } } \mathrm{div}\,   {{\tilde A}_{\xi ,k,n}} .
\end{equation*}
It follows from Lemma \ref{WPD-NHS-AA}, Lemma \ref{WPD-NHS-HUD-BB-2} and the stationary phase property (see \cite[Appendix G]{MR3374958}) that
\begin{align}\label{RSRI-P-PP-HUJ-O-10}
{\left\| {\mathcal{R}\mathrm{div}\left( {w_{q + 1}^{\left( p \right)} \otimes w_{q + 1}^{\left( p \right)}} \right)} \right\|_N} \lesssim & \lambda _{q + 1}^{ - 1}\mathop {\sup }\limits_{\xi ,k,n} {\left\| {{{\cos }^2}\left( {\xi  \cdot {\lambda _{q + 1}}{{\tilde \Phi }_k}} \right)} \right\|_N}{\left\| {{\mathfrak{D}_{\tilde \mu }}\left( {t,{\lambda _{q + 1}}{{\tilde \Phi }_k}} \right)} \right\|_1} \nonumber \\
& \cdot {\left\| {{\mathfrak{D}_{\tilde \mu }}\left( {t,{\lambda _{q + 1}}{{\tilde \Phi }_k}} \right)} \right\|_0}\left\| {{{\tilde \eta }_{\xi ,k,n}}} \right\|_0^2 \nonumber \\
\lesssim & {\delta _q} {\delta _{q + 1}}\lambda _{q + 1}^{N - \frac{1}{2}}    ,
\end{align}
and
\begin{align}\label{RSRI-P-PP-HUJ-O-11}
&{\left\| {{{\bar D}_{t,\Gamma }}\mathcal{R}\mathrm{div}\left( {w_{q + 1}^{\left( p \right)} \otimes w_{q + 1}^{\left( p \right)}} \right)} \right\|_N} \nonumber \\
\lesssim & \lambda _{q + 1}^{ - 1}\mathop {\sup }\limits_{\xi ,k,n} {\left\| {{{\cos }^2}\left( {\xi  \cdot {\lambda _{q + 1}}{{\tilde \Phi }_k}} \right)} \right\|_{N + 1}}{\left\| {{{\bar v}_{q,\Gamma }}} \right\|_0} {\left\| {{\mathfrak{D}_{\tilde \mu }}\left( {t,{\lambda _{q + 1}}{{\tilde \Phi }_k}} \right)} \right\|_1}{\left\| {{\mathfrak{D}_{\tilde \mu }}\left( {t,{\lambda _{q + 1}}{{\tilde \Phi }_k}} \right)} \right\|_0}\left\| {{{\tilde \eta }_{\xi ,k,n}}} \right\|_0^2 \nonumber \\
& + \lambda _{q + 1}^{ - 1}{\mu _{q + 1}}\mathop {\sup }\limits_{\xi ,k,n} {\left\| {{{\cos }^2}\left( {\xi  \cdot {\lambda _{q + 1}}{{\tilde \Phi }_k}} \right)} \right\|_N}{\left\| {{\mathfrak{D}_{\tilde \mu }}\left( {t,{\lambda _{q + 1}}{{\tilde \Phi }_k}} \right)} \right\|_1} {\left\| {{\mathfrak{D}_{\tilde \mu }}\left( {t,{\lambda _{q + 1}}{{\tilde \Phi }_k}} \right)} \right\|_0}\left\| {{{\tilde \eta }_{\xi ,k,n}}} \right\|_0^2 \nonumber \\
& + \lambda _{q + 1}^{ - 1}\mathop {\sup }\limits_{\xi ,k,n} {\left\| {{{\cos }^2}\left( {\xi  \cdot {\lambda _{q + 1}}{{\tilde \Phi }_k}} \right)} \right\|_N}{\left\| {{{\bar D}_{t,\Gamma }}\mathfrak{D}_{\tilde \mu }^2\left( {t,{\lambda _{q + 1}}{{\tilde \Phi }_k}} \right)} \right\|_0}\left\| {{{\tilde \eta }_{\xi ,k,n}}} \right\|_0^2 \nonumber \\
\lesssim & \delta _q{\delta _{q + 1}}\lambda _{q + 1}^{N - \frac{1}{2}} + {\mu _{q + 1}} {\delta _q}{\delta _{q + 1}}\lambda _{q + 1}^{N - \frac{1}{2}} + \delta _q^{\frac{3}{2}}{\delta _{q + 1}}\lambda _q^{\frac{1}{2}}\lambda _{q + 1}^N   .
\end{align}
\par
Thus, employing \eqref{RSRI-P-PP-HUJ-O-10} and Proposition \ref{WPD-NS-SD-P-main}, we deduce that
\begin{align}\label{RSRI-P-PP-O-10}
{\left\| {\mathcal{R}\mathrm{div}\left( {w_{q + 1}^{\left( s \right)} \otimes w_{q + 1}^{\left( s \right)}} \right)} \right\|_N}  \lesssim & {\left\| {\mathcal{R}\mathrm{div}\left( {w_{q + 1}^{\left( p \right)} \otimes w_{q + 1}^{\left( p \right)}} \right)} \right\|_N} + {\left\| {w_{q + 1}^{\left( p \right)}} \right\|_N}{\left\| {w_{q + 1}^{\left( c \right)} + w_{q + 1}^{\left( i \right)}} \right\|_0} \nonumber  \\
& + {\left\| {w_{q + 1}^{\left( p \right)}} \right\|_0}{\left\| {w_{q + 1}^{\left( c \right)} + w_{q + 1}^{\left( i \right)}} \right\|_N} + {\left\| {w_{q + 1}^{\left( c \right)} + w_{q + 1}^{\left( i \right)}} \right\|_N}{\left\| {w_{q + 1}^{\left( c \right)} + w_{q + 1}^{\left( i \right)}} \right\|_0} \nonumber  \\
\lesssim & {\delta _q} {\delta _{q + 1}}\lambda _{q + 1}^{N - \frac{1}{2}} + {{\tilde \mu }^{ - 1}}\delta _{q + 1}^{\frac{3}{2}}\lambda _{q + 1}^N + {{\tilde \mu }^{ - 2}}\delta _{q + 1}^2\lambda _{q + 1}^N \nonumber  \\
\lesssim & {\delta _{q + 2}}\lambda _{q + 1}^{N - \alpha }  .  
\end{align}
Furthermore, it follows from \eqref{WPD-NHS-MAIN-AA-P-3}, \eqref{RSRI-P-PP-HUJ-O-10}, \eqref{RSRI-P-PP-HUJ-O-11} and Proposition \ref{RSRI-ujdh-main} that
\begin{align}\label{RSRI-P-PP-O-11}
& {\left\| {{{\bar D}_{t,\Gamma }}\mathcal{R}\mathrm{div}\left( {w_{q + 1}^{\left( s \right)} \otimes w_{q + 1}^{\left( s \right)}} \right)} \right\|_N} \nonumber \\
\lesssim & {\left\| {{{\bar D}_{t,\Gamma }}\mathcal{R}\mathrm{div}\left( {w_{q + 1}^{\left( p \right)} \otimes w_{q + 1}^{\left( p \right)}} \right)} \right\|_N} + {\left\| {{{\bar D}_{t,\Gamma }}w_{q + 1}^{\left( p \right)}} \right\|_N}{\left\| {w_{q + 1}^{\left( c \right)} + w_{q + 1}^{\left( i \right)}} \right\|_0}   \nonumber  \\
& + {\left\| {{{\bar D}_{t,\Gamma }}w_{q + 1}^{\left( p \right)}} \right\|_0}{\left\| {w_{q + 1}^{\left( c \right)} + w_{q + 1}^{\left( i \right)}} \right\|_N} + {\left\| {w_{q + 1}^{\left( p \right)}} \right\|_N}{\left\| {{{\bar D}_{t,\Gamma }}\left( {w_{q + 1}^{\left( c \right)} + w_{q + 1}^{\left( i \right)}} \right)} \right\|_0}   \nonumber  \\
& + {\left\| {w_{q + 1}^{\left( p \right)}} \right\|_0}\left( {{{\left\| {{{\bar D}_{t,\Gamma }}w_{q + 1}^{\left( c \right)}} \right\|}_N} + {{\left\| {{\partial _t}w_{q + 1}^{\left( i \right)}} \right\|}_N} + {{\left\| {{{\bar v}_{q,\Gamma }}} \right\|}_0}{{\left\| {w_{q + 1}^{\left( i \right)}} \right\|}_{N + 1}}} \right) \nonumber  \\
& + {\left\| {{{\bar D}_{t,\Gamma }}\left( {w_{q + 1}^{\left( c \right)} + {\partial _t}w_{q + 1}^{\left( i \right)} + \left( {{{\bar v}_{q,\Gamma }} \cdot \nabla } \right)w_{q + 1}^{\left( i \right)}} \right)} \right\|_N}{\left\| {w_{q + 1}^{\left( c \right)} + w_{q + 1}^{\left( i \right)}} \right\|_0}  \nonumber  \\
\lesssim & \delta _q{\delta _{q + 1}}\lambda _{q + 1}^{N - \frac{1}{2}} + {\mu _{q + 1}} {\delta _q}{\delta _{q + 1}}\lambda _{q + 1}^{N - \frac{1}{2}} + \delta _q^{\frac{3}{2}}{\delta _{q + 1}}\lambda _q^{\frac{1}{2}}\lambda _{q + 1}^N + \delta _{q + 1}^{\frac{3}{2}}\lambda _{q + 1}^{N + \alpha }  \nonumber  \\
& + {\delta _{q + 1}}\lambda _{q + 1}^{N + 1}\left( {\lambda _{q + 1}^{ - \frac{1}{2}} + \lambda _q^{ - \frac{2}{3}}\lambda _{q + 1}^{ - \frac{1}{3}}} \right) + \delta _q^{\frac{1}{2}}{\delta _{q + 1}}\lambda _q^{ - \frac{2}{3}}\lambda _{q + 1}^{N + \frac{2}{3}} + {\mu _{q + 1}}\delta _q^{\frac{1}{2}}{\delta _{q + 1}}\lambda _{q + 1}^{N + \frac{1}{2}} \nonumber  \\
& + \left( {{\delta _{q + 1}}\lambda _{q + 1}^{N + \alpha } + \delta _q^{\frac{1}{2}}\delta _{q + 1}^{\frac{1}{2}}\lambda _q^{ - \frac{2}{3}}\lambda _{q + 1}^{N + \frac{2}{3}} + {\mu _{q + 1}}\delta _q^{\frac{1}{2}}\delta _{q + 1}^{\frac{1}{2}}\lambda _{q + 1}^{N + \frac{1}{2}}} \right)\left( {\delta _{q + 1}^{\frac{1}{2}}\lambda _{q + 1}^{ - \frac{1}{2}} + {{\tilde \mu }^{ - 1}}{\delta _{q + 1}}} \right)  \nonumber  \\
\lesssim & \delta _{q + 1}^{\frac{1}{2}}{\delta _{q + 2}}\lambda _{q + 1}^{N + 1 - \alpha }   .
\end{align}
Combining \eqref{RSRI-P-PP-O-4-5t}-\eqref{RSRI-P-PP-O-3}, \eqref{RSRI-P-PP-O-6}-\eqref{RSRI-P-PP-O-9}, \eqref{RSRI-P-PP-O-10} and \eqref{RSRI-P-PP-O-11}, we obtain
\begin{equation}\label{RSRI-P-SAMMLGOAL-O-1}
{\left\| { \mathring{R}_{q + 1}^{\left( i \right)} } \right\|_N} \lesssim {\delta _{q + 2}}\lambda _{q + 1}^{N - \alpha }   ,
\end{equation}
and
\begin{equation}\label{RSRI-P-SAMMLGOAL-O-2}
{\left\| { {{ D_{t,{q + 1} }}} \mathring{R}_{q + 1}^{\left( i \right)} } \right\|_N}  \lesssim \delta _{q + 1}^{\frac{1}{2}}{\delta _{q + 2}}\lambda _{q + 1}^{N + 1 - \alpha }   .
\end{equation}
\par
\vspace{1em}
{\textbf{Step 3. Estimates of the random oscillation error $\mathring{R}_{q + 1}^{\left( r \right)}$}}. 
\par
\vspace{1em}
Using \eqref{HIE-u-1.1}, \eqref{lemma-GLTNNRS-B-M-1}, \eqref{SHDUFY-A-3}, \eqref{Newtonsteps-A-B-36} and \eqref{Newtonsteps-A-B-41}, we get
\begin{equation}\label{RSRI-P-PP-rrr-1}
{\left\| {{\mathring{R}_{q,\Gamma }}} \right\|_N} \lesssim \mathop {\sup }\limits_{k,n} {\left\| {w_{k,n}^{\left( l \right)}} \right\|_N} \lesssim \mu _{q + 1}^{ - 1}{\delta _{q + 1}}  {\left( {\frac{{{\lambda _q}}}{{\lambda _{q + 1}}}} \right)^{\frac{1}{3} - \beta }}  \lambda _q^{N + 1}l_q^{ - \alpha } \lesssim {\delta _{q + 2}}\lambda _{q + 1}^{N - \alpha }    ,
\end{equation}
and
\begin{align}\label{RSRI-P-PP-rrr-2}
{\left\| {{D_{t,q + 1}}{\mathring{R}_{q,\Gamma }}} \right\|_N} \lesssim & \mathop {\sup }\limits_{k,n} \left( {{{\left\| {{{\bar D}_{t,q}}w_{k,n}^{\left( l \right)}} \right\|}_{N - 1}} + {{\left\| {\left( {{v_{q + 1}} - {{\bar v}_q}} \right) \otimes w_{k,n}^{\left( l \right)}} \right\|}_N}} \right)  \nonumber  \\
\lesssim & {\delta _{q + 1}} {\left( {\frac{{{\lambda _q}}}{{\lambda _{q + 1}}}} \right)^{\frac{1}{3} - \beta }}  \lambda _q^Nl_q^{ - \alpha } + \mu _{q + 1}^{ - 1}\delta _q^{\frac{1}{2}}{\delta _{q + 1}} {\left( {\frac{{{\lambda _q}}}{{\lambda _{q + 1}}}} \right)^{\frac{1}{3} - \beta }} \lambda _q^{N + 1}l_q^{ - \alpha }   \nonumber  \\
\lesssim & \delta _{q + 1}^{\frac{1}{2}}{\delta _{q + 2}}\lambda _{q + 1}^{N + 1 - \alpha } .  
\end{align}
Notice that
\begin{equation*}
{\left\| {{R_{q,0}} - {\mathring{R}_q}} \right\|_N} \lesssim l_q^2{\left\| {{\mathring{R}_q}} \right\|_{N + 2}} \lesssim {\delta _{q + 1}}\lambda _q^{N + 1 - \alpha }\lambda _{q + 1}^{ - 1} ,  
\end{equation*}
and
\begin{align*}
{\left\| {{D_{t,q + 1}}\left( {{R_{q,0}} - {\mathring{R}_q}} \right)} \right\|_N} \lesssim & {\left\| {{D_{t,q}}\left( {{R_{q,0}} - {\mathring{R}_q}} \right)} \right\|_N} + {\left\| {\left( {{w_{q + 1}} \cdot \nabla } \right)\left( {{R_{q,0}} - {\mathring{R}_q}} \right)} \right\|_N} \\
\lesssim & l_q^2\left( {{{\left\| {{D_{t,q}}{\mathring{R}_q}} \right\|}_{N + 2}} + {{\left\| {{w_{q + 1}}} \right\|}_N}{{\left\| {{\mathring{R}_q}} \right\|}_3}} \right) \\
\lesssim & \delta _q^{\frac{1}{2}}{\delta _{q + 1}}\lambda _q^{N + 2 - \alpha }\lambda _{q + 1}^{ - 1} + \delta _{q + 1}^{\frac{3}{2}}\lambda _q^{2 - \alpha }\lambda _{q + 1}^{N - 1} .
\end{align*}
It follows from \eqref{SHDUFY-A-2} and \eqref{parameters-S-22} that
\begin{align}\label{RSRI-P-PP-rrr-3}
{\left\| {{\mathfrak{P}_{q + 1, \Gamma}}} \right\|_N} \lesssim & {\left\| {{R_{q,0}} - {\mathring{R}_q}} \right\|_N} + \mathop {\sup }\limits_{\xi ,k,n} {{\left\| {{A_{\xi ,k,{n+1}}}} \right\|}_N} \nonumber \\
\lesssim & {\delta _{q + 1}}\lambda _q^{N + 1 - \alpha }\lambda _{q + 1}^{ - 1} + {\delta _q} {\delta _{q + 1}}{\left( {\frac{{{\lambda _q}}}{{\lambda _{q + 1}}}} \right)^{\frac{1}{3} - \beta }}\lambda _q^N \nonumber \\
\lesssim  & {\delta _{q + 2}}\lambda _{q + 1}^{N - \alpha } ,   
\end{align}
and
\begin{align}\label{RSRI-P-PP-rrr-4}
{\left\| {{D_{t,q + 1}}{\mathfrak{P}_{q + 1, \Gamma}}} \right\|_N} \lesssim & {\left\| {{D_{t,q + 1}}\left( {{R_{q,0}} - {\mathring{R}_q}} \right)} \right\|_N}  + \mathop {\sup }\limits_{\xi ,k,n} {{\left\| {{D_{t,q + 1}}{A_{\xi ,k,{n+1}}}} \right\|}_N} \nonumber \\
\lesssim & \delta _q^{\frac{1}{2}}{\delta _{q + 1}}\lambda _q^{N + 2 - \alpha }\lambda _{q + 1}^{ - 1} + \delta _{q + 1}^{\frac{3}{2}}\lambda _q^{2 - \alpha }\lambda _{q + 1}^{N - 1} + \delta _{q + 1}^{\frac{1}{2}}{\delta _{q + 2}}\lambda _{q + 1}^{N + 1 - \alpha } \nonumber \\
\lesssim & \delta _{q + 1}^{\frac{1}{2}}{\delta _{q + 2}}\lambda _{q + 1}^{N + 1 - \alpha } .
\end{align}
\par
Consider the random oscillation $w_{q + 1}^{\left( s \right)} \otimes {v_{q,\Gamma }} + {v_{q,\Gamma }} \otimes w_{q + 1}^{\left( s \right)}$ and the residual intermittent oscillation ${{\bar v}_{q,\Gamma }} \otimes w_{q + 1}^{\left( i \right)}$. According to \eqref{HIE-u-1.1}, \eqref{WPD-NP-main-BTO-1}, \eqref{WPD-NHS-MAIN-AA-P-3} and Proposition \ref{WPD-NS-SD-P-main}, we have
\begin{align*}
& {\left\| {w_{q + 1}^{\left( s \right)} \otimes {v_{q,\Gamma }} + {v_{q,\Gamma }} \otimes w_{q + 1}^{\left( s \right)} + {{\bar v}_{q,\Gamma }} \otimes w_{q + 1}^{\left( i \right)}} \right\|_N} \\
\lesssim & {\left\| {\mathcal{R}\left( {w_{q + 1}^{\left( s \right)} \cdot \nabla } \right){v_{q,\Gamma }}} \right\|_N} + {\left\| {{{\bar v}_{q,\Gamma }} \otimes w_{q + 1}^{\left( i \right)}} \right\|_N}   \\
\lesssim & \lambda _{q + 1}^{ - 1}{\left\| {w_{q + 1}^{\left( s \right)}} \right\|_N}{\left\| {{v_q} + w_{q + 1}^{\left( \Gamma \right)}} \right\|_1} + {\left\| {w_{q + 1}^{\left( i \right)}} \right\|_N}{\left\| {{{\bar v}_{q,\Gamma }}} \right\|_0}   \\
\lesssim  & {\lambda _q}\lambda _{q + 1}^{N - 1}\left( {\delta _q^{\frac{1}{2}} + \mu _{q + 1}^{ - 1}{\delta _{q + 1}}{\lambda _q}l_q^{ - \alpha }} \right) + {{\tilde \mu }^{ - 1}} {\delta _{q + 1}}\lambda _{q + 1}^N ,
\end{align*}
which implies that
\begin{equation}\label{RSRI-P-PP-rrr-6}
{\left\| {w_{q + 1}^{\left( s \right)} \otimes {v_{q,\Gamma }} + {v_{q,\Gamma }} \otimes w_{q + 1}^{\left( s \right)} + {{\bar v}_{q,\Gamma }} \otimes w_{q + 1}^{\left( i \right)}} \right\|_N}  \lesssim {\delta _{q + 2}}\lambda _{q + 1}^{N - \alpha } .
\end{equation}
Owing to
\begin{align*}
& {\left\| {{D_{t,q + 1}}\left( {w_{q + 1}^{\left( s \right)} \otimes {v_{q,\Gamma }} + {v_{q,\Gamma }} \otimes w_{q + 1}^{\left( s \right)} + {{\bar v}_{q,\Gamma }} \otimes w_{q + 1}^{\left( i \right)}} \right)} \right\|_N} \\
\lesssim & {\left\| {{{\bar D}_{t,\Gamma }}\left( {w_{q + 1}^{\left( s \right)} \otimes {v_{q,\Gamma }} + {v_{q,\Gamma }} \otimes w_{q + 1}^{\left( s \right)} + {{\bar v}_{q,\Gamma }} \otimes w_{q + 1}^{\left( i \right)}} \right)} \right\|_N}  \\
& + {\left\| {\left( {\left( {{v_{q + 1}} - {{\bar v}_{q,\Gamma }}} \right) \cdot \nabla } \right)\left( {w_{q + 1}^{\left( s \right)} \otimes {v_{q,\Gamma }} + {v_{q,\Gamma }} \otimes w_{q + 1}^{\left( s \right)} + {{\bar v}_{q,\Gamma }} \otimes w_{q + 1}^{\left( i \right)}} \right)} \right\|_N} \\
\lesssim & {\left\| {{{\bar D}_{t,\Gamma }}\left( {w_{q + 1}^{\left( p \right)} + w_{q + 1}^{\left( c \right)}} \right) + {\partial _t}w_{q + 1}^{\left( i \right)}} \right\|_N}{\left\| {{v_{q,\Gamma }}} \right\|_0} + \lambda _{q + 1}^{ - 1}{\left\| {{{\bar v}_{q,\Gamma }}} \right\|_0}{\left\| {w_{q + 1}^{\left( s \right)}} \right\|_1}{\left\| {{v_{q,\Gamma }}} \right\|_2}  \\
& + \lambda _{q + 1}^{ - 1}{\left\| {w_{q + 1}^{\left( s \right)}} \right\|_N}\left( {{{\left\| {{u_q}} \right\|}_1}{{\left\| {{u_q}} \right\|}_2} + {{\left\| {{\mathfrak{p}_q}} \right\|}_2} + {{\left\| {{\mathring{R}_q}} \right\|}_2}} \right) \lambda _{q + 1}^{ - 1}{\left\| {{{\bar v}_q}} \right\|_0}{\left\| {w_{q + 1}^{\left( \Gamma  \right)}} \right\|_0}{\left\| {w_{q + 1}^{\left( i \right)}} \right\|_{N + 2}} \\
&   + {\left\| {w_{q + 1}^{\left( s \right)}} \right\|_N}\left( {{{\left\| {{{\bar D}_{t,q}}w_{q + 1}^{\left( \Gamma  \right)}} \right\|}_0} + {{\left\| {{{\bar v}_q}} \right\|}_0}{{\left\| {w_{q + 1}^{\left( \Gamma  \right)}} \right\|}_1}} \right) + {\left\| {w_{q + 1}^{\left( s \right)}} \right\|_0}\left( {{{\left\| {{{\bar D}_{t,q}}w_{q + 1}^{\left( \Gamma  \right)}} \right\|}_N} + {{\left\| {{{\bar v}_q}} \right\|}_0}{{\left\| {w_{q + 1}^{\left( \Gamma  \right)}} \right\|}_{N + 1}}} \right) \\
&  + {\left\| {w_{q + 1}^{\left( i \right)}} \right\|_N}\left( {{{\left\| {{\partial _t}{{\bar v}_q}} \right\|}_0} + {{\left\| {{{\bar D}_{t,q}}w_{q + 1}^{\left( \Gamma  \right)}} \right\|}_0}} \right) + {\left\| {w_{q + 1}^{\left( i \right)}} \right\|_0}{\left\| {{{\bar D}_{t,q}}w_{q + 1}^{\left( \Gamma  \right)}} \right\|_N} \\
&  + {\left\| {{\partial _t}w_{q + 1}^{\left( i \right)}} \right\|_N}\left( {{{\left\| {{{\bar v}_q}} \right\|}_0} + {{\left\| {w_{q + 1}^{\left( \Gamma  \right)}} \right\|}_0}} \right) + \left\| {{{\bar v}_{q,\Gamma }}} \right\|_0^2{\left\| {w_{q + 1}^{\left( i \right)}} \right\|_{N + 1}}  \\
& + {\left\| {{v_{q + 1}} - {{\bar v}_{q,\Gamma }}} \right\|_0}{\left\| {w_{q + 1}^{\left( s \right)} \otimes {v_{q,\Gamma }} + {v_{q,\Gamma }} \otimes w_{q + 1}^{\left( s \right)} + {{\bar v}_{q,\Gamma }} \otimes w_{q + 1}^{\left( i \right)}} \right\|_{N + 1}} \\
& + {\left\| {{v_{q + 1}} - {{\bar v}_{q,\Gamma }}} \right\|_N}{\left\| {w_{q + 1}^{\left( s \right)} \otimes {v_{q,\Gamma }} + {v_{q,\Gamma }} \otimes w_{q + 1}^{\left( s \right)} + {{\bar v}_{q,\Gamma }} \otimes w_{q + 1}^{\left( i \right)}} \right\|_1} ,
\end{align*}
it follows from \eqref{HIE-u-1.1}, \eqref{HIE-R-1.1}, \eqref{WPD-NP-main-BTO-1}, \eqref{WPD-NHS-MAIN-AA-P-3}, \eqref{RSRI-main-PROVE-15FAKFOG}, Proposition \ref{WPD-NS-SD-P-main} and Proposition \ref{RSRI-ujdh-main} that
\begin{equation}\label{RSRI-P-PP-rrr-7}
{\left\| {{D_{t,q + 1}}\left( {w_{q + 1}^{\left( s \right)} \otimes {v_{q,\Gamma }} + {v_{q,\Gamma }} \otimes w_{q + 1}^{\left( s \right)} + {{\bar v}_{q,\Gamma }} \otimes w_{q + 1}^{\left( i \right)}} \right)} \right\|_N} \lesssim \delta _{q + 1}^{\frac{1}{2}}{\delta _{q + 2}}\lambda _{q + 1}^{N + 1 - \alpha } .  
\end{equation}
Therefore, employing \eqref{LIJIFNI-DKFJ-R-3} and \eqref{RSRI-P-PP-rrr-1}-\eqref{RSRI-P-PP-rrr-6}, we obtain
\begin{equation}\label{RSRI-P-SAMMLGOAL-rr-1}
{\left\| { \mathring{R}_{q + 1}^{\left( r \right)} } \right\|_N} \lesssim {\delta _{q + 2}}\lambda _{q + 1}^{N - \alpha } ,
\end{equation}
and
\begin{equation}\label{RSRI-P-SAMMLGOAL-rr-2}
{\left\| { {{ D_{t,{q + 1} }}} \mathring{R}_{q + 1}^{\left( r \right)} } \right\|_N}  \lesssim \delta _{q + 1}^{\frac{1}{2}}{\delta _{q + 2}}\lambda _{q + 1}^{N + 1 - \alpha } .  
\end{equation}
\par
Finally, combining \eqref{RSRI-P-SAMMLGOAL-L-1}, \eqref{RSRI-P-SAMMLGOAL-L-2}, \eqref{RSRI-P-SAMMLGOAL-O-1}, \eqref{RSRI-P-SAMMLGOAL-O-2}, \eqref{RSRI-P-SAMMLGOAL-rr-1} and \eqref{RSRI-P-SAMMLGOAL-rr-2}, we conclude that \eqref{RSRI-P-mainsd-AAJJS-1} and \eqref{RSRI-P-mainsd-AAJJS-2} hold. The proof is complete. \qed
\par
\begin{proposition}\label{RSRI-ESI-GAP}
For any $t \in \left[ {0,\mathfrak{t}} \right]$, the energy gap ${\mathfrak{E}_{q + 1}}$ defined by \eqref{LOKJI-A-FGH-A-1} admits
\begin{equation}\label{RSRI-P-mene-GAP}
\left| {{\mathfrak{E}_{q + 1}}} \right| \lesssim {\delta _{q + 2}}  .
\end{equation}
\end{proposition}
\par
\noindent{\textbf{Proof}}.
According to \eqref{Nashsteps-A-P-3}, Lemma \ref{GLTNNRS-A} and Lemma \ref{IBW-P-A-1}, we have
\begin{equation*}
w_{q + 1}^{\left( p \right)} \otimes w_{q + 1}^{\left( p \right)} = 4a_W^2\sum\limits_{n = 0}^\Gamma  {\sum\limits_{k \in {\mathbb{Z}_{q,n}}} {\sum\limits_{\xi  \in \Lambda } {{{\cos }^2}\left( {{\xi ^ \bot } \cdot {\lambda _{q + 1}}{{\tilde \Phi }_x}} \right)\mathfrak{D}_{\tilde \mu }^2\left( {t,{\lambda _{q + 1}}{{\tilde \Phi }_x}} \right)\chi _k^2g_{\xi ,k,n + 1}^2} } } \left( {{\mathfrak{E}_{q,\xi ,k,n}} \mathrm{Id} - {{\tilde R}_{q,n}}} \right) ,
\end{equation*}
which implies that
\begin{equation*}
\left\| {w_{q + 1}^{\left( p \right)}} \right\|_{{L^2}}^2 = {\mathfrak{E}_q} - \mathfrak{E}_{q + 1}^{\left( {Reynolds} \right)} .
\end{equation*}
For simplicity of notation, we denote
\begin{equation*}
\mathfrak{E}_{q + 1}^{\left( {Reynolds} \right)} \triangleq 4a_W^2\sum\limits_{n = 0}^\Gamma  {\sum\limits_{k \in {\mathbb{Z}_{q,n}}} {\sum\limits_{\xi  \in \Lambda } {\chi _k^2g_{\xi ,k,n + 1}^2\int_{{\mathbb{T}^2}} {{{\cos }^2}\left( {{\xi ^ \bot } \cdot {\lambda _{q + 1}}{{\tilde \Phi }_x}} \right)\mathfrak{D}_{\tilde \mu }^2\left( {t,{\lambda _{q + 1}}{{\tilde \Phi }_x}} \right){{\tilde R}_{q,n}}} \mathrm{d}x} } } .
\end{equation*}
\par
Notice that
\begin{equation*}
{{\hat u}_{q + 1}} = {{\hat u}_q} + {{\hat w}_{q + 1}} + \mathfrak{\hat W}_{q + 1}^{loc} + {u_q} * \left( {{{\hat \varsigma }_{{{\hat l}_{q + 1}}}} - {{\hat \varsigma }_{{{\hat l}_q}}}} \right) ,
\end{equation*}
\begin{equation*}
{{\hat v}_{q + 1}} = {{\hat v}_q} + {{\hat w}_{q + 1}} + {v_q} * \left( {{{\hat \varsigma }_{{{\hat l}_{q + 1}}}} - {{\hat \varsigma }_{{{\hat l}_q}}}} \right)  ,
\end{equation*}
and
\begin{equation*}
{\mathfrak{\hat W}_{q + 1}} = {\mathfrak{\hat W}_q} + \mathfrak{\hat W}_{q + 1}^{loc} + {\mathfrak{W}_q} * \left( {{{\hat \varsigma }_{{{\hat l}_{q + 1}}}} - {{\hat \varsigma }_{{{\hat l}_q}}}} \right)  .
\end{equation*}
We derive that
\begin{align}\label{RSRI-P-mene-GAP-P-A-1}
{\mathfrak{E}_{q + 1}} = & {\mathfrak{E}_q} - {\left\langle {{{\hat u}_q} + {{\hat u}_{q + 1}},{{\hat w}_{q + 1}} + \mathfrak{\hat W}_{q + 1}^{loc} + {u_q} * \left( {{{\hat \varsigma }_{{{\hat l}_{q + 1}}}} - {{\hat \varsigma }_{{{\hat l}_q}}}} \right)} \right\rangle _{{\mathbb{T}^2}}} \nonumber \\
& + {\left\langle {{\mathfrak{\hat W}_q} + {\mathfrak{\hat W}_{q + 1}},\mathfrak{\hat W}_{q + 1}^{loc} + {\mathfrak{W}_q} * \left( {{{\hat \varsigma }_{{{\hat l}_{q + 1}}}} - {{\hat \varsigma }_{{{\hat l}_q}}}} \right)} \right\rangle _{{\mathbb{T}^2}}} + \frac{1}{2}\left( {{\delta _{q + 2}} - {\delta _{q + 3}}} \right) \nonumber \\
& + 2\int_0^t {{{\left\langle {{{\hat w}_{q + 1}} + {v_q} * \left( {{{\hat \varsigma }_{{{\hat l}_{q + 1}}}} - {{\hat \varsigma }_{{{\hat l}_q}}}} \right), \mathrm{d}{\mathfrak{W}_{q + 1}}} \right\rangle }_{{\mathbb{T}^2}}}}  + 2\int_0^t {{{\left\langle {{{\hat v}_q},\mathrm{d}\mathfrak{W}_{q + 1}^{loc}} \right\rangle }_{{\mathbb{T}^2}}}} \nonumber \\
\triangleq & \mathfrak{E}_{q + 1}^{\left( {Reynolds} \right)} + \mathfrak{E}_{q + 1}^{\left( {osc} \right)} + \mathfrak{E}_{q + 1}^{\left( {rem} \right)} + \mathfrak{E}_{q + 1}^{\left( {sto} \right)} ,
\end{align}
where
\begin{equation*}
\mathfrak{E}_{q + 1}^{\left( {osc} \right)} \triangleq \left\| {w_{q + 1}^{\left( p \right)}} \right\|_{{L^2}}^2 - \left\| {\hat w_{q + 1}^{\left( p \right)}} \right\|_{{L^2}}^2 - 2{\left\langle {\hat w_{q + 1}^{\left( p \right)},{u_q} * {{\hat \varsigma }_{{{\hat l}_{q + 1}}}}} \right\rangle _{{\mathbb{T}^2}}}  ,
\end{equation*}
\begin{equation*}
\mathfrak{E}_{q + 1}^{\left( {rem} \right)} \triangleq \frac{1}{2}\left( {{\delta _{q + 2}} - {\delta _{q + 3}}} \right) - {\left\langle {{u_q} * \left( {{{\hat \varsigma }_{{{\hat l}_{q + 1}}}} - {{\hat \varsigma }_{{{\hat l}_q}}}} \right),{u_q} * \left( {{{\hat \varsigma }_{{{\hat l}_{q + 1}}}} + {{\hat \varsigma }_{{{\hat l}_q}}}} \right)} \right\rangle _{{\mathbb{T}^2}}} ,
\end{equation*}
and
\begin{align*}
\mathfrak{E}_{q + 1}^{\left( {sto} \right)} \triangleq & 2{\left\langle {{\mathfrak{\hat W}_q} + {\mathfrak{\hat W}_{q + 1}},{\mathfrak{W}_q} * \left( {{{\hat \varsigma }_{{{\hat l}_{q + 1}}}} - {{\hat \varsigma }_{{{\hat l}_q}}}} \right)} \right\rangle _{{\mathbb{T}^2}}} + \left\| {{\mathfrak{W}_q} * \left( {{{\hat \varsigma }_{{{\hat l}_{q + 1}}}} - {{\hat \varsigma }_{{{\hat l}_q}}}} \right)} \right\|_{{L^2}}^2  \\
& + 2{\left\langle {{\mathfrak{\hat W}_q} - {{\hat u}_q} - {{\hat w}_{q + 1}} + {u_q} * \left( {{{\hat \varsigma }_{{{\hat l}_{q + 1}}}} - {{\hat \varsigma }_{{{\hat l}_q}}}} \right),\mathfrak{\hat W}_{q + 1}^{loc}} \right\rangle _{{\mathbb{T}^2}}} \\
& + 2\int_0^t {{{\left\langle {{{\hat w}_{q + 1}} + {v_q} * \left( {{{\hat \varsigma }_{{{\hat l}_{q + 1}}}} - {{\hat \varsigma }_{{{\hat l}_q}}}} \right), \mathrm{d}{\mathfrak{W}_{q + 1}}} \right\rangle }_{{\mathbb{T}^2}}}}  + 2\int_0^t {{{\left\langle {{{\hat v}_q},\mathrm{d}\mathfrak{W}_{q + 1}^{loc}} \right\rangle }_{{\mathbb{T}^2}}}}  .
\end{align*}
\par
It follows from \eqref{JMNH-DKO-A-1-A-2} and \eqref{parameters-S-202568} that
\begin{equation}\label{RSRI-P-mene-GAP-P-A-2}
\left| {\mathfrak{E}_{q + 1}^{\left( {Reynolds} \right)}} \right| \lesssim {\left\| {{{\tilde R}_{q,n}}} \right\|_0} \lesssim {\delta _{q + 1}}\lambda _q^{ - \alpha } \lesssim {\delta _{q + 2}}  .
\end{equation}
Since $\left\| {w_{q + 1}^{\left( p \right)}} \right\|_{{L^2}}^2 - \left\| {\hat w_{q + 1}^{\left( p \right)}} \right\|_{{L^2}}^2 = {\left\langle {w_{q + 1}^{\left( p \right)} - \hat w_{q + 1}^{\left( p \right)},w_{q + 1}^{\left( p \right)} + \hat w_{q + 1}^{\left( p \right)}} \right\rangle _{{\mathbb{T}^2}}}$, we infer that
\begin{align*}
\left| {\left\| {w_{q + 1}^{\left( p \right)}} \right\|_{{L^2}}^2 - \left\| {\hat w_{q + 1}^{\left( p \right)}} \right\|_{{L^2}}^2} \right| \leqslant & \left( {{{\left\| {w_{q + 1}^{\left( p \right)}} \right\|}_0} + {{\left\| {\hat w_{q + 1}^{\left( p \right)}} \right\|}_0}} \right){\left\| {w_{q + 1}^{\left( p \right)} - \hat w_{q + 1}^{\left( p \right)}} \right\|_0}  \\
\lesssim & {{\hat l}_q}\left( {{{\left\| {w_{q + 1}^{\left( p \right)}} \right\|}_0} + {{\left\| {\hat w_{q + 1}^{\left( p \right)}} \right\|}_0}} \right){\left\| {w_{q + 1}^{\left( p \right)}} \right\|_1}.
\end{align*}
Furthermore, owing to ${\left\langle {\hat w_{q + 1}^{\left( p \right)},{u_q} * {{\hat \varsigma }_{{{\hat l}_{q + 1}}}}} \right\rangle _{{\mathbb{T}^2}}} = {\left\langle {\mathrm{div}\, \mathcal{R}\hat w_{q + 1}^{\left( p \right)},{u_q} * {{\hat \varsigma }_{{{\hat l}_{q + 1}}}}} \right\rangle _{{\mathbb{T}^2}}} =  - {\left\langle {\mathcal{R}\hat w_{q + 1}^{\left( p \right)},\nabla \left( {{u_q} * {{\hat \varsigma }_{{{\hat l}_{q + 1}}}}} \right)} \right\rangle _{{\mathbb{T}^2}}}$, we have
\begin{align*}
\left| {{{\left\langle {\hat w_{q + 1}^{\left( p \right)},{u_q} * {{\hat \varsigma }_{{{\hat l}_{q + 1}}}}} \right\rangle }_{{\mathbb{T}^2}}}} \right| \leqslant & \left| {{{\left\langle {\mathcal{R}\hat w_{q + 1}^{\left( p \right)},\nabla \left( {{u_q} * {{\hat \varsigma }_{{{\hat l}_{q + 1}}}}} \right)} \right\rangle }_{{\mathbb{T}^2}}}} \right| \\
\lesssim & \lambda _{q + 1}^{\tilde \alpha  - 1}{\left\| {w_{q + 1}^{\left( p \right)}} \right\|_0}{\left\| {{u_q}} \right\|_1} , \ \ \forall \, \tilde \alpha  \in \left( {0,1 - 3\beta } \right).
\end{align*}
Therefore,
\begin{equation}\label{RSRI-P-mene-GAP-P-A-3}
\left| {\mathfrak{E}_{q + 1}^{\left( {osc} \right)}} \right| \lesssim \hat l{\delta _{q + 1}}{\lambda _{q + 1}} + \delta _q^{\frac{1}{2}}\delta _{q + 1}^{\frac{1}{2}}\frac{{{\lambda _q}}}{{\lambda _{q + 1}^{1 - \tilde \alpha }}} \lesssim {\delta _{q + 2}}  .  
\end{equation}
\par
Using the fact that ${u_q} * \left( {{{\hat \varsigma }_{{{\hat l}_{q + 1}}}} - {{\hat \varsigma }_{{{\hat l}_q}}}} \right) = {u_q} * {{\hat \varsigma }_{{{\hat l}_{q + 1}}}} - {u_q} + {u_q} - {{\hat u}_q}$, we obtain
\begin{equation}\label{RSRI-P-mene-GAP-P-A-4}
{\left\| {{u_q} * \left( {{{\hat \varsigma }_{{{\hat l}_{q + 1}}}} - {{\hat \varsigma }_{{{\hat l}_q}}}} \right)} \right\|_0} \leqslant {\left\| {{u_q} * {{\hat \varsigma }_{{{\hat l}_{q + 1}}}} - {u_q}} \right\|_0} + {\left\| {{{\hat u}_q} - {u_q}} \right\|_0} \lesssim \left( {{{\hat l}_{q + 1}} + {{\hat l}_q}} \right){\left\| {{u_q}} \right\|_1}  ,
\end{equation}
which implies that
\begin{equation*}
\left| {{{\left\langle {{u_q} * \left( {{{\hat \varsigma }_{{{\hat l}_{q + 1}}}} - {{\hat \varsigma }_{{{\hat l}_q}}}} \right),{u_q} * \left( {{{\hat \varsigma }_{{{\hat l}_{q + 1}}}} + {{\hat \varsigma }_{{{\hat l}_q}}}} \right)} \right\rangle }_{{\mathbb{T}^2}}}} \right| \lesssim \left( {{{\hat l}_{q + 1}} + {{\hat l}_q}} \right) ^ 2{\left\| {{u_q}} \right\|_1^2} .
\end{equation*}
Then, we get
\begin{equation}\label{RSRI-P-mene-GAP-P-A-5}
\left| {\mathfrak{E}_{q + 1}^{\left( {rem} \right)}} \right| \lesssim \left| {\left( {{\delta _{q + 2}} - {\delta _{q + 3}}} \right)} \right| + \left( {{{\hat l}_{q + 1}} + {{\hat l}_q}} \right)^2 \delta _q{\lambda _q^2} \lesssim {\delta _{q + 2}}  .
\end{equation}
\par
Using the same argument as in \eqref{RSRI-P-mene-GAP-P-A-4}, we have
\begin{equation*}
{\left\| {{\mathfrak{W}_q} * \left( {{{\hat \varsigma }_{{{\hat l}_{q + 1}}}} - {{\hat \varsigma }_{{{\hat l}_q}}}} \right)} \right\|_0} \lesssim \left( {{{\hat l}_{q + 1}} + {{\hat l}_q}} \right){\left\| {{\mathfrak{W}_q}} \right\|_1} ,
\end{equation*}
and
\begin{equation*}
{\left\| {{v_q} * \left( {{{\hat \varsigma }_{{{\hat l}_{q + 1}}}} - {{\hat \varsigma }_{{{\hat l}_q}}}} \right)} \right\|_0} \lesssim \left( {{{\hat l}_{q + 1}} + {{\hat l}_q}} \right){\left\| {{v_q}} \right\|_1} .
\end{equation*}
Then, we infer that
\begin{align*}
& \left| {{{\left\langle {{\mathfrak{\hat W}_q} + {\mathfrak{\hat W}_{q + 1}},{\mathfrak{W}_q} * \left( {{{\hat \varsigma }_{{{\hat l}_{q + 1}}}} - {{\hat \varsigma }_{{{\hat l}_q}}}} \right)} \right\rangle }_{{\mathbb{T}^2}}} + \left\| {{\mathfrak{W}_q} * \left( {{{\hat \varsigma }_{{{\hat l}_{q + 1}}}} - {{\hat \varsigma }_{{{\hat l}_q}}}} \right)} \right\|_{{L^2}}^2} \right| \\
\lesssim & \left( {{{\hat l}_{q + 1}} + {{\hat l}_q}} \right)\left( {{{\left\| {{\mathfrak{\hat W}_q}} \right\|}_0} + {{\left\| {{\mathfrak{\hat W}_{q + 1}}} \right\|}_0}} \right){\left\| {{\mathfrak{W}_q}} \right\|_1} + {\left( {{{\hat l}_{q + 1}} + {{\hat l}_q}} \right)^2}\left\| {{\mathfrak{W}_q}} \right\|_1^2 ,
\end{align*}
and
\begin{align*}
& \left| {{{\left\langle {{\mathfrak{\hat W}_q} - {{\hat u}_q} - {{\hat w}_{q + 1}} + {u_q} * \left( {{{\hat \varsigma }_{{{\hat l}_{q + 1}}}} - {{\hat \varsigma }_{{{\hat l}_q}}}} \right),\mathfrak{\hat W}_{q + 1}^{loc}} \right\rangle }_{{\mathbb{T}^2}}}} \right|  \\
\lesssim & \left| {{{\left\langle {\nabla \left( {{\mathfrak{\hat W}_q} - {{\hat u}_q}} \right),R\mathfrak{\hat W}_{q + 1}^{loc}} \right\rangle }_{{\mathbb{T}^2}}}} \right| + \left| {{{\left\langle {{{\hat w}_{q + 1}},\mathfrak{\hat W}_{q + 1}^{loc}} \right\rangle }_{{\mathbb{T}^2}}}} \right| + \left( {{{\hat l}_{q + 1}} + {{\hat l}_q}} \right){\left\| {\mathfrak{\hat W}_{q + 1}^{loc}} \right\|_0}{\left\| {{u_q}} \right\|_1}  .
\end{align*}
It follows from the It\^{o} isometry that
\begin{align*}
& {\left| {\mathbb{E}\int_0^t {{{\left\langle {{{\hat w}_{q + 1}} + {v_q} * \left( {{{\hat \varsigma }_{{{\hat l}_{q + 1}}}} - {{\hat \varsigma }_{{{\hat l}_q}}}} \right), \mathrm{d}{\mathfrak{W}_{q + 1}}} \right\rangle }_{{\mathbb{T}^2}}}} } \right|^2} + {\left| {\mathbb{E}\int_0^t {{{\left\langle {{{\hat v}_q},\mathrm{d}\mathfrak{W}_{q + 1}^{loc}} \right\rangle }_{{\mathbb{T}^2}}}} } \right|^2} \\
\lesssim & \left\| {\mathcal{R}{{\hat w}_{q + 1}}} \right\|_0^2\sum\limits_{j \leqslant q + 1} {\delta _{j + 1}^{\frac{1}{2}}\left\| {\cos \left( {{\lambda _j}{x_1}} \right){{\mathbf{e}}_2}} \right\|_1^2}  + {\left( {{{\hat l}_{q + 1}} + {{\hat l}_q}} \right)^2}\left\| {{v_q}} \right\|_1^2  + {\delta _{q + 2}}\left\| {\mathcal{R}\cos \left( {{\lambda _{q + 1}}{x_1}} \right){{\mathbf{e}}_2}} \right\|_0^2\left\| {{v_q}} \right\|_1^2  .
\end{align*}
Therefore, combining the continuity argument in the temporal interval $\left[ {0,\mathfrak{t}} \right]$, we obtain
\begin{align}\label{RSRI-P-mene-GAP-P-A-6}
\left| {\mathfrak{E}_{q + 1}^{\left( {sto} \right)}} \right| \lesssim & {{\hat l}_q}{\delta _{q + 1}}{\lambda _q} + {{\hat l}_q}\delta _q^{\frac{1}{2}}\delta _{q + 1}^{\frac{1}{2}}{\lambda _q} + \hat l_q^2{\delta _q}\lambda _q^2 + \delta _q^{\frac{1}{2}}\delta _{q + 1}^{\frac{1}{2}}\frac{{{\lambda _q}}}{{\lambda _{q + 1}^{1 - \tilde \alpha }}} \nonumber \\
& + {\delta _q}{\delta _{q + 1}}\frac{{\lambda _q^2}}{{\lambda _{q + 1}^{2 - 2\tilde \alpha }}} + \hat l_q^2{\delta _q}\lambda _q^2 + {\delta _{q + 2}}{\delta _q}\frac{{\lambda _q^2}}{{\lambda _{q + 1}^{2 - 2\tilde \alpha }}}  \nonumber \\
\lesssim & {\delta _{q + 2}}  .
\end{align}
\par
Finally, it follows from \eqref{RSRI-P-mene-GAP-P-A-1}, \eqref{RSRI-P-mene-GAP-P-A-2}, \eqref{RSRI-P-mene-GAP-P-A-3}, \eqref{RSRI-P-mene-GAP-P-A-5} and \eqref{RSRI-P-mene-GAP-P-A-6} that \eqref{RSRI-P-mene-GAP} holds. The proof is complete. \qed


\addtocontents{toc}{\protect\setcounter{tocdepth}{2}}   

\vspace{4mm}

\section*{Declarations}
\noindent \textbf{Acknowledgements}: This work is supported by NSFC (Grants 12531009, 11925102 and 12171343).

\vspace{2mm}



\noindent \textbf{Conflict of Interest Statement}: The authors have no conflicts to disclose.

\vspace{2mm}

\noindent \textbf{Data Availability}: No data was used for the research described in the article.

\vspace{4mm}





\appendix
\addcontentsline{toc}{section}{Appendices} 
\addtocontents{toc}{\protect\setcounter{tocdepth}{0}}
\section{H\"{o}lder spaces and transport estimates}\label{AAA-A}
Suppose that $f\left( x \right):\ {\mathbb{T}^2} \longrightarrow \mathbb{R}^2$ is a vector function. For $N \in {\mathbb{Z}^ + } \cup \left\{ 0 \right\}$ and $\alpha \in \left[ {0,1} \right)$, define the semi-norm
\begin{equation*}
{\left[ f \right]_N} \triangleq \mathop {\sup }\limits_{\left| \theta  \right| = N} \mathop {\sup }\limits_{x \in {\mathbb{T}^2}} \left| {{D^\theta }f\left( x \right)} \right|,
\end{equation*}
and the H\"{o}lder semi-norm
\begin{equation*}
{\left[ f \right]_{N + \alpha }} \triangleq \mathop {\sup }\limits_{\left| \theta  \right| = N} \mathop {\sup }\limits_{x,y \in {\mathbb{T}^2},x \ne y} \frac{{\left| {{D^\theta }f\left( x \right) - {D^\theta }f\left( y \right)} \right|}}{{{{\left| {x - y} \right|}^\alpha }}} ,
\end{equation*}
where $\theta $ is an integer multi-index, and ${{D^\theta }}$ means the derivative operator. Let ${C^N}\left( {{\mathbb{T}^2}} \right)$ be the space of continuous functions, whose norm is
\begin{equation*}
{\left\| f \right\|_N} \triangleq \sum\limits_{n = 0}^N {{{\left[ f \right]}_n}}  <   \infty  ,  \ \ \forall \,  f \left( x \right) \in {C^N}\left( {{\mathbb{T}^2}} \right) .
\end{equation*}
Let ${C^{N + \alpha}}\left( {{\mathbb{T}^2}} \right)$ be the H\"{o}lder space with the norm
\begin{equation*}
{\left\| f \right\|_{N + \alpha }} \triangleq {\left\| f \right\|_N} + {\left[ f \right]_{N + \alpha }} <   \infty , \ \ \forall \,  f \left( x \right) \in {C^{N + \alpha}}\left( {{\mathbb{T}^2}} \right).
\end{equation*}
\par
\begin{lemma}[\cite{Zbl1556.35231}]\label{STN-HCS-A}
Assume that $f\left( x \right)$ and $g\left( x \right)$ belong to ${C^{N + \alpha}}\left( {{\mathbb{T}^2}} \right)$ for $N \in {\mathbb{Z}^ + } \cup \left\{ 0 \right\} $ and $\alpha \in \left[ {0,1} \right)$. If $f\left( x \right)g\left( x \right)$ is in ${C^{N + \alpha}}\left( {{\mathbb{T}^2}} \right)$, we then have
\begin{equation*}
{\left\| {fg} \right\|_{N + \alpha }} \lesssim {\left\| f \right\|_{N + \alpha }}{\left\| g \right\|_0} + {\left\| f \right\|_0}{\left\| g \right\|_{N + \alpha }}  .
\end{equation*}
\end{lemma}
\par
\begin{lemma}[\cite{Zbl1307.35205}]\label{ONCS}
Let $f:\ {\mathbb{R}^m}  \longrightarrow \mathbb{R}$ and $u:\ {\mathbb{R}^n} \longrightarrow {\mathbb{R}^m} $ be two smooth functions. There then exists a positive constant $C\left( {n,m,N} \right)$ such that
\begin{equation*}
{\left[ {f\left( u \right)} \right]_N} \leqslant C\left( {n,m,N} \right)\left( {{{\left[ f \right]}_1}{{\left\| {Du} \right\|}_{N - 1}} + {{\left\| {Df} \right\|}_{N - 1}}\left[ u \right]_1^N} \right) , \ \ \forall \, N \in {\mathbb{Z}^ + }  .
\end{equation*}
\end{lemma}
\par
\begin{lemma}[\cite{Zbl1255.53038}]\label{MMD-A-H}
Let ${\varsigma _l}$ be the symmetric mollifier with the scale $l > 0$. Then
\begin{equation*}
{\left\| {f - f*{\varsigma _l}} \right\|_N} \lesssim {l^2}{\left\| f \right\|_{N + 2}} , \ \ \forall \, N \in {\mathbb{Z}^ + } \cup \left\{ 0 \right\} ,
\end{equation*}
and
\begin{equation*}
{\left\| {\left( {fg} \right)*{\varsigma _l} - \left( {f*{\varsigma _l}} \right)\left( {g*{\varsigma _l}} \right)} \right\|_N} \lesssim {l^{2 - N + M}}\left( {{{\left\| f \right\|}_{M + 1}}{{\left\| g \right\|}_1} + {{\left\| f \right\|}_1}{{\left\| g \right\|}_{M + 1}}} \right) , \ \ \forall \, N , M \in {\mathbb{Z}^ + } \cup \left\{ 0 \right\} .
\end{equation*}
\end{lemma}
\par
We state the standard estimates for the transport equation driven by the velocity $u\left( t, x \right)$.
\par
\begin{lemma}[\cite{MR3302631,MR3374958}]\label{TSP-E-L-A-1}
Suppose $\left| {t - {t_0}} \right|{\left\| u \right\|_1} \lesssim 1$. Let $w{\left( t ,x \right)}$ be the solution of the transport equation
\begin{equation*}
\left\{ {\begin{array}{*{20}{l}}
{\partial _t}w + \left( {v \cdot \nabla } \right)w = u ,   &\mbox{on}\ \left[ {{t_0}, T } \right]  \times  \mathbb{T}^2, \\
w\left( {{t_0}} \right) = {w_0} ,
\end{array}} \right.
\end{equation*}
where $v\left( t, x \right)$ and $u\left( t, x \right)$ are the vector functions on $\left[ {{t_0}, T } \right]  \times   \mathbb{T}^2$. Then, for any $\alpha \in \left[ {0,1} \right)$ and $N \in {\mathbb{Z}^ + } $, we have
\begin{equation}\label{KSIDJIUF-A-1}
{\left\| {w\left( t \right)} \right\|_\alpha } \lesssim {\left\| {{w_0}} \right\|_\alpha } + \int_{{t_0}}^t {{{\left\| {u\left( s  \right)} \right\|}_\alpha }} \mathrm{d}s  ,
\end{equation}
and
\begin{equation}\label{KSIDJIUF-A-2}
{\left[ {w\left( t \right)} \right]_{N + \alpha }} \lesssim  {\left[ {{w_0}} \right]_{N + \alpha }} + \left| {t - {t_0}} \right|{\left[ v \right]_{N + \alpha }}{\left[ {{w_0}} \right]_1}  + \int_{{t_0}}^t {{{\left[ {u\left( s  \right)} \right]}_{N + \alpha }} + \left| {t - s } \right|{{\left[ v \right]}_{N + \alpha }}{{\left[ {u\left( s  \right)} \right]}_1}} \mathrm{d}s  .
\end{equation}
\end{lemma}
\par
We remark that \eqref{KSIDJIUF-A-1} and \eqref{KSIDJIUF-A-2} hold pathwise for every fixed sample path, when the driving flow $v\left( t, x \right)$ carries intrinsic randomness.

\section{Martingale solutions}\label{TPT-A-1}
For $m \in {\mathbb{Z}^ + } \cup \left\{ 0 \right\}$ and $p \geqslant 1$, let ${W^{m,p}}\left( {{\mathbb{T}^2}} \right)$ be the Sobolev space with the norm ${\left\|  \cdot  \right\|_{{W^{m,p}}}}$. It is clear that ${W^{0,p}}\left( {{\mathbb{T}^2}} \right)$ equals the Lebesgue space ${L^p}\left( {{\mathbb{T}^2}} \right)$. Define the Hilbert space
\begin{equation*}
\mathcal{H} \triangleq \left\{ {u\left( x \right) \in {L^{2}}\left( {{\mathbb{T}^2}} \right) \ \big{|}\  \mathrm{div}\,   u\left( x \right) = 0,\ \int_{{\mathbb{T}^2}} {u\left( {t,x} \right)}\ \mathrm{d}x = \mathbf{0}} \right\},
\end{equation*}
whose inner product is ${\left\langle {u\left( x \right),v\left( x \right)} \right\rangle _{{\mathbb{T}^2}}} = \int_{{\mathbb{T}^2}} {u \left( x \right) \cdot v \left( x \right)} \mathrm{d}x $. Assume that $\mathscr{W}$ is a collection of Wiener processes $\mathfrak{W} \left( {t,x} \right)$ satisfying
\begin{equation}\label{skidjfk-sdvf-3d-a}
\mathfrak{W}\left( {t,x} \right) = \sum\limits_{j \in {\mathbb{Z}^ + } \cup \left\{ 0 \right\}} {\frac{{{\varsigma _j}}}{\sqrt 2 \pi }{\mathfrak{B}_j}\left( t \right)\cos \left( {k_j {x_1}} \right){{\mathbf{e}}_2}} ,\ \ \mbox{for}\  {\varsigma _j} \in \mathbb{R},
\end{equation}  
where ${{\mathfrak{B} _j}\left( t \right)}$ are mutually independent $\mathbb{R}$-valued Brownian motions on the complete probability space $\left( {\Omega ,\mathcal{F},\mathbb{P}} \right)$, and $k_j \in \mathbb{Z}^+$. Suppose that $\sum\limits_{j \in \mathbb{Z}^+} {\varsigma _j^2}   <  \infty$. We then have
\begin{equation*}
\mathbb{E}{\left\langle {{\mathfrak{W}}\left( t \right),{\mathfrak{W}}\left( t \right)} \right\rangle _{{\mathbb{T}^2}}} = t \cdot \mathrm{Tr} \, Q = t \cdot \sum\limits_{j \in \mathbb{Z}^+} {\varsigma _j^2}  <  \infty , \ \ \forall \, t \in {\left[ {0,T} \right]} . 
\end{equation*}
Moreover, in view of the condition $\mathrm{Tr} \left( { {Q^{\frac{1}{2}}}{{\left( { - \Delta } \right)}^{\frac{1 + \gamma }{2} + \rho}}{{\left( {{Q^{\frac{1}{2}}}} \right)}^*}} \right) <  \infty$ for some $\rho > 0$, we have
\begin{equation*}
{\mathfrak{W}}\left( {t,x} \right) \in {L^2}\left( {\Omega ;{W^{1 + \gamma  + 2\rho ,2}}\left( {{\mathbb{T}^2}} \right)} \right) \subset  {L^2}\left( {\Omega ;{C^{\gamma}}\left( {{\mathbb{T}^2}} \right)} \right) , \ \ \forall \, {t \in \left[ {0,T} \right]} .
\end{equation*}
\par
In fact, according to $\mathbb{E}\left( {{\mathfrak{B} _i}\left( t \right){\mathfrak{B} _j}\left( t \right)} \right) = 0$ for any $i \ne j$, it follows from Sobolev embedding inequality that 
\begin{align}
\mathop {\sup }\limits_{t \in \left[ {0,T} \right]}   \mathbb{E}\left\| {\mathfrak{W}\left( t \right)} \right\|_\gamma ^2 \lesssim & \mathop {\sup }\limits_{t \in \left[ {0,T} \right]} \mathbb{E}{\left( {\sum\limits_{j \in \mathbb{Z}^+ \cup \left\{ 0 \right\}} {\frac{{\varsigma _j}}{\sqrt 2 \pi }{\mathfrak{B} _j}\left( t \right)} \left\| {{e_j}} \right\|_\gamma } \right)^2} \notag \\
\lesssim & \mathop {\sup }\limits_{t \in \left[ {0,T} \right]} \mathbb{E}{\left( {\sum\limits_{j \in \mathbb{Z}^+ \cup \left\{ 0 \right\}} {\frac{{\varsigma _j}}{\sqrt 2 \pi }{\mathfrak{B} _j}\left( t \right)} \left\| {{e_j}} \right\|_{{W^{1 + \gamma  + 2\rho ,2}}} } \right)^2} \notag \\
\lesssim & T \cdot \mathrm{Tr} \left( { {Q^{\frac{1}{2}}}{{\left( { - \Delta } \right)}^{\frac{1 + \gamma }{2} + \rho}}{{\left( {{Q^{\frac{1}{2}}}} \right)}^*}} \right)  .\label{SLOK-WLD}
\end{align}
Define the stopping time
\begin{equation}\label{IDJIFHISDF-ADS-1}
{\mathfrak{t}_\mathfrak{S}} \triangleq \inf \left\{ {t \geqslant 0 : {{\left\| \mathfrak{W} \right\|} ^2 _{C\left( {\left[ {0,t} \right];{C^{\gamma}}\left( {{\mathbb{T}^2}} \right)} \right)}} > \mathfrak{S}  T \cdot \mathrm{Tr} \left( { {Q^{\frac{1}{2}}}{{\left( { - \Delta } \right)}^{\frac{1 + \gamma }{2} + \rho}}{{\left( {{Q^{\frac{1}{2}}}} \right)}^*}} \right)} \right\} , \ \ \mbox{for some}\ \mathfrak{S} > 1 .
\end{equation}
It follows from \eqref{SLOK-WLD} that
\begin{equation*}
\mathbb{E} \left( { {\mathfrak{t}_\mathfrak{S}} } \right)  <   \infty ,
\end{equation*}
which implies that the stopping time ${\mathfrak{t}_\mathfrak{S}}$ is well-defined.
\par
\begin{definition} [\cite{Zbl1548.60081}]\label{WMStoBous}
The combination $\left\{ {\left( {{\Omega } ,\mathcal{F}_t,\mathbb{P}} \right),{u }\left( {t,x} \right),\mathfrak{p}{\left( {t , x} \right)},\mathfrak{W}\left( {t,x} \right)} \right\}$ is said to be a martingale solution to Eq.\eqref{2DEEUGE1}, if it satisfies the following conditions.
\par
\noindent$(\rmnum{1})$. The Wiener process $\mathfrak{W} \left( {t} \right)$ generates right-continuous filtration ${{\left\{ {{\mathcal{F}_t}} \right\}}_{t \in \left[ {0,T} \right]}}$, and $\left( {{\Omega } ,\mathcal{F}_t,\mathbb{P}} \right)$ is the complete probability space for any ${t \in \left[ {0,T} \right]}$;
\par
\noindent$(\rmnum{2})$. There exists a stopping time ${\mathfrak{t}_{\mathfrak{S}}}$ such that for any $\psi \left( x \right) \in {C^\infty }\left( {{\mathbb{T}^2}} \right) $ and $t \in \left[ {0,T \wedge {\mathfrak{t}_{\mathfrak{S}}}} \right]$,
\begin{equation*}
{\left\langle {u\left( t \right) - u\left( 0 \right),\psi } \right\rangle _{{\mathbb{T}^2}}} + \int_0^t {{{\left\langle {\left( {u\left( s \right) \cdot \nabla } \right)u\left( s \right) + \nabla \mathfrak{p} \left( s \right),\psi} \right\rangle }_{{\mathbb{T}^2}}}} \mathrm{d}s = {\left\langle {\mathfrak{W}\left( t \right),\psi } \right\rangle _{{\mathbb{T}^2}}} ;
\end{equation*}
\par
\noindent$(\rmnum{3})$. The process ${u}\left( {t} \right)$ is $\mathcal{F}_t$-adapted and satisfies ${{u }\left( {t,x} \right)} \in {L^\infty }\left( {\left[ {0, T \wedge {\mathfrak{t}_{\mathfrak{S}}}} \right];{\mathcal{H}}} \right)$;
\par
\noindent$(\rmnum{4})$. The process ${\left\langle {\mathfrak{W}\left( t \right),\psi } \right\rangle _{{\mathbb{T}^2}}}$ is a martingale.
\end{definition}
\section{Selection of parameters}\label{TPT-A-1-hu}
For any $\gamma \in \left( {0 ,\frac{1}{3}} \right)$, we choose the exponent $\beta \in \left( {\gamma ,\frac{1}{3}} \right)$ and the Newtonian index $\Gamma  = \left\lceil {{{\left( {\frac{1}{3} + \alpha - \beta } \right)}^{ - 1}}} \right\rceil $ with $\alpha \in \left( {{\bar \Gamma}^{-1},\beta} \right)$, where the positive integer ${\bar \Gamma}$ denotes the maximum bound of Newton iteration steps, satisfying $\bar \Gamma  < \left\lfloor {\frac{2}{{1 - 3\beta }}} \right\rfloor $. For $a > 1 $ and $b > 1$, we define the inductive parameters
\begin{equation}\label{parameters-S-2}
{\lambda _q} = \left\lceil {{a^{{b^q}}}} \right\rceil  , \ \ {\delta _q} = \lambda _q^{ - 2\beta } ,\ \ \forall \, q \in \mathbb{Z}^+  \cup \left\{ 0 \right\}   .
\end{equation}
By direct computation, we obtain
\begin{equation}\label{parameters-S-14}
\delta _q^{\frac{1}{2}}{\lambda _q} = \lambda _q^{1 - \beta } \leqslant \lambda _{q + 1}^{1 - \beta } = \delta _{q + 1}^{\frac{1}{2}}{\lambda _{q + 1}} .
\end{equation}
By the exponential growth of $\lambda _q$, we choose a sufficiently large parameter $a$ in \eqref{parameters-S-2} such that
\begin{equation}\label{parameters-S-20-ENDH-skdi-A}
\sum\limits_{j = 0}^q {{a^{\left( {N - \beta } \right){b^j}}}}  \leqslant 2 {a^{\left( {N - \beta } \right){b^{q + 1}}}} , \ \ \forall \, q \in \mathbb{Z}^+ \cup \left\{ 0 \right\} , \ N \in \left\{ {1, \cdots ,10} \right\}  ,
\end{equation}
which implies that (see \cite{MR3302631} for a similar setup)
\begin{equation}\label{parameters-S-20-ENDH-A}
\sum\limits_{j = 0}^q {{\delta _j}\lambda _j^{2N}}  \lesssim {\delta _{q + 1}}\lambda _{q + 1}^{2N} , \ \ \forall \, q \in \mathbb{Z}^+ \cup \left\{ 0 \right\} , \ N \in \left\{ {1, \cdots ,10} \right\}  .
\end{equation}
Moreover, there exist constants $\tilde \alpha  \in \left( {0,\beta } \right)$ and $ \alpha \in \left( {\frac{{2b + 1}}{3 } \tilde \alpha, \beta} \right)$ such that
\begin{equation}\label{parameters-S-20}
\lambda _q^{\tilde \alpha } {\left( {\frac{{{\lambda _q}}}{{\lambda _{q + 1}}}} \right)^{ - \tilde \alpha \left( {\frac{1}{3} - \beta } \right)}} = \lambda _q^{\tilde \alpha \left( {\frac{2}{3} + \beta } \right)}\lambda _{q + 1}^{\tilde \alpha b\left( {\frac{1}{3} - \beta } \right)}  \lesssim \lambda _{q + 1}^\alpha .
\end{equation}
\par
Fixing $1 < b \leqslant \min \left\{ {\frac{{1 - \alpha }}{{2\beta }} ,1 + 6\beta } \right\}$, we have
\begin{equation}\label{parameters-S-202568}
{\delta _{q + 1}}\lambda _q^{ - \alpha } \lesssim {\delta _{q + 2}}  ,
\end{equation}
\begin{equation}\label{parameters-S-21-A-2}
\delta _q^{\frac{3}{2}}\delta _{q + 1}^{\frac{1}{2}}  \lambda _q^{\frac{2}{3}}\lambda _{q + 1}^{2\alpha  - \frac{2}{3}} \lesssim {\delta _{q + 2}} , 
\end{equation}
\begin{equation}\label{parameters-S-21}
{\delta _q}  {\delta _{q + 1}}  {\left( {\frac{{{\lambda _q}}}{{\lambda _{q + 1}}}} \right)^{\frac{1}{3} - \beta }}\lambda _{q + 1}^{2\alpha } \lesssim {\delta _{q + 2}} ,
\end{equation}
and
\begin{equation}\label{parameters-S-22}
{\delta _{q + 1}}\lambda _q^{N + 1 - \alpha }\lambda _{q + 1}^{\alpha  - N - 1} \lesssim {\delta _{q + 2}},\ \ \forall \, N \in \left\{ {0,1, \cdots ,10} \right\} .
\end{equation}
We note that ${\lambda _q}$ is an increasing function with respect to $q$. To constrain the increasing speed, we choose $b \in \left( {1.005,1.3} \right]$ such that
\begin{equation*}
{\lambda _{q + 1}} \lesssim \lambda _{q }^{\min \left\{ {2,\frac{{1 + \beta }}{{3\beta }}} \right\}} \lesssim \lambda _q^{\frac{1}{\beta }} .
\end{equation*}
Here, an explicit range of $b$ is selected to guarantee \eqref{parameters-S-20-ENDH-skdi-A}, \eqref{parameters-S-20} and \eqref{parameters-S-21}, simultaneously. Then, there exists a parameter $\alpha \in \left( {\frac{\beta }{{b}},\beta } \right)$ such that
\begin{equation}\label{RSRI-main-PROVE-2-JIJSHU-posksij-a-1}
\delta _q^{\frac{1}{2}}\lambda _{q + 1}^{\alpha } \gtrsim 1  , 
\end{equation}
\begin{equation}\label{RSRI-main-PROVE-2-JIJDHUFH-A}
\delta _{q + 1}^{\frac{1}{2} - \frac{{\alpha  }}{{2\beta }}} \lambda _q^{ - \frac{2}{3}} \lambda _{q + 1}^{ - \frac{1}{3}} \lambda _{q + 2}^{2\beta } \lesssim 1 , 
\end{equation}
and
\begin{equation}\label{parameters-S-19}
\delta _q^{\frac{3}{2}}\delta _{q + 1}^{\frac{{1 - \alpha }}{{2\beta }}}  {\lambda _q}\lambda _{q + 2}^{2\beta } \lesssim 1 .
\end{equation}
\par
Define the mollifier scales
\begin{equation}\label{parameters-S-9}
{l_q} = {\lambda _q^{ - \frac{1}{2}}\lambda _{q + 1}^{ - \frac{1}{2}}} , 
\end{equation}
\begin{equation}\label{parameters-S-11}
{{\tilde l}_{q}} = \delta _q^{ - \frac{1}{2}}\lambda _q^{ - \frac{1}{3}}\lambda _{q + 1}^{ - \frac{2}{3}} ,
\end{equation}
\begin{equation}\label{parameters-S-11-A-1}
{{\hat l}_q} = \delta _q^{ - \frac{1}{8}}\lambda _q^{ - \frac{3}{4}}\lambda _{q + 1}^{ - \frac{2}{3}} ,
\end{equation}
the temporal scale
\begin{equation}\label{parameters-S-6}
{\tau _q} = \delta _q^{ - \frac{1}{2}}\lambda _q^{ - 1}\lambda _{q + 1}^{ - \alpha } ,
\end{equation}
and the temporal oscillation parameter
\begin{equation}\label{PASJD-MU-1}
{\mu _{q + 1}} = \delta _{q + 1}^{\frac{1}{2}}\lambda _q^{\frac{2}{3}}\lambda _{q + 1}^{\frac{1}{3} + \alpha },
\end{equation}
where $\alpha$ is determined by \eqref{parameters-S-20} and \eqref{RSRI-main-PROVE-2-JIJSHU-posksij-a-1}-\eqref{parameters-S-19}. It is easy to verify that
\begin{equation}\label{parameters-S-30}
{\tau _q}\delta _q^{\frac{1}{2}}{\lambda _q} = \lambda _{q + 1}^{ - \alpha } \leqslant 1 ,
\end{equation}
\begin{equation}\label{parameters-S-8}
\tau _q^{ - 1}\mu _{q + 1}^{ - 1} = \delta _q^{\frac{1}{2}}\delta _{q + 1}^{ - \frac{1}{2}}\lambda _q^{\frac{1}{3}}\lambda _{q + 1}^{ - \frac{1}{3}} = \tilde l_q^{ - \alpha }\lambda _q^{\left( {\alpha  - 1} \right)\left( {\beta  - \frac{1}{3}} \right)}\lambda _{q + 1}^{\beta  - \frac{1}{3} - \frac{{2\alpha }}{3}} \lesssim \tilde l_q^{ - \alpha } \lesssim l_q^{ - \alpha } ,
\end{equation}
and for any $N \in \left\{ {0,1, \cdots ,10} \right\} $,
\begin{equation}\label{parameters-S-13}
\mu _{q + 1}^{ - 1}{\delta _{q + 1}}\lambda _q^{N + 1}l_q^{ - \alpha } \lesssim \delta _q^{\frac{1}{2}}\lambda _q^N ,
\end{equation}
\begin{equation}\label{parameters-S-15}
\tau _q^{ - 1}\left( {\lambda _q^N + \lambda _{q + 1}^N} \right) = \delta _q^{\frac{1}{2}}\lambda _q^{N + 1 + \alpha } + \delta _q^{\frac{1}{2}}\lambda _q^{1 + \alpha }\lambda _{q + 1}^N \leqslant \lambda _{q + 1}^{N + 1} ,
\end{equation}
\begin{equation}\label{parameters-S-16}
{\mu _{q + 1}}\lambda _{q + 1}^N \leqslant \lambda _q^{\frac{2}{3}}\lambda _{q + 1}^{N + \frac{1}{3} + \alpha  - \beta } \leqslant \lambda _{q + 1}^{N + 1} .
\end{equation}
\par
Define the iteration parameter
\begin{equation}\label{parameters-S-4}
{\delta _{q + 1,n}} = {\delta _{q + 1}}\lambda _q^{\left( {\frac{1}{3} - \beta } \right)n}\lambda _{q + 1}^{\left( {\beta - \frac{1}{3}} \right)n},\ \ \forall \,  n \in \left\{ {1,2, \cdots ,\Gamma } \right\}.
\end{equation}
It is easy to see that ${\delta _{q + 1,n}} \leqslant {\delta _{q + 1}}$. Furthermore, for any $N \in \left\{ {0,1, \cdots ,10} \right\} $, we have
\begin{equation}\label{parameters-S-10}
{\tau _q}\mu _{q + 1}^{ - 1}{\delta _{q + 1,n}}\lambda _q^{N + 1} + \tau _q^2\mu _{q + 1}^{ - 1}\delta _q^{\frac{1}{2}}{\delta _{q + 1,n}}\lambda _q^{N + 2} \lesssim \tau _q^2\mu _{q + 1}^{ - 1}{\delta _q}{\delta _{q + 1,n}}l_q^{ - \alpha }\lambda _q^{N + 3} ,
\end{equation}
and
\begin{equation}\label{parameters-S-17}
\mu _{q + 1}^{ - 1}{\delta _{q + 1,n}}\lambda _q^{N + 1}l_q^{ - \alpha } = {\delta _{q + 1,n}}\lambda _q^{N + \frac{1}{3} + \frac{1}{2}\alpha }\lambda _{q + 1}^{\beta  - \frac{1}{2}\alpha  - \frac{1}{3}} \lesssim \delta _{q }^{\frac{1}{2}}\lambda _{q }^N .
\end{equation}
\par
Define the intermittent scale 
\begin{equation}\label{Phsudj-2}
\tilde \mu  = \delta _{q + 1}^{\frac{1}{2}}\lambda _q^{\frac{2}{3}}\lambda _{q + 1}^{\frac{1}{3}} \in \left[ {1,{\mu _{q+1}}} \right) .
\end{equation}
Then, it follows from \eqref{RSRI-main-PROVE-2-JIJDHUFH-A} that
\begin{equation}\label{RSRI-main-PROVE-2}
{{\tilde \mu }^{ - 1}}\delta _{q + 1}^{{\frac{1}{2}} - \frac{\alpha }{{2\beta }}}\lambda _{q + 2}^{2\beta } \lesssim 1 , 
\end{equation}
Notice that the intermittent scale $\tilde \mu$ is dominated by the temporal oscillation parameter ${\mu _{q+1}}$. Hence, the oscillations associated with $\tilde \mu$ in the Nash perturbation and the Reynolds stress are estimated by the argument similar to ${\mu _{q+1}}$.
\section{Intermittent building blocks}\label{NSAH-BSHDW-US}
To characterize the intermittent behavior, Frisch \cite[Section $8.2$]{TurbulenceFrisch} introduced a fundamental intermittency criterion: the $L^2$-norm of intermittent functions associated with high-frequency filtered signals is substantially smaller than their $L^4$-norm. Inspired by this observation, we formulate the general $\left( {p_1,p_2} \right)$-type definition of intermittency as follows.
\par
\begin{definition}[\cite{TurbulenceFrisch}]\label{DUFJ-INT-SJ-A}
A random function $v\left( {x} \right)$ is said to be intermittent if for two positive exponents ${p_1} > {p_2}$, there exists an increasing function $\mathfrak{h}_1\left( r \right)$ such that
\begin{equation}\label{D-INT-A-1}
\frac{{{{\left\| {{{\dot \Delta }_r}u} \right\|}_{{L^{{p_1}}}}}}}{{{{\left\| {{{\dot \Delta }_r}u} \right\|}_{{L^{{p_2}}}}}}}  \sim  \mathfrak{h}_1\left( r \right) \to  \infty , \ \ \mbox{as}\ r \to  \infty  , \ \mathbb{P}\mbox{-a.s.},
\end{equation}
where ${{{\dot \Delta }_r}}$ is the Littlewood--Paley projector, and $r$ is the cutoff wavenumber associated with the Fourier transform of $v\left( {x} \right)$.
\end{definition}
\par
The flatness \eqref{D-INT-A-1} characterizes intermittent influence on the function itself, yet it fails to capture intermittent behavior within the inertial range of turbulence. In physical experiments, Batchelor and Townsend \cite{Zbl0036.25602} discovered inertial-range intermittency, which was later systematically elaborated by Frisch and Morf \cite{Zbl0563.76057,IntermittencyFrisch}. After their Littlewood--Paley analysis in \cite{Zbl1295.76010}, Cheskidov and Shvydkoy \cite{Zbl1512.76040} proposed the $\left( {p_1,p_2} \right)$-type flatness $\frac{{{{\left\| {v\left( { \cdot  + y} \right) - v\left(  \cdot  \right)} \right\|}_{{L^{{p_1}}}}}}}{{{{\left\| {v\left( { \cdot  + y} \right) - v\left(  \cdot  \right)} \right\|}_{{L^{{p_2}}}}}}}$ as a measurement of intermittency. Then, the $\left( {p_1,p_2} \right)$-type definition of inertial-range intermittency for isotropic turbulence is given as follows. 
\par
\begin{definition}[\cite{TurbulenceFrisch}]\label{DUFJ-INT-SJ-B}
A random function $v\left( {x} \right)$ is said to be inertial-range intermittent if for two positive exponents ${p_1} > {p_2}$, there exist increasing function $\mathfrak{h}_2\left( r \right)$ and dissipation scale $l_D \gtrsim r^{ - 1} $ such that
\begin{equation}\label{D-INT-A-2}
\frac{{{{\left\| {v\left( { \cdot  + l\mathbf{n}} \right) - v\left(  \cdot  \right)} \right\|}_{{L^{{p_1}}}}}}}{{{{\left\| {v\left( { \cdot  + l\mathbf{n}} \right) - v\left(  \cdot  \right)} \right\|}_{{L^{{p_2}}}}}}} \sim {\mathfrak{h}_2}\left( {r} \right) \to  \infty , \ \  \forall \, l \in \left[ {{l_D},{l_I}} \right] , \  \mbox{as}\ r \to  \infty , \ \mathbb{P}\mbox{-a.s.} 
\end{equation}
\end{definition}
\par
We now introduce the intermittent analogues inspired by Beltrami waves, and we simply call them intermittent building blocks for convenience. Let $\xi$ and $\Lambda $ be determined by Lemma \ref{GLTNNRS-A}. The simplified building block reads (see \cite[Section $4.2$]{Zbl1556.35231})
\begin{equation}\label{UJIJY-DJI}
{W_\xi }\left( x \right) = {a_W}\left( {\exp \left( {\mathrm{i}{{\xi ^ \bot }  } \cdot x} \right) + \exp \left( { - \mathrm{i}{{\xi ^ \bot }  } \cdot x} \right)} \right){\xi }, \ \ \mbox{with}\  {\xi ^ \bot } \cdot \xi  = 0 , \ {\xi ^ \bot } \neq \mathbf{0} , \ \xi \in \Lambda ,
\end{equation}
where ${a_W}$ is the normalization constant for the $L^2$-norm, and $\mathrm{i}$ is the unit imaginary number. The stream function of ${W_\xi }\left( x \right)$ is given by
\begin{equation}\label{IUHYU-DJ}
{\mathfrak{S}_\xi }\left( x \right) = \mathrm{i}{a_W}\left( {\exp \left( {\mathrm{i}{{\xi ^ \bot }  } \cdot x} \right) - \exp \left( { - \mathrm{i}{{\xi ^ \bot }  } \cdot x} \right)} \right)  .
\end{equation}
Then, it is easy to verify that
\begin{equation}\label{BWsdfgg-P-A-1}
{W_\xi }\left( x \right) = {\nabla ^ \bot }{\mathfrak{S}_\xi } \left( x \right) .
\end{equation}
\par
\begin{lemma}\label{BW-P-A-1}
The simplified building block ${W_\xi }\left( x \right)$ is real-valued, and admits
\begin{equation*}
\mathrm{div}\,   {W_\xi } = 0 , \ \ \int_{{\mathbb{T}^2}} {{W_\xi }\left( x \right) \otimes {W_\xi }\left( x \right)}\ \mathrm{d}x = {\xi } \otimes {\xi} ,
\end{equation*}
and
\begin{equation*}
\mathrm{div}\left( {{W_\xi } \otimes {W_\xi }} \right) = \left( {{W_\xi } \cdot \nabla } \right){W_\xi } = \mathbf{0} .
\end{equation*}
\end{lemma}
\par
\noindent{\textbf{Proof}}. 
In view of the Euler's formula, we have ${W_\xi }\left( x \right) = 2{a_W}{\xi  }\cos \left( {{\xi ^ \bot }  \cdot x} \right)$, which implies that ${W_\xi }\left( x \right)$ is real-valued. We further calculate that
\begin{equation*}
\mathrm{div}\,   {W_\xi } =  - 2{a_W}\left( {{\xi ^ \bot } \cdot \xi } \right)\sin \left( {{\xi ^ \bot }  \cdot x} \right) = 0 ,
\end{equation*}
\begin{equation*}
\int_{{\mathbb{T}^2}} {{W_\xi }\left( x \right) \otimes {W_\xi }\left( x \right)}\ \mathrm{d}x = {\xi  } \otimes {\xi  }\int_{{\mathbb{T}^2}} {4a_W^2{{\cos }^2}\left( {{\xi ^ \bot }  \cdot x} \right)}\ \mathrm{d}x = {\xi  } \otimes {\xi  } ,
\end{equation*}
and
\begin{equation*}
\mathrm{div}\left( {{W_\xi } \otimes {W_\xi }} \right) = - 4a_W^2\left( {{\xi ^ \bot } \cdot \xi } \right){\xi  }\sin \left( {{\xi ^ \bot }  \cdot x} \right)\cos \left( {{\xi ^ \bot }  \cdot x} \right) = \mathbf{0} .
\end{equation*}
The proof is complete. \qed
\par
Let ${\mathfrak{H}_r} \triangleq \left\{ {\xi  = \left( {j,k} \right) \ \mid \ j,k \in \left\{ { - r, \cdots ,r} \right\}, \ r \in \mathbb{Z}^+  \cup \left\{ 0 \right\}} \right\}$ be the two-dimensional integer cube. We choose a sufficiently large $r$ so that the decomposition set $\Lambda$ in Lemma \ref{GLTNNRS-A} satisfies $\Lambda   \subset {\mathfrak{H}_r}$. Define the Dirichlet kernel as
\begin{equation}\label{XINHUI-SJI-A}
{\mathfrak{D}_r}\left( x \right) = {\left( {2r + 1} \right)^{ - 1}}\sum\limits_{\tilde \xi  \in {\mathfrak{H}_r}} {\exp \left( {\mathrm{i}\tilde \xi  \cdot x} \right)} .
\end{equation}
Then, we readjust the scale of ${\mathfrak{D}_r}\left( x \right)$ to construct the intermittent Dirichlet kernel
\begin{equation}\label{IBW-GJI-A-1}
{\mathfrak{D}_{\tilde \mu }}\left( {t,x} \right) = {\mathfrak{D}_r}\left( { \sigma \left( {\tilde \xi  \cdot x + \tilde \mu t} \right), \sigma {{\tilde \xi }^ \bot } \cdot x} \right)   ,
\end{equation}
where ${{\tilde \xi }^ \bot } \cdot \tilde \xi  = 0$, ${\left| {{{\tilde \xi }^ \bot }} \right|^2} = 1$, and the intensity parameter ${ \sigma }$ is a positive real number. We have
\begin{equation}\label{IBW-GJI-A-3}
{\partial _t}{\mathfrak{D}_{\tilde \mu }}\left( {t,x} \right) = {\tilde \mu } \left( {\tilde \xi  \cdot \nabla } \right){\mathfrak{D}_{\tilde \mu }}\left( {t,x} \right) ,
\end{equation}
and
\begin{equation}\label{IBW-GJI-A-4}
{\partial _t}\mathfrak{D}_{\tilde \mu }^2\left( {t,x} \right) = {\tilde \mu } \left( {\tilde \xi  \cdot \nabla } \right)\mathfrak{D}_{\tilde \mu }^2\left( {t,x} \right) .
\end{equation}
An analogue of Beltrami wave is given by
\begin{equation}\label{IBW-GJI-A-2}
{\mathbb{W}_{\xi }}\left( {t,x} \right) = {\mathfrak{D}_{\tilde \mu }}\left( {t,x} \right){W_\xi }\left( x \right) .
\end{equation}
\par
The building block ${\mathbb{W}_{\xi }}\left( {t,x} \right)$ defined above inherits the temporally oscillating intermittent structure of original three-dimensional Beltrami waves, while it is not a genuine Beltrami wave in the planar setting. In this paper, we do not need the collinearity condition between vorticity and wave velocity. In the following, we characterize divergence, convective contribution and intermittent structure of ${\mathbb{W}_{\xi }}\left( {t,x} \right)$.
\par
\begin{lemma}\label{WPD-WIEJDUHFJ-HUD}
Suppose that $m$ and $n$ are two non-negative integers. Then, there exists a bounded function $\mathfrak{f}\left( t \right)$ on $\left[ 0,T \right]$ such that
\begin{equation}\label{WIEJDUHFJ-DFIJDJ-A-1}
{\left\| {{{\nabla ^m}\partial _t^n}{\mathfrak{D}_{\tilde \mu }}} \right\|_{{L^p}}} \sim \left| {\mathfrak{f}\left( t \right)} \right|  {\sigma} ^{m + n}{{\tilde \mu }^n}{r^{1 + m - \frac{2}{p}}} , \ \ \forall \, p > 1 ,
\end{equation}
and
\begin{equation}\label{WIEJDUHFJ-DFIJDJ-A-2}
{\left\| {\sum\limits_{\xi  \in \Lambda } {{\nabla ^m}\partial _t^n{\mathbb{W}_\xi }} } \right\|_{{L^p}}} \sim  \left| {\mathfrak{f}\left( t \right)} \right|  {\sigma} ^{m + n}{{\tilde \mu }^n}{r^{1 + m - \frac{2}{p}}} , \ \ \forall \, p > 1  .
\end{equation}
\end{lemma}
\par
\noindent{\textbf{Proof}}.
\textit{Case} $1$. If the exponent $p$ is a positive integer, then
\begin{equation*}
\left\| {{\mathfrak{D}_r}} \right\|_{{L^p}}^p = {\left( {2r + 1} \right)^{ - p}}\int_{{\mathbb{T}^2}} {{{\left| {\sum\limits_{\tilde \xi  \in {\mathfrak{H}_r}} {\exp \left( {\mathrm{i}\tilde \xi  \cdot x} \right)} } \right|}^p}}\ \mathrm{d}x = \left| {{\mathbb{T}^2}} \right|\left| {{N_r}\left( p \right)} \right|{\left( {2r + 1} \right)^{ - p}} ,
\end{equation*}
where $\left| {{N_r}\left( p \right)} \right|$ means the amount of resonant set ${N_r}\left( p \right) \triangleq \left\{ {\left( {{\tilde \xi _1}, \cdots ,{\tilde \xi _p}} \right) \in \mathfrak{H}_r^p : {\tilde \xi _1} +  \cdots  + {\tilde \xi _p} = \mathbf{0}} \right\} $, satisfying $\left| {{N_r}\left( p \right)} \right| \sim {r^{2p - 2}}$. Then
\begin{equation}\label{WPD-prove-HUD-1}
{\left\| {{\mathfrak{D}_r}} \right\|_{{L^p}}} \sim {r^{1 - \frac{2}{p}}}  .
\end{equation}
Since ${\left( {\tilde \xi ,{{\tilde \xi }^ \bot }} \right)}$ is the orthonormal basis of ${{\mathfrak{H}_r}}$, the linear combination ${\tilde \xi  \cdot \left( {\tilde \xi ,{{\tilde \xi }^ \bot }} \right)}$ is an orthogonal basis of ${{\mathfrak{H}_r}}$. Thus, the resonance condition for ${\mathfrak{D}_{\tilde \mu }}\left( {x} \right)$ corresponds to linearization of that for ${\mathfrak{D}_r}\left( x \right)$. There exists a bounded function $\mathfrak{f}\left( t \right)$ associated with the Fourier basis $\exp \left( {\mathrm{i} \sigma {{\tilde \xi }_1}\tilde \mu t} \right)$ such that
\begin{align*}
\left\| {{\mathfrak{D}_{\tilde \mu }}} \right\|_{{L^p}}^p = & {\left( {2r + 1} \right)^{ - p}}\left\| {\sum\limits_{\tilde \xi  \in {\mathfrak{H}_r}} {\exp \left( {\mathrm{i}\tilde \xi  \cdot \left( { \sigma \left( {\tilde \xi  \cdot x + \tilde \mu t} \right), \sigma {{\tilde \xi }^ \bot } \cdot x} \right)} \right)} } \right\|_{{L^p}}^p  \\
= & {\left( {2r + 1} \right)^{ - p}}\int_{{\mathbb{T}^2}} {{{\left| {\sum\limits_{\tilde \xi  \in {\mathfrak{H}_r}} {\exp \left( {\mathrm{i} \sigma \tilde \xi  \cdot \left( {\tilde \xi ,{{\tilde \xi }^ \bot }} \right) \cdot x} \right)\exp \left( {\mathrm{i} \sigma {{\tilde \xi }_1}\tilde \mu t} \right)} } \right|}^p}}\ \mathrm{d}x \\
\sim & \left| {\mathfrak{f}\left( t \right)} \right|^p \left\| {{\mathfrak{D}_r}} \right\|_{{L^p}}^p .
\end{align*}
Employing \eqref{WPD-prove-HUD-1}, we obtain
\begin{equation*}
{\left\| {{\mathfrak{D}_{\tilde \mu }}} \right\|_{{L^p}}} \sim \left| {\mathfrak{f}\left( t \right)} \right| {r^{1 - \frac{2}{p}}}  ,
\end{equation*}
which implies that
\begin{equation*}
{\left\| {{\nabla ^m}\partial _t^n{\mathfrak{D}_{\tilde \mu }}} \right\|_{{L^p}}} \sim \left| {\mathfrak{f}\left( t \right)} \right| {\sigma} ^m{{\tilde \mu }^n}{\left| {{\mathfrak{H}_r}\backslash \left\{ 0 \right\}} \right|^m}{\left\| {{\mathfrak{D}_{\tilde \mu }}} \right\|_{{L^p}}} \sim \left| {\mathfrak{f}\left( t \right)} \right| {\sigma} ^{m + n}{{\tilde \mu }^n}{r^{1 + m - \frac{2}{p}}} .
\end{equation*}
\par
According to \eqref{UJIJY-DJI}, \eqref{XINHUI-SJI-A}, \eqref{IBW-GJI-A-1} and \eqref{IBW-GJI-A-2}, we have
\begin{align*}
{\mathbb{W}_\xi }\left( {t,x} \right) = & {a_W}{\left( {2r + 1} \right)^{ - 1}}\sum\limits_{\tilde \xi  \in {\mathfrak{H}_r}} {{\xi } \exp \left( {\mathrm{i}\left( {\sigma \tilde \xi  \cdot \left( {\tilde \xi ,{{\tilde \xi }^ \bot }} \right) + {\xi ^ \bot } } \right) \cdot x} \right) \exp \left( {\mathrm{i}\sigma \tilde \mu {{\tilde \xi }_1}t} \right) } \\
& +  {a_W}{\left( {2r + 1} \right)^{ - 1}}\sum\limits_{\tilde \xi  \in {\mathfrak{H}_r}} {{\xi  } \exp \left( {\mathrm{i}\left( {\sigma \tilde \xi  \cdot \left( {\tilde \xi ,{{\tilde \xi }^ \bot }} \right) - {\xi ^ \bot } } \right) \cdot x} \right)\exp \left( {\mathrm{i}\sigma \tilde \mu {{\tilde \xi }_1}t} \right) } .
\end{align*}
Denote
\begin{equation*}
{K^ + } \triangleq \exp \left( {\mathrm{i}\left( {\sigma \tilde \xi  \cdot \left( {\tilde \xi ,{{\tilde \xi }^ \bot }} \right) + {\xi ^ \bot } } \right) \cdot x} \right)  ,
\end{equation*}
and
\begin{equation*}
{K^ - } \triangleq \exp \left( {\mathrm{i}\left( {\sigma \tilde \xi  \cdot \left( {\tilde \xi ,{{\tilde \xi }^ \bot }} \right) - {\xi ^ \bot } } \right) \cdot x} \right)  .
\end{equation*}
Then, the building block ${\mathbb{W}_{\xi }}\left( {t,x} \right)$ is rewritten as
\begin{equation}\label{NSHYDGBVD-F-G-1}
{\mathbb{W}_\xi }\left( {t,x} \right) = {a_W}{\left( {2r + 1} \right)^{ - 1}} {\xi } \sum\limits_{\tilde \xi  \in {\mathfrak{H}_r}} {\left( {{K^ + } + {K^ - }} \right)\exp \left( {\mathrm{i}\sigma \tilde \mu  k _1 t} \right)} ,
\end{equation}
where $k_1$ is the second component of $\tilde \xi$. Fixing $\xi \in  \Lambda $, the resonance condition for ${\mathbb{W}_{\xi }}\left( {t,x} \right)$ corresponds to that of linearization for ${\mathfrak{D}_r}\left( x \right)$. In addition, the decomposition set $\Lambda$ is finite (see Lemma \ref{GLTNNRS-A}), and $1 \leqslant \left| {\xi  }  \right| \leqslant {C}\left| \Lambda  \right|$ for any $\xi \in  \Lambda $. Hence, there exists a bounded function $\mathfrak{f}\left( t \right)$ depending on the Fourier basis $\exp \left( {\mathrm{i} \sigma k_1 \tilde \mu t} \right)$ such that
\begin{align*}
\left\| {{\mathbb{W}_\xi }\left( t \right)} \right\|_{{L^p}}^p = & a_W^{ p}{\left( {2r + 1} \right)^{ - p}} {\left| {{\xi }} \right|^p}\int_{{\mathbb{T}^2}} {{{\left| {\sum\limits_{\tilde \xi  \in {\mathfrak{H}_r}} {\left( {{K^ + } + {K^ - }} \right)\exp \left( {\mathrm{i}\sigma \tilde \mu {{\tilde \xi }_1}t} \right)} } \right|}^p}} \;{\text{d}}x \\
\sim &\left| {\mathfrak{f}\left( t \right)} \right| ^p {\left\| {{\mathfrak{D}_r}} \right\|_{{L^p}}^p} ,
\end{align*}
Employing \eqref{WPD-prove-HUD-1}, we obtain
\begin{equation*}
{\left\| {\sum\limits_{\xi  \in \Lambda } {{\mathbb{W}_\xi }} } \right\|_{{L^p}}} \sim \left| {\mathfrak{f}\left( t \right)} \right| {r^{1 - \frac{2}{p}}}  ,
\end{equation*}
which implies that
\begin{equation*}
{\left\| {\sum\limits_{\xi  \in \Lambda } {{\nabla ^m}\partial _t^n{\mathbb{W}_\xi }} } \right\|_{{L^p}}} \sim {\sigma} ^m{{\tilde \mu }^n}{\left| {{\mathfrak{H}_r}\backslash \left\{ 0 \right\}} \right|^m}  {\left\| {\sum\limits_{\xi  \in \Lambda } {{\mathbb{W}_\xi }} } \right\|_{{L^p}}} \sim \left| {\mathfrak{f}\left( t \right)} \right|  {\sigma} ^{m + n}{{\tilde \mu }^n}{r^{1 + m - \frac{2}{p}}} .
\end{equation*}
\par
\textit{Case} $2$. If the exponent $p$ is in ${\mathbb{R} ^ +} \backslash {\mathbb{Z}^ + }$, there then exist two positive integers $p_1$ and $p_2$ such that $p_1 < p < p_2$. Employing Riesz--Thorin interpolation inequality, there exists a unique real number $\alpha  \in \left( {0,1} \right)$ with $\frac{1}{p} = \frac{{1 - \alpha }}{{{p_1}}} + \frac{\alpha }{{{p_2}}}$ such that 
\begin{align*}
{\left\| {{\nabla ^m}\partial _t^n{\mathfrak{D}_{\tilde \mu }}} \right\|_{{L^p}}} \lesssim & \left\| {{\nabla ^m}\partial _t^n{\mathfrak{D}_{\tilde \mu }}} \right\|_{{L^{{p_1}}}}^{1 - \alpha }\left\| {{\nabla ^m}\partial _t^n{\mathfrak{D}_{\tilde \mu }}} \right\|_{{L^{{p_2}}}}^\alpha \\
\lesssim & \left| {\mathfrak{f}\left( t \right)} \right| {\sigma ^{\left( {m + n} \right)\left( {1 - \alpha } \right)}}{{\tilde \mu }^{n\left( {1 - \alpha } \right)}}{r^{\left( {1 + m - \frac{2}{{{p_1}}}} \right)\left( {1 - \alpha } \right)}}{\sigma ^{\left( {m + n} \right)\alpha }}{{\tilde \mu }^{n\alpha }}{r^{\left( {1 + m - \frac{2}{{{p_2}}}} \right)\alpha }}  \\
\sim & \left| {\mathfrak{f}\left( t \right)} \right| {\sigma ^{m + n}}{{\tilde \mu }^n}{r^{1 + m - \frac{2}{p}}} .
\end{align*}
Furthermore, it follows from H\"{o}lder inequality that
\begin{equation*}
\left\| {{\nabla ^m}\partial _t^n{\mathfrak{D}_{\tilde \mu }}} \right\|_{{L^2}}^2 \lesssim {\left\| {{\nabla ^m}\partial _t^n{\mathfrak{D}_{\tilde \mu }}} \right\|_{{L^p}}}{\left\| {{\nabla ^m}\partial _t^n{\mathfrak{D}_{\tilde \mu }}} \right\|_{{L^{p'}}}} \lesssim {\sigma ^{m + n}}{{\tilde \mu }^n}{r^{1 + m - \frac{2}{{p'}}}}{\left\| {{\nabla ^m}\partial _t^n{\mathfrak{D}_{\tilde \mu }}} \right\|_{{L^p}}} ,
\end{equation*}
where $\frac{1}{p} + \frac{1}{{p'}} = 1$. By $\left\| {{\nabla ^m}\partial _t^n{\mathfrak{D}_{\tilde \mu }}} \right\|_{{L^2}}^2 \sim \left| {\mathfrak{f}\left( t \right)} \right|^2 {\sigma ^{2\left( {m + n} \right)}}{{\tilde \mu }^{2n}}{r^{2m}}$, we deduce that
\begin{equation*}
{\left\| {{\nabla ^m}\partial _t^n{\mathfrak{D}_{\tilde \mu }}} \right\|_{{L^p}}} \gtrsim \left| {\mathfrak{f}\left( t \right)} \right| {\sigma ^{m + n}}{{\tilde \mu }^n}{r^{m + \frac{2}{{p'}} - 1}} \sim  \left| {\mathfrak{f}\left( t \right)} \right| {\sigma ^{m + n}}\tilde \mu ^ n {r^{1 + m - \frac{2}{p}}}  .
\end{equation*}
Using the same argument, we obtain
\begin{equation*}
{\left\| {\sum\limits_{\xi  \in \Lambda } {{\nabla ^m}\partial _t^n{\mathbb{W}_\xi }} } \right\|_{{L^p}}} \sim \left| {\mathfrak{f}\left( t \right)} \right| {\sigma ^{m + n}}{{\tilde \mu }^n}{r^{1 + m - \frac{2}{p}}} .
\end{equation*}
The proof is complete. \qed
\par
\begin{lemma}\label{prop:transported_dirichlet_increment}
For any $p > 1$ and $l \in \left[ {\frac{{{C_0}}}{{\sigma r}},\frac{{{c_1}}}{\sigma }} \right]$ with ${C_0} \geqslant  {2\sqrt 2 } \pi $ and ${c_0} \in \left( {0,1} \right)$, there exists a bounded function $\mathfrak{f}\left( t \right)$ on $\left[ 0,T \right]$ such that
\begin{equation}\label{prop:transported_dirichlet_increment-A-1}
{\left\| {{\mathfrak{D}_{\tilde \mu }}\left( {t,x + ln} \right) - {\mathfrak{D}_{\tilde \mu }}\left( {t,x} \right)} \right\|_{{L^p}}} \sim \left| {\mathfrak{f}\left( t \right)} \right| {r^{1 - \frac{2}{p}}} ,
\end{equation}
and
\begin{equation}\label{prop:transported_dirichlet_increment-A-2}
{\left\| {\sum\limits_{\xi  \in \Lambda } {\left( {{\mathbb{W}_\xi }\left( {x + l{\mathbf{n}}} \right) - {\mathbb{W}_\xi }\left( x \right)} \right)} } \right\|_{{L^p}}} \sim \left| {\mathfrak{f}\left( t \right)} \right| {r^{1 - \frac{2}{p}}}  .
\end{equation}
\end{lemma}
\par
\noindent{\textbf{Proof}}. 
Denote $\tilde \xi  \triangleq \left( {{k_1},{k_2}} \right)$. Using \eqref{XINHUI-SJI-A} and \eqref{IBW-GJI-A-1}, we have
\begin{equation*}
{\mathfrak{D}_{\tilde \mu }}(t,x) = \frac{1}{{2r + 1}}\sum\limits_{\tilde \xi  \in {\mathfrak{H}_r}} {\exp \left( {{\text{i}}\sigma ({k_1}\tilde \xi  + {k_2}{{\tilde \xi }^ \bot }) \cdot x} \right)\exp \left( {{\text{i}}\sigma {k_1}\tilde \mu t} \right)}  .
\end{equation*}
Denote ${\eta _{\tilde \xi }} \triangleq \sigma ({k_1}\tilde \xi  + {k_2}{{\tilde \xi }^ \bot })$ and ${a_{\tilde \xi }}(l) \triangleq \exp \left( {{\text{i}}{\eta _k} \cdot l{\mathbf{n}}} \right) - 1$. The spatial increment satisfies
\begin{equation*}
{\mathfrak{D}_{\tilde \mu }}\left( {t,x + ln} \right) - {\mathfrak{D}_{\tilde \mu }}(t,x) = \frac{1}{{2r + 1}}\sum\limits_{\tilde \xi  \in {\mathfrak{H}_r}} {{a_{\tilde \xi }}(l)\exp \left( {{\text{i}}{\eta _k} \cdot x} \right)\exp \left( {{\text{i}}\sigma {k_1}\tilde \mu t} \right)}  .
\end{equation*}
By the Fourier orthogonality, there exists a bounded function $\mathfrak{f}\left( t \right)$ depending on the Fourier basis $\exp \left( {\mathrm{i} \sigma k_1 \tilde \mu t} \right)$ such that
\begin{align}\label{KSIFJKF-A-1}
\left\| {{\mathfrak{D}_{\tilde \mu }}\left( {t,x + l\mathbf{n}} \right) - {\mathfrak{D}_{\tilde \mu }}(t,x)} \right\|_{{L^2}}^2  \sim & \left| {\mathfrak{f}\left( t \right)} \right|{(2r + 1)^{ - 2}}\sum\limits_{\tilde \xi  \in {\mathfrak{H}_r}} {{{\left| {{a_{\tilde \xi }}\left( l \right)} \right|}^2}} \nonumber \\
\sim & 2 - 2 \cdot {(2r + 1)^{ - 2}}\operatorname{Re} \left( {\sum\limits_{{k_1} =  - r}^r {\sum\limits_{{k_2} =  - r}^r {\exp \left( {{\text{i}}{k_1}{h_1}} \right)\exp \left( {{\text{i}}{k_2}{h_1}} \right)} } } \right) ,
\end{align}
where ${h_1} = \sigma l\tilde \xi  \cdot \mathbf{n}$ and ${h_2} = \sigma l{{\tilde \xi }^ \bot } \cdot \mathbf{n}$.
\par
Using the fact that $\sum\limits_{k =  - r}^r {\exp \left( {{\text{i}}kh} \right)}  = \frac{{\sin ((r + \frac{1}{2})h)}}{{\sin (\frac{h}{2})}}$ and $\sin \left( {\frac{h}{2}} \right) \geqslant \frac{{\left| h \right|}}{\pi }$ for any $\left| h \right| \in \left( {0,\pi } \right)$, we obtain
\begin{equation*}
\left| {\sum\limits_{k =  - r}^r {{e^{ikh}}} } \right| = \left| {\frac{{\sin ((r + \frac{1}{2})h)}}{{\sin (\frac{h}{2})}}} \right| \leqslant \frac{\pi }{{\left| h \right|}} \leqslant \frac{{\sqrt 2 \pi }}{{\sigma l}}  ,
\end{equation*}
which implies that
\begin{equation}
{(2r + 1)^{ - 1}}\left| {\sum\limits_{k =  - r}^r {\exp \left( {{\text{i}}kh} \right)} } \right| \leqslant \sqrt 2 \pi {(2r + 1)^{ - 1}}{\sigma ^{ - 1}}{l^{ - 1}} \leqslant \frac{{\sqrt 2 \pi }}{{{C_0}}} \leqslant \frac{1}{2}  .  \label{KSIFJKF-A-134}
\end{equation}
Since $h_1^2 + h_2^2 = {\sigma ^2}{l^2}$, we have $\left| {{h_1}} \right| \geqslant \frac{{\sigma l}}{{\sqrt 2 }}$ or $\left| {{h_2}} \right| \geqslant \frac{{\sigma l}}{{\sqrt 2 }}$. Then, combining \eqref{KSIFJKF-A-1} and \eqref{KSIFJKF-A-134}, we derive that
\begin{equation}
{\left\| {{\mathfrak{D}_{\tilde \mu }}\left( {t,x + l\mathbf{n}} \right) - {\mathfrak{D}_{\tilde \mu }}(t,x)} \right\|_{{L^2}}} \gtrsim 1  .  \label{KSIFJKF-A-135}
\end{equation}
\par
For any $p > 1$, it follows from \eqref{WIEJDUHFJ-DFIJDJ-A-1} that
\begin{equation}\label{KSIFJKF-A-4}
{\left\| {{\mathfrak{D}_{\tilde \mu }}\left( {t,x + l\mathbf{n}} \right) - {\mathfrak{D}_{\tilde \mu }}(t,x)} \right\|_{{L^p}}} \lesssim {\left\| {{\mathfrak{D}_{\tilde \mu }}(t,x)} \right\|_{{L^p}}} \sim \left| {\mathfrak{f}\left( t \right)} \right| {r^{1 - \frac{2}{p}}}  ,
\end{equation}
and it follows from \eqref{KSIFJKF-A-135}, \eqref{KSIFJKF-A-4} and H\"{o}lder inequality that
\begin{align*}
{\left\| {{\mathfrak{D}_{\tilde \mu }}\left( {t,x + l\mathbf{n}} \right) - {\mathfrak{D}_{\tilde \mu }}(t,x)} \right\|_{{L^2}}} \leqslant & {\left\| {{\mathfrak{D}_{\tilde \mu }}\left( {t,x + l{\mathbf{n}}} \right) - {\mathfrak{D}_{\tilde \mu }}(t,x)} \right\|_{{L^p}}}{\left\| {{\mathfrak{D}_{\tilde \mu }}\left( {t,x + l{\mathbf{n}}} \right) - {\mathfrak{D}_{\tilde \mu }}(t,x)} \right\|_{{L^{\frac{{p - 1}}{p}}}}} \\
\lesssim &  {\left\| {{\mathfrak{D}_{\tilde \mu }}\left( {t,x + l{\mathbf{n}}} \right) - {\mathfrak{D}_{\tilde \mu }}(t,x)} \right\|_{{L^p}}}{r^{\frac{2}{p} - 1}} .
\end{align*}
Therefore,
\begin{equation*}
{\left\| {{\mathfrak{D}_{\tilde \mu }}\left( {t,x + l{\mathbf{n}}} \right) - {\mathfrak{D}_{\tilde \mu }}(t,x)} \right\|_{{L^p}}} \gtrsim {r^{1 - \frac{2}{p}}} .
\end{equation*}
Furthermore, in view of $\left| \xi  \right| \leqslant \left| \Lambda  \right| \leqslant C$, the upper and lower bounds of ${\left\| {\sum\limits_{\xi  \in \Lambda } {\left( {{\mathbb{W}_\xi }\left( {x + l{\mathbf{n}}} \right) - {\mathbb{W}_\xi }\left( x \right)} \right)} } \right\|_{{L^p}}}$ are controlled by ${\left\| {{\mathfrak{D}_{\tilde \mu }}\left( {t,x + ln} \right) - {\mathfrak{D}_{\tilde \mu }}(t,x)} \right\|_{{L^p}}}$. Therefore, it is easy to verify that \eqref{prop:transported_dirichlet_increment-A-2} holds. The proof is complete. \qed
\par
\begin{corollary}\label{COR-WIEJDUHFJ-HUD}
For any $1 < {p_2} < {p_1} $ and $(\sigma r)^{-1}\lesssim l \ll 1 $, the building block ${\mathbb{W}_\xi } \left( {t, x } \right)$ exhibits intermittency
\begin{equation}\label{COR-WIEJDUHFJ-A-1}
\frac{{{{\left\| {  {\sum\limits_{\xi  \in \Lambda } {{\mathbb{W}_\xi }} }  } \right\|}_{{L^{{p_1}}}}}}}{{{{\left\| { {\sum\limits_{\xi  \in \Lambda } {{\mathbb{W}_\xi }} } } \right\|}_{{L^{{p_2}}}}}}} \sim {r^{\frac{2}{{{p_2}}} - \frac{2}{{{p_1}}}}} , \  \ \forall \, t \in \left[ 0,T \right] ,
\end{equation}
and inertial-range intermittency
\begin{equation}\label{COR-WIEJDUHFJ-A-2}
\frac{{{{\left\| {\sum\limits_{\xi  \in \Lambda } {\left( {{\mathbb{W}_\xi }\left( {x + l\mathbf{n}} \right) - {\mathbb{W}_\xi }\left( x \right)} \right)}}  \right\|}_{{L^{{p_1}}}}}}}{{{{\left\| {\sum\limits_{\xi  \in \Lambda } {\left( {{\mathbb{W}_\xi }\left( {x + l\mathbf{n}} \right) - {\mathbb{W}_\xi }\left( x \right)} \right)}} \right\|}_{{L^{{p_2}}}}}}} \sim {r^{\frac{2}{{{p_2}}} - \frac{2}{{{p_1}}}}} , \ \ \forall \, t \in \left[ 0,T \right] .
\end{equation}
\end{corollary}
\par
In view of \eqref{WIEJDUHFJ-DFIJDJ-A-2} and \eqref{prop:transported_dirichlet_increment-A-2}, it is easy to verify that the intermittency \eqref{COR-WIEJDUHFJ-A-1} and the inertial-range intermittency \eqref{COR-WIEJDUHFJ-A-2} hold.
\par
\begin{lemma}\label{IBW-P-A-1}
The building block ${\mathbb{W}_{\xi }}\left( {t,x} \right)$ is real-valued, and admits
\begin{equation*}
\int_{{\mathbb{T}^2}} {{\mathbb{W}_\xi }\left( x \right) \otimes {\mathbb{W}_\xi }\left( x \right)}\ \mathrm{d}x = \mathscr{M} \left( {t,r} \right) {\xi } \otimes {\xi } ,
\end{equation*}
where the function $\mathscr{M} \left( {t,r} \right)$ satisfies the uniform bound ${C_2} \leqslant \left| {\mathscr{M}\left( {t,r} \right)} \right| \leqslant {C_1}$. Furthermore, if $\xi  = \tilde \xi $, then
\begin{equation}\label{IBW-P-A-FGHY}
\mathrm{div}\left( {{\mathbb{W}_{\xi }} \otimes {\mathbb{W}_{\xi }}} \right) = \left( {{\mathbb{W}_\xi } \cdot \nabla } \right){\mathbb{W}_\xi } + \left( {{\mathbb{W}_{\xi ^ \bot } } \cdot \nabla } \right){\mathbb{W}_{\xi ^ \bot } } - 4 a_W^2{{\tilde \mu }^{ - 1}} \xi {\cos ^2}\left( {{\xi ^ \bot }  \cdot x} \right){\partial _t}\left( {\mathfrak{D}_{\tilde \mu }^2} \right) .
\end{equation}
\end{lemma}
\par
\noindent{\textbf{Proof}}. 
We directly calculate that
\begin{align*}
\int_{{\mathbb{T}^2}} {{\mathbb{W}_\xi }\left( x \right) \otimes {\mathbb{W}_\xi }\left( x \right)}\ \mathrm{d}x = & \int_{{\mathbb{T}^2}} {\mathfrak{D}_{\tilde \mu }^2\left( {{W_\xi }\left( x \right) \otimes {W_\xi }\left( x \right)} \right)}\ \mathrm{d}x \\
= & {\xi } \otimes {\xi }\int_{{\mathbb{T}^2}} {4a_W^2{{\cos }^2}\left( {{\xi ^ \bot }  \cdot x} \right)\mathfrak{D}_{\tilde \mu }^2}\ \mathrm{d}x .
\end{align*}
Denote $\mathscr{M} \left( {t,r} \right) =  \int_{{\mathbb{T}^2}} {4a_W^2{{\cos }^2}\left( {{\xi ^ \bot }  \cdot x} \right)\mathfrak{D}_{\tilde \mu }^2}\ \mathrm{d}x$. It follows from \eqref{WIEJDUHFJ-DFIJDJ-A-1} that
\begin{equation*}
\left| {\mathscr{M}\left( {t,r} \right)} \right| \leqslant C\left\| { \mathfrak{D}_{\tilde \mu } } \right\|_{{L^2}}^2 \leqslant C  \mathop {\sup }\limits_{t \in \left[ {0,T} \right]} \left| {{\mathfrak{f}}\left( t \right)} \right| .
\end{equation*}
Since $\xi \in  \Lambda   \subset {\mathfrak{H}_r}$, the spatial‑frequency vector of ${\cos } \left( {{\xi ^ \bot }   \cdot x} \right)$ must correspond to one of the Fourier modes of $\mathfrak{D}_{\tilde \mu }$. Thus, 
\begin{equation*}
\left| {\mathscr{M}\left( {t,r} \right)} \right| \geqslant 4a_W^2 \mathop {\inf }\limits_{t \in \left[ {0,T} \right]} \left| {{\mathfrak{f}}\left( t \right)} \right| \int_{{\mathbb{T}^2}} {{{\cos }^2}\left( {{\xi ^ \bot }   \cdot x} \right){{\cos }^2}\left( {a{\xi ^ \bot }   \cdot x} \right)} \;\mathrm{d}x \geqslant C \mathop {\inf }\limits_{t \in \left[ {0,T} \right]} \left| {{\mathfrak{f}}\left( t \right)} \right| , \ \ \mbox{with}\ a \in \mathbb{Z} .
\end{equation*}
\par
We note that
\begin{equation*}
\frac{{2r + 1}}{{2\mathrm{i} \sigma {\mathfrak{D}_{\tilde \mu }}}}\left( {\left( {{\xi ^ \bot } \cdot \nabla } \right)\mathfrak{D}_{\tilde \mu }^2} \right){\xi ^ \bot } \\
=  \sum\limits_{\tilde \xi  \in {\mathfrak{H}_r}} {{\tilde \xi  \cdot \left( {\tilde \xi  \cdot {\xi ^ \bot },{{\tilde \xi }^ \bot } \cdot {\xi ^ \bot }} \right){\xi ^ \bot }} \exp \left( {\mathrm{i} \sigma \tilde \xi  \cdot \left( {\left( {\tilde \xi  \cdot x + \tilde \mu t} \right),{{\tilde \xi }^ \bot } \cdot x} \right)} \right)}  ,
\end{equation*}
\begin{equation*}
\frac{{2r + 1}}{{2\mathrm{i} \sigma {\mathfrak{D}_{\tilde \mu }}}}\left( {\left( {\xi  \cdot \nabla } \right)\mathfrak{D}_{\tilde \mu }^2} \right)\xi =  \sum\limits_{\tilde \xi  \in {\mathfrak{H}_r}} {{\tilde \xi  \cdot \left( {\tilde \xi  \cdot \xi ,{{\tilde \xi }^ \bot } \cdot \xi } \right)\xi } \exp \left( {\mathrm{i} \sigma \tilde \xi  \cdot \left( {\left( {\tilde \xi  \cdot x + \tilde \mu t} \right),{{\tilde \xi }^ \bot } \cdot x} \right)} \right)} ,
\end{equation*}
and
\begin{equation*}
\frac{{2r + 1}}{{2\mathrm{i} \sigma {\mathfrak{D}_{\tilde \mu }}}}\nabla \left( {\mathfrak{D}_{\tilde \mu }^2} \right) =  \sum\limits_{\tilde \xi  \in {\mathfrak{H}_r}} {\tilde \xi \cdot \left( {\tilde \xi ,{{\tilde \xi }^ \bot }} \right)\exp \left( {\mathrm{i} \sigma \tilde \xi  \cdot \left( {\left( {\tilde \xi  \cdot x + \tilde \mu t} \right),{{\tilde \xi }^ \bot } \cdot x} \right)} \right)} .
\end{equation*}
Then, taking $\xi ^ \bot = \tilde \xi $, we have
\begin{equation}
\nabla \left( {\mathfrak{D}_{\tilde \mu }^2} \right) = \left( {\left( {{\xi ^ \bot } \cdot \nabla } \right)\mathfrak{D}_{\tilde \mu }^2} \right){\xi ^ \bot } + \left( {\left( {\xi  \cdot \nabla } \right)\mathfrak{D}_{\tilde \mu }^2} \right)\xi . \label{JSIUHDJHF-D}
\end{equation}
Using the fact that
\begin{align*}
\mathrm{div}\left( {{\mathbb{W}_\xi } \otimes {\mathbb{W}_\xi }} \right) = & \left( {{W_\xi } \cdot \nabla } \right)\left( {\mathfrak{D}_{\tilde \mu }^2{W_\xi }} \right) + \mathfrak{D}_{\tilde \mu }^2{W_\xi }\mathrm{div} \, {W_\xi } \\
= & \left( {\left( {{W_\xi } \cdot \nabla } \right)\mathfrak{D}_{\tilde \mu }^2} \right){W_\xi } + \mathfrak{D}_{\tilde \mu }^2\left( {{W_\xi } \cdot \nabla } \right){W_\xi } \\
= & 4a_W^2{\cos ^2}\left( {{\xi ^ \bot }  \cdot x} \right)\left( {\left( {{\xi } \cdot \nabla } \right)\mathfrak{D}_{\tilde \mu }^2} \right){\xi  } ,
\end{align*}
it follows from \eqref{IBW-GJI-A-4} and \eqref{JSIUHDJHF-D} that
\begin{align*}
\mathrm{div}\left( {{\mathbb{W}_\xi } \otimes {\mathbb{W}_\xi }} \right) = & 4a_W^2{\cos ^2}\left( {{\xi ^ \bot }  \cdot x} \right)\left( {\nabla \left( {\mathfrak{D}_{\tilde \mu }^2} \right) - \left( {\left( {{\xi ^ \bot }  \cdot \nabla } \right)\mathfrak{D}_{\tilde \mu }^2} \right){\xi ^ \bot } } \right) \\
= & \left( {{\mathbb{W}_\xi } \cdot \nabla } \right){\mathbb{W}_\xi } + \left( {{\mathbb{W}_{\xi ^ \bot } } \cdot \nabla } \right){\mathbb{W}_{\xi ^ \bot } } - 4 a_W^2 {{\tilde \mu }^{ - 1}} \xi  {\cos ^2}\left( {{\xi ^ \bot }  \cdot x} \right){\partial _t}\left( {\mathfrak{D}_{\tilde \mu }^2} \right) .
\end{align*}
The proof is complete. \qed

\end{document}